\documentclass{amsart}

\usepackage{amsmath, accents, amsfonts, amssymb, amsthm, amscd, graphicx, float, epstopdf, tabu
}
\usepackage[pdfpagemode={UseOutlines},bookmarks=true,bookmarksopen=true, bookmarksopenlevel=0,bookmarksnumbered=true,hypertexnames=false, colorlinks,linkcolor={blue},citecolor={blue},urlcolor={blue}, pdfstartview={FitV},unicode,breaklinks=true,backref=page]{hyperref}

\usepackage{eulervm}

\usepackage[T1]{fontenc}  
\usepackage{amsmath}      
\usepackage{tikz}         

\makeatletter
\newlength\xvec@height%
\newlength\xvec@depth%
\newlength\xvec@width%
\newcommand{\xvec}[1]{%
  \ifmmode%
    \settoheight{\xvec@height}{$#1$}%
    \settodepth{\xvec@depth}{$#1$}%
    \settowidth{\xvec@width}{$#1$}%
  \else%
    \settoheight{\xvec@height}{#1}%
    \settodepth{\xvec@depth}{#1}%
    \settowidth{\xvec@width}{#1}%
  \fi%
  \def\xvec@arg{#1}%
  \raisebox{.2ex}{\raisebox{\xvec@height}{\rlap{%
    \kern.05em
    \begin{tikzpicture}[scale=1]
    \pgfsetroundcap
    \draw (.05em,0)--(\xvec@width-.05em,0);
    \draw (\xvec@width-.05em,0)--(\xvec@width-.15em, .095em);
    \draw (\xvec@width-.05em,0)--(\xvec@width-.15em,-.095em);
    \end{tikzpicture}%
  }}}%
  #1%
}
\makeatletter

\makeatletter
\newlength\subvec@height%
\newlength\subvec@depth%
\newlength\subvec@width%
\newcommand{\subvec}[1]{%
  \ifmmode%
    \settoheight{\subvec@height}{$_{#1}$}%
    \settodepth{\subvec@depth}{$_{#1}$}%
    \settowidth{\subvec@width}{$_{#1}$}%
  \def\xvec@arg{#1}%
  \raisebox{.4ex}{\raisebox{\subvec@height}{\rlap{%
    \kern.05em
    \begin{tikzpicture}[scale=1]
    \pgfsetroundcap
    \draw (.05em,0)--(\subvec@width-.05em,0);
    \draw (\subvec@width-.05em,0)--(\subvec@width-.15em, .095em);
    \draw (\subvec@width-.05em,0)--(\subvec@width-.15em,-.095em);
    \end{tikzpicture}%
  }}}%
  #1%
}
\makeatother

\renewcommand{\vec}[1]{\xvec{#1}}

\DeclareMathSymbol{\widetildesym}{\mathord}{largesymbols}{"65}

\newtheorem{thm}{Theorem}[section]
\numberwithin{thm}{section}
\newtheorem*{thm*}{Theorem}
\newtheorem{rmk}[thm]{Remark}
\newtheorem{notation}[thm]{Notation}
\newtheorem{fact}[thm]{Fact}
\newtheorem{conj}[thm]{Conjecture}
\newtheorem{defn}[thm]{Definition}
\newtheorem{lem}[thm]{Lemma}
\newtheorem{prop}[thm]{Proposition}
\newtheorem{cor}[thm]{Corollary}
\newtheorem{exa}[thm]{Example}

\begin{document}

\title[McShane Identities for higher Teichm\"uller theory and the GS potential]{McShane identities for higher Teichm\"uller theory and the Goncharov--Shen potential}


\author{Yi Huang}
\address{Yau Mathematical Sciences Center, Tsinghua University}
\email{yihuangmath@mail.tsinghua.edu.cn}
\author{Zhe Sun}
\address{Mathematics Research Unit, University of Luxembourg}
\email{sunzhe1985@gmail.com}

\date{\today}

\begin{abstract} 
We derive generalizations of McShane's identity for higher ranked surface group representations by studying a family of mapping class group invariant functions introduced by Goncharov and Shen which generalize the notion of horocycle lengths. In particular, we obtain McShane-type identities for finite-area cusped convex real projective surfaces by generalizing the Birman--Series geodesic scarcity theorem. More generally, we establish McShane-type identities for positive surface group representations with loxodromic boundary monodromy, as well as McShane-type inequalities for general rank positive representations with unipotent boundary monodromy. Our identities are systematically expressed in terms of projective invariants, and we study these invariants: we establish boundedness and Fuchsian rigidity results for triple and cross ratios. We apply our identities to derive the simple spectral discreteness of unipotent-bordered positive representations, collar lemmas, and generalizations of the Thurston metric.
\end{abstract}

\maketitle
\parindent = 0cm

\section{Introduction}
\label{intro}
The aim of this paper is to generalize McShane identities for higher Teichm\"uller theory, a goal previously considered by Labourie and McShane in \cite{LM09}. The McShane identities we obtain are expressed in terms of geometric quantities such as simple root lengths, triple ratios and edge functions, and naturally generalize those employed by Mirzakhani in her computation of Weil--Petersson volumes of moduli spaces of bordered hyperbolic surfaces of fixed boundary lengths \cite{mirz_simp} and her proof \cite{mirz_witten} of the Witten--Kontsevich theorem. We establish geometric applications for our identities, yielding properties of simple root lengths and triple ratios along the way.\medskip

Let $S=S_{g,m}$ denote a  genus $g$ oriented surface with $m\geq 1$ boundary components and negative Euler characteristic. In the classical hyperbolic setting, horocycle lengths define regular functions on Penner's decorated Teichm\"uller space of horocycle-decorated hyperbolic metrics on $S_{g,m}$ \cite{pennercoords}, and the decomposition of horocycle lengths leads to the classical McShane identities \cite{mcshane_allcusps}. The natural analog of this picture in higher Teichm\"uller theory is that of Goncharov and Shen's family of mapping class group invariant regular functions \cite{GS15} on the Fock--Goncharov $\mathcal{A}$ moduli space $\mathcal{A}_{\operatorname{SL}_n, S}$ \cite{FG06}. The \emph{Goncharov--Shen potential}  (Definition~\ref{defn:FGA}) is the starting point for our family of McShane identities for positive surface group representations into $\operatorname{PGL}_n(\mathbb{R})$.

\subsection{The classical McShane identity}
In his doctoral dissertation, McShane \cite{mcshane_thesis} established the following stunning result:

\begin{thm*}[{McShane identity \cite{mcshane_thesis}}]
Given an \emph{arbitrary} $1$-cusped hyperbolic torus $\Sigma_{1,1}$, let $\mathcal{C}_{1,1}$ denote the collection of \emph{unoriented} simple closed geodesics $\bar{\gamma}$ on $\Sigma_{1,1}$ up to homotopy and let $\ell(\bar{\gamma})$ denote their respective hyperbolic lengths. 
\begin{align}
\label{equation:cmc}
\sum_{ \bar{\gamma}\in \mathcal{C}_{1,1}}
\frac{2}{1+e^{\ell(\bar{\gamma})}}
= 1.
\end{align}
\end{thm*}

The above theorem has led to an ever-growing list of identities for rich families of hyperbolic geometric objects including the following direct generalizations of McShane's identity \cite{akiyoshi2004refinement,MR2258748, bowditch1997variation, markofftriples, huangclosed, huang2018mcshane,  lee2013variation, mcshane_allcusps, mirz_simp, MR2399656, tan_zhang_cone, tan2008mcshane}, as well as the closely related Basmajian identity \cite{basmajian1993orthogonal}, the Bridgeman-Kahn identity \cite{bridgeman2010hyperbolic,orthospectra} and the Luo-Tan dilogarithm identity \cite{luo_tan}. There has also been progress in establishing similar identities for higher Teichm\"uller theory \cite{FP16, He19,LM09,VY17} and super Teichm\"uller theory \cite{HPZ19}.

\subsection{McShane identity for bordered hyperbolic surfaces}

Let us highlight the McShane identity for bordered hyperbolic surfaces due independently to Mirzakhani \cite{mirz_simp} and Tan--Wong--Zhang \cite{tan_zhang_cone}. For simplicity, we state this identity only for genus $g$ hyperbolic surfaces $\Sigma_{g,1}$ with a single geodesic border $\bar{\alpha}$: 

\begin{thm*}[McShane identity for bordered hyperbolic surfaces {\cite{mirz_simp}, \cite{tan_zhang_cone}}]
Let $\mathcal{P}_{\bar{\alpha}}$ denote the collection of homotopy classes of embedded pairs of pants which contain $\bar{\alpha}$ as an unoriented boundary component. Then,

\begin{align}
\sum_{(\bar{\beta},\bar{\gamma})\in \mathcal{P}_{\bar{\alpha}}}
2\log \left( \frac{e^{\frac{\ell(\bar{\alpha})}{2}} + e^{\frac{\ell(\bar{\beta})+\ell(\bar{\gamma})}{2}}}{e^{\frac{-\ell(\bar{\alpha})}{2}} + e^{\frac{\ell(\bar{\beta})+\ell(\bar{\gamma})}{2}}} \right) = \ell(\bar{\alpha})\label{equation:mirD}.
\end{align}
\end{thm*}
Equation~\eqref{equation:mirD} is the basis for arguably the most celebrated application of McShane identities: Mirzakhani's integration scheme and volume recursion formula for the Weil-Petersson volume of moduli spaces of hyperbolic surfaces of fixed boundary length \cite{mirz_simp}.

\subsection{Labourie--McShane's identity for Hitchin representations}

The \emph{Hitchin component} $\mathrm{Hit}_n(S_{g,0})$ \cite{Hit92} is a contractible component of the representation variety
\[
\mathrm{Hom}(\pi_1(S_{g,0}),\operatorname{PGL}_n(\mathbb{R}))/\operatorname{PGL}_n(\mathbb{R}), 
\]
characterized as the deformation space of the $n$-Fuchsian representations of $\pi_1(S_{g,0})$ --- compositions of any Fuchsian representation with an irreducible representation from $\operatorname{PSL}_2(\mathbb{R})$ to $\operatorname{PGL}_n(\mathbb{R})$. Representations in the Hitchin component are referred to as \emph{Hitchin representations} and are the central object in higher Teichm\"uller theory\footnote{We highly recommend Wienhard's beautiful overview \cite{W18} of higher Teichm\"uller theory.}.\medskip

In \cite{LM09}, Labourie and McShane generalized the notion of a Hitchin component $\mathrm{Hit}_n(S_{g,m})$ for bordered surfaces, and established a very general family of McShane-type identities for these Hitchin representations of bordered surfaces via \emph{ordered cross ratios} \cite{Lab07}. In the $S_{g,1}$ setting, their identity takes the following form:\medskip

\begin{thm*}[{Labourie--McShane identity \cite{LM09}}]
Consider a Hitchin representation $\rho:\pi_1(S_{g,1})\to\operatorname{PGL}_n(\mathbb{R})$ and let $\alpha$ denote the boundary component of $S_{g,1}$ oriented so that $S_{g,1}$ is on the left of $\alpha$. Given any ordered cross ratio $\mathbb{B}^\rho$ defined with respect to $\rho$,
\begin{align*}
\sum_{(\beta,\gamma)\in\subvec{\mathcal{P}}_\alpha}
 \log \mathbb{B}^{\rho}(\alpha^-,\alpha^+,\gamma^+,\beta^+)=\ell^{\mathbb{B}^\rho}(\alpha),\quad\text{where}
\end{align*}
\begin{itemize}
\item
Labourie--McShane define $\xvec{\mathcal{P}}_\alpha$ as the collection of homotopy classes of \emph{embeddings} of a fixed pair of pants into $S_{g,1}$ whose image is marked by simple homotopy classes $\alpha,\beta,\gamma$ satisfying $\alpha\beta^{-1}\gamma=1$, for $\alpha$ a homotopy representative of the oriented boundary. We interpret $\xvec{\mathcal{P}}_\alpha$ as the set of boundary-parallel pairs of pants on $S_{g,1}$ which contain $\alpha$ (Definition~\ref{definition:Pp}). 
\item
The summands are logarithms of the $\mathbb{B}^{\rho}$ ordered cross ratio of quadruples of ideal points arising as attracting and/or repelling fixed points of $\alpha,\beta$ and $\gamma$.  
\item
The quantity $\ell^{\mathbb{B}^\rho}(\alpha)$ is a length-type quantity defined via cross ratios.
\end{itemize}
\end{thm*}

It is perhaps more accurate to view Labourie and McShane's formula as \emph{very} powerful machinery for producing McShane-type identities. The summands $\log \mathbb{B}^{\rho}(\alpha^-,\alpha^+,\gamma^+,\beta^+)$ --- often referred to as \emph{gap functions} --- are generally complex expressions of standard moduli of $\rho$ restricted to the underlying pair of pants. As an example, some summands for the rank $n$ weak cross ratio \cite[Section 10]{LM09} for $n=3$ Hitchin representations require fifteen lines to state \cite[Equation~(55)]{LM09}, let alone for $n\geq4$ and above.\medskip

We offer up identities which structurally resemble Equations~\eqref{equation:cmc} and \eqref{equation:mirD}.

\subsection{Positive representations}

Two of the main directions in which Hitchin representations have been generalized are \emph{positive representations} (Definition \ref{definition:posrep}) and \emph{Anosov representations}. The former, due to Fock and Goncharov \cite{FG06}, is based upon an algebraic property called positivity, whereas the latter hails from Labourie's \cite{Lab06} dynamical approach to higher Teichm\"uller theory via Anosov flows.\medskip

Positive and Anosov representations share key traits which make their theoretical development interesting and tractable. For example, both approaches yield discrete and faithful representations. As another example, consider the loxodromic property:

\begin{defn}[Loxodromic matrices]
An element in $\operatorname{PGL}_n(\mathbb{R})$ is \emph{loxodromic} if and only if it has a lift into $\operatorname{SL}_n(\mathbb{R})$ such that it is conjugate to a diagonal matrix with eigenvalues 
\[
\lambda_1>\cdots>\lambda_n>0.
\]
\end{defn}

Fock and Goncharov show that:
\begin{thm}[{\cite[Theorem 9.3]{FG06}}]
\label{theorem:loxo}
Given any positive representation $\rho:\pi_1(S_{g,m})\to\operatorname{PGL}_n(\mathbb{R})$, for any non-trivial $\gamma \in \pi_1(S_{g,m})$, 
\begin{itemize}
\item
if $\gamma$ is non-peripheral then $\rho(\gamma)$ is conjugate to a totally positive matrix and thus loxodromic;
\item 
if  $\gamma$ is peripheral then $\rho(\gamma)$ is conjugate to a totally positive upper triangular matrix.
\end{itemize}
\end{thm}
Labourie \cite{Lab06} also shows that non-trivial non-peripheral homotopy classes for Anosov representations are loxodromic.\medskip

One powerful geometric consequence of this loxodromic property is that it enables us to define the notion of \emph{$i$-th lengths} for curves on $S_{g,m}$.
\begin{defn}[$i$-th length]
\label{definition:ilength}
Given a positive representation $\rho:\pi_1(S_{g,m})\to\operatorname{PGL}_n(\mathbb{R})$, for any non-trivial $\gamma\in \pi_1(S_{g,m})$, we denote the eigenvalues of $\rho(\gamma)$ by 
\[
\lambda_1(\rho(\gamma))\geq \cdots \geq \lambda_n(\rho(\gamma))>0.
\]
For $i=1,\ldots,n-1$ we define the \emph{$i$-th length} (also called \emph{simple root length}) of $\gamma$ with respect to $\rho$ as 
\[
\ell_i(\gamma):=\log\left(\frac{\lambda_i(\rho(\gamma))}{\lambda_{i+1}(\rho(\gamma))}\right).
\]
\end{defn}
Note that whilst loxodromic elements always produce positive $i$-th lengths, it is possible for peripheral $\gamma$ to admit $i$-th length $\ell_i(\gamma)=0$ (e.g.: when the boundary is unipotent).\medskip

We focus on positive representations, and denote the \emph{$\operatorname{PGL}_n(\mathbb{R})$-positive representation variety} by $\mathrm{Pos}_n(S)$  (Definition~\ref{defn:prepvar}). For closed surfaces $S_{g,0}$, the positive representation variety $\mathrm{Pos}_n(S_{g,0})$ \emph{is} the Hitchin component $\mathrm{Hit}_n(S_{g,0})$ \cite[Theorem 1.15]{FG06}. More generally, for bordered surfaces $S_{g,m}$ it includes Hitchin representations (Remark~\ref{remark:hitpos}), and hence the $n$-Fuchsian representations.

\subsection{McShane identities for convex real projective $1$-cusped tori}

The theory of strictly convex $\mathbb{RP}^2$ surfaces, which generalizes the Beltrami-Klein approach to hyperbolic surfaces, is an important geometric manifestation of non-Fuchsian higher Teichm\"uller theory. To clarify: positive representations $\rho:\pi_1(S)\to\operatorname{PGL}_3(\mathbb{R})$ of closed surfaces and surfaces with unipotent boundary monodromy are holonomy representations of strictly convex $\mathbb{RP}^2$ surfaces $\Sigma$. Namely, $\Sigma$ may be expressed as $\Omega/\rho(\pi_1(S))$ where $\Omega$ is a strictly convex domain in $\mathbb{RP}^2$ preserved by the properly discontinuously action of $\rho(\pi_1(S))$ \cite{G1990convex,choi1993convex,Mar10}. \medskip

Ideal triangles are fundamental building pieces for hyperbolic and convex real projective surfaces. It is well-known that all hyperbolic ideal triangles are isometric. In contrast, oriented convex real projective ideal triangles are geometrically richer and are classified by their \emph{triple ratios} $T\in\mathbb{R}_{>0}$ \cite{FG06}. 
\begin{figure}[h!]
\includegraphics[scale=1.2]{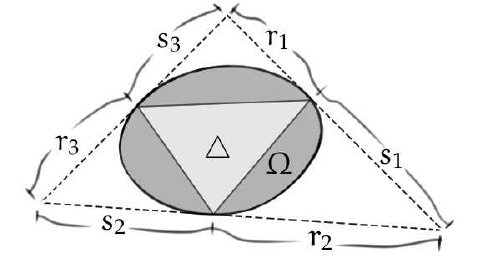}
\caption{The triple ratio of $\Delta$.}
\label{fig:triprat}
\end{figure}
In Figure~\ref{fig:triprat}, the triple ratio of the anticlockwise-oriented ideal triangle $\triangle$ inside the convex domain $\Omega\subset\mathbb{RP}^2$ is defined as $T=\frac{s_1s_2s_3}{r_1r_2r_3}$, where the $s_j$ and $r_k$ denote Euclidean segment lengths which are possible to be infinite. We denote the logarithm of the triple ratio by 
$\tau(\Delta):=\log(T(\Delta))\in\mathbb{R}$,
 and refer to this quantity as the \emph{triangle invariant} \cite{bonahon2014parameterizing,bonahon2017hitchin}.\medskip

We establish McShane identities for all (finite-type) cusped strictly convex $\mathbb{RP}^2$ surfaces (Theorems \ref{theorem:equsl3}, \ref{thm:equsl3pp}). For $1$-cusped tori, our result takes the form:

\begin{thm}[McShane identity for convex real projective $1$-cusped tori, Theorem~\ref{theorem:inequsl3s11} and Proposition~\ref{prop:dualidentity}]
\label{thm:firstn3S11}
Given a strictly convex $\mathbb{RP}^2$ $1$-cusped torus $\Sigma$, let ${\xvec{\mathcal{C}}_{1,1}}$ be the set of oriented simple closed geodesics on $\Sigma$ up to homotopy. Then 
\begin{align}
\label{equation:mc3c}
\sum_{\gamma\in{\subvec{\mathcal{C}}_{1,1}}}
\frac{1}{1+e^{\ell_1(\gamma)+\tau(\gamma)}}
= 
\sum_{\gamma\in{\subvec{\mathcal{C}}_{1,1}}}
\frac{1}{1+e^{\ell_2(\gamma)-\tau(\gamma)}}
=
1,
\end{align}
where $\tau(\gamma)$ is the triangle invariant for either of the two oriented ideal triangles on $\Sigma$ with one side being the unique ideal geodesic disjoint from $\gamma$ and the other two sides spiraling parallel to $\gamma$ (see Figure~\ref{fig:torussummand}).
\end{thm}

\begin{figure}[h!]
\includegraphics[scale=1]{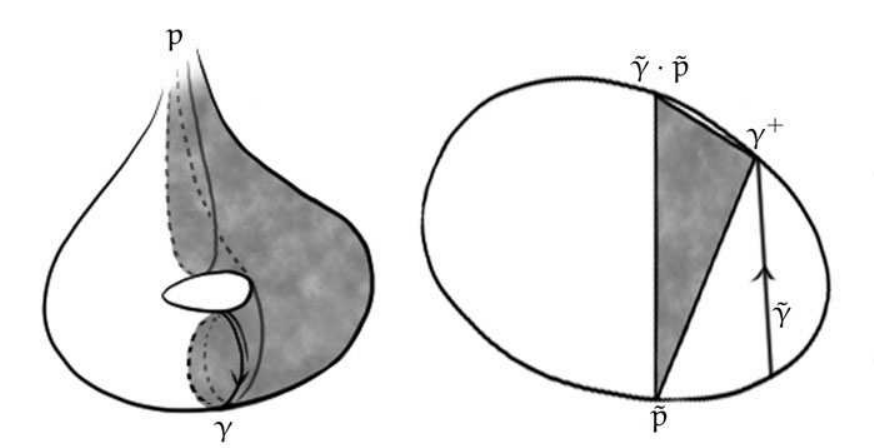}
\caption{Cutting the shaded pair of half-pants on the left figure along the spiraling geodesic depicted produces an ideal triangle $\triangle_\gamma$, and we use it to define $\tau(\gamma)=\tau(\triangle_\gamma)$. The right figure depicts a single lift $(\tilde{p},\gamma\cdot\tilde{p},\gamma^+)$ of $\triangle_\gamma$ to the universal cover of $\Sigma$.}
\label{fig:torussummand}
\end{figure}

We make two remarks before moving on to more general identities:
\begin{itemize}
\item
For $3$-Fuchsian representations, Equation~\eqref{equation:mc3c} recovers the classical McShane identity (Remark~\ref{remark:c11}).
\item
There are in fact two possible ideal triangles $\triangle_{\gamma}$ and $\triangle_{\gamma}'$ with one side being the unique ideal geodesic disjoint from $\gamma$ and the other two sides spiraling parallel to $\gamma$, and provided that one marks them to agree with the lift $(\tilde{p},\gamma\cdot\tilde{p},\gamma^+)$ (Figure~\ref{fig:torussummand}), their triple ratios agree and $\tau(\gamma)$ is well-defined (Remark~\ref{remark:sl3s11td}). Moreover, we show in Lemma~\ref{lem:12termcompare} that $\tau(\gamma^{-1})=-\tau(\gamma)$ and $\ell_1(\gamma^{-1})=\ell_2(\gamma)$, which leads to the ``two'' McShane identities above (see Proposition~\ref{prop:dualidentity}). We study these \emph{a priori} unexpected symmetries for $S_{1,1}$, $n=3$ positive representations in \S\ref{sec:surprisesymmetry}.
\end{itemize}

\subsection{McShane identities for general convex real projective surfaces}

We begin by introducing the requisite summation index.

\begin{defn}[Boundary-parallel pairs of pants]
\label{definition:Pp}
Assume that $S_{g,m}$ is endowed with an auxiliary hyperbolic metric and let $p$ be a distinguished cusp of $S_{g,m}$. An (embedded) \emph{boundary-parallel pairs of pants} containing $p$ is a pair $(\beta,\gamma)$ of (disjoint) oriented closed geodesics so that $p,\beta,\gamma$ bound a pair of pants on $S_{g,m}$, and $\beta,\gamma$ are positioned and oriented as per Figure~\ref{fig:pidealquad}). We denote the collection of all boundary-parallel pairs of pants on $S$ containing $p$ up to homotopy by $\xvec{\mathcal{P}}_p$. We similarly define $\xvec{\mathcal{P}}_\alpha$ for bordered hyperbolic surfaces by supplanting the role of the cusp $p$ by a distinguished oriented boundary geodesic $\alpha$.
\end{defn}

\begin{figure}[h!]
\includegraphics[scale=0.25]{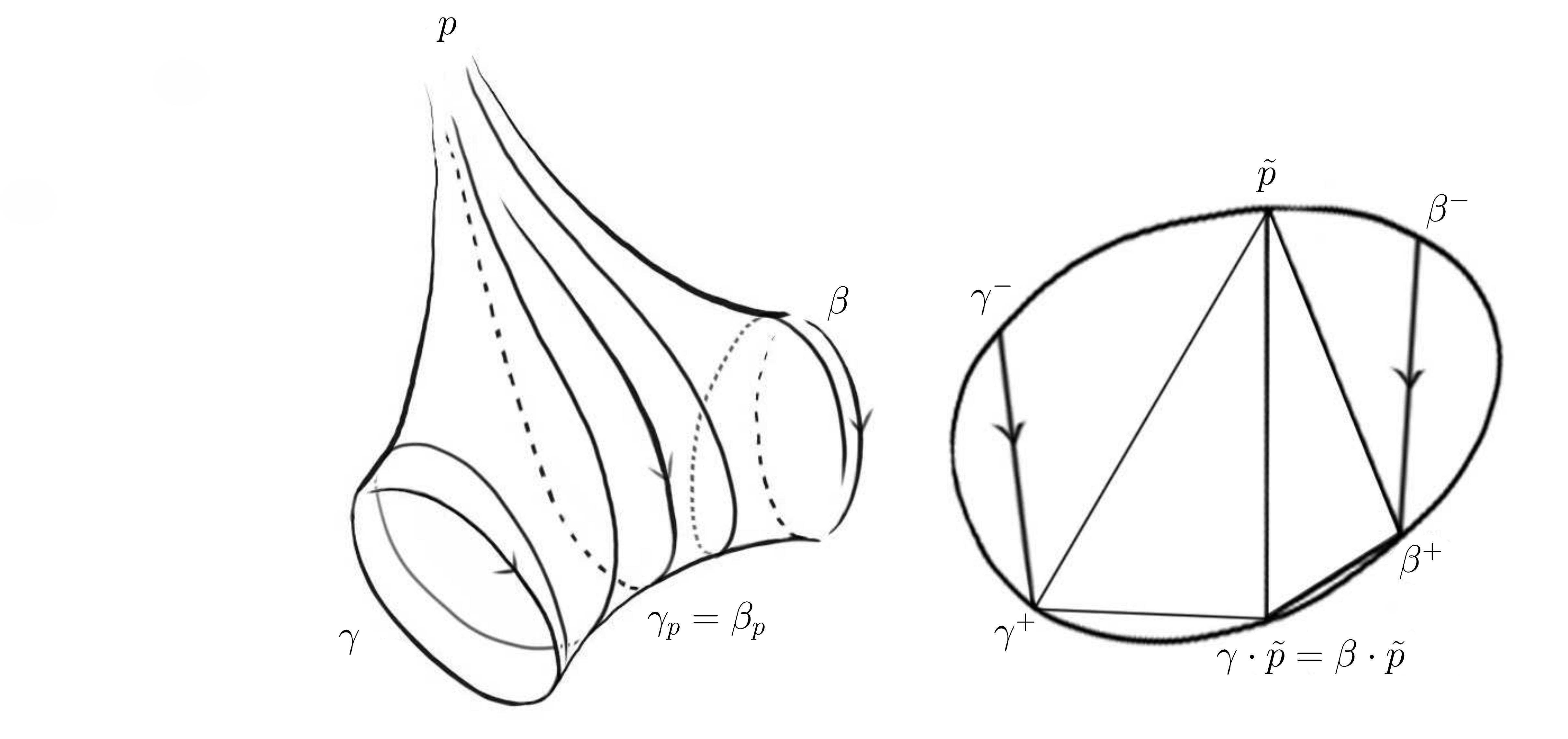}
\caption{Cutting along the spiraling geodesics on the boundary-parallel pair of pants $(\beta,\gamma)$ (left figure) results in an ideal quadrilateral whose lift is the marked quadrilateral $(\tilde{p},\beta^+,\beta\cdot\tilde{p}=\gamma\cdot\tilde{p},\gamma^+)$ (right figure).}
\label{fig:pidealquad}
\end{figure}

Fock and Goncharov \cite[Section 9]{FG06} parameterize $\mathrm{Pos}_n(S_{g,m})$ by two types of projective invariants: the \emph{triple ratios} (Definition \ref{definition:tripleratio}) and \emph{edge functions} (Definition \ref{definition:edgefunction}), the latter of which generalize Thurston's shearing coordinates \cite{Thu98}. We have already seen the importance of the former in defining triangle invariants, and we now introduce the latter in the guise of edge invariants.

\begin{defn}[Edge invariants]
\label{definition:edgeinvariant}
For any boundary parallel pair of pants $(\beta,\gamma)\in\xvec{\mathcal{P}}_p$, let $\gamma_p=\beta_p$ denote the unique boundary-parallel oriented simple bi-infinite geodesic on $(\beta,\gamma)$, with both ends going up $p$, which separates $(\beta,\gamma)$ into two boundary-parallel pairs of half-pants $(\beta,\beta_p),(\gamma,\gamma_p)\in\xvec{\mathcal{H}}_p$. On $(\beta,\beta_p)$, there is a unique simple bi-infinite geodesic which emanates from $p$ and spirals towards $\beta$ (in the same direction as $\beta$) and likewise on $(\gamma,\gamma_p)$ there is also such a geodesic spiraling towards $\gamma$. Cutting $(\beta,\gamma)$ along these two spiral geodesics results in an ideal quadrilateral $\square_{\beta,\gamma}$ (see Figure~\ref{fig:pidealquad}). Let $(\,\gamma^+,\gamma\tilde{p},\beta^+,\tilde{p})$ be anticlockwise oriented ideal quadrilateral lift of $\square_{\beta,\gamma}$. We define \emph{edge invariants}
\[
d_1(\beta,\gamma):=\log \left(D_1(\tilde{p},\gamma \tilde{p}, \beta^+,\gamma^+)\right)\text{ and }e_1(\beta,\gamma):=\log \left(D_2(\tilde{p},\gamma \tilde{p}, \beta^+,\gamma^+)\right),
\]
where the $D_i$ are \emph{edge functions} (Definition \ref{definition:edgefunction}) of $(\tilde{p},\gamma \tilde{p}, \beta^+,\gamma^+)$.
\end{defn}

\begin{thm}[McShane identity for general cusped convex real projective surfaces, {Theorem~\ref{thm:equsl3pp}}]
Let $\rho:\pi_1(S_{g,m})\rightarrow\operatorname{PGL}_3(\mathbb{R})$ be a positive representation with unipotent boundary monodromy and let $p$ be a distinguished cusp on $S_{g,m}$. Then, 
\begin{align}
\sum_{(\beta,\gamma)\in\subvec{\mathcal{P}}_p}
\left(1+\tfrac{\cosh \frac{e_1(\beta,\gamma)}{2}}
{\cosh \frac{d_1(\beta,\gamma)}{2}}
\cdot 
e^{\frac{1}{2}\left(\ell_1(\beta)+\tau(\beta,\beta_p)+\ell_1(\gamma)+\tau(\gamma,\gamma_p)\right)}
\right)^{-1}
= 1.\label{eq:generalcuspedidentity}
\end{align}
where $d_1(\beta,\gamma):=\log \left(D_1(x,\gamma x, \beta^+,\gamma^+)\right)$ and $e_1(\beta,\gamma):=\log \left(D_2(x,\gamma x, \beta^+,\gamma^+)\right)$ are edge invariants (Definition~\ref{definition:edgeinvariant}), and $\tau(\gamma,\gamma_p)$ and $\tau(\beta,\beta_p)$ are triangle invariants (Definition~\ref{definition:Tgammap}).
\end{thm}

In the $(g,m)=(1,1)$ case, the quantity $d_1(\beta,\gamma)=-e_1(\beta,\gamma)$ and Equation~\eqref{eq:generalcuspedidentity} simplifies to Equation~\eqref{equation:mc3c}. See Remark \ref{remark:sl3s11td} for concrete details.

\subsection{McShane identities of loxodromic bordered positive representations}

We again state the McShane identity only in the special case of positive representations with one (loxodromic) boundary to simplify notation.

\begin{thm}[{McShane identities for loxodromic bordered positive representations, Theorem~\ref{theorem:boundaryith}}]
\label{theorem:b}
Let $\rho:\pi_1(S_{g,1})\to\operatorname{PGL}_n(\mathbb{R})$ be a positive representation with loxodromic boundary monodromy, and let $\alpha$ be a distinguished oriented boundary component of $S_{g,1}$ such that $S_{g,1}$ is on the left of $\alpha$. For each $i=1,\cdots,n-1$, we have:
\begin{align}
\sum_{(\beta,\gamma)\in \overrightarrow{\mathcal{P}}_\alpha}
\log \left( \frac{e^{\frac{\ell_i(\alpha)}{2}} + e^{\frac{2\phi_i(\beta,\gamma)+\kappa_i(\beta,\beta_{\alpha^-})+\ell_i(\beta)+\kappa_i(\gamma,\gamma_{\alpha^-})+\ell_i(\gamma)}{2}}}{e^{\frac{-\ell_i(\alpha)}{2}} + e^{\frac{2\phi_i(\beta,\gamma)+\kappa_i(\beta,\beta_{\alpha^-})+\ell_i(\beta)+\kappa_i(\gamma,\gamma_{\alpha^-})+\ell_i(\gamma)}{2}}} \right)=\ell_i(\alpha)\label{equation:idhs},
\end{align}
where
\begin{itemize}
\item
$\kappa_i(\delta,\delta_{\alpha^-}),\text{ for }\delta=\beta\text{ or }\gamma$, is the logarithm of a rational function of triple ratios associated to an ideal triangle embedded in $(\beta,\gamma)$, and 
\item
$\phi_i(\beta,\gamma)$ is an analytic function of triple ratios and edge functions associated with the boundary parallel pair of pants $(\beta,\gamma)$. 
\end{itemize}
\end{thm}

\begin{rmk}
We highlight the fact that both our summands as well as those from the McShane identity for bordered hyperbolic surfaces, see Equation~\eqref{equation:mirD}, take the following general form
\begin{align}
\mathcal{D}(x, y, z) =
\log \left( \frac{e^{\frac{x}{2}} + e^{\frac{y+z}{2}}}{e^{\frac{-x}{2}} + e^{\frac{y+z}{2}}} \right)\label{equation:mirid}.
\end{align}
For the bordered hyperbolic surface identity, $(x,y,z)$ equals $(\ell(\bar{\alpha}),\ell(\bar{\beta}),\ell(\bar{\gamma}))$, whereas for our identity, we take 
\[
(x,y,z)=(\ell_i(\alpha), \phi_i(\beta,\gamma)+\kappa_i(\beta,\beta_{\alpha^-})+\ell_i(\beta), \phi_i(\beta,\gamma)+\kappa_i(\gamma,\gamma_{\alpha^-})+\ell_i(\gamma)).
\]
Moreover, for $n$-Fuchsian representations, 
\begin{align*}
\kappa_i(\beta,\beta_{\alpha^-})&=\kappa_i(\gamma,\gamma_{\alpha^-})=\phi_i(\beta,\gamma)=0\text{,}\\
\ell_i(\beta)&=\ell_j(\beta), \ell_i(\gamma)=\ell_j(\gamma)\text{ for all }i,j=1,\cdots,n-1,
\end{align*}
and hence each of our family of $(n-1)$ McShane identities reduces to Equation~\eqref{equation:mirD}. 
\end{rmk}

\begin{rmk}
We direct those curious about the precise relationship between Theorem~\ref{theorem:b} and  the Labourie--McShane identity for the rank $n$ weak cross ratio to Corollary~\ref{corollary:LM-HS}.
\end{rmk}

\subsection{Birman--Series theorem and McShane-type inequalities for $n\geq4$ unipotent bordered representations}

The Birman--Series theorem \cite{birman_series} asserts the sparsity of complete simple geodesics on hyperbolic surfaces, and is a crucial ingredient (albeit sometimes only implicitly appearing) in almost all proofs of McShane identities. For example, the classical McShane identity has a probabilistic interpretation: the summand $2({1+e^{\ell(\bar{\gamma})}})^{-1}$ is precisely the probability that a geodesic uniformly randomly launched from the cusp of a $1$-cusped torus self-intersects before intersecting $\bar{\gamma}$. In order for the classical McShane identity to hold true, it is necessary that the event of a geodesic launched from the cusp never self-intersecting has probability $0$. This is ensured by the Birman--Series theorem.\medskip

Labourie and McShane also depend on the classical Birman--Series theorem in establishing identities for rank $n$ weak cross ratios associated to Hitchin representations \cite[Theorem~4.1.2.1]{LM09}. For representations with loxodromic boundary monodromy, they combine the Birman--Series theorem and the Anosov property (see Remark \ref{remark:holderanosov}) in order to prove the identity is indeed an equality. As Hitchin representations with loxodromic boundary monodromy deform to positive representations with unipotent boundary monodromy, however, the Anosov property is lost. In this setting, they establish their identity under a regularity hypothesis \cite[Definition~4.2.1]{LM09}. Loxodromic bordered positive representations are Hitchin (Remark~\ref{remark:hitpos}) and we are able to employ the same trick as Labourie and McShane to establish Theorem \ref{theorem:b} (or rather, Theorem~\ref{theorem:boundaryith}). However, it is generally unknown if unipotent bordered positive representations satisfy their regularity hypothesis and in order to prove our McShane identities for cusped convex real projective surfaces, we generalize the Birman--Series theorem:

\begin{thm}[Birman--Series theorem for convex real projective surfaces, Theorem~\ref{thm:birmanseries}] Given a strictly convex $\mathbb{RP}^2$ surface $\Sigma$, the \emph{Birman--Series set} defined as
\begin{align*}
\mathcal{BS}(\Sigma):=
\left\{
x\in\Sigma\mid
x\text{ lies on a complete simple geodesic on }\Sigma
\right\}
\end{align*}
is nowhere dense, closed and has $0$ Hilbert area.
\end{thm}

\begin{rmk}
We owe Benoist an enormous debt of gratitude for helping us to prove the above result. Particularly in explaining to us the proof for the exponentially shrinking ball property (Lemma~\ref{thm:expshrinkball}).\end{rmk}

Unfortunately, we are currently unable to formulate (let alone prove) a natural Birman--Series-type theorem for  $n\geq4$ unipotent bordered positive representations. Thus, instead of a McShane identity, we obtain the following McShane-type inequality:

\begin{thm}[Inequalities for unipotent bordered positive representations, Theorem \ref{theorem:puncturecase}]
\label{theorem:pc}
Given a positive representation $\rho\in \mathrm{Pos}_n(S_{g,m})$ with unipotent boundary monodromy, for the cusp $p$ and $i=1,\cdots,n-1$, we have 
\begin{align*}
&\sum_{(\beta,\gamma)\in \subvec{\mathcal{P}}_p}
 \frac{1}{1 + e^{\phi_i(\beta,\gamma)+\frac{1}{2}\left(\kappa_i(\gamma,\gamma_{p})+\ell_i(\gamma)+\kappa_i(\beta,\beta_{p})+\ell_i(\beta)\right)}}
\leq 1.
\end{align*}
\end{thm}

We conjecture that the above (non-strict) inequality should indeed by an equality, and outline a possible strategy of proof based upon establishing and obtaining enough control on the polynomial growth rate of $i$-th lengths, see Theorem~\ref{theorem:equuni} for details.

\subsection{Triple ratio boundedness and rigidity}

Of the three types of invariants we use to express our identities, $i$-th lengths directly generalize hyperbolic lengths, edge invariants generalize Thurston's shearing coordinates and so in some sense triangle invariants, or rather, triple ratios are perhaps the most mysterious. We undertake to dispel a little of this mystery by demonstrating that triple ratios satisfy a boundness property. 
Let us begin with an observation in the $n=3$ case, where triangle invariants and p-areas on convex real projective surfaces are related as follows:

\begin{thm}[{\cite[Proposition~0.3]{adeboye2015area}}]
Let $m$ be the Lebesgue measure with respect to the standard inner product on $\mathbb{R}^2$. Set 
\begin{equation*}
K = \rm{Sup}_{a,b \in \mathbb{R}^2, ||a||, ||b||\leq 1} m(\{p \cdot a+q \cdot b\;|\; 0 \leq p,q\leq 1\})).
\end{equation*}
The {\em p-area} is $\omega_{||\cdot||}= K^{-1} m$.

Given an embedded ideal triangle $\triangle\subset\Sigma$ on a finite p-area convex $\mathbb{RP}^2$ surface $\Sigma$, the p-area $\mathit{Parea}(\triangle)$ of $\triangle$ satisfies: 
\[
\mathit{Parea}(\triangle)\geq\tfrac{1}{8}(\pi^2+\tau(\triangle)^2).
\]
\end{thm}
For any $\operatorname{PGL}_3(\mathbb{R})$-positive representation $\rho$ with unipotent boundary monodromy, by \cite{Mar10}, the strictly convex real projective surface with holonomy representation $\rho$ has finite \emph{Hilbert area} with respect to the Hilbert metric. The p-area and the Hilbert area are uniformly comparable because of the Benz\'ecri compactness theorem \cite{B60}, and thus the p-area for $\rho$ is also finite. An immediate consequence of this result is that the triangle invariant $\tau(\triangle)=\log(T(\triangle))$ of any embedded ideal triangle on $\Sigma$ is necessarily bounded between
\[
\pm\sqrt{8\mathit{Parea}(\triangle)-\pi^2}.
\]

\begin{rmk}
One immediate corollary of this observation is that the collection of triangle invariants which arise in any given McShane identity, such as $\{\tau(\gamma)\}_{\gamma\in \subvec{\mathcal{C}}_{1,1}}$ in Equation~\eqref{equation:mc3c}, is bounded and hence the convergence properties of Equation~\eqref{equation:mc3c} are governed purely by the growth rates of the $i$-th lengths. We shall see that this phenomenon extends to arbitrary $n$.
\end{rmk}

Before proceeding, we clarify that, in the context of positive representation theory, triple ratios are actually functions defined on a finite ramified cover of $\mathrm{Pos}_n(S_{g,m})$ called the \emph{Fock--Goncharov $\mathcal{X}$-moduli space $\mathcal{X}_m(S_{g,m})$} (Definition~\ref{defn:positive}). The covering is bijective over unipotent-bordered positive  representations, and we implicitly made use of this when stating our McShane identities in terms of triple ratios. For loxodromic bordered positive representations, there is a \emph{canonical lift} (Definition~\ref{definition:ratioperiod}) of the positive representation variety to $\mathcal{X}_n(S_{g,m})$, thus also enabling us to consider triple ratios of loxodromic-bordered positive representations. As a further clarification: triple ratios $T_{i,j,k}:\mathcal{X}_n(S_{g,m})\to\mathbb{R}_{>0}$ are indexed by triples of positive integers $i,j,k$ summing to $n$  (Definition \ref{definition:tripleratio}). 

\begin{thm}[Triple ratio boundedness, Theorem~\ref{thm:bounded}]
\label{thm:triplebounded}
Given any unipotent or loxodromic-bordered positive representation $\rho:\pi_1(S)\rightarrow\operatorname{PGL}_n(\mathbb{R})$, the set of triple ratios taken over 
\begin{itemize}
\item
all lifts $(\rho,\xi)$ of $\rho$ in $\mathcal{X}_m(S_{g,m})$,
\item
all triple ratio indices $i,j,k$ summing to $n$, and
\item
all \emph{embedded} ideal triangles on $S$
\end{itemize}
is bounded within a compact interval $[T^{\xi_\rho}_{\min},T^{\xi_\rho}_{\max}]\subset\mathbb{R}_{>0}$.
\end{thm}

\begin{rmk}
The closed surface case for Theorem~\ref{thm:triplebounded} is due independently to Fran\c{c}ois Labourie and Tengren Zhang via private communication. Our proof for the above result is essentially topological and holds also for the positive representations with quasihyperbolic boundary monodromy.
\end{rmk}

\begin{rmk}
The $\kappa_i(\beta,\beta_{\alpha^-})$ terms in Equation~\eqref{equation:idhs} are logarithms of positive rational functions of triple ratios (Equation~\eqref{equation:kappa}). Thus, Theorem~\ref{thm:triplebounded} ensures that the spectrum $\{\kappa_i(\beta,\beta_{\alpha^-}),\kappa_i(\gamma,\gamma_{\alpha^-})\}_{(\beta,\gamma)\in \subvec{\mathcal{P}}_\alpha}$ of $\kappa_i$-terms is bounded in $\mathbb{R}$. This informs us that the convergence properties of the McShane identity series are governed by the $i$-th lengths and the edge functions. 
\end{rmk}

When a given positive representation $\rho$ is $n$-Fuchsian, by Lemma \ref{lemma:nfuchrig}, there exists a lift $(\rho,\xi)\in\mathcal{X}_n(S_{g,m})$ such that all of its triple ratios are equal to $1$. We show that this is in fact a characterizing condition for $n$-Fuchsian representations:

\begin{thm}[Fuchsian rigidity, Theorem~\ref{thm:highrankrigidity}]
A positive representation $\rho\in\mathrm{Pos}_n(S_{g,m})$ with unipotent boundary monodromy (including $S$ being a closed surface) is $n$-Fuchsian if and only if all of its triple ratios are equal to $1$.
\end{thm}

The following corollary is somewhat unrelated to the theme of our paper. We state it due to independent interest: ellipsoid characterization is a classical area of research with over a century's worth of history (see \cite{guo2013characterizations} for a nice survey). For any $k$-dimensional strictly convex open domain $\Omega\subset\mathbb{R}^k$ with $C^1$ boundary, one can define an alternative generalization of the notion of triple ratios. Specifically, any oriented \emph{ideal triangle} $\triangle$  (i.e.: a Euclidean triangle with all vertices on $\partial\Omega$) lies on the intersection of $\Omega$ and a unique $2$-dimensional affine plane $H\subset\mathbb{R}^k$. One may then define the triple ratio for $\triangle$ as the triple ratio of $\triangle$ in the strictly convex planar domain $\Omega\cap H$. 

\begin{cor}[Ellipsoid characterization]
A $k$-dimensional $C^1$ open strictly convex domain in $\mathbb{R}^k$ is a $k$-dimensional ellipsoid if and only if the triple ratios for all of its ideal triangles are equal to $1$.
\end{cor}

\subsection{Applications of the McShane identity}

We have already alluded to Mirzakhani \cite{mirz_simp}'s spectacular application of McShane identities to derive a recursive algorithm for computing the volumes of moduli spaces of Riemann surfaces. In \cite{Sun20b}, the second author builds upon Mirzakhani's ideas and employs Theorem~\ref{theorem:b} and \cite[Corollary 8.18]{SZ17} to study the volumes of certain bounded subspaces of the mapping class group quotient of fixed boundary monodromy subslices of $\mathrm{Pos}_n(S_{g,m})$.\medskip

We are aware of the following applications for the McShane-type identities in the literature:
\begin{itemize}
\item
various authors \cite{akiyoshi2004refinement,MR2258748,markofftriples,bowditch1997variation,huang2018mcshane,lee2013variation} use them to study the geometry of the convex core or the cuspidal tori for various hyperbolic $3$-manifolds;
\item
Miyachi uses them to bound the Teichm\"uller distance between two marked surfaces \cite{miyachi2005limit}.
\end{itemize}

We illustrate several novel applications of the McShane identity.\medskip

To begin with, a refinement (Theorem~\ref{theorem:puncturecase}) of Theorem~\ref{theorem:pc}, combined with Theorem~\ref{thm:triplebounded} and Lemma~\ref{thm:halfpantslbound}, yields the following:

\begin{thm}[Simple $\ell_i$-spectrum discreteness]
For $m\geq 1$, let $\rho:\pi_1(S_{g,m})\rightarrow\operatorname{PGL}_n(\mathbb{R})$ be a positive representation with unipotent boundary monodromy. For any $i=1,\cdots,n-1$, the simple $\ell_i$-spectrum for $\rho$ is discrete. As a consequence, let $\ell=\sum_{i=1}^{n-1} \ell_i$, then the simple $\ell$-spectrum for $\rho$ is discrete.
\end{thm}

\begin{rmk}For a positive representations with (only) loxodromic boundary monodromy, the above result can be obtained via the Anosov property \cite{Lab06}. However, positive representations with unipotent boundary monodromy are not Anosov. In particular, our proof uses positivity in a fundamental way.
\end{rmk}

When $n=3$, we strengthen the above result in two different directions (Appendix~\ref{sec:lengthspecproj}), we show that:
\begin{itemize}
\item
both the simple $i$-length and the $\ell$-length spectra of every unipotent bordered positive representation $\rho\in\mathrm{Pos}^u_3(S_{g,m})$ grow polynomially of order $6g-6+2m$, and
\item
both the $i$-length and the $\ell$-length spectra of every $\rho\in\mathrm{Pos}^u_3(S_{g,m})$ is discrete.
\end{itemize}
Kim utilizes different techniques in \cite{kim2019} to generalize the above discreteness of the $\ell$-spectrum for all $n$.

\begin{thm}[{\cite{lee2017collar}, Collar lemma, Theorem~\ref{thm:collar}}]
Given any positive representation $\rho\in\mathrm{Pos}_3(S)$, the Hilbert lengths of any two intersecting simple closed curves $\beta,\gamma$ satisfy the following inequality:
\begin{align*}
(e^{\frac{1}{2}\ell(\beta)}-1)(e^{\frac{1}{2}\ell(\gamma)}-1)>4.
\end{align*}
\end{thm}

\begin{rmk}
The above collar lemma is due to Lee and Zhang \cite[Equation~(3-2)]{lee2017collar}. Naturally, the above result translates into a collar lemma for convex real projective surfaces with cusps (unipotent boundary) and/or closed geodesic boundaries (loxodromic boundary).
\end{rmk}

\subsection{Applications to Thurston-type metrics}

The remaining applications are all related to asymmetric ratio metrics on various character varieties. These results require the full strength of the McShane-type identity and not just an inequality. We begin with our results for the Fuchsian representations:

\begin{thm}[Fuchsian non-domination]
Given two marked hyperbolic surfaces $\Sigma_1,\Sigma_2\in\mathit{Teich}_{g,m}(L_1,\ldots,L_m)$ with fixed boundary lengths $L_1,\ldots,L_m\geq 0$. Then the simple closed geodesic spectrum for $\Sigma_1$ dominates the simple closed geodesic spectrum $\Sigma_2$ if and only if $\Sigma_1=\Sigma_2$.
\end{thm}

Non-domination fails when the boundary length is allowed to vary \cite{papadopoulos2010shortening}, meaning that naive length ratio-based generalizations of Thurston's length ratio metric do not satisfy positivity (compare with \cite[Theorem~3.1]{Thu98}). Liu--Papadopoulos--Su--Th{\'e}ret resolve this by introducing the arc metric. We do so by fixing boundary lengths:

\begin{cor}[Length ratio metric for fixed bordered hyperbolic surfaces]
The non-negative real function $d_{\mathit{Th}}: \mathit{Teich}_{g,m}(L_1,\ldots,L_m)\times \mathit{Teich}_{g,m}(L_1,\ldots,L_m)\rightarrow \mathbb{R}_{\geq0}$ defined by 
\begin{align*}
d_{\mathit{Th}}(\Sigma_1,\Sigma_2)
:=
\log
\sup_{\bar{\gamma}\in\mathcal{C}_{g,m}}
\frac{\ell^{\Sigma_1}(\bar{\gamma})}{\ell^{\Sigma_2}(\bar{\gamma})},
\end{align*}
is a mapping class group invariant asymmetric metric on the Teichm\"uller space $\mathit{Teich}_{g,m}$ $(L_1,\ldots,L_m)$ of genus $g$ surfaces with $m$ boundaries of fixed lengths $L_1,\ldots, L_m$.
\end{cor}

Let $\mathrm{Pos}_3^u(S_{g,m})$ be the $\operatorname{PGL}_3(\mathbb{R})$-positive representation variety with unipotent boundary monodromy, which corresponds to the moduli space of strictly convex cusped $\mathbb{RP}^2$ structures on $S_{g,m}$. We propose the following candidate for a metric on the space $\mathrm{Pos}_3^u(S_{1,1})$:
\begin{align*}
d_{\mathit{Gap}}(\rho_1,\rho_2)
:=
\log
\sup_{\gamma\in\subvec{\mathcal{C}}_{1,1}}
\frac
{\log(1+e^{\ell_1^{\rho_1}(\gamma)+\tau^{\rho_1}(\gamma)})}
{\log(1+e^{\ell_1^{\rho_2}(\gamma)+\tau^{\rho_2}(\gamma)})}.
\end{align*}

\begin{thm}[Gap metric for $\mathrm{Pos}_3^u(S_{1,1})$]
The non-negative function $d_{\mathit{Gap}}$ defines a mapping class group invariant asymmetric metric on $\mathrm{Pos}_3^u(S_{1,1})$. Moreover, the restriction of the metric $d_{\mathit{Gap}}$ to the Fuchsian locus of $\mathrm{Pos}_3^u(S_{1,1})$ is equal to the Thurston metric.
\end{thm}

We also generalize the notion of a gap metric to include $\mathrm{Pos}_3^u(S_{g,m})$ (Definitions~\ref{defn:pantsgapmetric} and \ref{defn:totalgapmetric}). The resulting asymmetric metric is mapping class group invariant. When restricted to the Fuchsian locus, the novel metric is at least as large as the Thurston metric, but it remains to be seen whether these two metrics are equal.

\subsection{Section overview and reading guide}

This paper consists of the following:\medskip

\noindent
\textsc{\S\ref{sec:fgmodcor}: Preliminary.} We introduce positive representations (Definition \ref{definition:posrep}), triple ratios (Definitions \ref{definition:tripleratio} and edge functions (Definition~\ref{definition:edgefunction}), and Fock and Goncharov's theory of positive higher Teichm\"uller spaces (\ref{defn:positive}).\medskip

\noindent
\textsc{\S\ref{sec:xcoords}: Properties of projective invariants.} We show that the set of triple ratios associated to any given positive representation is bounded (Theorem~\ref{thm:bounded}). We then show that triple ratios all being equal to $1$ or edge functions along the same edge being all the same are characterizing properties for $n$-Fuchsian representations (Proposition~\ref{prop:triple34} for $n=3,4$, Proposition~\ref{prop:edge3} for $n=3$ and Theorem~\ref{thm:highrankrigidity} for general $n$ with unipotent boundary monodromy).\medskip

\noindent
\textsc{\S\ref{sec:GS}: Goncharov--Shen potentials.} We introduce and familiarize ourselves with Goncharov--Shen potentials.\medskip

\noindent
\textsc{\S\ref{sec:splitting}: Identities for $\operatorname{PGL}_3(\mathbb{R})$-representations with unipotent boundary.} We describe the strategy for proving McShane-type identities, before rigorously establishing the McShane identity for all $\operatorname{PGL}_3(\mathbb{R})$-positive representation with unipotent boundary monodromy (Theorem~\ref{thm:equsl3pp}). A key ingredient of the proof --- the Birman--Series geodesic sparsity theorem for cusped convex real projective surfaces is delayed to \S\ref{sec:sparsity}. We focus on the $1$-cusped torus case (\ref{theorem:inequsl3s11}), highlighting certain surprising symmetries (\ref{sec:surprisesymmetry}). We also introduce a finer McShane-type identity (Theorem~\ref{theorem:equsl3}) summing over half-pants rather than pants.\medskip

\noindent
\textsc{\S\ref{sec:sparsity}: Geodesic sparsity for convex real projective surfaces.} We prove the Birman--Series geodesic sparsity theorem for convex real projective surfaces (Theorem~\ref{thm:birmanseries}).\medskip

\noindent
\textsc{\S\ref{sec:highermcshane}: McShane identities for higher Teichm\"uller space.} 
We show that ratios of Goncharov--Shen potentials are projective invariants (Proposition~\ref{prop:BC}). We dub these objects $i$-th potential ratios and relate them to rank $n$ weak cross ratios (Corollary~\ref{corollary:LM-HS}) and simple root lengths (Corollary~\ref{cor:iratiocycle}).\medskip

We adapt (Theorem~\ref{thm:boundary}) Labourie and McShane's ideas from \cite{LM09} to establish a family of McShane identities for loxodromic-bordered positive representations of arbitrary rank (Theorem~\ref{theorem:boundaryith}), deriving regular expressions for their summands in terms of $i$-th lengths, triple ratios and edge functions using $i$-th potential ratios. We then obtain McShane-type inequality for unipotent-bordered positive representations of arbitrary rank (Theorem \ref{theorem:puncturecase}) by deforming the loxodromic bordered identities. We pose a conjectural condition which would promote these inequalities to identities (Theorem~\ref{theorem:equuni}).\medskip

\noindent
\textsc{\S\ref{section:app}: Applications.} We employ our McShane identities to show the discreteness of simple $i$-th length spectrum (Theorem~\ref{thm:discretespectra}), to demonstrate the collar lemma (Theorem~\ref{thm:collar}) for $\operatorname{PGL}_3(\mathbb{R})$-positive representations and to generalize the Thurston metric (Theorem~\ref{thm:genthurstonmetric} and Definition~\ref{defn:pantsgapmetric}) for cusped strictly convex real projective surfaces.\medskip

\begin{rmk}
Readers mainly interested in convex real projective surfaces ($\mathrm{Pos}_3(S)$) may wish to focus on \S\ref{sec:splitting}, \S\ref{sec:sparsity} and the McShane identity applications in \S\ref{section:app}. On the other hand, those with background in and predominantly interested in (arbitrary rank) Fock--Goncharov higher Teichm\"uller theory may be primarily interested in \S\ref{sec:xcoords}, \S\ref{sec:GS}, \S\ref{sec:splitting} and \S\ref{sec:highermcshane}, with secondary interests in our McShane identity applications in \S\ref{section:app}.
\end{rmk}

\clearpage

\section{Preliminary}
\label{sec:fgmodcor}
The results we derive in this article center on \emph{positive representations} --- focal objects in Fock and Goncharov's approach to the higher Teichm\"uller theory \cite{FG06}. We review the definition of positive representations, projective invariants associated to them, as well as their relationship to Fock and Goncharov's $\mathcal{X}$ and $\mathcal{A}$ moduli spaces.

\subsection{Positive representations}

The notion of positive surface group representations are motivated by totally positive matrices and positive configurations of flags. We first present these concepts.\medskip

Let $S=S_{g,m}$ be a topological surface of genus $g$ with $m$ holes, with negative Euler characteristic $\chi(S_{g,m})=2-2g-m<0$. Moreover, consider the vector space $\mathbb{R}^n$ endowed with the standard Euclidean volume form $\Delta$.

\begin{defn}[Flags and decorated flags]
\label{defn:flag}
A \emph{flag} $F$ in $\mathbb{R}^n$ is a maximal filtration of vector subspaces of $\mathbb{R}^n$:
\[
\{0\}=F^{(0)}\subset F^{(1)}\subset \cdots \subset F^{(n-1)} \subset F^{(n)}=\mathbb{R}^n,\quad \dim F^{(i)}=i.
\]
A \emph{basis} for a flag $F$ is a basis $(f_1,\ldots,f_n)$ for $\mathbb{R}^n$ such that, for any $i=1,\ldots,n$, the first $i$ basis vectors form a basis for $F_i$.\medskip

A \emph{decorated flag} $(F,\varphi)$ is a pair consisting of a flag $F$ and a collection $\varphi$ of $(n-1)$ non-zero vectors
\[
\varphi=\left\{
\check{f}_i \in F^{(i)}/F^{(i-1)}
\right\}_{i=1,\ldots,n-1}.
\] 
A \emph{basis} for a decorated flag $(F,\varphi)$ is a basis $(f_1,\ldots,f_n)$ for the vector space $\mathbb{R}^n$ such that
\[
f_i+F^{(i-1)}=\check{f}_i\in F^{(i)}/F^{(i-1)}\quad \text{for}\quad i=1,\ldots, n-1.
\]

We refer to the set $\mathcal{B}$ of flags in $\mathbb{R}^n$ as the \emph{flag variety} and the set $\mathcal{A}$ of decorated flags in $\mathbb{R}^n$ as the \emph{principal affine space}. We note the obvious ``forgetful" projection map
\begin{equation}
\label{equation:projpi}
\pi: \mathcal{A} \rightarrow \mathcal{B}\text{, }(F,\varphi)\mapsto F.
\end{equation}
\end{defn}

\begin{notation}
\label{notation:flag}
Given a basis $(f_1,\ldots,f_n)$ for a flag or a decorated flag $F$, for $i=1,\cdots,n$, we use $f^i$ to denote:
\[
f^i:=f_1 \wedge f_2 \wedge \cdots \wedge f_{i-1} \wedge f_i.
\]
We set $f^0=1$ by convention. Moreover, without loss of generality, we only consider bases such that $f_n$ satisfies $\Delta(f^n)=1$.
\end{notation}

\begin{defn}[Generic position]
\label{defn:genpos}
For an integer $d\geq 2$, We say that a $d$-tuple of flags $(F_1,\cdots,F_d)$ is in \emph{generic position} if, for any collection of non-negative integers $n_1,\cdots,n_d\in\mathbb{Z}_{\geq0}$ satisfying $n_1+\cdots+n_d\leq n$, the sum $\sum_{i=1}^d F_i^{(n_i)}$ of vector spaces is a direct sum. Likewise, a $d$-tuple of decorated flags is in generic position if the underlying $d$-tuple of flags is in generic position.
\end{defn}

In \cite{Lu94}, Luztig expanded upon the theory of \emph{totally positive matrices} originally developed by Gantmacher--Krein \cite{GK41} and Schoenberg \cite{Sch33} to include arbitrary semi-simple real Lie groups. For our purposes, it means the following:

\begin{defn}[{Totally positive matrices (see, e.g., \cite[\S1.5]{FG06})}]
A real matrix is \emph{totally positive} if and only if all of its matrix minors are positive. A real upper triangular matrix is totally positive if and only if all of its minors, apart from those which are necessarily $0$ are positive. 
\end{defn}

Positive $d$-tuples of flags are defined in \cite[Definition 1.4]{FG06} for a very general context. Again, we restrict to the $\operatorname{PGL}_n(\mathbb{R})$ case (e.g.: \cite[Definition 2.14]{sun2017flows}).
\begin{defn}
\label{definition:posconf}
For $d\geq 3$, a generic $d$-tuple of flags $(F_1,\cdots,F_d)$ is \emph{positive} if for some fixed basis $B=\{f_i\}_{i=1}^n$ of $\mathbb{R}^n$ such that $f_i \in F_1^{(i)}\cap F_2^{(n-i+1)}$ for $i=1,\cdots,n$,
\begin{enumerate}
\item there are projective transformations $u_1,\cdots,u_{d-2}$ in $\operatorname{PGL}_n(\mathbb{R})$ that are totally positive upper triangular unipotent matrices with respect to the basis $B$, and
\item there exists a $g \in \operatorname{PGL}_n(\mathbb{R})$ which fixes $F_1$ and $F_2$
\end{enumerate}
 such that 
\[
g(F_1,F_2,F_3,\cdots,F_d)=(F_1,F_2, u_1 \cdot F_2, \cdots, u_1 \cdots u_{d-2} \cdot F_2).
\]
\end{defn}

Note that if $(F_1,\cdots,F_d)$ is positive, then for any collection of indices $1\leq i_1<\cdots<i_l\leq d$, the $l$-tuple of flags $(F_{i_1},\cdots,F_{i_l})$ is positive.

\begin{defn}[Auxiliary metric]
\label{definition:aux}
Let $S=S_{g,m}$ be a surface of genus $g$ with $m$ holes with negative Euler characteristic. For any discrete faithful representation $\rho:\pi_1(S)\rightarrow\operatorname{PGL}_n(\mathbb{R})$, we say that a complete hyperbolic metric $h_\rho$ on $S$ is an \emph{auxiliary metric} for $\rho$ if it satisfies the following conditions:
\begin{enumerate}
\item if the monodromy $\rho(\alpha)$ of a boundary component $\alpha$ of $S$ is unipotent, then choose $h_\rho$ such that the boundary $\alpha$ is a cusp;
\item if the monodromy $\rho(\alpha)$ of a boundary component $\alpha$ of $S$ is non-unipotent, then choose $h_\rho$ such that the boundary $\alpha$ is a closed geodesic (of strictly positive length).
\end{enumerate}
\end{defn}

\begin{defn}[Boundary at infinity]
\label{definition:bin}
Consider a surface $S$ with negative Euler characteristic and let $h_\rho$ denote an auxiliary metric for a discrete faithful representation $\rho:\pi_1(S)\to\operatorname{PGL}_n(\mathbb{R})$. Further let $(\tilde{S},\tilde{h}_\rho)$ denote the universal cover of $(S,h_\rho)$. We define the \emph{boundary at infinity} $
\partial_\infty \pi_1(S,h_\rho)=\partial_\infty \pi_1(S)$ 
for $\rho$ as the intersection of $\mathbb{RP}^1=\partial \mathbb{H}^2$ with the set of metric completion points of $(\tilde{S},\tilde{h}_\rho)$. 
\end{defn}
To clarify: when every boundary component of $(S,h_\rho)$ is cuspidal (including the scenario when $S$ is closed), then the boundary at infinity $\partial_\infty \pi_1(S)$ is homeomorphic to a circle. Conversely, when $(S,h_\rho)$ has non-cuspidal boundary components, the boundary at infinity $\partial_\infty \pi_1(S)$ is homeomorphic to a Cantor set (regarded as a subset of a circle).  

\begin{defn}[Positive maps]
\label{dfn:pmaps}
Consider a subset $\mathbb{X}\subseteq \mathbb{RP}^1$, we say that a map $\xi:\mathbb{X}\to\mathcal{B}$ from $\mathbb{X}$ to the flag variety $\mathcal{B}$ is a \emph{positive map} if and only if: for any collection of distinct points $x_1,\ldots,x_d\in\mathbb{X}$ cyclically anti-clockwise ordered around $\mathbb{RP}^1$, the $d$-tuple of flags $(\xi(x_1),\cdots,\xi(x_d))$ is a positive $d$-tuple of flags.
\end{defn}

\begin{defn}[Positive representations]
\label{definition:posrep}
We say that a representation $\rho:\pi_1(S_{g,m})\rightarrow\operatorname{PGL}_n(\mathbb{R})$ is a \emph{positive representation} if and only if there exists a $\rho$-equivariant positive map 
\[
\xi_\rho:\partial_\infty \pi_1(S_{g,m})\rightarrow \mathcal{B}.
\]
In situations where we wish to emphasize that a positive representation $\rho$ maps into $\operatorname{PGL}_n(\mathbb{R})$, we refer to $\rho$ as a \emph{$\operatorname{PGL}_n(\mathbb{R})$-positive representation}.
\end{defn}
By \cite[Theorem 1.14]{FG06}, the above $\rho$-equivariant positive map $\xi_\rho$ is continuous.

\subsection{Representation varieties}

We create specialized notation for three types of positive representation varieties. Strictly speaking, the elements constituting these representation varieties are \emph{$\operatorname{PGL}_n(\mathbb{R})$-conjugacy classes} of $\operatorname{PGL}_n(\mathbb{R})$-positive representations. However, it is standard nomenclature to simply refer to these spaces as representation varieties.

\begin{defn}[Positive representation varieties]
\label{defn:prepvar}
We adopt the following notation:
\begin{itemize}
\item
\emph{$\operatorname{PGL}_n(\mathbb{R})$-positive representation variety}:
$\operatorname{Pos}_n(S_{g,m})$ denotes the space of $\operatorname{PGL}_n(\mathbb{R})$-conjugacy classes of $\operatorname{PGL}_n(\mathbb{R})$-positive representations $\rho:\pi_1(S_{g,m})\to\operatorname{PGL}_n(\mathbb{R})$.
\item
\emph{loxodromic-bordered $\operatorname{PGL}_n(\mathbb{R})$-positive representation variety:}
$\operatorname{Pos}_n^h(S_{g,m})$ denotes the subspace of $\operatorname{Pos}_n(S_{g,m})$ consisting of (conjugacy classes of) $\operatorname{PGL}_n(\mathbb{R})$-positive representations with loxodromic boundary monodromy for all $m$ boundary components.
\item
\emph{ unipotent-bordered $\operatorname{PGL}_n(\mathbb{R})$-positive representation variety:}
$\operatorname{Pos}_n^u(S_{g,m})$ denotes the subspace of $\operatorname{Pos}_n(S_{g,m})$ consisting of $\operatorname{PGL}_n(\mathbb{R})$-positive representations with unipotent boundary monodromy for all $m$ boundary components.
\end{itemize}
\end{defn}

Before venturing further, we first elucidate the relationship between positive representations and Hitchin representations \cite{Hit92}. To begin with, for the closed surface $S_{g,0}$, Fock and Goncharov \cite[Theorem 1.15]{FG06} show that the Hitchin representation variety $\operatorname{Hit}_n(S_{g,0})$ is equal to the positive representation variety $\operatorname{Pos}_n(S_{g,0})$. When the underlying surface has boundaries, however, the situation is as follows:

\begin{rmk}[Hitchin versus positive for $S_{g,m\geq1}$]
\label{remark:hitpos}
In \cite[\S9]{LM09}, Labourie and McShane generalize the notion of {\em Hitchin representations} to the negative Euler characteristic bordered surface $S_{g,m}$ context by defining them as representations which 
have loxodromic boundary and are deformed along a path of loxodromic-bordered representations from a $n$-Fuchsian representation. Thanks to \cite[Theorem~9.1]{LM09}, we know that
\begin{align*}
\operatorname{Hit}_n(S_{g,m})
&:=\left\{
\rho\mid
\rho:\pi_1(S_{g,m})\to\operatorname{PGL}\text{ is a Hitchin representation }
\right\}/\operatorname{PGL}_n(\mathbb{R})\\
&\subseteq
\operatorname{Pos}_n^h(S_{g,m}).
\end{align*}
Conversely, any positive representation $\rho \in \operatorname{Pos}_n^h(S_{g,m})$ may be extended to a positive representation $d\rho$ for the doubled surface $S_{2g-1+m,0}$ via a canonical doubling construction described in \cite[Definition~9.2.2.3]{LM09} (or more generally, via constructions described in \cite[\S7]{FG06} based on gluing conditions defined in \cite[Definition~7.2]{FG06}). Since positive representations for closed surfaces are Hitchin \cite[Theorem 1.15]{FG06}, the representation $d\rho$ definitionally deforms to a $n$-Fuchsian representation via a path $\{d\rho_t\}$. The restriction of $d\rho_t$ to $\pi_1(S_{g,m})\leq\pi_1(S_{2g-1+m,0})$ all have loxodromic boundary and deform from $\rho$ to a $n$-Fuchsian representation of $\pi_1(S_{g,m})$ along a path of loxodromic-bordered representations. Therefore,
\[
\operatorname{Hit}_n(S_{g,m})= \operatorname{Pos}_n^h(S_{g,m})\text{, for }m\geq1.
\]

\end{rmk}

\subsection{Configuration space and projective invariants}

\begin{defn}[Configuration space]
We denote the space of generic $d$-tuple of flags up to diagonal projective transformations by $\operatorname{Conf}_d$, and refer to elements of $\operatorname{Conf}_d$ as \emph{configurations}. We further denote the subspace of $\operatorname{Conf}_d$ consisting of positive $d$-tuple of flags up to diagonal projective transformations by $\operatorname{Conf}_d^+$. The elements of $\operatorname{Conf}_d^+$ are \emph{positive configurations}.
\end{defn}

In the classical (i.e.: $n=2$) setting, the (pure) mapping class group is trivial and the positive configuration space is equal to both the moduli space and the Teichm\"uller space of hyperbolic ideal $d$-gons. The positive configuration space serves as a building block for Teichm\"uller spaces of surfaces of greater topological complexity, a similar picture persists for general $n$. This in turn means that projective invariants of $d$-tuples of flags, which define functions on $\operatorname{Conf}_d$, are candidates for local and/or global coordinates for higher Teichm\"uller spaces. We focus on two types of projective invariants: \emph{triple ratios} and \emph{edge functions}.\medskip

\begin{defn}[Triple ratio]
\label{definition:tripleratio}
Consider a triple of flags $(F,G,H)$ in generic position, with bases 
\[
(f_1,\cdots,f_n),\;\;(g_1,\cdots,g_n),\;\;(h_1,\cdots,h_n).
\] 
Then for any triple of positive integers $(i,j,k)$ with $i+j+k=n$, the \emph{triple ratio} $T_{i,j,k}(F,G,H)$ is defined as:
\begin{align*}
T_{i,j,k}(F,G,H):= 
\frac
{\Delta\left(f^{i+1} \wedge g^{j} \wedge h^{k-1}\right) \Delta\left(f^{i-1} \wedge g^{j+1} \wedge h^{k}\right)  \Delta\left(f^{i} \wedge g^{j-1} \wedge h^{k+1}\right)} 
{\Delta\left(f^{i+1} \wedge g^{j-1} \wedge h^{k}\right) \Delta\left(f^{i} \wedge g^{j+1} \wedge h^{k-1}\right) \Delta\left(f^{i-1} \wedge g^{j} \wedge h^{k+1}\right)}.
\end{align*}
Properties of determinants ensure the following cyclic symmetry: 
\[
T_{i,j,k}(F,G,H)=T_{j,k,i}(G,H,F)=T_{k,i,j}(H,F,G).
\]
\end{defn}

\begin{rmk}
For $n=3$, the triple $(i,j,k)$ is necessarily equal to $(1,1,1)$. We will often omit the indices $(1,1,1)$ and simply write $T(F,G,H)$.
\end{rmk}

We give an interpretation for the triple ratio, noting that it serves also as a geometrically flavored definition:

\begin{rmk}[{\cite{FG07}, Geometric definition for the triple ratio}]
\label{remark:triple}
Consider three flags
\[
A=(a,L_1),B=(b,L_2),C=(c,L_3)
\] 
 in $\mathbb{RP}^2$ in generic position. Let $u=L_2 \cap L_3$, $v=L_1 \cap L_3$, $w=L_1 \cap L_2$ (see Figure~\ref{fig:ceva}), and let $|\cdot|$ denote the Euclidean distance. We stated in the introduction that the triple ratio of $(A,B,C)$ is given by
\begin{align}
T(A,B,C) :=
\frac
{|wb|\cdot |uc|\cdot |va|}
{|bu|\cdot |cv| \cdot |aw|},
\end{align}
where the Euclidean distance of the segments may be infinite.

\begin{figure*}[h!]
\includegraphics[scale=0.2]{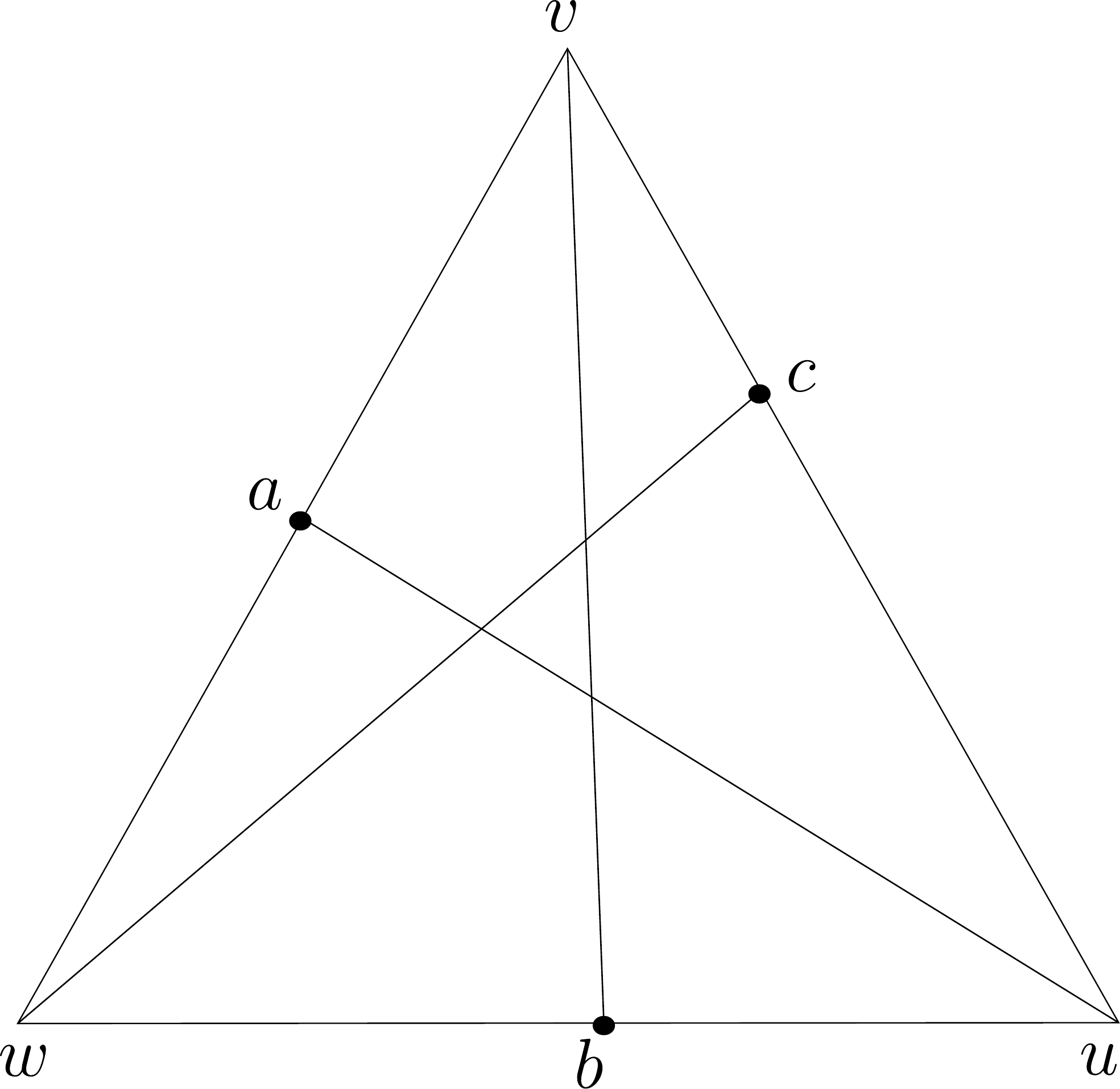}
\caption{Triple ratio.}
\label{fig:ceva}
\end{figure*}

To interpret triple ratios for flags $A,B,C\in\mathcal{B}$ in higher rank contexts, we project $\mathbb{R}^n$ down to the following $3$-dimensional vector space
\[
\mathbb{R}^n/\left(A^{(i-1)}\oplus  B^{(j-1)} \oplus C^{(k-1)}\right),
\]
and note that $A,B,C$ project to flags in the quotient vector space. The triple ratio $T_{i,j,k}(A,B,C)$ is then equal to the triple ratio of the respective projected flags for $A,B,C$.
\end{rmk}

\begin{rmk}
For $n=3$, Ceva's theorem asserts that $T(A,B,C)=1$ if and only if the lines $\overline{au}$, $\overline{bv}$, $\overline{cw}$ intersect at one point. 
\end{rmk}

\begin{defn}[Edge function]
\label{definition:edgefunction}
Let $(X,Y,Z,W)$ be quadruple of flags in generic position, and choose bases
\[
(x_1,\cdots,x_n),\;\;(y_1,\cdots,y_n),\;\;(z_1,\cdots,z_n),\;\;(w_1,\cdots,w_n).
\]
For $i=1,\cdots,n-1$, the \emph{edge function} is defined as
\begin{eqnarray*}
&&D_i(X,Y,Z,W):=- \frac{\Delta\left(x^{n-i} \wedge y^{i-1}\wedge z^{1} \right)}{\Delta\left(x^{n-i-1}  \wedge y^{i}\wedge z^1\right)}\cdot \frac{ \Delta\left(x^{n-i-1} \wedge y^{i} \wedge w^1\right)}{\Delta\left(x^{n-i}  \wedge y^{i-1} \wedge w^{1}\right)}.
\end{eqnarray*}
Properties of determinants ensure the following symmetry: 
\[
D_i(X,Y,Z,W)=D_{n-i}(Y,X,W,Z).
\]
\end{defn}

We once again emphasize that both triple ratios and edge functions are projective invariants. Moreover, we emphasize that they are well-defined, which is to say that their values do not depend on the chosen bases for the input flags.\medskip

\begin{defn}[Marked ideal triangle]
A \emph{marked ideal triangle} on a hyperbolic surface $S$ is a pair $(\triangle,\mathit{imm})$, where $\triangle$ is an oriented ideal triangle with vertices labeled anticlockwise by $1,2,3$ and $\mathit{imm}:\triangle\to S$ is an isometric immersion of $\triangle$ into $S$. When $S$ is a subset of the hyperbolic plane $\mathbb{H}^2$ (e.g.: a hyperbolic ideal $d$-gon, or perhaps the entire hyperbolic plane), the data of the immersion $\mathit{imm}$ is equivalent to giving an ordered $3$-tuple $(v_1,v_2,v_3)$ of ideal points $v_i\in\partial\mathbb{H}^2$. We say that a marked ideal triangle is a \emph{marked oriented ideal triangle} if $\mathit{imm}$ is orientation-preserving.
\end{defn}

\begin{notation}
We henceforth adopt the following notation conventions:
\begin{itemize}
\item
$\overline{xy}$ denotes the unoriented edge between $x$ and $y$;
\item
$\overline{xyz}$ denotes an unoriented triangle;
\item
$\xvec{xyz}$ denotes an oriented triangle; 
\item
$(x,y)$ denotes the oriented edge from $x$ to $y$;
\item
$(x,y,z)$ denotes a marked triangle; 
\item
$\tilde{X}$ denotes the union of all lifts of an object $X$ in $S$ to the universal cover $\tilde{S}$ of $S$.
\end{itemize}
We continue to use this notation throughout the paper except when explicitly stated otherwise, especially when carrying out computations.
\end{notation}

Consider a hyperbolic ideal $d$-gon and label its $d$ cusps, which may be regarded as vertices, by $\{v_1,\cdots,v_d\}$. To each such vertex $v$, we assign a flag $F(v)\in\mathcal{B}$. Let $\mathcal{T}$ be an ideal triangulation of the $d$-gon, and arbitrarily fix one marked (anticlockwise) oriented ideal triangle $\triangle$ representative for each (unmarked) ideal triangle in the triangulation $\mathcal{T}$ and denote this collection of marked oriented ideal triangle representatives by $\Theta$. We represent each marked oriented ideal triangle $\triangle\in\Theta$ by its ideal vertices $(x,y,z)$ so that $x,y,z\in\{v_1,\cdots,v_d\}$ arise anticlockwise along the boundaries of the $d$-gon. We associate to each triangle $\triangle\in \Theta$ the $3$-tuple $\mathbf{F}(\triangle)$ of flags associated to its vertices:
\[
\mathbf{F}(\triangle):=\mathbf{F}(x,y,z):=(F(x),F(y),F(z)).
\]
Similarly, fix one oriented edge $\xvec{e}$ representative for each (unoriented) interior edge in $\mathcal{T}$ and denote the collection of oriented interior edges by $\Xi$. We represent each oriented edge $\xvec{e}\in\Xi$ by its ideal vertices $(x,z)$ for $x,z\in\{v_1,\cdots,v_d\}$. The edge $e\in E$ underlying $\xvec{e}$ is shared by two triangles in $\mathcal{T}$ and hence arises as the diagonal of an anticlockwise oriented ideal quadrilateral $(x,t,z,y)$ with $x,t,z,y\in\{v_1,\ldots,v_d\}$. We assign the following $4$-tuple $\mathbf{F}(\vec{e})$ of flags:
\[
\mathbf{F}(\vec{e}):=(F(x),F(y),F(z),F(t)).
\]
\begin{thm}[{\cite[Theorem 9.1]{FG06}}]
\label{thm:posconf}
For the integers $d\geq 3$ and $n\geq 2$, the map
\begin{align*}
\rm{Conf}_d^+
&\to
\mathbb{R}_{>0}^{(n-1)(n-2)(d-2)/2} \times  \mathbb{R}_{>0}^{(n-1)(d-3)}\\
 (F(v_1),\cdots, F(v_d)) 
 &\mapsto 
 \left(\left(T_{i,j,k}\left(\mathbf{F}(\triangle)\right)  \right)_{i+j+k=n;\; \triangle \in \Theta}, \left(D_l\left(\mathbf{F}(\xvec{e})\right)  \right)_{l=1,\ldots,n-1;\; \xvec{e} \in \Xi} \right)
\end{align*}
is a real analytic diffeomorphism.
\end{thm}
The above proposition gives an algebraic characterization of $\rm{Conf}_d^+$: a $d$-tuple of flags is positive if and only if there exists an ideal triangulation $\mathcal{T}$ of a $d$-gon such that:
\begin{itemize}
\item
for every marked ideal triangle $\triangle\in\Theta$ and every triple of positive integers $(i,j,k)$ summing to $n$, the quantity $T_{i,j,k}\left(\mathbf{F}(\triangle)\right)$ is (strictly) positive and 
\item
for every oriented interior edge $\xvec{e}\in\Xi$ and every integer $l=1,\cdots,n-1$, the quantity $D_l\left(\mathbf{F}(\xvec{e})\right)$ is positive.
\end{itemize}

\subsection{Fock--Goncharov moduli spaces $\mathcal{X}_{\operatorname{PGL}_n,S_{g,m}}$ and $\mathcal{A}_{\operatorname{SL}_n,S_{g,m}}$}

We have already mentioned that Fock and Goncharov's version of higher Teichm\"uller theory \cite{FG06} is deep and applies to a \emph{very} broad context. We do not utilize the full force of their machinery, and concern ourselves with higher Teichm\"uller spaces of the form $\mathcal{X}_{\operatorname{PGL}_n,S_{g,m}}$ and $\mathcal{A}_{\operatorname{SL}_n,S_{g,m}}$, where $m\geq1$. The latter space $\mathcal{A}_{\operatorname{SL}_n,S_{g,m}}$ is concerned only with the representations with unipotent boundary monodromy, and it will be convenient to regard the boundaries of $S_{g,m}$ as either punctures or cusps. The former space $\mathcal{X}_{\operatorname{PGL}_n,S_{g,m}}$ is generically concerned with positive representations with loxodromic boundary monodromy (although uipotent is also permitted), and we generally regard the boundaries of $S_{g,m}$ as holes. We shall regard the boundaries of $S_{g,m}$ flexibly throughout this paper.

\medskip

\textbf{A reductionist approach: flags and decorated flags}

The flags and decorated flags are keys to Thurston's enhanced Teichm\"uller theory and Penner's decorated Teichm\"uller theory \cite{pennercoords} respectively --- the respective classical archetypes for Fock--Goncharov's $\mathcal{X}_{\operatorname{PGL}_n,S_{g,m}}$ and $\mathcal{A}_{\operatorname{SL}_n,S_{g,m}}$ moduli space theory. 
Let $m_p$ denote the set of punctures of the topological surface $S_{g,m}$. The central idea is that each element in $\mathcal{X}_{\operatorname{PGL}_n,S_{g,m}}$ or $\mathcal{A}_{\operatorname{SL}_n,S_{g,m}}$ moduli space may be described in terms of:
\begin{enumerate}
\item surface group representation $\rho$;
\item the flags (or decorated flags) invariant with respect to the holonomy of $\rho$ at each puncture in $m_p$.
\end{enumerate}
Crucially, Fock and Goncharov realized that all of these necessary data may be stored in terms of $\rho$ invariant flags (or decorated flags) assigned to all the lifts $\tilde{m}_p$ in the universal cover by deck transformations.\medskip

Fix a collection of $m$ based homotopy classes $\alpha_1,\ldots,\alpha_m\in\pi_1(S_{g,m},*)$, respectively winding the punctures $p_1,\ldots,p_m\in m_p$, oriented so that when $S_{g,m}$ is endowed with a geodesic-bordered hyperbolic metric, the surface lies on the left of the respective geodesic representatives of $\alpha_i$.


\begin{notation}
\label{notation:opensur}
In latter sections of this paper, we perform computations involving objects determined by ideal points (e.g.: ideal triangles), and we shall find it convenient to canonically identify $p_1,\ldots,p_m$ with the subset in $\partial_\infty\pi_1(S_{g,m})$ (Definition~\ref{definition:bin}) consisting of the respective fixed points of $\alpha_1,\ldots,\alpha_m$. If the monodromy matrix $\rho(\alpha_i)$ is
\begin{itemize}
\item
 unipotent, then $\alpha_i$ has precisely one fixed point on $\partial_\infty\pi_1(S_{g,m})$; 
 \item
 loxodromic, then it has precisely one attracting fixed point $\alpha_i^+$ and precisely one repelling fixed point $\alpha_i^-$. 
 \end{itemize}
 We choose $\alpha_i^-$ between the two different fixed points to relate to $p_i$. Such a choice corresponds to a choice of the spiraling direction of the ideal edge of $\mathcal{T}$ going into $\alpha_i$ when the puncture $p_i$ deforms into the hole $\alpha_i$. 
\end{notation}

\begin{defn}[{\cite[Definition 2.1]{FG06}, $\mathcal{X}$-moduli space $\mathcal{X}_{\operatorname{PGL}_n,S_{g,m}}$}]
\label{defn:FGX}
A \emph{framed $\operatorname{PGL}_n$-local system} on $S_{g,m}$ is a pair $(\rho,\xi)$ consisting of
\begin{itemize}
\item
a (surface group) representation $\rho\in\mathrm{Hom}(\pi_1(S_{g,m}),\operatorname{PGL}_n)$, and
\item
a map $\xi :m_p \rightarrow \mathcal{B}$, such that $\rho(\alpha_i)$ fixes the flag $\xi(p_i)\in\mathcal{B}$ for each $i=1,\ldots, m$.
\end{itemize}
Two framed $\operatorname{PGL}_n$-local systems $(\rho_1,\xi_1),(\rho_2,\xi_2)$ are equivalent if and only if there exists some $g\in\operatorname{PGL}_n$ such that $\rho_2=g\rho_1 g^{-1}$ and $\xi_2=g\xi_1$. We denote the moduli space (that is the space of equivalence classes) of all framed $\operatorname{PGL}_n$-local systems on $S_{g,m}$ by $\mathcal{X}_{\operatorname{PGL}_n,S_{g,m}}$.
\end{defn}

\begin{rmk}
Although the elements of the $\mathcal{X}$-moduli space $\mathcal{X}_{\operatorname{PGL}_n,S_{g,m}}$ are equivalence classes, we choose to conflate notation and denote them by $(\rho,\xi)$. We also adopt this convention later for elements of the $\mathcal{A}$-moduli space.
\end{rmk}

\begin{defn}[Farey set]\label{defn:fareyset}
Endow $S_{g,m}$ with an auxiliary complete hyperbolic metric as in Definition \ref{definition:bin}.
Let us assume for the moment that the surface $S=S_{g,m}$ is cusped, and let $\tilde{m}_p$ denote the set consisting of all the lifts $\tilde{m}_p$ of $m_p$ in the boundary at infinity for the universal cover $\tilde{S}$. We refer to $\tilde{m}_p$ as the \emph{Farey set}. 
\end{defn}

\begin{defn}[Equivariant map for framed local systems]
\label{definition:rhoequivariant}
The data contained in $(\rho,\xi)\in\mathcal{X}_{\operatorname{PGL}_n,S_{g,m}}$ is equivalent to that contained in the {\em $\rho$-equivariant map} $\xi_{\rho}:\tilde{m}_p\rightarrow\mathcal{B}$ induced by deck transformations ($\rho$-action) applied to $\xi$. When $\rho \in \mathrm{Pos}_n(S_{g,m})$, we identify $m_p$ with a finite subset of $\partial_\infty \pi_1(S)$ by Notation \ref{notation:opensur}, then by \cite[Theorem 1.14]{FG06}, the map $\xi_{\rho}$ extends uniquely to a $\rho$-equivariant map, still denoted by $\xi_{\rho}$, from $\partial_\infty \pi_1(S)$ to $\mathcal{B}$. 
\end{defn}

The analogous definition for the $\mathcal{A}_{\operatorname{SL}_n,S_{g,m}}$ moduli space is slightly more involved (only) when
$n$ is even. Let $T^1 S$ denote the unit tangent bundle over $S=S_{g,m}$ and fix an arbitrary point $\hat{x}\in T^1_xS\subset T^1 S$ over $x\in S$. Consider the short exact sequence for the unit tangent bundle fibration: 
\begin{align*}
1\rightarrow
\pi_1(T^1_xS)=\mathbb{Z}=\langle\sigma_S \rangle\rightarrow
\pi_1(T^1S,\hat{x})\rightarrow
\pi_1(S,x),
\end{align*}
where $\sigma_S$ is either of the two generators for $\pi_1(T_x^1 S)$, and define the quotient group $\bar{\pi}_1(S)=\bar{\pi}_1(S,x):=\pi_1(T^1S,\hat{x})/\langle\sigma_S^2\rangle$, and observe that $\bar{\pi}_1(S)$ is the central extension of $\pi_1(S,x)$ by $\mathbb{Z}/2\mathbb{Z}$. We fix the lifts $\hat{\alpha}_1,\ldots,\hat{\alpha}_m\in\bar{\pi}_1(S,x)$ respectively covering $\alpha_1,\ldots,\alpha_m$.
\begin{defn}[{\cite[Definition 2.4, page 38]{FG06}, $\mathcal{A}$-moduli space $\mathcal{A}_{\operatorname{SL}_n,S_{g,m}}$}]
\label{defn:FGA}
A \emph{decorated twisted $\operatorname{SL}_n$-local system on $S=S_{g,m}$} is a pair $(\bar{\rho},\bar{\xi})$ consisting of
\begin{itemize}
\item
a (twisted surface group) representation $\bar{\rho}\in\mathrm{Hom}(\bar{\pi}_1(S_{g,m}),\operatorname{SL}_n)$ with unipotent boundary monodromy, such that $\bar{\rho}(\bar{\sigma}_S)=(-1)^{n-1}\mathrm{Id}_{n\times n}$, and 
\item
a map $\bar{\xi}:m_p \rightarrow \mathcal{A}$, such that each $\bar{\rho}(\hat{\alpha}_i)$ fixes the decorated flag $\bar{\xi}(p_i)\in\mathcal{A}$.
\end{itemize}
Two decorated twisted $\operatorname{SL}_n$-local systems $(\bar{\rho}_1,\bar{\xi}_1),(\bar{\rho}_2,\bar{\xi}_2)$ are equivalent if and only if there exists some $g\in\operatorname{SL}_n$ such that $\bar{\rho}_2=g\bar{\rho}_1 g^{-1}$ and $\bar{\xi}_2=g\bar{\xi}_1$. We denote the moduli space of all decorated twisted $\operatorname{SL}_n$-local systems on $S_{g,m}$ by $\mathcal{A}_{\operatorname{SL}_n,S_{g,m}}$.
\end{defn}

\begin{rmk}[From $\mathcal{A}_{\operatorname{SL}_n,S_{g,m}}$ to $\mathcal{X}_{\operatorname{PGL}_n,S_{g,m}}$]
\label{remark:AX}
Consider the following natural projection maps:
\begin{itemize}
\item
$\mathbf{pr}:\operatorname{SL}_n\rightarrow \operatorname{PGL}_n$ killing the center, and
\item
$\pi:\mathcal{A}\rightarrow \mathcal{B}$ which forgets flag decorations.
\end{itemize}
Given a twisted surface group representation $\bar{\rho}:\bar{\pi}_1(S_{g,m})\to\operatorname{SL}_n$, the representation $\mathbf{pr}\circ\bar{\rho}$ necessarily kills off $\bar{\rho}(\bar{\sigma}_S)$ as it's in the center. Therefore, $\mathbf{pr}\circ\bar{\rho}$ induces a representation $\rho:\pi_1(S_{g,m})\to\operatorname{PGL}_n$. We intentionally conflate notation, and denote $\mathbf{pr}\circ\bar{\rho}$ simply as $\rho$, and refer to it as the \emph{underlying representation} for $\bar{\rho}$ or $(\bar{\rho},\bar{\xi})$. This in turn induces a projection map from $\mathcal{A}_{\operatorname{SL}_n,S_{g,m}}$ to $\mathcal{X}_{\operatorname{PGL}_n,S_{g,m}}$ given by 
\[
(\bar{\rho},\bar{\xi})\rightarrow (\mathbf{pr}(\bar{\rho}),\pi(\bar{\xi}) ):=(\rho,\xi) ,
\]
whose image consists of all framed $\operatorname{PGL}_n$-local systems with (only) unipotent boundary monodromy. 
\end{rmk}

\begin{defn}[Equivariant map for decorated twisted local systems]
\label{definition:Arep1}
Let $\mathbb{S}^1$ be the double cover of $\partial \mathbb{H}^2\cong \mathbb{RP}^1$. 
The data contained in a pair $(\bar{\rho},\bar{\xi})\in \mathcal{A}_{\operatorname{SL}_n,S_{g,m}}$ is equivalent to that contained in the {\em $\bar{\rho}$-equivariant map} $\bar{\xi}_{\bar{\rho}}$ from the double cover $\bar{\tilde{m}}_p$ of the Farey set in $\mathbb{S}^1$ to the principle affine space $\mathcal{A}$ induced by the $\bar{\rho}$ action applied to $\bar{\xi}$.
\end{defn}

\begin{defn}[Positivity for $\bar{\rho}$-equivariant maps]
\label{definition:Arep2}
Let $s$ be the antipodal involution on $\mathbb{S}^1$, then $\bar{\xi}_{\bar{\rho}}(sx)=\bar{\rho}(\bar{\sigma}_S)\bar{\xi}_{\bar{\rho}}(x)$ for any $x$. By \cite[Definition 8.5, page 121-122]{FG06}, the notion of \emph{positivity} can also be defined for maps from $s$-invariant subsets $\overline{\mathbb{X}}$ of $\mathbb{S}^1$ to $\mathcal{A}$, where the role of cyclically anticlockwise ordered $d$-tuples of points $x_1,\ldots,x_d\in\mathbb{X}=\overline{\mathbb{X}}/s\subset\mathbb{RP}^1$ (Definition~\ref{dfn:pmaps}) is replaced by coherent lifts of such $d$-tuples. To clarify: we reinterpret any such cyclically anticlockwise ordered $d$-tuple $x_1,\ldots,x_d$ of points as an embedded path in $\mathbb{RP}^1$ traversing from $x_1$ to $x_d$, a \emph{coherent lift} of such a $d$-tuple is a cyclically anticlockwise ordered $d$-tuple of points in $\overline{\mathbb{X}}$ corresponding to a lift of the path for $x_1,\ldots,x_d$. Crucially, any $\bar{\rho}$-equivariant map $\bar{\xi}_{\bar{\rho}}$ is positive if and only if the corresponding $\rho$-equivariant map $\xi_{\rho}$ is positive.
\end{defn}

\subsection{Fock--Goncharov $\mathcal{A}$-coordinates}
\label{subsection:FGcoor}

We now introduce coordinates for $\mathcal{X}_{\operatorname{PGL}_n,S_{g,m}}$ and $\mathcal{A}_{\operatorname{SL}_n,S_{g,m}}$ moduli spaces. Going forward, we only consider $\xi_\rho$ ($\bar{\xi}_{\bar{\rho}}$ resp.) which satisfy the following generic position condition: any pairwise distinct triple $(x,y,z)$ is mapped to a triple $(\xi_\rho(x), \xi_\rho(y), \xi_\rho(z))$ of flags ($(\bar{\xi}_{\bar{\rho}}(x), \bar{\xi}_{\bar{\rho}}(y), \bar{\xi}_{\bar{\rho}}(z))$ of decorated flags resp.) in generic position (Definition~\ref{defn:genpos}).

\begin{defn}[Ideal triangulation]
Let $m_p$ denote the set of punctures of $S_{g,m}$. An \emph{ideal triangulation} $\mathcal{T}$ of $S_{g,m}$ is a maximal collection of (unoriented) essential arcs which join the elements of $m_p$, such that these arcs are: 
\begin{itemize}
\item
pairwise disjoint on the interior of $S_{g,m}$ and 
\item
non-homotopic with respect to homotopies of $S_{g,m}$. 
\end{itemize}
We regard ideal triangulations up to homotopy. Moreover, we identify an ideal triangulation $\mathcal{T}$ with the graph $(V_{\mathcal{T}},E_{\mathcal{T}})$, where $V_\mathcal{T}=m_p$ is the set of vertices of $\mathcal{T}$ and $E_\mathcal{T}$ is the set of (unoriented) edges of $\mathcal{T}$.
\end{defn}

\begin{defn}[$n$-Triangulation]
\label{definition:ntri}
Given an ideal triangulation $\mathcal{T}= (V_{\mathcal{T}}, E_{\mathcal{T}})$ of $S_{g,m}$, we define the \emph{$n$-triangulation} $\mathcal{T}_n$ of $\mathcal{T}$ to be the triangulation of $S_{g,m}$ obtained by subdividing each triangle of $\mathcal{T}$ into $n^2$ triangles (as per Figure~\ref{fig:ntriang}). We also identify $\mathcal{T}_n$ with the graph $(V_{\mathcal{T}_n}, E_{\mathcal{T}_n})$, just as we did for ideal triangulations.
\end{defn}

\begin{notation}[Vertex notation]
We define the following vertex sets.
\begin{align*}
 \mathcal{I}_n:=
\left\{
\begin{array}{c|c}
V\in V_{\mathcal{T}_n} \setminus V_{\mathcal{T}}
&
V\text{ lies on an edge } e\in E_{\mathcal{T}} 
\end{array}
\right\}\text{ and }
\mathcal{J}_n:= V_{\mathcal{T}_n} \setminus \left( V_{\mathcal{T}} \cup \mathcal{I}_n\right).
\end{align*}
We also adopt the following vertex labeling conventions:

\begin{itemize}
\item
we denote a vertex $V\in\mathcal{I}_n$ on an oriented ideal edge $(x,y)$ by $v_{i,n-i}^{x,y}=v_{n-i,i}^{y,x}$, where $i\geq1$ is the least number of $E_{\mathcal{T}_n}$ edges from $V$ to $y$ (see Figure~\ref{fig:ntriang}).
\item
we denote a vertex $V\in\mathcal{I}_n \cup \mathcal{J}_n$ on a triangle $(x,y,z)$ by $v_{i,j,k}^{x,y,z}$, where $i\geq0$, $j\geq0$ and $k=n-i-j\geq0$ respectively denote: the least number of $E_{\mathcal{T}_n}$ edges from $V$ to $\overline{yz}$, from $V$ to $\overline{xz}$ and from $V$ to $\overline{xy}$ (see Figure~\ref{fig:ntriang}). 
\end{itemize}
\end{notation}

\begin{figure*}[h!]
\includegraphics[scale=0.75]{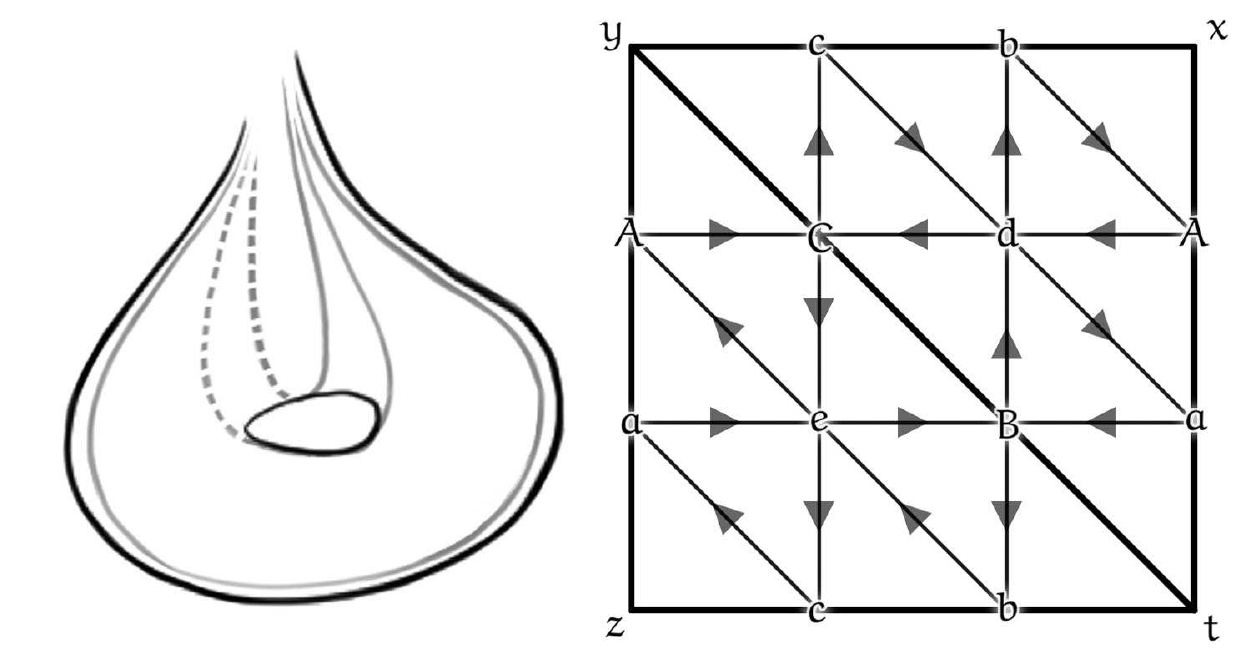}
\caption{Left: an ideal triangulation for $S_{1,1}$. Right: a (lift of a) $3$-triangulation $\mathcal{T}_3$ for $S_{1,1}$, with opposite edges identified. Edges endowed with arrows constitute edges of the quiver $\Gamma_{\mathcal{T}_3}$. For example: the edge vertices $v_{1,2}^{x,y}$ and $v_{2,1}^{z,t}$ identify to the same vertex $c$ when projected to $S_{1,1}$. The interior vertices $v_{1,1,1}^{x,y,t}$ and $v_{1,1,1}^{y,z,t}$ respectively correspond to the vertices $d$ and $e$.}
\label{fig:ntriang}
\end{figure*}

\begin{defn}[Quiver]
\label{definition:quiver}
Consider the largest subgraph of $\mathcal{T}_n$ with vertex set $\mathcal{I}_n\cup \mathcal{J}_n$. By placing orientations on this graph as per Figure~\ref{fig:ntriang}, we obtain a \emph{quiver} $\Gamma_{\mathcal{T}_n}$.
\end{defn}

Quivers are combinatorially useful both in defining Fock--Goncharov coordinates, as well as in describing their coordinate transformations. We now describe Fock--Goncharov $\mathcal{A}$-coordinates.

\begin{defn}[{\cite[\S9]{FG06}} $\mathcal{A}$-coordinates]
\label{defn:FGAcoordinate}
Fix an ideal triangulation $\mathcal{T}$ of $S_{g,m}$ and its $n$-triangulation $\mathcal{T}_n$. Let $\tilde{\mathcal{T}}$ be all the lifts of $\mathcal{T}$ in the universal cover $\tilde{S}_{g,m}$. Given a vertex $V=v_{i,j,k} \in \mathcal{I}_n \cup \mathcal{J}_n$, let $(f,g,h)$ be a marked ideal triangle in $\tilde{\mathcal{T}}$ containing a lift of $V$. For $(\bar{\rho},\bar{\xi})\in\mathcal{A}_{\operatorname{SL}_n,S_{g,m}}$, choose bases 
\begin{align*}
(f_1,...,f_{n}),\; (g_1,...,g_{n}),\; (h_1,...,h_{n})
\end{align*}
for the respective decorated flags $\bar{\xi}_{\bar{\rho}}(\bar{f})$, $\bar{\xi}_{\bar{\rho}}(\bar{g})$, $\bar{\xi}_{\bar{\rho}}(\bar{h})$, where $(\bar{f},\bar{g},\bar{h})$ is a coherent lift of the marked ideal triangle $(f,g,h)$ in the double cover $\partial_\infty\bar{\pi}_1(S)\subset \mathbb{S}^1$ of $\partial_\infty \pi_1(S)$ (Definitions~\ref{definition:Arep1} and \ref{definition:Arep2}). The \emph{vertex function} for $V$ is defined by
\[
\Delta_{V}(\bar{\rho},\bar{\xi}):=\Delta_{V}:= \Delta\left(f^i \wedge g^j \wedge h^k\right).
\]
The \emph{(Fock--Goncharov) $\mathcal{A}$-coordinate} $A_{V}(\bar{\rho},\bar{\xi}):=A_{V}$ is equal to $\Delta_V$ up to sign.
\end{defn}

\begin{rmk}
The choice of sign for $A_V$ is technical and dependent upon a choice of spin structure on $S_{g,m}$ \cite{FG06}. In this paper, we focus on the part of $\mathcal{A}_{\operatorname{SL}_n,S_{g,m}}$ where all the $\mathcal{A}$-coordinate are positive. In these cases, we have $A_{V}=|\Delta_{V}|$. Hence the complete definition of $A_{V}$ is not necessary for the content of the paper.
\end{rmk}

\subsection{Fock--Goncharov $\mathcal{X}$-coordinates}

There are two types of Fock--Goncharov $\mathcal{X}$-coordinates respectively corresponding to edge functions and triple ratios. The former are labeled by vertices in $\mathcal{I}_n$, correspond to degree four vertices in the quiver $\Gamma_{\mathcal{T}_n}$, and generalize Thurston's shear coordinate \cite{MR1413855, Thu98}. The latter are labeled by vertices in $\mathcal{J}_n$ and are degree $6$ vertices in $\Gamma_{\mathcal{T}_n}$.

\begin{defn}[{\cite[\S9]{FG06}} $\mathcal{X}$-coordinates]
\label{defn:FGXcoordinate}
We define one $\mathcal{X}$-coordinate for each vertex in $\mathcal{I}_n\cup\mathcal{J}_n$. For a vertex $V\in\mathcal{I}_n$, let $(x,y)$ denote an oriented edge in $E_{\tilde{\mathcal{T}}}$ containing a lift $\tilde{V}=v_{x,y}^{n-i,i}$ of $V$. Further let $\xvec{xyz}$ and $\xvec{xty}$ denote the two (anticlockwise) oriented ideal triangles in $\tilde{\mathcal{T}}$ which contain the edge $\overline{xy}$. The \emph{(Fock--Goncharov) $\mathcal{X}$-coordinate} for $V$, evaluated at $(\rho,\xi)\in\mathcal{X}_{\operatorname{PGL}_n,S_{g,m}}$, is defined as the edge function in Definition \ref{definition:edgefunction}:
\[
X_V(\rho,\xi):=X_V:=D_i(x,y,z,t):=D_i(\xi_\rho(x),\xi_\rho(y),\xi_\rho(z),\xi_\rho(t)).
\]
For a vertex $V\in\mathcal{J}_n$, let $(f,g,h)$ be a marked (anticlockwise) oriented ideal triangle in $\tilde{\mathcal{T}}$ containing a lift $\tilde{V}=v_{i,j,k}^{f,g,h}$ of $V$. The \emph{(Fock--Goncharov) $\mathcal{X}$-coordinate} $X_V$, evaluated at $(\rho,\xi)\in\mathcal{X}_{\operatorname{PGL}_n,S_{g,m}}$, is defined as the triple ratio in Definition \ref{definition:tripleratio}:
\[
X_V(\rho,\xi):=X_V:=T_{i,j,k}(f,g,h):=T_{i,j,k}(\xi_\rho(f),\xi_\rho(g),\xi_\rho(h)).
\]
\end{defn}

As with configuration spaces, the $\mathcal{X}$-coordinates are crucial examples of projective invariants as rational functions of $\mathcal{A}$-coordinates and define rational functions on the $\mathcal{X}$-moduli space.

\subsection{Positivity for $\mathcal{X}$ and $\mathcal{A}$-moduli spaces}

\begin{defn}[\cite{FG06} Positive higher Teichm\"uller spaces]
\label{defn:positive}
The \emph{positive} Fock--Goncharov higher Techm\"uller space $\mathcal{A}_n(S_{g,m})$ and $\mathcal{X}_n(S_{g,m})$ are the respective subsets of $\mathcal{A}_{\operatorname{SL}_n,S_{g,m}}$ and $\mathcal{X}_{\operatorname{PGL}_n,S_{g,m}}$ consisting of points which are positive in every coordinate with respect to some $\mathcal{A}$ or $\mathcal{X}$-coordinate chart. 
\end{defn}

One key advantage of the Fock--Goncharov approach to higher Teichm\"uller theory is that we can explicitly write down rational functions specifying the transition maps between coordinate patches. This aspect of the story is an example of the powerful theory of cluster ensembles \cite[\S10]{FG06}. We do explicitly utilize these coordinate transformations in our derivation of McShane identities --- especially the $\mathcal{A}$-coordinate transformations. It is worth noting that these coordinate changes  are always \emph{positive rational maps} in the sense that they are fractions of two polynomials (of $\mathcal{A}$-coordinates) with positive coefficients, and hence send positive coordinates to positive coordinates.

\begin{defn}[Flips for $n=3$]
\label{definition:flip}
Consider two adjacent ideal triangles $\overline{xyt}$ and $\overline{yzt}$ sharing a common edge $\overline{yt}$. A flip along $\overline{yt}$ produces a new ideal triangulation by replacing $\overline{yt}$ with $\overline{xz}$. For $n=3$, we now explicitly write down the corresponding coordinate change for such a flip (See Figure~\ref{Figure:flipgeneral}): denote the $\mathcal{A}$-coordinates for $\mathcal{A}_{\operatorname{SL}_3,S_{1,1}}$ by $\{a,b,c,d,r,s,q,w\}$. After successive mutations at the vertices corresponding to $r, s , p , q$, we obtain new coordinates $\{a,b,c,d,r',s',q',w'\}$ given by: 
\begin{align}
\label{equation:mutation3}
r'= \tfrac{bq+cw}{r},\quad s'= \tfrac{aw+dq}{s},\quad w'= \tfrac{as'+cr'}{w},\quad q'= \tfrac{b r'+d s'}{q}.
\end{align}
\begin{figure}[h!]
\includegraphics[scale=0.2]{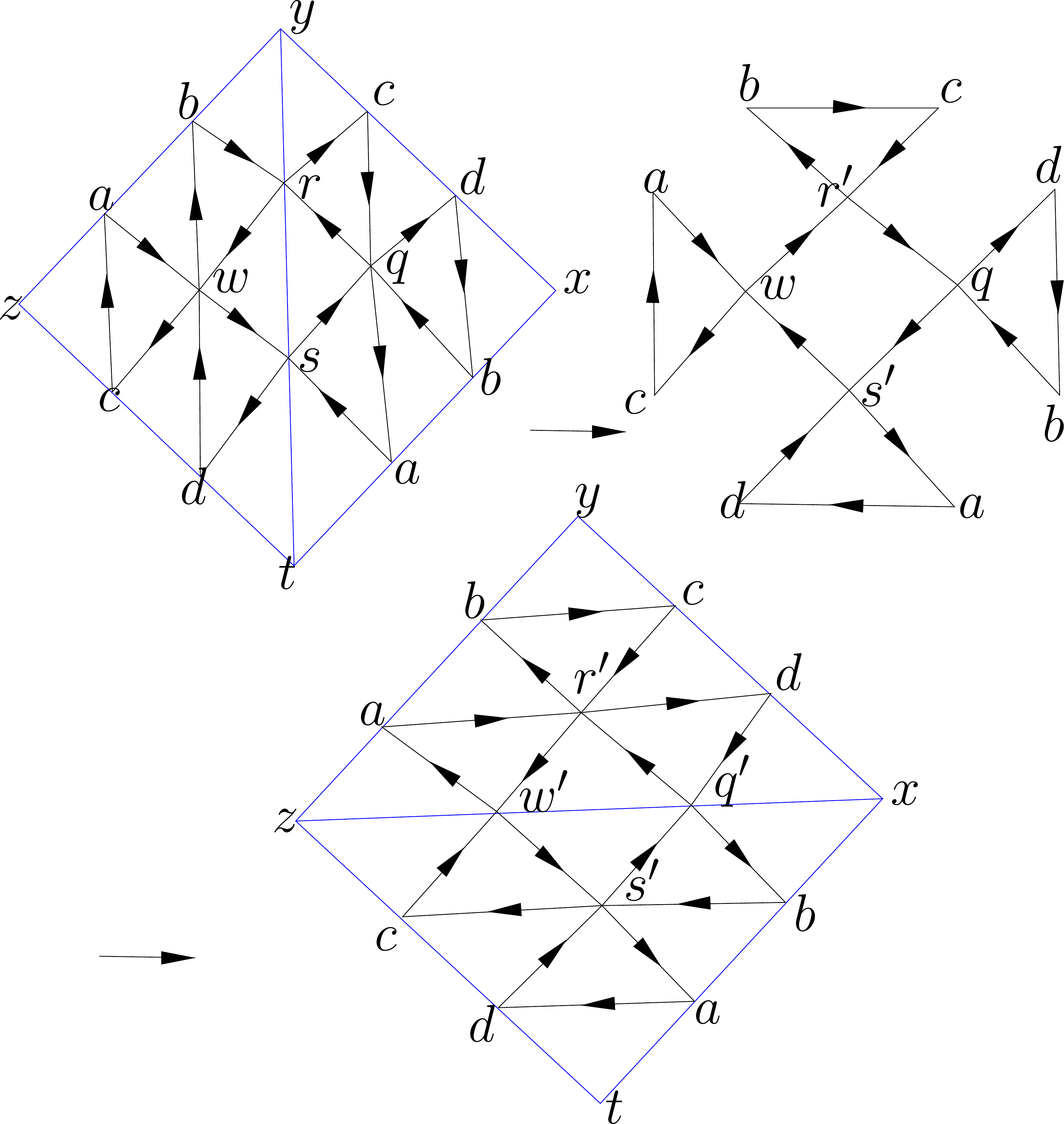}
\caption{For $\mathcal{A}_{\operatorname{SL}_3,S_{1,1}}$, given an ideal triangulation $\mathcal{T}$ with $V_{\mathcal{T}}=\{x,y,z,t\}$ and $E_{\mathcal{T}}=\{\overline{xy},\overline{yt},\overline{tx},\overline{yz}, \overline{zt}\}$, we have its $n$-triangulation $\mathcal{T}_n$.}
\label{Figure:flipgeneral}
\end{figure}
\end{defn}
For general $n$ as well as for the $\mathcal{X}$-moduli space, the coordinate changes for a flip are described in \cite[\S10.3, pg. 147]{FG06}.\medskip

\subsubsection{Relation to positive representation varieties}

The set consisting of the underlying $\operatorname{PGL}_n(\mathbb{R})$-representations (see Definitions~\ref{definition:Arep1} and \ref{definition:Arep2}) of decorated twisted $\operatorname{SL}_n$-local systems in $\mathcal{A}_n(S_{g,m})$ is precisely the unipotent-bordered $\operatorname{PGL}_n(\mathbb{R})$-positive representation variety $\mathrm{Pos}_{n}^u(S_{g,m})$. Moreover, the set consisting of the underlying $\operatorname{PGL}_n(\mathbb{R})$-representations of positive $\operatorname{PGL}_n$-local systems in $\mathcal{X}_n(S_{g,m})$ is precisely the $\operatorname{PGL}_n(\mathbb{R})$-positive representation variety $\mathrm{Pos}_{n}(S_{g,m})$. The map from $\mathcal{X}_n(S_{g,m})$ to $\mathrm{Pos}_{n}(S_{g,m})$ which takes $(\rho,\xi)$ to $\rho$ is a finite to one map. In fact, it is generically a finite covering map in the following sense:

\begin{prop}[{\cite[Proposition 7.1]{FG06} and \cite[Proposition 10.2.1.1]{LM09}}]
Let $\mathcal{W}$ be the Weyl group of $\operatorname{PGL}_n(\mathbb{R})$. For any $\rho \in \mathrm{Pos}_{n}^h(S_{g,m})$, there exists $|\mathcal{W}|^m=(n!)^m$  many lifts of $\rho$ in $\mathcal{X}_n(S_{g,m})$, where the lifts are parameterized by $\mathcal{W}^m$. 
\end{prop}

\begin{defn}[Canonical lift]
\label{definition:ratioperiod}
For any $\rho \in \mathrm{Pos}_n^h(S_{g,m})$ with loxodromic boundary monodromy, there is a \emph{canonical lift} $(\rho,\xi)\in \mathcal{X}_n(S_{g,m})$ with the following property: 

for any peripheral $\delta \in \pi_1(S_{g,m})$ around a boundary component $p\in m_p$ of $S_{g,m}$ oriented such that $S_{g,m}$ is on the left of $\delta$, 
\begin{itemize}
\item by Theorem \ref{theorem:loxo} there exists a lift of $\rho(\delta)$ into $\operatorname{SL}_n(\mathbb{R})$ with eigenvectors $\delta_1,\cdots,\delta_n$ and corresponding eigenvalues $\lambda_1,\cdots,\lambda_n$ labeled in decreasing magnitude
\[
\lambda_1>\ldots>\lambda_n>0;
\] 
\item $(\delta_n,\cdots,\delta_1)$ is a basis for the flag $\xi(p)=\xi_\rho(\delta^-)$ (Notation \ref{notation:opensur}).
\end{itemize}
\end{defn}

Note that $(\delta_1,\cdots,\delta_n)$ is also a basis for $\xi_\rho(\delta^+)$. Further note that every other lift of $\rho$ can be obtained by the group action of $\mathcal{W}^m$ given by permuting the basis $(\delta_n,\cdots,\delta_1)$ for the flag $\xi(p)$ for each of the $m$ boundaries.

\clearpage

\section{Properties of projective invariants}
\label{sec:xcoords}

\subsection{Uniform boundedness of the triple ratio}

Let $S=S_{g,m}$ be a topological surface of genus $g$ with $m$ holes with negative Euler characteristic. Given a positive representation $\rho \in \mathrm{Pos}_n(S)$, recall that we define $\partial_\infty\pi_1(S)$ (Definition~\ref{definition:bin}) for $\rho$ by taking the boundary at infinity for an auxiliary complete hyperbolic metric $h_\rho$  (Definition~\ref{definition:aux}) chosen so as to satisfy the following: 
\begin{itemize}
\item
when $\rho$ admits at least one boundary with non-unipotent monodromy, the auxiliary hyperbolic metric $(S,h_\rho)$ has geodesic boundary components (as well as possibly cusps), and the boundary at infinity for $\rho$ is homeomorphic to the Cantor set;
\item
for all other positive $\rho$, the auxiliary hyperbolic metric $(S,h_\rho)$ is cusped (or closed) surface, and the boundary at infinity $\partial_\infty\pi_1(S)$ for $\rho$ is homeomorphic to $\mathbb{RP}^1$.
\end{itemize}
In either case, the inclusion of $\partial_\infty\pi_1(S)$ as a subset of $\partial\mathbb{H}^2$ imposes an anticlockwise cyclic ordering on $\partial_\infty\pi_1(S)$.

\begin{defn}[Set of marked ideal triangles]
We define the set of marked (oriented) ideal triangles on the universal cover $\tilde{S}$ of $S$ by:
\begin{align*}
\mathit{Tri}(\tilde{S}):=
\mathit{Tri}(\tilde{S},\tilde{h}_\rho):=
\left\{
(a,b,c)\in(\partial_\infty\pi_1(S))^3\; 
\begin{array}{|c}
a,b,c\text{ are distinct elements arranged}\\
\text{in anticlockwise order along $\partial_\infty\pi_1(S)$}
\end{array}
\right\}.
\end{align*}
We define the set of \emph{ideal triangles on $S$} as
\begin{align*}
\mathit{Tri}(S):=\mathit{Tri}(S,h_\rho):=\mathit{Tri}(\tilde{S})/\pi_1(S),
\end{align*}
where $\pi_1(S)$ acts diagonally on $\mathit{Tri}(\tilde{S})$. Moreover, we denote the $\pi_1(S)$ orbit of $(a,b,c)$ representing an element in $\mathit{Tri}(S,h_\rho)$ by $[a,b,c]_\rho$. We regard each $[a,b,c]_\rho$ as an immersed marked ideal triangle on $S$ and denote its anticlockwise-oriented sides by $[a,b]_\rho$, $[b,c]_\rho$ and $[c,a]_\rho$.
\end{defn}

\begin{fact}[e.g.: {\cite[Section~4.1, pg.~7]{pressure}}]
When $S$ is closed, the set $\mathit{Tri}(S)$ of marked (oriented) ideal triangles on $S$ is homeomorphic to the unit tangent bundle $T^1S$  on $S$.
\end{fact}

\begin{defn}[$k$-intersecting ideal triangle]
Given an auxiliary complete hyperbolic metric $h_\rho$ on $S$ for a positive representation $\rho\in\mathrm{Pos}_n(S)$, we say that an ideal triangle $[a,b,c]_\rho$ on $S$ is \emph{$k$-intersecting} if the unique geodesic representatives of each of the three sides $[a,b]_\rho,[b,c]_\rho,[c,a]_\rho$ of $[a,b,c]_\rho$ on $(S,h_\rho)$ have:
\begin{itemize}
\item
at most $k$ self-intersections, and
\item
at most $k$ pairwise intersections.
\end{itemize}
We denote the set of $k$-intersecting ideal triangle on $S$ by $\mathit{Tri}_k(S):=\mathit{Tri}_k(S,h_\rho)$.
\end{defn}

The goal of this subsection is to prove the following:

\begin{thm}[Triple-ratio boundedness]
\label{thm:bounded}
Let $S=S_{g,m}$ be a surface with negative Euler characteristic. For a positive framed $\operatorname{PGL}_n$-local system $(\rho,\xi)\in \mathcal{X}_n(S)$ where $\rho\in\mathrm{Pos}_n(S)$, let us consider the unique $\rho$-equivariant map $\xi_\rho:\partial_\infty \pi_1(S)\rightarrow \mathcal{B}$ associated to $(\rho,\xi)$ by Definition~\ref{definition:rhoequivariant}. For any triple of positive integers $(i_0,j_0,k_0)$ satisfying $i_0+j_0+k_0=n$, by Definition \ref{definition:tripleratio}, the triple ratio function $T^{\xi_\rho}:\mathit{Tri}(S)\rightarrow\mathbb{R}_{>0}$ of the form:
\begin{align}
T^{\xi_\rho}(x,y,z):=T_{i_0,j_0,k_0}(\xi_\rho(x),\xi_\rho(y),\xi_\rho(z)),
\end{align}
restricted to the set $\mathit{Tri}_k(S)$ of $k$-intersecting ideal triangles on $S$, is bounded within some closed interval $[T^{\xi_\rho}_{\min}(k),T^{\xi_\rho}_{\max}(k)]\subset\mathbb{R}_{>0}$.
\end{thm}

\begin{rmk}
We need not only $\rho$ but also $\xi$ to induce $\xi_\rho:\partial_\infty \pi_1(S)\rightarrow \mathcal{B}$ which provides enough data to define $T^{\xi_\rho}$.
The above theorem can be also stated for any $ \rho \in \mathrm{Pos}_n^h(S)\cup \mathrm{Pos}_n^u(S)$, since for $ \rho \in \mathrm{Pos}_n^h(S)$ we choose a canonical lift $(\rho,\xi)\in \mathcal{X}_n(S)$ (Definition \ref{definition:ratioperiod}), and for $ \rho \in \mathrm{Pos}_n^u(S)$ (including $S$ being a closed surface) we have a unique lift $(\rho,\xi)\in \mathcal{X}_n(S)$.

It is clear that $T^{\xi_\rho}$ defines a strictly positive function on $\mathit{Tri}(\tilde{S})$. However, triple ratios are projective invariants and hence invariant with respect to the diagonal action of the fundamental group $\pi_1(S)$ on $\mathit{Tri}(\tilde{S})$ and hence $T^{\xi_\rho}$ descends to a well-defined continuous function on $\mathit{Tri}(S)$.

\end{rmk}

\begin{rmk}[Boundedness under flips]
Consider the mapping class group orbit\footnote{Or more generally, any group/groupoid of which the mapping class group is a finite-index subgroup, e.g.: the Ptolemy groupoid.} of an arbitrary point in the $\mathcal{X}$-moduli space $\mathcal{X}_{\operatorname{PGL}_n, S_{g,m}}$, Theorem~\ref{thm:bounded} then asserts that the triple ratio coordinates of this orbit of points are all bounded away from $0$ and $\infty$. In other words, from the cluster dynamic point of view, triple ratios are bounded under the flips. We believe this to be a novel observation.
\end{rmk}

Let us first consider the special case when $S$ is closed. The following argument comes from Fran\c{c}ois Labourie and also independently from Tengren Zhang:

\begin{prop}[Labourie, Zhang]
When $S=S_{g,0}$ is a closed surface, the triple ratio function $T^{\xi_\rho}$ is bounded within some closed interval $[T^{\xi_\rho}_{\min},T^{\xi_\rho}_{\max}]\subset\mathbb{R}_{>0}$.
\end{prop}

\begin{proof}
The proposition follows from the simple fact that the domain $\mathit{Tri}(S)\cong T^1 S$ is compact when $S$ is compact.
\end{proof}

We now proceed onto the general case where $(S,h_\rho)$ has geodesic boundary holes and punctures. Our proof in this case is also based on compactness, with the adjustment that the role of $\mathit{Tri}(S_{g,0})$ is supplanted by $\mathit{Tri}_k(S_{g,m})$.

\begin{prop}\label{thm:kcompact}
For $S=S_{g,m}$ with negative Euler characteristic, let $\rho:\pi_1(S)\rightarrow \operatorname{PGL}_n(\mathbb{R})$ be a positive representation and let $h_\rho$ denote an auxiliary hyperbolic metric for $\rho$ (see Definition~\ref{definition:aux}). The set $\mathit{Tri}_k(S)$ of $k$-intersecting ideal triangles on $S$ is a compact subset of $\mathit{Tri}(S)$.
\end{prop}

\begin{proof}
Let us first consider the case when none of the boundary monodromies of $\rho$ are unipotent. In this case, $(S,h_\rho)$ is a hyperbolic surface with $m$ closed geodesic boundary components and no cusps. Let $(dS,dh_\rho)$ denote the closed orientable double of $(S,h_\rho)$, then the inclusion isometric embedding
\[
\iota:(S,h_\rho)\hookrightarrow (dS,dh_\rho)
\] 
induces an embedding of ideal triangles $\iota_*:\mathit{Tri}(S)\hookrightarrow\mathit{Tri}(dS)$. In particular, observe that $\iota_*(\mathit{Tri}(S))$ is precisely the set of ideal triangles on $dS$ which lie completely on $\iota(S)$. The intersection of the closed condition of being contained on $\iota(S)$ and the closed condition of being a $k$-intersecting ideal triangle is a closed condition. To see that being a $k$-intersecting ideal triangle is a closed condition, we show that the complement is an open condition. Consider the lifts of the relevant geodesics to the universal cover which has at least $(k+1)$ self-intersections or at least $(k+1)$ pairwise intersections. Each intersection depends purely on how their end points are configured. Hence small perturbations in the vertices of an ideal triangle can only increase the number of intersection points and therefore describes an open condition. Thus, $\mathit{Tri}_k(S)$ is a closed subset of $\mathit{Tri}(dS)$. Since $\mathit{Tri}(dS)$ is compact, the set $\mathit{Tri}_k(S)$ must be compact.\medskip

We now consider the case when one or more of the boundary monodromies of $\rho$ are unipotent, which is to say that the auxiliary metric $h_\rho$ is cuspidal at those corresponding boundary components. The boundary at infinity $\partial_\infty\pi_1(S,h_\rho)$ for the universal cover $(\tilde{S},\tilde{h}_\rho)$ of $(S,h_\rho)$ is equal to the set of limit points with respect to the holonomy representation action of $\pi_1(S,h_\rho)$ on $\overline{\mathbb{H}^2}$. The second main theorem of \cite[pg. 207]{Flo80} tells us that there is a $\pi_1(S)$-equivariant map from the Floyd boundary of $\pi_1(S)$ to the boundary at infinity $\partial_\infty\pi_1(S,h_\rho)$ which is injective everywhere except over parabolic fixed points, where the map is $2:1$. Moreover, the Floyd boundary of $\pi_1(S)$ is homeomorphic to the Gromov boundary of the hyperbolic group $\pi_1(S)$ \cite[see, e.g., Corollary~2.3]{BK09}.
The Gromov boundary of $\pi_1(S)$ naturally identifies with $\partial_\infty\pi_1(S,h_0)$ where $h_0$ is a geodesic bordered hyperbolic metric on  $S$, and we therefore obtain a continuous $\pi_1(S)$-equivariant map
\[
\mathit{u}:\partial_\infty\pi_1(S,h_0)\to\partial_\infty\pi_1(S,h_\rho).
\]
The image of $\mathit{u}$ is:
\begin{itemize}
\item
closed, since $\partial_\infty\pi_1(S,h_0)$ is compact, and
\item
dense in $\partial_\infty\pi_1(S,h_\rho)$, since parabolic fixed points are dense.
\end{itemize}
Therefore, $\mathit{u}$ is surjective and defines a quotient map from $\partial_\infty\pi_1(S,h_0)$ to $\partial_\infty\pi_1(S,h_\rho)$ which identifies the $2$ lifts of each parabolic fixed point in $\partial_\infty\pi_1(S,h_\rho)$. In fact, $\mathit{u}$ in turn induces a continuous map $\mathit{u}_*:\mathit{Tri}_k(S,h_0)\rightarrow\mathit{Tri}_k(S,h_\rho)$. It is crucial to note that $\mathit{u}_*$ is well-defined on $\mathit{Tri}_k(S,h_0)$, for any $k$, but does not extend to a map $\mathit{Tri}(S,h_0)\to\mathit{Tri}(S,h_\rho)$ because 
\[
\mathit{u}_*([a,b,c]_0)=[u(a),u(b),u(c)]_\rho
\] 
does not produce a triangle if $u(a),u(b)$ and $u(c)$ are not pairwise distinct. This cannot happen to a triangle $[a,b,c]_0\in\mathit{Tri}_k(S,h_0)$: if (without loss of generality) $a$ and $b$ are the two endpoints of a lift of a boundary geodesic of $(S,h_0)$, then the geodesics $[b,c]_0$ and $[c,a]_0$ spiral toward the same boundary in opposite directions and hence intersect infinitely often. Finally, since $\mathit{Tri}_k(S,h_\rho)$ is the image of a compact set, it is compact.
\end{proof}

\begin{proof}[Proof of Theorem~\ref{thm:bounded} for $S_{g,m}$, $m\geq1$ case.]
The triple ratio function $T^{\xi_\rho}:\mathit{Tri}(S)\rightarrow\mathbb{R}_{>0}$ restricts to a positive continous function $T^{\xi_\rho}|_{\mathit{Tri}_k(S)}$ defined over the compact set $\mathit{Tri}_k(S)$. We then take $T^{\xi_\rho}_{\min}(k)$ and $T^{\xi_\rho}_{\max}(k)$ to be the respective minimum and the maximum for the restricted function $T^{\xi_\rho}|_{\mathit{Tri}_k(S)}$.
\end{proof}

\begin{rmk}
Our proof is essentially topological, and so Theorem~\ref{thm:bounded} holds true even for $(\rho,\xi)\in \mathcal{X}_3(S)$ where $\rho$ is the holonomy representation of a convex real projective surface with quasihyperbolic boundary monodromy (see, e.g., \cite{Mar12}).
\end{rmk}

\subsection{$n$-Fuchsian rigidity conditions}
\label{sec:fuchsiancharacterization}

We now shift from the study of triple ratio boundedness to that of Fuchsian rigidity. 

\begin{rmk}
\label{remark:fuchsianvero}
A $n$-Fuchsian representation $\rho$ is a composition of the discrete faithful homomorphism $\rho_0$ from $\pi_1(S)$ to $\operatorname{PSL}_2(\mathbb{R})$ with the unique irreducible representation $\iota$ from $\operatorname{PSL}_2(\mathbb{R})$ to $\operatorname{PGL}_n(\mathbb{R})$. There exists a $\rho_0$-equivariant map $\xi_0$ from $\partial_\infty \pi_1(S)$ to $\mathbb{RP}^1$. The Veronese curve is $v:\mathbb{RP}^1\rightarrow \mathbb{RP}^{n-1}$
\[
[x,y]\mapsto [x^{n-1}, x^{n-2}y,\cdots, y^{n-1}].
\]
The unique irreducible representation $\iota$ is defined by 
\[v(M \cdot [x,y]^T)=\iota(M)\cdot [x^{n-1}, x^{n-2}y,\cdots, y^{n-1}]^T.\]
Thus the $\rho$-equivariant map $v\circ \xi_0$ for $\rho=\iota\circ \rho_0$ is a reparameterization of the Veronese curve. 
\end{rmk}

The following lemma is probably well-known to the experts, but we do not find a proper reference.
\begin{lem}
\label{lemma:nfuchrig}
For a $n$-Fuchsian representation $\rho=\iota\circ \rho_0$, there exists a lift $(\rho,\xi)\in\mathcal{X}_n(S_{g,m})$ such that all the triple ratios of $(\rho,\xi)$ are equal to $1$, and for any quadrilateral with a diagonal edge, the $(n-1)$ edge functions along the diagonal edge are equal.
\end{lem}
\begin{proof}
Take $(\rho,\xi)\in\mathcal{X}_n(S_{g,m})$ such that the $\rho$-equivariant map $\xi_\rho:\partial_\infty \pi_1(S) \rightarrow \mathcal{B}$ is the osculating curve of the $\rho$-equivariant map $v\circ \xi_0: \partial_\infty \pi_1(S) \rightarrow \mathbb{RP}^{n-1}$ in Remark \ref{remark:fuchsianvero}. For any triple ratio $T_{i,j,k}(A,B,C)$ where $A, B, C$ are in the image of $\xi_\rho$, let us use the notations in Remark \ref{remark:triple} and consider the space $Q=\mathbb{P}\left(\mathbb{R}^n/\left(A^{(i-1)}\oplus  B^{(j-1)} \oplus B^{(k-1)}\right)\right)$. The projection of the Veronese curve in $Q$ is a conic, thus a circle up to projective transformations. Thus 
\[|wb|=|aw|, \;\; |uc|=|bu|, \;\; |va|=|cv| .\] 
Hence
\[T_{i,j,k}(A,B,C)=
\frac
{|wb|\cdot |uc|\cdot |va|}
{|bu|\cdot |cv| \cdot |aw|}=1.\]
The proof for the edge functions is similar.
\end{proof}

We propose the following candidate conditions for characterizing when a positive representation (or a positive framed local system) is a $n$-Fuchsian one:

\begin{defn}[Candidate $n$-Fuchsian characterizing conditions]
We define the following conditions for positive framed $\operatorname{PGL}_n$-local system $(\rho,\xi)\in\mathcal{X}_n(S_{g,m})$:
\begin{description}
\item[Triple ratio rigidity] 
for every ideal triangle in every ideal triangulation, the triple ratios (Definition~\ref{defn:FGXcoordinate}) are all equal to $1$.

\item[Strong triple ratio rigidity]
the $(n-1)(n-2)/2$ triple ratio functions $T^{\xi_\rho}_{i,j,k}:\mathrm{Tri}(S)\to\mathbb{R}_{>0}$ are identically equal to $1$.

\item[Edge function rigidity]
for every (interior) ideal edge on every ideal triangulation, the $(n-1)$ edge functions (Definition~\ref{defn:FGXcoordinate}) along said edge are equal.

\item[Strong edge function rigidity] 
for each diagonal of every ideal quadrilateral (i.e.: quadrilateral with cyclically ordered vertices in $\partial_\infty\pi_1(S)$), the $(n-1)$ edge functions along the diagonal are all equal.
\end{description}
Furthermore, whenever there exists a map $\xi$ such that $(\rho,\xi)\in \mathcal{X}_n(S)$ satisfies one of the above stated conditions, we say that the underlying positive representation $\rho$ (also) satisfies the corresponding condition.
\end{defn}

\begin{rmk}
\label{rmk:rigidityproof}
By Lemma \ref{lemma:nfuchrig}, all four of the above rigidity conditions are necessary conditions for $n$-Fuchsian representations. Conversely, the triple ratio rigidity and edge function rigidity combine to give defining equations for a $n$-Fuchsian slice of $\mathcal{X}_n(S_{g,m})$. Therefore, to show that the triple ratio rigidity condition characterizes $n$-Fuchsian representations, we need only to show that triple ratio rigidity implies edge function rigidity, or vice versa. Using these observations, we show:
\end{rmk}

\begin{prop}[Triple ratio rigidity for $n=3,4$]
\label{prop:triple34}
For $n=3,4$, a positive representation $\rho\in \mathrm{Pos}_n(S_{g,m})$ is $n$-Fuchsian if and only if $\rho$ satisfies the triple ratio rigidity condition.
\end{prop}

\begin{prop}[Edge function rigidity for $n=3$]
\label{prop:edge3}
For $n=3$, a positive representation $\rho\in \mathrm{Pos}_n(S_{g,m})$ is $n$-Fuchsian if and only if $\rho$ satisfies the edge function rigidity condition.
\end{prop}

We establish Propsitions~\ref{prop:triple34} and \ref{prop:edge3} via explicit algebraic computation (see Appendix~\ref{sec:appendixrigidity}). The advantage of such a proof is not merely in its simplicity, but also in its extensibility:
\begin{itemize}
\item 
it applies to $\mathcal{X}_n(\hat{S})$ \cite[Definition~1.2]{FG07}, where $\hat{S}$ is a surface with marked points on the boundary;
\item
it applies to the universal higher Teichm\"uller space context \cite[Definition~1.9]{FG07};
\item
and it also applies to general coefficient fields.
\end{itemize}
This method of proof does, however, quickly become difficult upon increasing $n$.

\subsubsection{Strong triple ratio rigidity}

We will show that the strong triple ratio rigidity condition characterizes $n$-Fuchsian representations for positive representations with unipotent boundary monodromy. We turn to the geometry of Frenet curves to help establish these rigidity conditions. 

\begin{defn}[\cite{Lab06} Frenet curve and osculating curve]
\label{defn:frenet}
A continuous curve $\xi^1:\mathbb{RP}^1\rightarrow\mathbb{RP}^{n-1}$ is called a \emph{Frenet curve} if there exists a curve 
\[
\xi=(\xi^1,\ldots,\xi^{n-1}):\mathbb{RP}^1\rightarrow\mathcal{B}
\] 
such that for every $k$-tuple of positive integers $(j_1,\ldots,j_k)$ such that $j_1+\ldots+j_k=j\leq n$, the curve $\xi$ satisfies the following properties
\begin{itemize}
\item
hyperconvexity: for every $k$-tuple of distinct points $x_1,\ldots,x_k\in\mathbb{RP}^1$, the following sum is direct
\[
\bigoplus^k_{i=1}\xi^{j_i}(x_i)\subset \mathbb{R}^n.
\]
\item
for every $x\in\mathbb{RP}^1$, the following limit exists and satisfies
\[
\lim_{(x_i)\to x}
\bigoplus^k_{i=1}\xi^{j_i}(x_i)=\xi^j(x),
\]
where the limit is taken over $k$-tuples $(x_1,\ldots,x_k)$ of pairwise distinct points.
\end{itemize}
We refer to $\xi=(\xi^1,\ldots,\xi^{n-1})$ as the \emph{osculating curve} of the Frenet curve $\xi^1$.
\end{defn}

Frenet curves are central to the study of positive representations. 

\begin{rmk}[Positive representations that have Frenet curves]
\label{remark:mpcont}
One important geometric property of positive representations $\rho$ is that for any positive framed $\operatorname{PGL}_n(\mathbb{R})$-local system $(\rho,\xi)$, by Definition~\ref{definition:rhoequivariant}, its respective associated map $\xi_{\rho}:\tilde{m}_p\rightarrow\mathcal{B}$ extends uniquely to the positive $\rho$-equivariant map $\xi_\rho:\partial_\infty\pi_1(S)\to\mathcal{B}$. When $\rho$ only admits unipotent boundary monodromy, $\partial_\infty\pi_1(S)\cong \mathbb{RP}^1$, then $\xi_\rho$ is the osculating curve of a Frenet curve where hyperconvexity follows \cite[Proposition 9.4]{FG06} and second Frenet property follows \cite[Lemma 5.1]{Lab06}. For positive representations with only (at least one) loxodromic boundary monodromy, let $dS$ denote the closed surface obtained by taking $S=S_{g,m}$ and an orientation-reversed copy of $S_{g,m}$ and identifying all corresponding boundary components. In this setting, there exists an osculating curve $d\xi:\partial_\infty\pi_1(dS)\cong \mathbb{RP}^1\rightarrow \mathcal{B}$ which restricts to $\xi_\rho$ on $\tilde{m}_p\subset\partial_\infty\pi_1(S)\subset\partial_\infty\pi_1(dS)$. This extension is far from being unique, but the osculating curve $d\xi$ for Hitchin double representation $d\rho$ \cite[Definition~9.2.2.3]{LM09} gives a canonical construction for such an extended Frenet curve. The ``extensions'' we describe in this remark are all $\rho$-equivariant. 
\end{rmk}

\begin{rmk}
Frenet curves have low regularity. They are $C^1$, and usually they are not $C^\infty$. By \cite[Theorem~D]{potrie2017eigenvalues} (or \cite[Proposition~6.1]{benoist2001convexes} for the $n=3$ case), for a positive (Hitchin) representation of a closed surface $S_{g,0}$ with $g\geq 2$, the Frenet curve $\xi_\rho^1$ is $C^\infty$ if and only if $\rho$ is a $n$-Fuchsian representation. 
\end{rmk}

We now prove a ``local'' version of the claim that strong triple ratio rigidity condition characterizes $3$-Fuchsian representations.

\begin{lem}[Elliptical subarc]
\label{lem:conicsegment}
For $n=3$, consider the restricted osculating curve 
\[
\xi=(\xi^1,\xi^2):[0,1]\rightarrow\mathcal{B}
\] 
for a subarc of a Frenet curve. If the triple ratio $T(\xi(0),\xi(1),\xi(s))$ is equal to $1$ for every $s\in(0,1)$, then the image of $\xi^1$ in $\mathbb{RP}^2$ is the subarc of an ellipse.
\end{lem}

\begin{proof}
We first observe that we may freely apply $\operatorname{PGL}_3(\mathbb{R})$ to $\xi$ without affecting the smoothness of $\xi$ or its triples ratios. In particular, our degree of freedom is high enough so that we may assume without loss of generality that
\begin{enumerate}
\item
the subarc maps to $\mathbb{R}^2=\{(x,y)\}\cong\{[x,y:1]^t \in\mathbb{RP}^{2}\}\subset\mathbb{RP}^2$;
\item
$\xi^1(0)$ and $\xi^1(1)$ are respectively positioned at $(0,0)$ and $(1,0)$;
\item
$\xi^2(0)$ and $\xi^2(1)$ are vertical lines $x=0$ and $x=1$ respectively;
\item
and $\xi^1$ is parameterized so that $\xi^1(s)=(s,f(s))$ for some $C^1$ function $f(s)$ such that $f(s)>0$ for $s\in(0,1)$.
\end{enumerate}
Note that conditions $(2)$ and $(3)$ mean that the lines $\xi^2(0),\xi^2(1)$ intersect at $[0,1,0]^t$. Further note that condition~$(4)$ is possible because Frenet curves are necessarily hyperconvex and the subarc $\xi^1$ is forced to be entirely on the upper half plane or lower half plane of $\mathbb{R}^2$, then choosing one of the two possible cases is equivalent up to a projective transformation. Consider the triple ratio rigidity condition
\[
T(\xi(0),\xi(1), \xi(s))=1.
\]
Explicitly writing out this condition for a $C^1$ curve $(s,f(s))$ yields the following:
\[
\frac{(1-s)(f(s)-sf'(s))}
{s(f(s)+(1-s)f'(s))}
=1,
\]
which leads to: 
\[
d (\log f(s))=\frac{f'(s)}{f(s)}
=\frac{1-2s}{2s(1-s)}
=\tfrac{1}{2} d (\log(s)+ \log(1-s)).
\]
We conclude that the family of half-ellipses of the form $f(s)=C_0 \sqrt{s(1-s)}$ for a positive constant $C_0$ constitute the full set of possible solutions for this ODE.
\end{proof}

We use Lemma~\ref{lem:conicsegment} to show the main result of this subsection:

\begin{thm}[Strong triple ratio rigidity characterizes $n$-Fuchsian]
\label{thm:highrankrigidity}
Given $S=S_{g,m}$ a surface with negative Euler characteristic and $n\geq2$, a positive representation $\rho:\pi_1(S)\rightarrow \operatorname{PGL}_n(\mathbb{R})$ with unipotent boundary monodromy (including $S$ being a closed surface) is $n$-Fuchsian if and only if $\rho$ satisfies the strong triple ratio condition.
\end{thm}

\begin{prop}
\label{prop:triplefrenet}
Let $\xi=(\xi^1,\cdots,\xi^{n-1}):\mathbb{RP}^1 \rightarrow \mathcal{B}$ be an osculating curve of a Frenet curve $\xi^1$.
If for every triple of distinct points in $\mathbb{RP}^1$ and any positive integers $i,j,k$ sum to $n$ the triple ratio equals to $1$, the Frenet curve $\xi^1$ is the Veronese curve up to projective equivalence.
\end{prop}
\begin{proof}
Let $X:[0,1]\rightarrow\mathbb{RP}^{n-1}$ be a subarc of $\xi^1$. By applying the action of $\operatorname{PGL}_n(\mathbb{R})$, we assume without loss of generality that: 
\begin{itemize}
\item
the standard basis $(e_1,e_2,\ldots,e_n)$ is a basis for the flag $\xi(0)$;
\item
the reversed standard basis $(e_n,e_{n-1},\ldots,e_1)$ is a basis for the flag $\xi(1)$.
\end{itemize}
We identify $X(t)$ with the following lift to $\mathbb{R}^n$:
\[
X(t)=x_1(t) e_1+\ldots+x_{n-1}(t) e_{n-1}+x_n(t)e_n.
\]
The hyperconvexity of $\xi$ (Definition~\ref{defn:frenet}) ensures that, for $i=1,\cdots,n$,
\[\xi^{(i-1)}(0)+ \xi^{1}(t) +\xi^{(n-i)}(1)=\mathbb{R}^n.\]
Thus 
\[x_i(t)\neq 0\quad for \;\;\;t\neq 0,1\;\;\;\text{ and }\;\;\; i=1,\cdots,n.\]  

We now show that there exists an algebraic
relation among $\frac{x_{n-k-1}(t)}{x_{n}(t)}$, $\frac{x_{n-k}(t)}{x_n(t)}$ and $\frac{x_{n-k+1}(t)}{x_n(t)}$ for $k=1,\ldots,n-2$.

\textbf{Step $\mathbf{k}$:} We know from the given assumption that the triple ratios
\[
T_{n-k-1,k,1}(X(0),X(1),X(t))=1\text{ for all }t\neq 0,1.
\] 
Remark~\ref{remark:triple} tells us that these triple ratios are still equal to $1$ after projecting $X(t)$ into the orthogonal complement
$\mathbb{V}_k^{\perp}$ of
\begin{align*}
\mathbb{V}_k:=
\mathrm{Span}\{
e_1,e_2,\ldots,e_{n-k-2},e_{n-k+2},\ldots,e_{n}
\}.
\end{align*}
By Lemma~\ref{lem:conicsegment}, the projected image 
\begin{align*}
\mathit{proj}_{\mathbb{V}_k^\perp}(X(t))
=x_{n-k-1}(t)e_{n-k-1} + x_{n-k}(t)e_{n-k} + x_{n-k+1}(t)e_{n-k+1}
\end{align*}
defines a subsegment of an ellipse when further projected into $\mathbb{RP}^2$. Thus there exists an algebraic relation between $\frac{x_{n-k-1}(t)}{x_{n-k+1}(t)}$ and 
\[\frac{x_{n-k}(t)}{x_{n-k+1}(t)}=\left(\frac{x_{n-k}(t)}{x_{n}(t)}\right)\big/\left(\frac{x_{n-k+1}(t)}{x_{n}(t)}\right).\]
Hence there exists an algebraic relation among
\[\frac{x_{n-k-1}(t)}{x_n(t)}=\frac{x_{n-k-1}(t)}{x_{n-k+1}(t)} \cdot \frac{x_{n-k+1}(t)}{x_{n}(t)},\]
 $\frac{x_{n-k}(t)}{x_{n}(t)}$ and $\frac{x_{n-k+1}(t)}{x_{n}(t)}$ for $k=1,\cdots, n-2$.

Thus any subarc $X$ of $\xi^1$ is an algebraic arc. Hence the Frenet curve $\xi^1$ from $\mathbb{RP}^1$ to $\mathbb{RP}^{n-1}$ is a reparameterization of an algebraic curve $c$. By hyperconvexity of $\xi^1$ in Definition~\ref{defn:frenet}: 
\begin{enumerate}
\item the algebraic curve $c$ is nondegenerate since it does not lie in any hyperplane;
\item for any mutually distinct points $x_1,\cdots,x_{n-1}$ on the curve $c$, there exists a unique hyperplane passing through these $n-1$ points, thus $\deg(c)\geq n-1$;
\item a generic hyperplane intersects $c$ transversely and contains at most $n-1$ points of $c$, thus $\deg(c)\leq n-1$.
\end{enumerate}
Hence $\deg(c)= n-1$. If there is a singular point $q$ in $c$, then a hyperplane passing through $q$ and any other $n-2$ points will imply that $\deg (c)\geq n$, which is impossible. Thus the curve $c$ is non-singular. Since $\xi^1$ is continuous and $\mathbb{RP}^1$ is connected, the curve $c$ is connected. Thus the non-singular algebraic curve $c$ is irreducible. 
By \cite[Proposition in Chapter 1, Section 4, Line Bundles and Maps to Projective Spaces, pg.~179]{GH78}, every irreducible nondegenerate algebraic curve of degree $n-1$ in $\mathbb{RP}^{n-1}$ is projectively isomorphic to the Veronese curve.
Thus the Frenet curve $\xi^1$ is the Veronese curve up to projective equivalence.
\end{proof}

\begin{proof}[Proof of Theorem \ref{thm:highrankrigidity}]
For such positive representation $\rho$, the boundary at infinity $\partial_\infty \pi_1(S)$ is a circle $\mathbb{RP}^1$. By Proposition \ref{prop:triplefrenet}, the Frenet curve $\xi_\rho^1:\partial_\infty \pi_1(S) \rightarrow \mathbb{RP}^{n-1}$ is projectively equivalent to the Veronese curve after reparameterizing $\partial_\infty \pi_1(S)$. Hence $\rho$ is a $n$-Fuchsian representation.
\end{proof}


\clearpage

\section{Goncharov--Shen potentials}
\label{sec:GS}

The positive $\mathcal{A}$-moduli space $\mathcal{A}_2(S_{1,1})$ is better known as Penner's decorated Teichm\"uller space \cite{pennercoords}. The elements of this space correspond to marked hyperbolic surfaces decorated with a horocycle around its solitary cusp. Let $(x,y,z)$ be the $\mathcal{A}$-coordinates (i.e.: $\lambda$-length coordinates) for $\mathcal{A}_2(S_{1,1})\cong\{(x,y,z)\in\mathbb{R}_{>0}^3\}$ with respect to an ideal triangulation $\mathcal{T}$ of $S_{1,1}$. Penner showed that the length $P$ of the decorating horocycle is a rational function of these coordinates:
\begin{align}
P=
2\left(
\frac{x}{y z}
+\frac{y}{x z}
+\frac{z}{x y}
\right).
\label{equation:SL2S11P}
\end{align}
Any ideal triangulation $\mathcal{T}$ of $S_{1,1}$ decomposes $S_{1,1}$ into two ideal triangles, each of which may be expressed as three different marked (oriented) ideal triangles. The first coordinate of a marked ideal triangle distinguishes one of its vertices, and \eqref{equation:SL2S11P} is obtained from summing the horocyclic segments at the distinguished vertices of these $6$ marked ideal triangles.

\subsection{Goncharov--Shen potentials}
Goncharov and Shen \cite{GS15} generalize $P$ for Fock--Goncharov $\mathcal{A}$-moduli spaces $\mathcal{A}_{\mathrm{SL}_n,S_{g,m}}$ of surfaces $S=S_{g,m}$ with negative Euler characteristic and at least one boundary component (i.e.: $m\geq1$) and relate to Knuston--Tao's hives \cite{KT98}. They associate $n-1$ types of expressions (Definition~\ref{definition:ichar}) to each marked ideal triangle and sum each type of expression over the marked oriented ideal triangles to obtain $n-1$ functions. The algebraic assignment of these $n-1$ expressions derives from the following fact:

\begin{fact}\label{fact:torsor}
For any triple of decorated flags $(F,G,H)\in \mathcal{A}^{3}$, if $(F,G,H)$ are in generic position, there is a unique linear transformation $g$, which can be expressed by an upper triangular unipotent matrix with respect to any basis for $F$, such that 
\[
g\cdot (F,\pi(G))=(F,\pi(H)),
\] where $\pi$ is the decoration-forgetting projection map from $\mathcal{A}$ to $\mathcal{B}$ (see Equation~\eqref{equation:projpi}).
\end{fact}

\begin{defn}[$i$-th character]
\label{definition:ichar}
For any generic triple  $(F,G,H)\in \mathcal{A}^{3}$, let $(g_{ij})$ be the upper triangular unipotent matrix for $g$ in Fact \ref{fact:torsor} with respect to some basis for the \emph{decorated flag} $F$. For $i=1,\cdots, n-1$, we define the \emph{$i$-th character} $P_i(F;G,H)$ of $(F,G,H)$ to be
\[
P_i(F;G,H):=g_{n-i,n-i+1},
\]
which does not depend on the basis that we choose. The $i$-th character $P_i$ satisfies the following additive properties:
\begin{align*}
P_i(F;G,H)&=P_i(F;G,W)+P_i(F;W,H);\\
P_i(F;G,H)&=-P_i(F;H,G).
\end{align*}
\end{defn}

\begin{rmk}[$i$-th character for triangles]
\label{remark:ichartri}
For a decorated twisted $\operatorname{SL}_n$-local system $(\bar{\rho},\bar{\xi}) \in \mathcal{A}_{\operatorname{SL}_n,S}$, recall the $\bar{\rho}$-equivariant map from the double cover $\bar{\tilde{m}}_p\subset \partial_\infty\bar{\pi}_1(S)\subset \mathbb{S}^1$ of $\tilde{m}_p\subset \partial_\infty\pi_1(S)\subset \mathbb{RP}^1$ to $\mathcal{A}$ in Definition~\ref{definition:Arep1}. 
Recall $w=\bar{\rho}(\bar{\sigma}_S):=(-1)^{n-1}\mathrm{Id}_{n\times n}$ in Definition \ref{defn:FGA}. Since
\[w g w^{-1} \cdot (w F, w\pi(G))= (w F, w\pi(H))\]
and $w g w^{-1}=g$, we obtain
\begin{equation}
\label{equation:wtri}
P_i(F;G,H)=P_i(wF;wG,wH). 
\end{equation}
Given a marked ideal triangle $(f,g,h)$ in $\tilde{\mathcal{T}}$ ( in the universal cover), we define
\[
P_i(f;g,h):=P_i(\bar{\xi}_{\bar{\rho}}(\bar{f});\bar{\xi}_{\bar{\rho}}(\bar{g}),\bar{\xi}_{\bar{\rho}}(\bar{h})),
\] 
to see that this is well-defined, we note that every marked ideal triangle $(f,g,h)$ has two coherent lifts $(\bar{f},\bar{g},\bar{h})$ and $(s\bar{f},s\bar{g},s\bar{h})$ to the double cover of $\mathbb{RP}^1=\tilde{S}$ (see Definition~\ref{definition:Arep2}), related by the antipodal involution $s:\mathbb{S}^1\to\mathbb{S}^1$. Equation~\eqref{equation:wtri} then ensures that 
\begin{align*}
P_i(\bar{\xi}_{\bar{\rho}}(\bar{f});\bar{\xi}_{\bar{\rho}}(\bar{g}),\bar{\xi}_{\bar{\rho}}(\bar{h}))&=P_i(\bar{\rho}(\bar{\sigma}_S)\bar{\xi}_{\bar{\rho}}(\bar{f});\bar{\rho}(\bar{\sigma}_S)\bar{\xi}_{\bar{\rho}}(\bar{g}),\bar{\rho}(\bar{\sigma}_S)\bar{\xi}_{\bar{\rho}}(\bar{h}))\\
&=P_i(\bar{\xi}_{\bar{\rho}}(s\bar{f});\bar{\xi}_{\bar{\rho}}(s\bar{g}),\bar{\xi}_{\bar{\rho}}(s\bar{h})),
\end{align*}
which is to say that $P_i(f;g,h)$ is independent of the choice of coherent lift for $(f,g,h)$.
\end{rmk}

\begin{rmk}
Given a positive decorated twisted local system $(\bar{\rho},\bar{\xi}) \in \mathcal{A}_n(S)$, for any cyclically ordered $(e,f,g,h)\in \partial_\infty\pi_1(S)$, the $i$-th characters satisfy the following positivity property:
\[
\frac{P_i(e;f,h)}{P_i(e;f,g)}>1.
\] 
\end{rmk}

\begin{defn}[\cite{GS15} Goncharov--Shen potential]
\label{definition:GSp}
Given $(\bar{\rho},\bar{\xi})\in\mathcal{A}_{\operatorname{SL}_n,S_{g,m}}$, we fix an ideal triangulation $\mathcal{T}$ of $S_{g,m}$. Fix a puncture $p\in m_p$ of $S$ and let $\Theta_p$ denote the set of marked anticlockwise-oriented ideal triangles with the first vertex being $p$. For each $i=1,\cdots,n-1$, the \emph{$i$-th Goncharov--Shen potential $P_i^p$ at $p$} is a regular function $\mathcal{A}_{\operatorname{SL}_n,S_{g,m}}$
given by
\begin{align}
\label{equation:gspotentialp}
P_i^p := \sum_{ \Delta \in \Theta_p} P_i(\Delta).
\end{align}
\end{defn}

Goncharov and Shen show that $P_i^p$ is well-defined, independent of the chosen ideal triangulation $\mathcal{T}$ and hence mapping class group invariant. They further demonstrate the following beautiful fact: 

\begin{thm}[{\cite[Theorem 10.7]{GS15}}]
The $m(n-1)$ Goncharov--Shen potentials $\left\{P_i^p\right\}_{p,i}$ generate the algebra of mapping class group invariant regular functions on the moduli space $\mathcal{A}_{\operatorname{SL}_n,S_{g,m}}$.
\end{thm}

\begin{rmk}
Goncharov and Shen refer to these potentials as Landau--Ginzberg partial potentials because an important aspect of their hitherto unproven homological mirror symmetry conjecture asserts that these potentials should correspond to Landau--Ginzburg partial potentials from Landau-Ginzburg theory. We opt to refer to these potentials as \emph{Goncharov--Shen potentials} to acknowledge their contribution in discovering this geometrically fascinating object.
\end{rmk}

\subsection{Constructing Goncharov--Shen potentials}

We now demonstrate how one might motivate and construct the aforementioned $P_i(F;G,H)$ expressions. This is essentially taken from \cite[Section 3]{GS15} which can be understood as the $\mathcal{A}$-coordinate version of Fock--Goncharov's snakes \cite[Section 9]{FG06}. We include this section both for expositional completeness and because many of our later derivations depend upon these foundational computations.\medskip

Consider a triple of decorated flags $(F,G,H)\in \mathcal{A}^3$ is in generic position with respective bases $(f_1,\cdots, f_n)$, $(g_1,\cdots, g_n)$, and $(h_1,\cdots, h_n)$. For any non-negative integers $a,b,c$ with $a+b+c=n$, define a $1$-dimensional vector space
\[
L_{a}^{b,c}:=F^{a+1} \cap (G^{b} \oplus H^{c}),
\]
and choose $e_{a}^{b,c}$ to be the unique vector in $L_{a}^{b,c}$ such that $e_{a}^{b,c} - f_{a+1} \in F^{a}$ (to clarify: $e_0^{i,n-i}=f_1$ for every $i=0,\cdots,n$). Then 
\[e_{a}^{b-1,c+1} - e_{a}^{b,c}\in G^{b} \oplus H^{c+1},\]
\[e_{a}^{b-1,c+1} - e_{a}^{b,c} = (e_{a}^{b-1,c+1}-f_{a+1}) - (e_{a}^{b,c}-f_{a+1})  \in F^{a}.\]
Thus there exist $\alpha^{F;G,H}_{a,b,c}\in\mathbb{R}$ such that 
\begin{align}
\label{equ:trans}
e_{a}^{b-1,c+1} - e_{a}^{b,c}  = \alpha^{F;G,H}_{a,b,c} \cdot e_{a-1}^{b,c+1}\in L_{a-1}^{b,c+1}.
\end{align}

\begin{lem}[{\cite[Lemma 3.1]{GS15}}]
\label{lemma:lozenge}
\begin{align*}
\alpha^{F;G,H}_{a,b,c} = 
\frac{\Delta\left(f^{a-1} \wedge h^{c+1} \wedge g^{b}\right)\cdot \Delta\left(f^{a+1} \wedge h^{c} \wedge g^{b-1}\right)}{\Delta\left(f^{a} \wedge h^{c} \wedge g^{b}\right) \cdot \Delta\left(f^{a} \wedge h^{c+1} \wedge g^{b-1} \right)}.
\end{align*}
\end{lem}

\begin{rmk}
The above formula differs from Goncharov--Shen's in that we construct $g$ satisfying $g\cdot (F,\pi(G))=(F,\pi(H))$, instead of $g'$ such that $g'\cdot(F,\pi(H)) =(F,\pi(G))$.
\end{rmk}

The following relationship between $\alpha^{F;G,H}_{a,b,c}$ and $T_{i,j,k}(F,G,H)$ is an immediate consequence of Lemma~\ref{lemma:lozenge}:

\begin{lem}
\label{lem:triplea}
For positive integers $a,b,c$ with $a+b+c=n$, we have
\begin{align*}
\frac{\alpha^{F;G,H}_{a,b+1,c-1} }{\alpha^{F;G,H}_{a,b,c}}= T_{a,b,c}(F,G,H).
 \end{align*}
\end{lem}

Equation~\eqref{equ:trans} tells us that there is a change of bases 
\begin{align*}
N_a(\alpha^{F;G,H}_{a,b,c})\cdot &\left(e_{0}^{a+b-1,c+1}, \cdots, e_{a-1}^{b,c+1}, e_{a}^{b,c}, e_{a+1}^{b-1,c}, \cdots, e_{a+b-1}^{1,c}, e_{a+b}^{0,c}, \cdots, e_{n-1}^{0,1}\right)
\\
=&\left(e_{0}^{a+b-1,c+1}, \cdots, e_{a-1}^{b,c+1}, e_{a}^{b-1,c+1}, e_{a+1}^{b-1,c}, \cdots, e_{a+b-1}^{1,c}, e_{a+b}^{0,c}, \cdots, e_{n-1}^{0,1}\right)
\end{align*} 
encoded by unipotent matrices of the form
\[N_a(x)=\left(\begin{array}{cccc}
  Id_{a-1} & 0 & 0 & 0 \\ 
 0 & 1 & x & 0 \\ 
 0 & 0 & 1 & 0 \\ 
 0 & 0 & 0 &  Id_{b+c-1}
 \end{array} \right).\]

For $c=0$, applying a $n-1$ chain of such transformations, we obtain:
\[
N_{n-1}(\alpha^{F;G,H}_{n-1,1,0})\cdots N_{1}(\alpha^{F;G,H}_{1,n-1,0}) \cdot \left(f_1,e_{1}^{n-1,0},\cdots, e_{n-1}^{1,0} \right)
=  
\left(f_1,e_{1}^{n-2,1},\cdots, e_{n-1}^{0,1}\right).
\]

Similarly, for $c=k, 1\leq k\leq n-2$, applying a $n-1-k$ chain of such transformations, we obtain:
\begin{align*}
N_{n-1-k}(\alpha^{F;G,H}_{n-1-k,1,k}) \cdots N_{1}(\alpha^{F;G,H}_{1,n-1-k,k})\cdot 
&\left(f_1,e_{1}^{n-1-k,k},\cdots,e_{n-k-1}^{1,k}, e_{n-k}^{0,k},\cdots, e_{n-1}^{0,1} \right)
\\
=&\left(f_1,e_{1}^{n-2-k,k+1},\cdots, e_{n-k-1}^{0,k+1},e_{n-k}^{0,k},\cdots, e_{n-1}^{0,1} \right).
\end{align*}
Precomposing the above $n-1$ (chains of) transformations, starting from $c=0$ all the way until $c=n-1$, we get:
\begin{align}
g&\cdot \left(e_{0}^{n,0},e_{1}^{n-1,0},\cdots, e_{n-1}^{1,0} \right)
=  \left(e_{0}^{0,n},e_{1}^{0,n-1},\cdots, e_{n-1}^{0,1}\right)\text{ ,where}\notag\\
g&:=\prod_{c=n-2}^{0} \prod_{a=n-c-1}^{1} N_a(\alpha^{F;G,H}_{a,b,c}),\quad\text{ for }b=n-a-c.\label{equ:rotationmatrix}
\end{align}
The above process explicitly constructs $g$ so that $g\cdot (F, \pi(G)) = (F, \pi(H))$, thereby verifying Fact~\ref{fact:torsor}. In particular, the $i$-th character:
\begin{align}
\label{equ:potential}
P_i(F;G,H) =\sum_{c = 0}^{i-1} \alpha^{F;G,H}_{n-i,i-c,c}.
\end{align}
As a consequence of Lemma \ref{lemma:lozenge}
\begin{cor}
\label{cor:triuni}
Let $(\bar{\rho},\bar{\xi})\in \mathcal{A}_n(S)$ and $\gamma \in \pi_1(S)$. Then for any ideal triangle $(f,g,h)$ of $\tilde{\mathcal{T}}$ in the universal cover,
\[
P_i(\gamma f;\gamma g,\gamma h).
=P_i(f;g,h).
\]
\end{cor}
\begin{exa}
\label{example:lozenge}
Let us consider the simple example of the $n=3$ case. We denote 
\[
R^{f}_{g,h}:=\alpha^{F;G,H}_{2,1,0},\quad
S^{f}_{g,h}:=\alpha^{F;G,H}_{1,2,0},\quad 
T^{f}_{g,h}:=\alpha^{F;G,H}_{1,1,1}. 
\]
Figure~\ref{Figure:Psl3} encodes a schematic for the construction procedure for $g$ we describe above. Each ``lozenge'' (or diamond) here is labeled by one of the $\alpha^{F;G,H}_{i,j,k}$, which specifies the linear transformation to pass from the basis (encoded by a colored path) on the left of the lozenge to the right of the lozenge. In this particular case, Equation~\eqref{equ:rotationmatrix} yields: 
\begin{align*}
g = N_1(T^{f}_{g,h}) \cdot N_2(R^{f}_{g,h})\cdot  N_1(S^{f}_{g,h})
=\left(
\begin{array}{ccc}
1 &  S^{f}_{g,h}+ T^{f}_{g,h} & T^{f}_{g,h} R^{f}_{g,h} \\
0 & 1 &R^{f}_{g,h}  \\
0 & 0 & 1 
\end{array} 
\right),
\end{align*}
and the $i$-th characters are
\[
P_1(F;G,H) =R^{f}_{g,h}\text{ and }P_2(F;G,H) =S^{f}_{g,h}+T^{f}_{g,h}.
\]
\end{exa}

\begin{figure}[h!]
\includegraphics[scale=0.5]{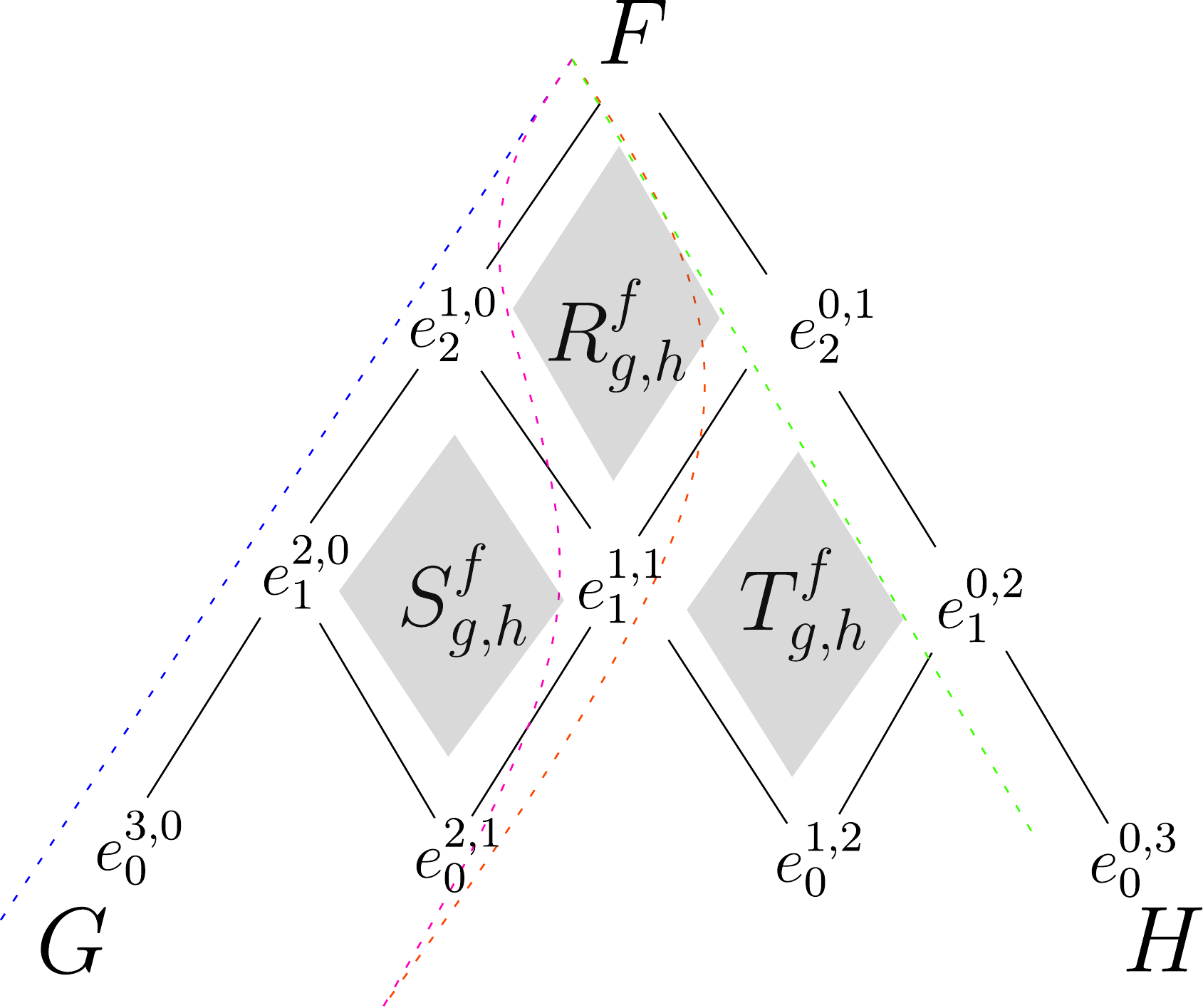}
\caption{The above figures encodes how to construct the unipotent matrix taking $(F,\pi(G))$ to $(F,\pi(H))$. Each colored path corresponds to a basis; basis $1$ is blue, basis $2$ is magenta, basis $3$ is red, and basis $4$ is green.}
\label{Figure:Psl3}
\end{figure}

\subsection{Goncharov--Shen potential and $\mathcal{A}$-coordinates}

One beautiful achievement of Goncharov and Shen's work we reviewed above is that \eqref{equ:potential} combined with Lemma~\ref{lemma:lozenge} (as well as and Notation~\ref{notation:flag}) tells us how to express Goncharov--Shen potentials in terms of rational functions of $\mathcal{A}$-coordinates. Let us see this explicated with an important example: the $\mathcal{A}_3(S_{1,1})$ case.\medskip

Given $(\bar{\rho},\bar{\xi})\in \mathcal{A}_3(S_{1,1})$ and an ideal triangulation $\mathcal{T}$ of $S_{1,1}$, we lift $\mathcal{T}$ into the universal cover $\tilde{\mathcal{T}}$. Choosing one fundamental domain, we denote the $\mathcal{A}$-coordinates as in Figure~\ref{Figure:potentialp31}. In this case, we have $P_1(y;z,t)=\frac{w}{br}$. Then by Equation~\eqref{equation:gspotentialp}, we have:
\begin{align*}
 P_1^p
&={\color{red}\tfrac{w}{br}+ \tfrac{q}{cr}}
 +{\color{blue}\tfrac{w}{ds}+\tfrac{q}{as}}
 +{\color{magenta}\tfrac{w}{ac}} + 
 {\color{orange}\tfrac{q}{bd}},\text{ and}\\
P_2^p 
&= \tfrac{bc}{aw} + \tfrac{rd}{ws}+  \tfrac{bs}{wr} + \tfrac{ad}{wc}+  \tfrac{ar}{bw} + \tfrac{cs}{dw}+ \tfrac{ar}{sq} + \tfrac{cb}{dq}+  \tfrac{dr}{cq} + \tfrac{bs}{aq}+  \tfrac{ad}{bq} + \tfrac{cs}{rq}.
\end{align*}

\begin{figure}[h!]
\includegraphics[scale=0.2]{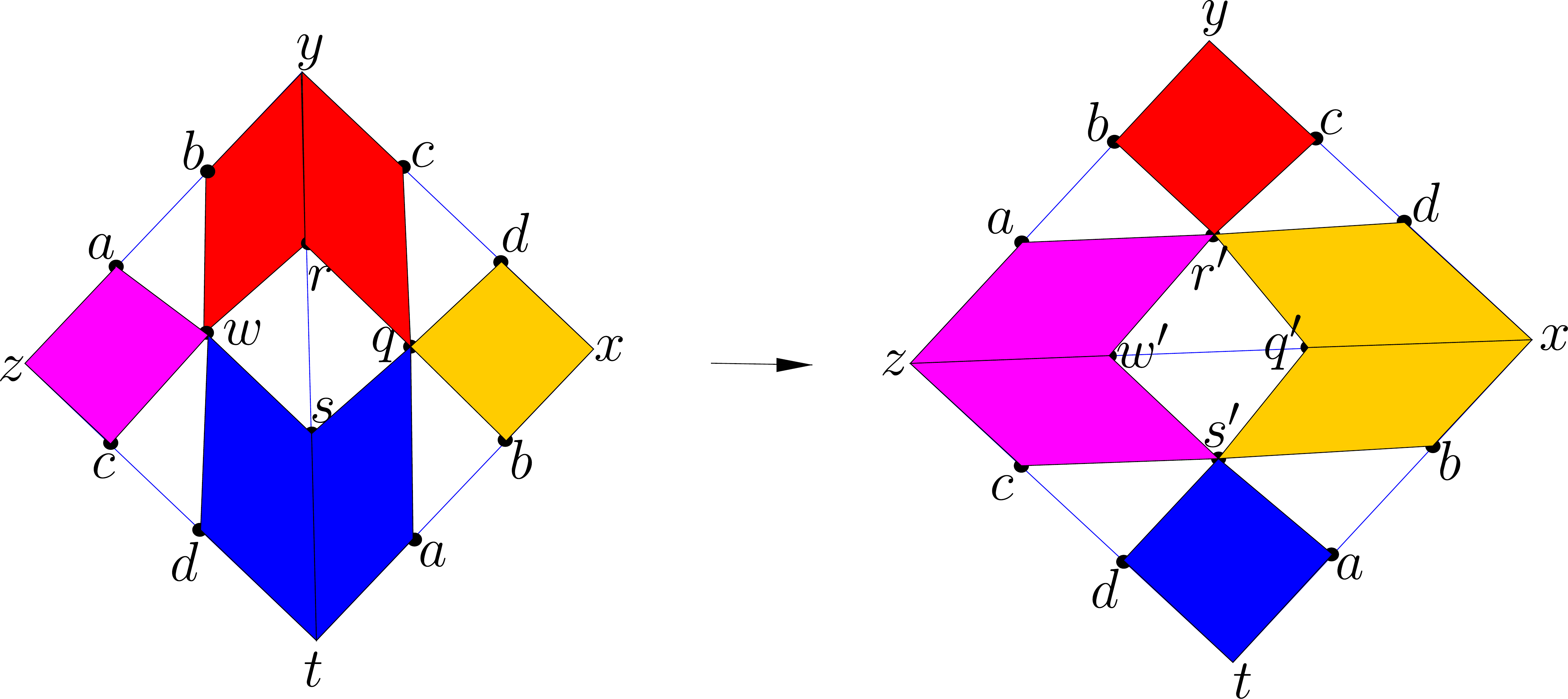}
\caption{The colored lozenges each correspond to the $1$-character $P_1(\triangle)$ of a marked ideal triangle $\triangle$.}
\label{Figure:potentialp31}
\end{figure}

Goncharov--Shen potentials are invariant under flips, and so let us now observe what happens to the above expressions of $P_1^p$ and $P_2^p$ under the change of $\mathcal{A}$-coordinates corresponding to a flip along the edge $\overline{yt}$ as depicted in Figure~\ref{Figure:potentialp31}. We saw previously (description adjacent to Figure~\ref{Figure:flipgeneral}) that such a flip is composed of four successive cluster mutations. In fact, the algebraic relations for these four mutation may be expressed in the following manner:

\begin{lem}
\label{lem:add}
Given $\mathcal{A}$-coordinates for $\mathcal{A}_3(S_{1,1})$ as depicted in Figure~\ref{Figure:potentialp31}, we have
\begin{align*}
{\color{red}P_1(y;z,t)+P_1(y;t,x)}&{\color{red}=\tfrac{w}{br}+ \tfrac{q}{cr} = \tfrac{r'}{bc}=P_1(y;z,x)},\\
{\color{blue}P_1(t;x,y)+P_1(t;y,z)}&{\color{blue}=\tfrac{q}{sa}+\tfrac{w}{ds} = \tfrac{s'}{ad}=P_1(t;x,z)},\\
{\color{magenta}P_1(z;t,y)}&{\color{magenta}=\tfrac{w}{ac}=\tfrac{s'}{c w'}+\tfrac{r'}{a w'}=P_1(z;t,x)+P_1(z;x,y)},\\
{\color{orange}P_1(x;y,t)}&{\color{orange}=\tfrac{q}{bd}=\tfrac{r'}{d q'}+ \tfrac{s'}{b q'}=P_1(x;y,z)+P_1(x;z,t)}.
\end{align*}
There are analogous formulae for $P_2$-related terms.
\end{lem}

Lemma~\ref{lem:add} tells us that, with respect to the $\mathcal{A}$-coordinates after flipping in the ideal edge corresponding to $\overline{yt}$, the Goncharov--Shen potential $P_1^p$ is equal to:
\[
P_1^p =
{\color{red}\tfrac{r'}{bc}}
+{\color{blue}\tfrac{s'}{ad}}
+{\color{magenta}\tfrac{s'}{c w'}+\tfrac{r'}{a w'}}+
{\color{orange}\tfrac{r'}{d q'}+ \tfrac{s'}{b q'}}.
\]
We need not restrict ourselves to using just a single set of $\mathcal{A}$-coordinates. For instance, by utilizing both sets of coordinates, we obtain the following compact expressions for $P_1^p$:
\[
P_1^p = 
{\color{red}\tfrac{r'}{bc}}
+{\color{blue}\tfrac{s'}{ad}}
+{\color{magenta}\tfrac{w}{ac}}+
{\color{orange}\tfrac{q}{bd}}.
\]
This an important idea that we make use of in the proofs of Theorem~\ref{theorem:inequsl3s11}, (implicitly in) Theorem~\ref{theorem:equsl3} and Theorem~\ref{thm:equsl3pp}.\medskip

Conversely, we may split the terms up as much as possible to obtain:
\begin{align*}
P_1^p=
{\color{red}\tfrac{w}{br}+ \tfrac{q}{cr}}
+{\color{blue}\tfrac{q}{sa}+\tfrac{w}{ds}}
+{\color{magenta}\tfrac{s'}{c w'}+\tfrac{r'}{a w'}}
+{\color{orange}\tfrac{r'}{d q'}+ \tfrac{s'}{b q'}}.
\end{align*}
We can then go further: flipping $\mathcal{T}$ in the edge covered by $\overline{yz}$ (or equivalently, $\overline{xt}$) produces yet another ideal triangulation and splits ${\color{red}\tfrac{q}{cr}}$ and ${\color{blue}\tfrac{w}{ds}}$ into two new summands each. Similar, flipping $\mathcal{T}$ in the edge covered by $\overline{xy}$ (or equivalently, $\overline{zt}$) splits ${\color{red}\tfrac{w}{br}}$ and ${\color{blue}\tfrac{q}{sa}}$ into two new summands each. Of course, we need not stop here, we can keep flipping to new ideal triangulations and deriving finer and finer expressions of $P_1^p$ as a (finite) series. It is natural, and tempting, to pose:\medskip

\emph{What happens if we flip to all possible ideal triangulations of $S_{1,1}$ whilst always splitting $P_1^p$ as finely as possible?}\medskip

In the Fuchsian (and indeed, quasi-Fuchsian \cite{markofftriples}) setting, Bowditch \cite{Bow96} shows that there is a precise sense in which this procedure limits to the McShane identity for $1$-cusped hyperbolic tori. This idea has been exploited in several papers \cite{MR2399656, HN17, HSY18, HPZ19} to obtain McShane identities for various types of geometric objects. We adopt this idea as a starting point, and show that one does indeed obtain McShane identities for positive representations of surface groups with unipotent boundary as a consequence (Theorems~\ref{theorem:inequsl3s11}, \ref{theorem:equsl3} and \ref{thm:equsl3pp}). However, the type of analysis conducted in \cite{Bow96} is difficult to completely replicate in our setting and we draw upon McShane's classical strategy of proof. We explain this in \S\ref{sec:proofideasummary}.

\clearpage

\section{Identities for $\operatorname{PGL}_3(\mathbb{R})$-representations with unipotent boundary}
\label{sec:splitting}

The goal of this section is to establish McShane identities for positive $\operatorname{PGL}_3(\mathbb{R})$-representations of hyperbolic surface groups $\pi_1(S_{g,m})$ with (only) unipotent boundary monodromy. These representations are not Anosov and so many ``standard'' higher Teichm\"uller theoretic techniques do not apply. Our proof takes advantage of the fact that, for $n=3$, positive representations with unipotent boundary monodromy arise as the holonomy representations of finite-area cusped convex real projective surfaces (see \S\ref{sec:sparsity}). The picture is then sufficiently geometric that we may adapt the ideas of the classical proof of McShane's identity with computational techniques availed by the cluster algebraic structure of Fock--Goncharov coordinates.

\subsection{Structure of proof in the classical case}
\label{sec:proofideasummary}
We begin with an overview of the general strategy for proving McShane identities in the hyperbolic case due to McShane \cite{mcshane_thesis,mcshane_allcusps}. Let $S_{g,m}$ denote a cusped hyperbolic surface with $m\geq1$ cusps, and distinguish one of these cusps by labeling it as $p$. McShane's identity for $S_{g,m}$ (and hence its Fuchsian holonomy representation) may be obtained via the following steps:\medskip

\textbf{Step 1: a probabilistic partition of geodesics.}\newline
The set of geodesics on $S_{g,m}$ emanating from cusp $p$ naturally identifies with the length $1$ horocycle $\eta$ based at cusp $p$, and hence inherits a natural probability measure via the horocyclic length measure on $\eta$. The points on $\eta$ partition into:
\begin{itemize}
\item
a Cantor set $\mathfrak{C}$ corresponding to \emph{simple} geodesics which emanate from $p$ and spiral towards a geodesic lamination;
\item
a countable set $\mathfrak{A}$ corresponding to \emph{simple ideal geodesic arcs} emanating from $p$ and ending also at a cusp; 
\item
and a countable collection of open horocylic intervals, hitherto referred to as \emph{gap intervals}, corresponding to geodesics which self-intersect (generic) as well as simple geodesics with both ends at $p$ (non-generic, only countably many such points). 
\end{itemize}
The Birman-Series geodesic sparsity theorem ensures that $\mathfrak{C}\cup\mathfrak{A}$ has horocyclic length measure $0$ because its thickening to an $\epsilon$-neighborhood of $\eta$ has hyperbolic area $0$. The McShane identity then comes from expressing the total measure (i.e.: $1$) of $\eta$ as the sum of the horocyclic lengths of the gap intervals which make up $\eta-(\mathfrak{C}\cup\mathfrak{A})$.\medskip

\textbf{Step~2: indexing the gap intervals}\newline
The aforementioned horocyclic gap intervals are in $4:1$ correspondence with embedded pairs of pants containing cusp $p$ (as $I_1,I_2,I_3,I_4$ in Figure~\ref{Figure:gap}). In fact, the subsegment of any geodesic launched from $p$ within a given gap interval, up to the first point of self-intersection, lies completely on the pair of pants corresponding to the given gap interval. This can be shown, for example, via local Gauss-Bonnet based arguments \cite[Lemma~4.6]{huang_thesis}. In any case, the classical McShane identity is a series taken over the set of embedded pairs of pants on $S_{g,m}$ containing cusp $p$, with the summands given by the sum of the four intervals referred to above.\medskip

\begin{figure}
\includegraphics[scale=0.35]{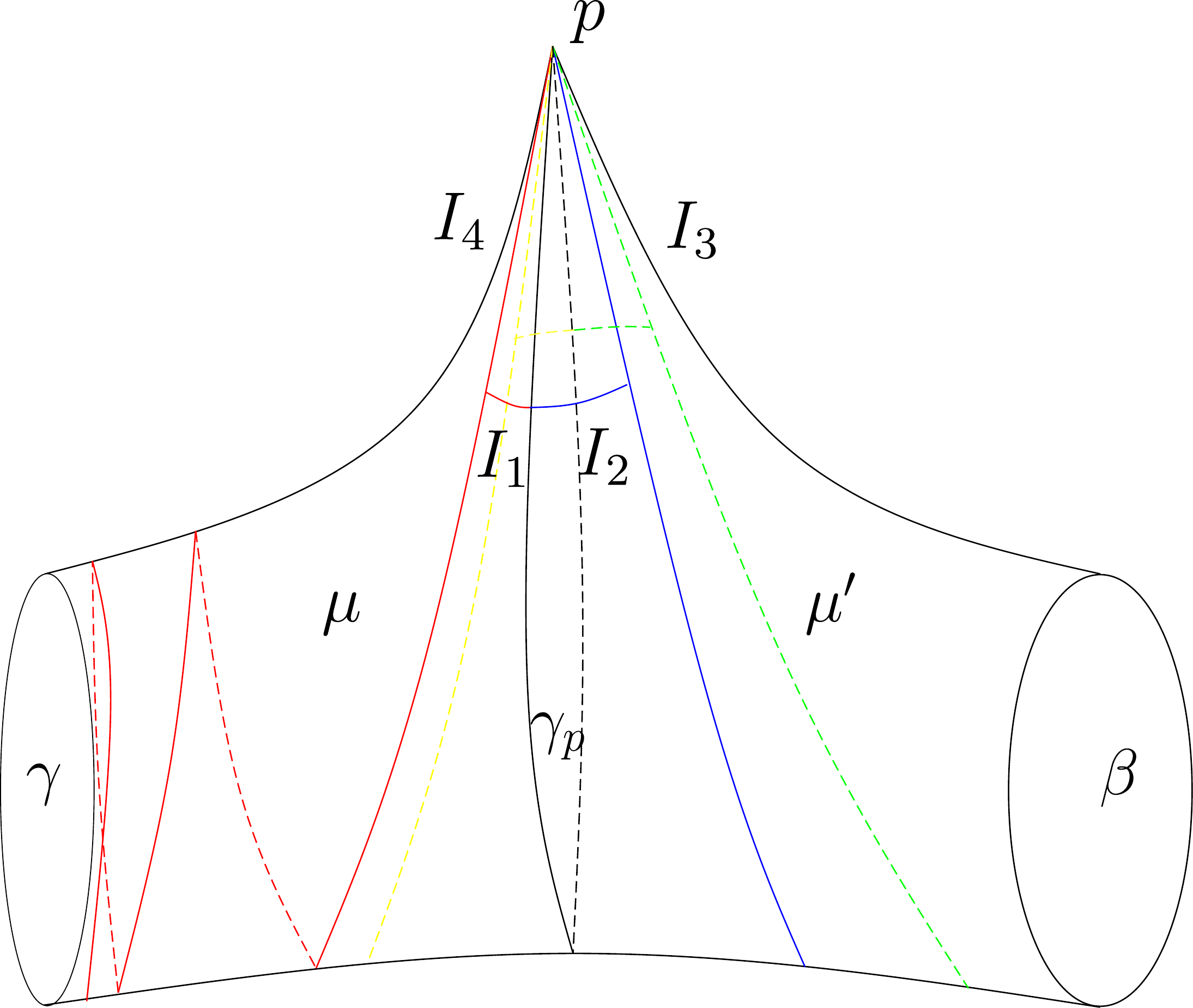}
\caption{The red and yellow simple curves spiral around the left side hole to the infinity in two different directions, while the blue and the green simple curves spiral around the right side hole to the infinity.}
\label{Figure:gap}
\end{figure}

\textbf{Step~3: computing the lengths of the gaps.}\newline 
The gaps $I_1,I_2,I_3,I_4$ are purely dependent upon the geometry of the pair of pants they lie on, and hence may be expressed purely in terms of the boundary lengths of the relevant pair of pants. McShane computes this directly in the Poincar\'{e} upper half plane model \cite[page 619]{mcshane_allcusps}, whereas Mirzakhani \cite[Lemma~3.1]{mirz_simp} and Tan-Wong-Zhang \cite[\S7]{tan_zhang_cone} do so by invoking hyperbolic trigonometric identities.

\subsubsection{Adapting the proof for finite-area cusped convex real projective surfaces}

The strategy of proof of the McShane identity for a cusped convex real projective surface $\Sigma_{g,m}$ with $m\geq1$ cusps, negative Euler characteristic and a distinguished cusp $p$ is fundamentally the same as for hyperbolic surfaces, but with the following adjustments for each of the three steps:

\begin{description}
\item[Step 1]
We again identify the set of geodesics emanating from cusp $p$ with a horocycle $\eta$ (with respect to the Hilbert metric on $S_{g,m}$). We further identify $\eta$ with a $\mathbb{Z}$ quotient of $\partial\Omega-\{\tilde{p}\}$, where 
\begin{itemize}
\item
$\Omega\subset\mathbb{RP}^2$ is the convex domain universal cover of $\Sigma_{g,m}$;
\item
$\tilde{p}\in\partial\Omega$ is a lift of the cusp $p$;
\item
the $\mathbb{Z}$ quotient is taken with respect to the stabilizer subgroup of $\tilde{p}$ in the group of deck transformations $\pi_1(\Sigma_{g,m})$ acting on $\Omega$. This group consists of unipotent linear transformations which preserve $\tilde{p}$ and $\Omega$.
\end{itemize}
One inherits from this identification a partition of $\eta$ into $\mathfrak{C}$, $\mathfrak{A}$ and countably gap intervals. We then normalize and reinterpret each of the $i=1,2$ Goncharov--Shen potentials as a probability measure on $\eta$ --- we refer to these probability measures as \emph{Goncharov--Shen potential measures}. We invoke our generalization of the Birman-Series theorem (Theorem~\ref{thm:birmanseries}) to ensure that $\mathfrak{C}\cup\mathfrak{A}$ has measure $0$ with respect to the Goncharov--Shen potential measure. This is the basis for the McShane identity for cusped convex real projective surfaces: the total sum of the Goncharov--Shen potential measures of the gap intervals is equal to $1$. 

\item[Step 2]
Although the gap intervals on $\eta$ satisfy the same $4:1$ correspondence with the set of embedded pairs of pants on $\Sigma_{g,m}$ (and hence may again be used to index the McShane identity), we choose to adopt finer summation indices. In the classical setting, the lengths of $I_1$ and $I_4$ are the same (and the lengths of $I_2$ and $I_3$ are the same) due to all hyperbolic pairs of pants admitting a boundary-component-fixing ``reflection'' isometry. The richness of convex real projective structures generically breaks this symmetry, and this is one reason why we instead sum over the set $\xvec{\mathcal{P}}_p$ of boundary-parallel pairs of pants (Definition~\ref{definition:Pp}) containing cusp $p$, which is a  $2:1$ covering set of the set $\mathcal{P}_p$ of pairs of pants on $\Sigma_{g,m}$. We consider also a even finer summation index (Theorem~\ref{theorem:equsl3}), consisting of boundary-parallel pairs of half-pants (Definition~\ref{defn:hp}), which are in natural bijection with the gap intervals.

\item[Step 3]
We use the cluster algebraic structure of Fock and Goncharov's $\mathcal{A}$-coordinates to compute the Goncharov--Shen potential measure of gap intervals (and pairs of gap intervals). This strategy is implicit in Bowditch's \cite{Bow96}, which essentially uses Penner's $\lambda$-length to perform the requisite horocyclic length computations, albeit expressed in terms of traces of Fuchsian holonomy representations for hyperbolic surfaces.
\end{description}

\begin{rmk}
One key idea that we utilize in Step~3 is the derivation of gap terms for a given positive representation $\rho\in\operatorname{Pos}_n^u(S_{g,m})$ via decorated twisted local systems $(\bar{\rho},\bar{\xi})\in\mathcal{A}_n(S_{g,m})$. To be precise, given such a positive representation $\rho$, we know by \cite[Theorems 1.12, 1.14]{FG06} that there exists a positive decorated twisted $\operatorname{SL}_n$-local system $(\bar{\rho},\bar{\xi})\in \mathcal{A}_n(S_{g,m})$ such that the underlying framed $\operatorname{PGL}_n$-local system for $(\bar{\rho},\bar{\xi})$ takes the form $(\rho,\xi)$ (see Remark~\ref{remark:AX}). Generally speaking, the decorated twisted local system $(\bar{\rho},\bar{\xi})$ is not unique, however, by imbuing $\rho$ with extra data, we enable computations regarding the properties of $\rho$ via the cluster algebraic language of Fock and Goncharov's $\mathcal{A}$-coordinates. The resulting gap terms are purely expressed in terms of projective invariants associated to $\rho$ and hence depend only on $\rho$ and not on any of the decorating data in $(\bar{\rho},\bar{\xi})$.
\end{rmk}

We now go through each of these three steps in detail, starting with Step~1 (\S\ref{sec:step1}). 
Step~2 has already been covered after Definition \ref{definition:Pp} for McShane identities summed over boundary-parallel pairs of pants. We describe the boundary-parallel half-pants case in \S\ref{sec:halfpantsseries}. The majority of the remainder is focused on steps~3 --- the computation of the actual summands, which we shall handle on a case-by-case basis depending on the type of summation index used (\S\ref{subsection:s11summand}, \S\ref{subsection:bphpsummand}, \S\ref{subsection:bpppsummand}).

\subsection{Generalizing step~1: a probabilistic partition of geodesics.}
\label{sec:step1}
Let $S_{g,m}$ denote a genus $g$ oriented surface with $m\geq1$ boundary components, negative Euler characteristic and a distinguished cusp $p$. For any $\operatorname{PGL}_3(\mathbb{R})$ positive representation $\rho:\pi_1(S_{g,m})\to\mathrm{PGL}_3(\mathbb{R})$ with unipotent boundary monodromy, there is a unique framed $\operatorname{PGL}_3(\mathbb{R})$ local system $(\rho,\xi) \in \mathcal{X}_3(S_{g,m})$ with underlying holonomy representation $\rho$. By Remark~\ref{remark:mpcont}, the $\rho$-equivariant map $\xi_\rho:\tilde{m}_p \rightarrow \mathcal{B}$ extends uniquely to a positive map
\[
\xi_\rho=(\xi_\rho^1,\xi_\rho^2): \partial_\infty \pi_1(S_{g,m})\to\mathcal{B}.
\]
In particular, the image $\xi_\rho^1(\pi_1(S_{g,m}))$ is a $C^1$ smooth curve in $\mathbb{RP}^2$ and bounds a simply connected region $\Omega$. The quotient of $\Omega$ by $\rho(\pi_1(S_{g,m}))$ is a finite-area cusped convex real projective surface $\Sigma_{g,m}$ homeomorphic to $S_{g,m}$ by \cite{Mar10}. More accurately speaking, $\Omega$ defines a convex real projective structure on $S_{g,m}$ with $\rho$ as the holonomy representation.\medskip

We henceforth identify $\partial_\infty\pi_1(S_{g,m})$ with $\partial\Omega$ via the following result: 

\begin{thm}
\label{lem:boundaryidentification}
The map $\xi_\rho^1:\partial_\infty\pi_1(S_{g,m})\to\partial\Omega$ is a homeomorphism. 
\end{thm}
\begin{proof}
By \cite[Theorem 1.14]{FG06}, the $\rho$-equivariant map $\xi_\rho: \partial_\infty\pi_1(S_{g,m})\to \mathcal{B}$ is a continuous positive map. Thus $\xi_\rho^1: \partial_\infty\pi_1(S_{g,m})\to \partial\Omega$ is continuous order-preserving injective map. By \cite[Theorem 6.14]{Mar12}, $\xi_\rho^1$ is surjective. Hence $\xi_\rho^1$ is a homeomorphism.
\end{proof}

Let $\tilde{p}\in\partial\Omega=\xi_\rho^1(\partial_\infty\pi_1(S_{g,m}))$ denote a lift of the cusp $p$, and let $\alpha_{\tilde{p}}\in\pi_1(S_{g,m})$ be the unique primitive peripheral homotopy class which fixes $\tilde{p}$ oriented such that $S_{g,m}$ is on the left side $\alpha_{\tilde{p}}$. In addition, let $\eta$ denote a cusp $p$ horocycle on $\Sigma_{g,m}$ small enough so as to be embedded, and let $\tilde{\eta}\subset\Omega$ denote the unique lift of $\eta$ which limits to $\tilde{p}$ in both directions (see Figure~\ref{fig:projection}). The ``fan'' of complete Hilbert metric geodesics (i.e.: Euclidean straight lines) on $\Omega$ emanating from $\tilde{p}$ gives a natural identification between: the horocycle $\tilde{\eta}$ and $\partial\Omega-\{\tilde{p}\}$. This identifcation is preserved under quotienting by $\alpha_{\tilde{p}}$ and hence descends to an identification between:
\begin{itemize}
\item
the set of (Hilbert metric) geodesics on $\Sigma_{g,m}$ emanating from cusp $p$
\item
the (Hilbert metric) horocycle $\eta$, and
\item
the quotient curve $(\partial\Omega-\{\tilde{p}\})/\langle\rho(\alpha_{\tilde{p}})\rangle$.
\end{itemize}
We intentionally conflate these sets and refer to them all as $\eta$.

\begin{defn}[Goncharov--Shen potential measure]
Consider the collection of closed intervals on $\eta$ with measure defined as follows: for $i=1$ or $2$, given an arbitrary closed interval $[q_0,q_1]\subset\eta$, choose a lift $[\tilde{q}_0,\tilde{q}_1]\subset\partial\Omega$ (see Figure~\ref{fig:projection}) and assign the outer measure of $[q_0,q_1]$ to be
\[ {P_i(\tilde{p},\tilde{q}_0,\tilde{q}_1)}
/P_i^p
.\]
The value so assigned is independent of the choice of the lift $\tilde{p}$ of $p$ and the lift $[\tilde{q}_0.\tilde{q}_1]$ of $[q_0,q_1]$ thanks to Corollary \ref{cor:triuni} and hence suffices to generate a well-defined Borel measure on $\eta$ by the Caratheodory procedure. We refer to this measure as the $i$-th \emph{Goncharov--Shen potential measure}.
\end{defn}

\begin{lem}
\label{lem:gsmeasure}
The Goncharov--Shen potential measure on the horocycle $\eta$ is a probability measure and is in the same measure class as the Hilbert length measure on $\eta$.
\end{lem}

\begin{proof}
We know from Fact~\ref{fact:torsor} that there is a unique unipotent matrix $M$ that takes $(\bar{\xi}_\rho(\tilde{p}),\xi_\rho(\tilde{q}_0))$ to $(\bar{\xi}_\rho(\tilde{p}),\xi_\rho(\tilde{q}_1))$. Fix a basis for $\rho(\tilde{p})$ and let $M(t)$ denote the path of unipotent matrices, expressed with respect to said basis,
\begin{align*}
M(t):=\left(
\begin{array}{ccc}
1& m_{12}(t)& m_{13}(t)\\
0&1&tP_1^p\\
0&0&1
\end{array}
\right)\text{, such that:}
\end{align*}
\begin{itemize}
\item
the parametrized path $M(t)\cdot \tilde{q}_0$ traces out the interval $[\tilde{q}_0,\tilde{q}_1]$ on $\partial\Omega$ and
\item
$M(t)$ takes the tangent space $T_{\tilde{q}_0}\partial\Omega$ to $T_{M(t)\cdot \tilde{q}_0}\partial\Omega$. 
\end{itemize}
The entries $m_{12}(t)$ and $tP_1^p$ are unaffected by the choice of basis for $\bar{\xi}_\rho(\tilde{p})$. This allows us pullback the $M(t)\cdot q_0$ parametrization of $\partial\Omega-\{\tilde{p}\}$ onto $\tilde{\eta}$. Note that this is a parametrization of $\tilde{\eta}$ which is $C^1$ compatible to the Hilbert length parameterization of $\tilde{\eta}$  because $\partial\Omega$ is $C^1$-smooth. Moreover, the $t\in[0,1]$ parameter, by construction, precisely parametrizes the $\mathbb{Z}$-invariant lift of the Goncharov--Shen measure on $\eta$ to $\tilde{\eta}$, thereby telling us that the Goncharov--Shen potential measure on $\eta$ is the Hilbert length measure on $\eta$ weighted by an almost everywhere positive $C^1$ function. Thus, the two measures are in the same measure class. Finally, the Goncharov--Shen potential measure of $\eta$ is $1$ by construction (see Example~\ref{example:lozenge}).
\end{proof}

\begin{figure}[h]
\includegraphics[scale=0.7]{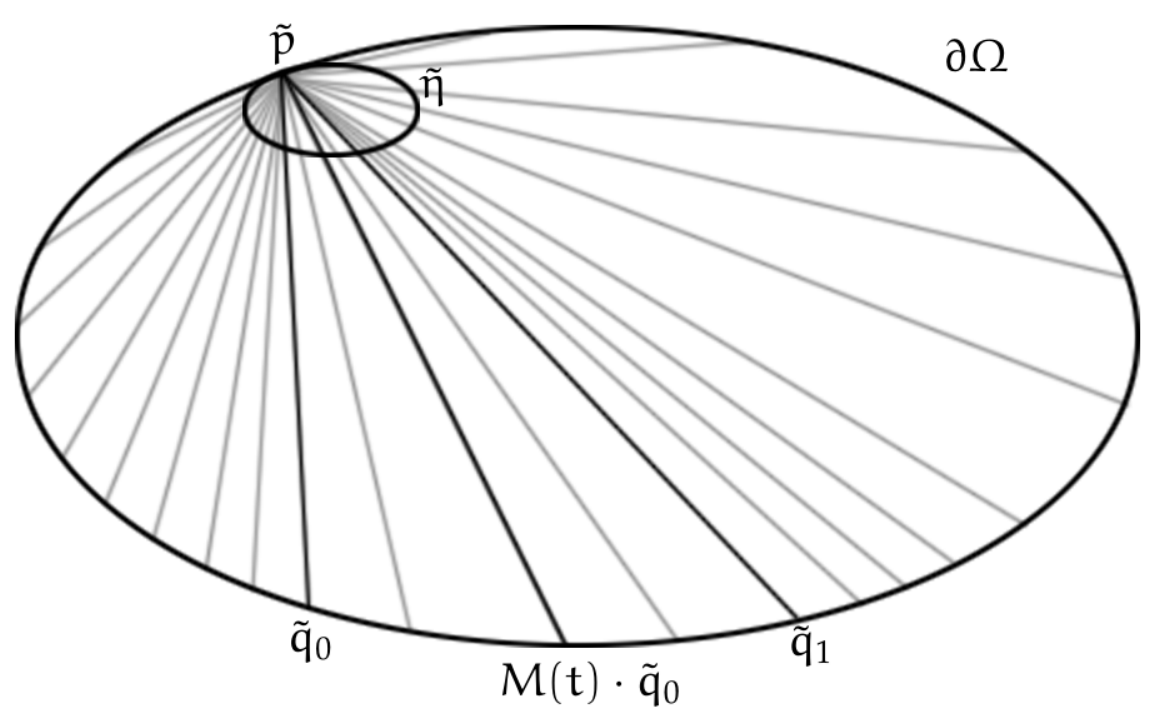}
\caption{The lighter grey lines specify a $C^1$ identification between $\tilde{\eta}$ and $\partial\Omega-\{\tilde{p}\}$.}
\label{fig:projection}
\end{figure}

\begin{lem}
The Goncharov--Shen potential measure of $\mathfrak{C}\cup\mathfrak{A}\subset\eta$ is $0$.
\end{lem}

\begin{proof}
In Theorem~\ref{thm:birmanseries} (we postpone the proof until the next Section), we show that the set of simple geodesics with respect to the Hilbert metric of $\Omega$ occupies zero Busemann area on $S_{g,m}$. This implies that the set of points on an $\epsilon$-neighborhood of $\eta$ which lie on simple geodesics occupies zero Busemann area. For sufficiently small $\epsilon>0$, the neighborhood of $\eta$ is annular and identifies with $\eta\times (-\epsilon,\epsilon)$. The restriction of the Busemann area on $\eta\times (-\epsilon,\epsilon)$ is in the same measure class as the product measure of the geodesic Hilbert length measure on $\eta$ multiplied by the Hilbert length measure on $(-\epsilon,\epsilon)$ because they differ by a strictly positive density function bounded away from $0$ and $\infty$. This in turn tells us that $\mathfrak{C}\cup\mathfrak{A}$ occupies $0$ horocyclic Hilbert length measure on $\eta$. By Lemma~\ref{lem:gsmeasure}, the Goncharov--Shen potential measure of $\mathfrak{C}\cup\mathfrak{A}$ is $0$. 
\end{proof}

\subsection{Generalizing step 2: indexing the gap intervals} 
\label{sec:halfpantsseries}

We first emphasize that $\mathfrak{C}\cup\mathfrak{A}$, when regarded as a subset of the ideal boundary $\partial_\infty\pi_1(S_{g,m})$, is a purely topological condition (see \cite[pg. 290 and 291]{LM09}) and is independent of the choice of auxiliary metric for $S_{g,m}$. Indeed, Labourie and McShane take advantage of this fact to express their identities in \cite{LM09} as series over homotopy classes of embeddings of a given pair of pants into $S_{g,m}$. It is clear that Labourie and McShane's summation index are in natural bijection with our boundary-parallel pairs of pants simply by choosing, once-and-for-all, boundary orientations on the domain pair of pants in the Labourie--McShane summation scheme so as to cause the domain to be a boundary-parallel pair of pants.

\subsubsection{Summation indices}
Our first summation scheme comes from a $2:1$ correspondence between the countable collection of open intervals in $\partial_\infty\pi_1(S_{g,m})-(\mathfrak{C}\cup\mathfrak{A})$ and the set of boundary-parallel pairs of pants on $S_{g,m}$ in Definition \ref{definition:Pp}. We consider refined summation scheme over boundary-parallel pairs of half-pants in Theorem~\ref{theorem:equsl3}. The idea of refining the summation scheme in this manner makes an appearance in both \cite{huangclosed} and  \cite[Theorem~4.5]{huang_thesis}.

\begin{defn}[Boundary-parallel pairs of half-pants]
\label{defn:hp}
Given a surface $S_{g,m}$ with negative Euler characteristic, an (embedded) \emph{boundary-parallel pair of half-pants} $\mu$ containing $p$ is one of the two pieces obtained by cutting along the unique simple bi-infinite geodesic of an embedded pair of pants on $S_{g,m}$ containing $p$, equipped with parallel orientations on the simple bi-infinite geodesic and the boundary component. We denote the collection of all boundary-parallel pairs of half-pants containing $p$ up to homotopy by $\xvec{\mathcal{H}}_p$. When $S_{g,m}$ is not a $1$-cusped torus, knowing the oriented boundaries $\gamma$ and $\gamma_p$ of a boundary-parallel pair of half-pants $\mu$, where $\gamma$ is closed and $\gamma_p$ is bi-infinite, suffices to uniquely specify $\mu$ and we adopt the notation $\mu=(\gamma,\gamma_p)$.
\end{defn}

\begin{figure}[h!]
\includegraphics[scale=0.2]{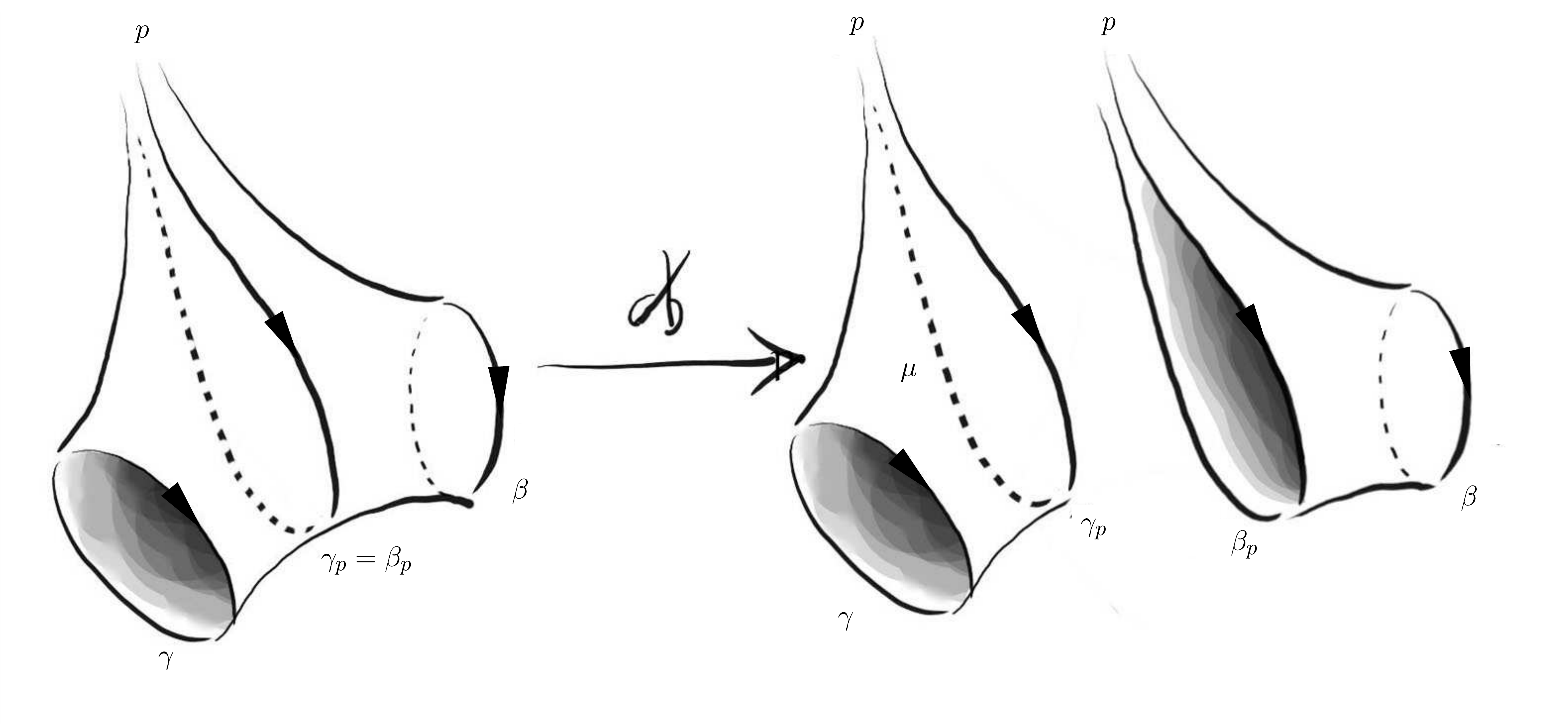}
\caption{Cutting a boundary-parallel pair of pants into two boundary-parallel pairs of half-pants $\mu=(\gamma,\gamma_p)$ and $(\beta,\beta_p)$.}
\label{fig:hpdecomp}
\end{figure}

\begin{rmk}[Notation for half-pants without boundary orientation]
We write $\bar{\mu}$ to refer to the underlying pair of half-pants for a boundary-parallel pair of half-pants $\mu$. Similarly, we use $(\bar{\gamma},\bar{\gamma}_p)$ (when $S_{g,m}$ is not a $1$-cusped torus) to refer to the underlying pair of half-pants without boundary orientation for $\mu=(\gamma,\gamma_p)$. Finally, we denote the collection of all pairs of half-pants containing $p$, up to homotopy, by $\mathcal{H}_p$.
\end{rmk}

\subsubsection{Geometric invariants}
We now introduce two types of geometric invariants of boundary-parallel pairs of half-pants used to express our McShane identity summands. The first (triangle invariants) is based on triple ratios and the second (half-pants ratios) is based on Goncharov--Shen potentials.

\begin{defn}[Triangle invariant for $\xvec{\mathcal{H}}_p$]
\label{definition:Tgammap}
For each boundary-parallel pair of half-pants $\mu=(\gamma,\gamma_p)$, the unique simple bi-infinite geodesic which shoots out from $p$ and spirals towards $\gamma$ parallel to its orientation cuts the underlying pair of half-pants $\bar{\mu}$ into a marked ideal triangle $\triangle_{\gamma,\gamma_p}$ as in Figure~\ref{fig:hpidealtriangle}. We adopt the notation
\begin{align*}
T(\gamma,\gamma_p):=T(\tilde{p},\gamma\cdot\tilde{p},\gamma^+)\text{ and }\tau(\gamma,\gamma_p):=\log T(\gamma,\gamma_p).
\end{align*}
\end{defn}

\begin{figure}[h!]
\includegraphics[scale=0.25]{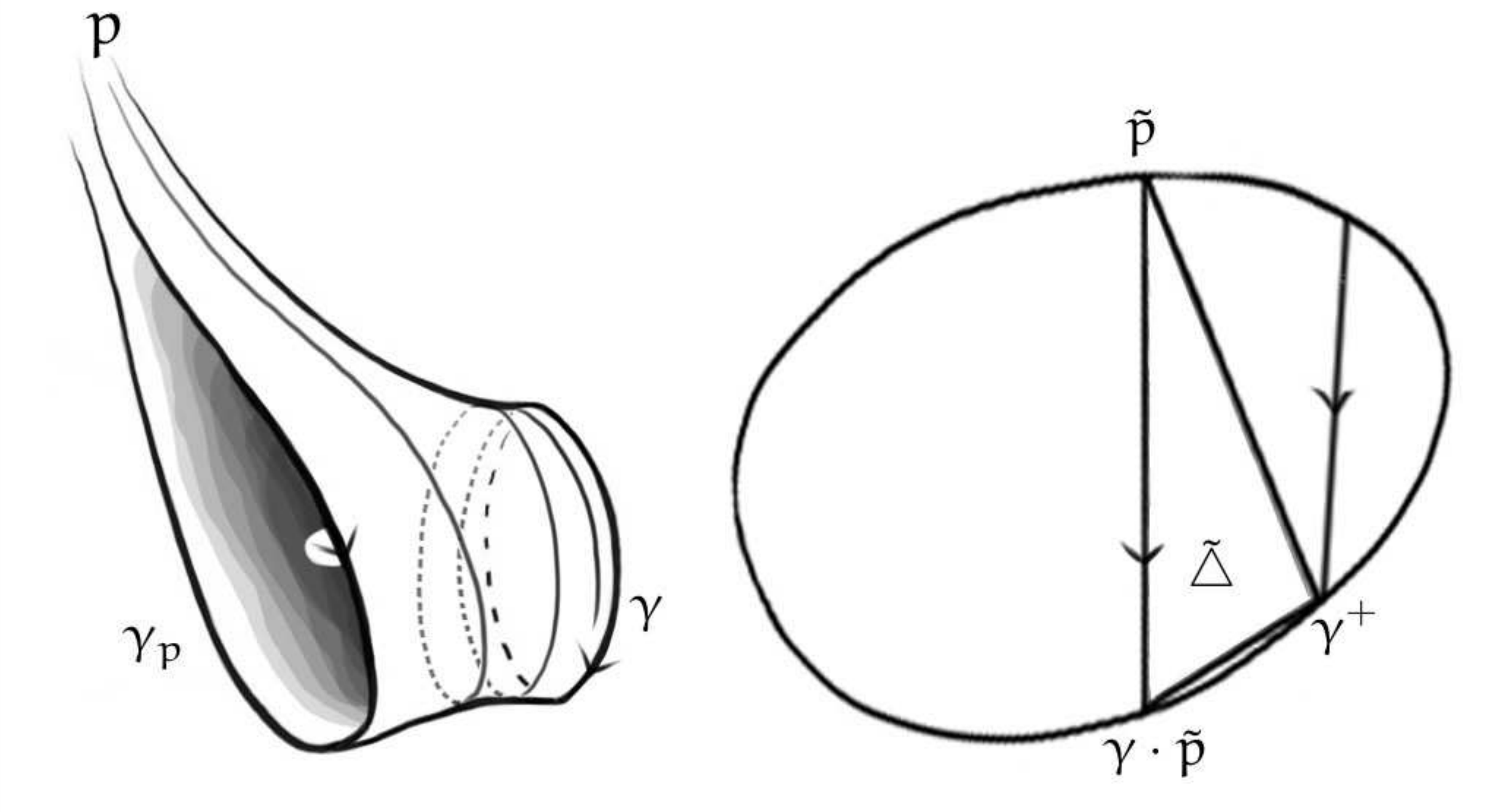}
\caption{Cutting along the spiraling geodesic on the boundary-parallel pair of half-pants $(\gamma,\gamma_p)$ (left figure) results in a marked ideal triangle $\triangle_{\gamma,\gamma_p}$, and the marked triangle $\tilde{\triangle}=(\tilde{p},\gamma\cdot\tilde{p},\gamma^+)$ (right figure) is a lift of $\triangle_{\gamma,\gamma_p}$.}
\label{fig:hpidealtriangle}
\end{figure}

The next geometric invariant we consider is naturally phrased in terms of the Goncharov--Shen potential measure. 

\begin{defn}[Half-pants ratio]
\label{definition:hpr}
Consider a surface $S_{g,m}$ with negative Euler characteristic endowed with a cusped strictly convex real projective surface structure $\Sigma_{g,m}$ (given via a holonomy representation which is positive and has unipotent boundary monodromy). Given any embedded pair of half-pants $\bar{\mu}\in\mathcal{H}_p$ on $S_{g,m}$. We define the \emph{$i$-th half-pants ratio} as the $i$-th Goncharov--Shen potential measure of the subinterval of any embedded horocycle $\eta$ around cusp $p$ lying on (the unique geodesic bordered homotopy representative of) $\bar{\mu}$, and denote it by $B_i(\bar{\mu})$. For any boundary-parallel pair of half-pants $\mu$ with underlying pair of half-pants $\bar{\mu}$, we also define $B_i(\mu):=B_i(\bar{\mu})$.
\end{defn}

\begin{rmk}\label{remark:Bgammap}
When $S_{g,m}$ is not the once-punctured torus, pairs of half-pants $\bar{\mu}$ are uniquely specified by its (unoriented) cuff $\bar{\gamma}$ and its (unoriented) seam $\bar{\gamma}_p$ (see Figure~\ref{fig:hpdecomp}). In these cases, we may write $B_i(\bar{\mu})$ as $B_i(\bar{\gamma},\bar{\gamma}_p)$ and $B_i(\mu)$ as $B_i(\gamma,\gamma_p)$.
\end{rmk}

Practically speaking, it is convenient to also be able to express half-pants ratios in terms of $\mathcal{A}$-coordinates.

\begin{defn}[$(\mu,i)$-Goncharov--Shen potential]
\label{definition:phalfgsp}
For $(\bar{\rho},\bar{\xi})\in \mathcal{A}_{\operatorname{SL}_n, S_{g,m}}$, and a pair of half-pants $\bar{\mu}$ with bi-infinite boundary $\bar{\gamma}_p$, let $\mathcal{T}$ be an ideal triangulation of $S_{g,m}$ which contains $\bar{\gamma}_p$ as one of its ideal edges. Choose a collection $\Theta_p$ of marked anticlockwise-oriented ideal triangles with the first vertex being $p$ as per Definition~\ref{definition:GSp}. Then, for $\bar{\mu} \in \mathcal{H}_p$, let $\Theta_{\bar{\mu}}$ denote set of marked ideal triangles $\triangle\in\Theta_p$ such that a small neighborhood of the first cusp of $\triangle$ is contained  in $\mu$. We define \emph{$(\mu,i)$-Goncharov--Shen potential} to be
\begin{align*}
P_i^\mu := P_i^{\bar{\mu}}:= \sum_{ (f;g,h) \in \Theta_\mu} P_i(f;g,h).
\end{align*}
\end{defn}

\begin{rmk}
The following relationship between the half-pants ratio and the $(\mu,i)$-Goncharov--Shen potential follows by definition:
\begin{align*}
B_i(\mu):=B_i(\bar{\mu})=
\frac{P^{\mu}_i}{P^p_i},
\end{align*}
where $P^{\mu}_i$ is the $(\mu,i)$-th Goncharov--Shen potential. Moreover, suppose that $\bar{\mu},\bar{\mu}'\in \mathcal{H}_p$ have the same bi-infinite geodesic boundary and (hence) glue to an embedded pair of pants containing $p$. Then we know from Lemma~\ref{lem:gsmeasure} that $B_i(\bar{\mu})+B_i(\bar{\mu}')=1$, or equivalently:
\begin{align*}
P_i^p=P_i^\mu+P_i^{\mu'}.
\end{align*}
\end{rmk}

\subsection{Generalizing step~3: McShane identity for the $S_{1,1}$ case}
\label{subsection:s11summand}
We begin by spelling out the $S_{1,1}$ case in detail. This is to motivate and familiarize readers to what needs to occur in general. We show the following:

\begin{thm}[McShane identity for $S_{1,1}$, $n=3$]
\label{theorem:inequsl3s11}
Let $\rho:\pi_1(S_{1,1})\rightarrow\operatorname{PGL}_3(\mathbb{R})$ be a positive representation with unipotent boundary monodromy. Let ${\xvec{\mathcal{C}}_{1,1}}$ denote the collection of oriented simple closed curves up to homotopy on $S_{1,1}$. Then 
\begin{align}
\label{equation:SL3S11identity}
\sum_{\gamma\in{\subvec{\mathcal{C}}_{1,1}}}
\frac{1}{1+e^{\ell_1(\gamma)+\tau(\gamma)}}
= 1,
\end{align}
where $\tau(\gamma)=\log T(\tilde{p},\gamma \tilde{p}, \gamma^+)$ is defined as per Figure~\ref{fig:torussummand}.
\end{thm}

\begin{rmk}
There are unexpected (not just topological) ``coincidences'' (see Equations~\eqref{equation:triplesym} and \eqref{equation:d121s11}, and Remark~\ref{rmk:trianglewelldefined}) in the specialized $\mathcal{A}_3(S_{1,1})$ setting which lead to the above elegant expression over the (relatively simpler) summation index. Furthermore, the identities obtained via the $i=1,2$ Goncharov--Shen potentials are equal (\S\ref{sec:surprisesymmetry}), this is specialized to the $S_{1,1}$ setting. 
\end{rmk}

\subsubsection{A sequence of ideal triangulations}

Our strategy for obtaining the gap term (i.e.: summands) of \eqref{equation:SL3S11identity} indexed by $\gamma\in\xvec{\mathcal{C}}_{1,1}$ is to compute the Goncharov--Shen potential measures of the intervals $I_1,I_2,I_3,I_4$ in Figure~\ref{Figure:gap} by approximation. Namely, consider the following bi-infinite sequence of ideal triangulations $\{\mathcal{T}^k\}_{k\in\mathbb{Z}}$ of $S_{1,1}$. Let
\begin{itemize}
\item
$\gamma_p$ denote the unique oriented simple bi-infinite geodesic $\gamma_p$ with both ends going up cusp $p$ such that $\gamma$ and $\gamma_p$ are the boundary components of some boundary parallel pair of half-pants on $S_{1,1}$; 
\item
$\mathcal{T}^0$ be any ideal triangulation of $S_{1,1}$ that contains $\gamma_p$ as an ideal edge.
\item
$\mathcal{T}^k$ be the ideal triangulation of $S_{1,1}$ obtained by applying the $k$-fold Dehn twist $\mathrm{tw}_\gamma^k$ along $\gamma$ to $\mathcal{T}^{0}$.
\end{itemize}
We construct a bi-infinite sequence of marked ideal triangles $\{\triangle^k\}_{k\in\mathbb{Z}}$, where $\triangle^k$ lies in $\mathcal{T}^k$, such that $\triangle^k$ converges to $\triangle_\gamma$ as $k$ tends to $\infty$ and to $\triangle_{\gamma^{-1}}$ as $k$ tends to $-\infty$. We then express, using $\mathcal{A}$-coordinates with respect to $\mathcal{T}^k$, the Goncharov--Shen potential measure of the horocyclic segment based at the first ideal vertex of $\triangle^k$ and take the limit as $k\to\pm\infty$. In order to obtain a sensible limit, we then need to have an understanding of the behavior of $\mathcal{A}$-coordinates with respect to $\mathcal{T}^k$ as $k\to\pm\infty$. This strategy was previously applied in \cite[Chapter 4.4.1]{huang_thesis} to compute McShane identity summands for hyperbolic surface via the $n=2$ Goncharov--Shen potential (i.e.: horocycle length).

 \subsubsection{Asymptotic behavior of $\mathcal{A}$-coordinates under Dehn twists} 

In order to facilitate our discussions regarding the behavior of $\mathcal{A}$-coordinates with respect to $\mathcal{T}^k$, we first develop some notation.\medskip

To begin with, given a positive decorated twisted local system $(\bar{\rho},\bar{\xi})\in\mathcal{A}_3(S_{1,1})$, the underlying surface group representation $\rho=\mathbf{pr}\circ \bar{\rho}$ (Remarks~\ref{remark:AX} and \ref{remark:ichartri}) provides the $\rho$-equivariant map $\bar{\xi}_\rho:\partial_\infty \pi_1(S)\rightarrow \mathcal{A}$ by deck transformations. We abuse notation slightly and let $\gamma\in\pi_1(S_{1,1})$ denote an arbitrary homotopy class representative of the oriented simple closed geodesic $\gamma$. Further let
\begin{itemize}
\item
$x=\tilde{p}$ be a lift of $p$ such that there exists a path from $x$ to the axes of $\gamma$ which projects to a simple path on $S_{1,1}$; 
\item
$t:=\gamma\cdot x$ be another lift of $p$;
\item
$y_0$ be yet another lift of $p$ such that the oriented ideal triangle $(x,y_0,t)$ is (the marked representative of) a lift of an ideal triangle in $\mathcal{T}^0$;
\item
$\{z_k:=y_{k+1}:=\gamma^{k+1}\cdot y_0\}$ be the orbit of $y_0$ with respect to the subgroup $\langle\gamma\rangle\leq\pi_1(S_{1,1})$ generated by $\gamma$.
\end{itemize}

\begin{figure}[H]
\includegraphics[scale=0.4]{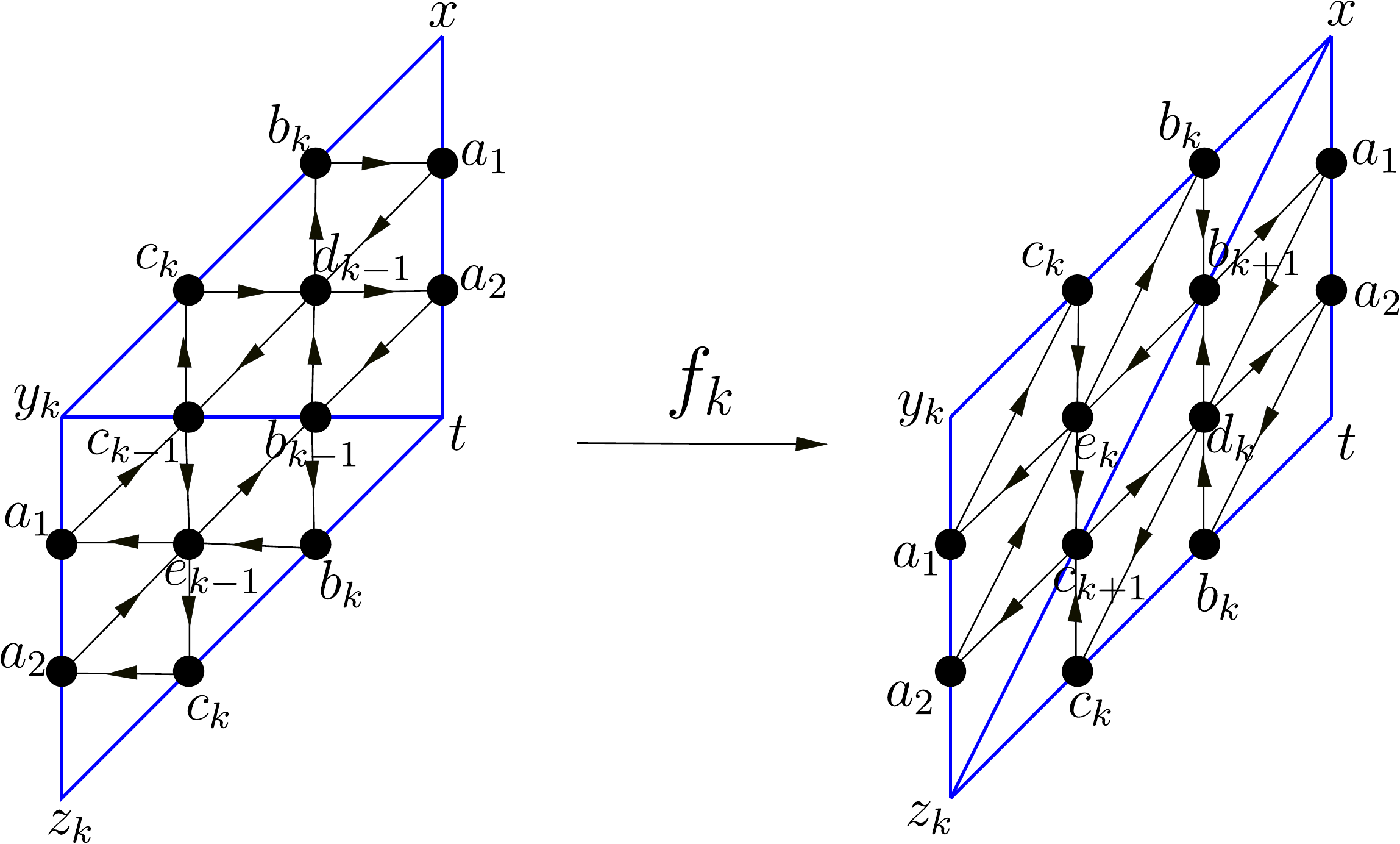}
\caption{A flip $f_k$ along the edge $\overline{y_k t}$.}
\label{Figure:cludyn}
\end{figure}

It is straight-forward to verify the following topological consequences:
\begin{itemize}
\item
the oriented geodesic from $x$ to $t$ and the oriented geodesics from $y_k$ to $z_k$ are all lifts of the oriented bi-infinite geodesic $\gamma_p$;
\item
$x,y_k,z_k,t$ constitute the ideal vertices of an ideal quadrilateral $\overline{xy_kz_kt}$, which is a fundamental domain for $S_{1,1}$;
\item
the ideal triangulation $\hat{\mathcal{T}}^k$ consisting of the two ideal triangles $\overline{xy_kt},\overline{y_ktz_k}$ is a lift of $\mathcal{T}^k$ (see Figure~\ref{Figure:cludyn}) and an ideal triangulation of $\overline{xy_kz_kt}$.
\end{itemize}

Let us coordinatize $\mathcal{A}_3(S_{1,1})$ with respect to $\mathcal{T}^k$ (or $\hat{\mathcal{T}}^k$) as per Figure~\ref{Figure:cludyn}. Then the action of the Dehn twist $tw_\gamma$ on $\mathcal{A}$-coordinates with respect to $\mathcal{T}^k$ is:
\[
(a_1,a_2, b_{k-1}, c_{k-1}, c_k, b_k, d_{k-1}, e_{k-1})\mapsto (a_1,a_2, b_{k}, c_{k}, c_{k+1}, b_{k+1}, d_{k}, e_{k}).
\]
Before continuing, we would like to thank Binbin Xu for clarifying our thinking regarding the following lemma:
\begin{lem}
\label{proposition:lim}
Consider $(\bar{\rho},\bar{\xi})\in \mathcal{A}_3(S_{1,1})$ and let $\gamma$ be a non-peripheral oriented simple closed curve on $S_{1,1}$. By Theorem~\ref{theorem:loxo} we know that the eigenvalues of $\rho(\gamma)$ satisfy 
\[
\lambda_1(\rho(\gamma))> \lambda_2(\rho(\gamma))> \lambda_3(\rho(\gamma))>0.
\]
Let $\{tw_\gamma^k\}_{k\in\mathbb{Z}}$ denote the bi-infinite sequence of Dehn-twists along $\gamma$, then the $\mathcal{A}$-coordinates (as per Figure~\ref{Figure:cludyn}) for $\mathcal{A}_3(S_{1,1})$ under the action of $\{tw_\gamma^k\}$ satisfy the following:
\begin{align}
\label{equ:bdce1}
\lim_{k\rightarrow + \infty} \tfrac{b_{k+1}}{b_k}
=\lim_{k\rightarrow + \infty} \tfrac{d_{k+1}}{d_k}
=\lim_{k\rightarrow - \infty} \tfrac{c_k}{c_{k+1}}
=\lim_{k\rightarrow - \infty} \tfrac{e_k}{e_{k+1}}
=\lambda_1(\rho(\gamma)),
\end{align}
and conversely,
\begin{align}
\label{equ:bdce2}
\lim_{k\rightarrow + \infty} \tfrac{c_{k+1}}{c_k}
=\lim_{k\rightarrow + \infty} \tfrac{e_{k+1}}{e_k}
=\lim_{k\rightarrow - \infty} \tfrac{b_{k}}{b_{k+1}}
=\lim_{k\rightarrow - \infty} \tfrac{d_k}{d_{k+1}}
=\lambda_1(\rho(\gamma^{-1})),
\end{align}
noting that $\lambda_1(\rho(\gamma^{-1}))=\lambda_1(\rho(\gamma)) \lambda_2(\rho(\gamma))$. Moreover, the following limits exist and yield strictly positive real numbers
\begin{align}
\label{equ:bdce3}
\lim_{k\rightarrow +\infty} 
\tfrac{b_k}{d_k}, \;
\lim_{k\rightarrow -\infty} 
\tfrac{b_k}{d_k},\;
\lim_{k\rightarrow +\infty} 
\tfrac{c_{k}}{e_k},\;
\lim_{k\rightarrow -\infty} \tfrac{c_k}{e_k}\in\mathbb{R}_{>0}.
\end{align}
\end{lem}

\begin{proof}
Recall that $t=\gamma \cdot x$, $y_k=\gamma^k\cdot y_0$ and $z_k=\gamma^{k+1} \cdot y_0$. Let 
\[
(x_1,x_2,x_3),\quad (y_{k,1},y_{k,2},y_{k,3}),\quad (z_{k,1},z_{k,2},z_{k,3}), \quad(t_1,t_2,t_3)
\] 
denote respective bases for the flags $\xi_\rho(x),\xi_\rho(y_k),\xi_\rho(z_k)$, $\xi_\rho(t)$ and let $v_1,v_2,v_3$ denote the respective eigenvectors for the eigenvalues $\lambda_1(\rho(\gamma)), \lambda_2(\rho(\gamma)),\lambda_3(\rho(\gamma))$ of $\rho(\gamma)$. The $\mathcal{A}$-coordinates $(a_1,a_2, b_{k-1}, c_{k-1}, c_k, b_k, d_{k-1}, e_{k-1})$ for $(\bar{\rho},\bar{\xi})$ are functions of these bases (Definition~\ref{defn:FGAcoordinate}), and in particular, we have:
\[
\frac{d_{k+1}}{d_{k}} 
=\left|
\frac{\Delta\left(x_1 \wedge \rho(\gamma^{k+2}) y_{0,1} \wedge \rho(\gamma) x_1\right)}
{\Delta\left(x_1 \wedge \rho(\gamma^{k+1}) y_{0,1} \wedge \rho(\gamma) x_1\right)}
\right|.
\]
Since $y\neq\gamma^-$, the sequence of ideal points  $\{\gamma^k\cdot y\}$ necessarily converges to the attracting fix point $\gamma^+$ as $k$ approaches $+\infty$, and hence the vector $\rho(\gamma^{k}) y_{0,1}$ converges to the eigenvector $v_1$. Thus the ratio $\frac{d_{k+1}}{d_k}$ converges to $\lambda_1(\rho(\gamma))$ as $k$ goes to $+\infty$. All other cases in Equations~\eqref{equ:bdce1}, \eqref{equ:bdce2} follow by essentially the same argument.\medskip

Furthermore, we have 
\[
\lim_{k\rightarrow +\infty} 
\frac{b_k}{d_k}
=\lim_{k\rightarrow +\infty}  
\left|
\frac{\Delta\left(x_1 \wedge x_2 \wedge \rho(\gamma^k) y_{0,1} \right)}
{\Delta\left(x_1 \wedge \rho(\gamma^{k+1}) y_{0,1} \wedge \rho(\gamma) x_1\right)}
\right|
=
\left|
\frac{\Delta\left(x_1 \wedge x_2 \wedge v_1 \right)}
{\Delta\left(x_1 \wedge \rho(\gamma) v_1 \wedge \rho(\gamma) x_1\right)}
\right|,
\]
which is well-defined and strictly positive because $\xi_\rho$ is hyperconvex. All other cases in Equation~\eqref{equ:bdce3} follow by the same argument.
\end{proof}

\begin{rmk}
\label{remark:generallim}
Although Lemma~\ref{proposition:lim} is stated for $\mathcal{A}$-coordinates with respect to a triangulation on $S_{1,1}$, the proof itself does not use the full topology of $S_{1,1}$ and merely requires that we have an annulus around some $\gamma$ with one puncture on each of the boundaries of the annulus. The resulting $\mathcal{A}$-coordinates for an ideal triangulation of such an annulus is depicted in Figure~\ref{Figure:cpp}.
\end{rmk}

\subsubsection{Limiting summand for $S_{1.1},\;n=3$ case}

\begin{proof}[Proof of Theorem~\ref{theorem:inequsl3s11}] 
We continue to use the notation developed in the previous two subsubsections. Given an oriented simple closed geodesic $\gamma$, the goal of this subsection is compute the McShane identity summand indexed by $\gamma$. Recall that the relevant term is the sum of the $i=1$ Goncharov--Shen potential measures of the horocyclic segments $I_1$ and $I_2$ as depicted in Figure~\ref{Figure:gap} (albeit with $\beta$ identified with $\gamma$), and hence takes the form:
\begin{align*}
\lim_{k\rightarrow + \infty}
\tfrac{P_1(x;y_k,t)}{P_1^p}+  
\lim_{k\rightarrow -\infty} 
\tfrac{P_1(y_k;z_k,x)}{P_1^p}. 
\end{align*}
Expressed in terms of $\mathcal{A}$-coordinates as per Figure~\ref{Figure:cludyn}, this is equal to
\begin{align*}
\lim_{k\rightarrow + \infty}
\tfrac{d_{k-1}}{a_1 b_k P_1^p}
+ \lim_{k\rightarrow -\infty}
\tfrac{e_k}{ a_1c_kP_1^p}
\end{align*}
The closed oriented geodesic $\gamma$ and the bi-infinite ideal $\gamma_p$ constitute the oriented boundaries of two distinct pairs of boundary-parallel half pants $\mu,\mu'\in\xvec{\mathcal{H}}_p$. Let $\mu$ denote the pair of half-pants with $(\mu,1)$-Goncharov--Shen potential
\[
P_1^{\mu}=
{\color{red}P_1(x;y_k,t)}
+{\color{blue}P_1(t;x,z_k)},
\]
and let $\mu'$ denote the pair of half-pants with $(\mu',1)$-Goncharov--Shen potential
\[
P_1^{\mu'}=P_1(y_k;t,x)+P_1(z_k;t,y_k).
\]
Then, expressed in terms of $\mathcal{A}$-coordinates, we have: 
\[
P_1^\mu=
{\color{red}\tfrac{d_{k-1}}{a_1b_k}}
+{\color{blue}\tfrac{d_k}{a_2b_k}},
\]
where the {\color{red}red} term is expressed in terms of the $\mathcal{A}$-coordinates for $\mathcal{T}^k$ and the {\color{blue}blue} term is expressed with respect to $\mathcal{T}^{k+1}$. Invoking Lemma~\ref{proposition:lim} to take limits, we obtain the following:
\begin{align}
\begin{aligned}
\label{equation:s11sl3bmu}
\lim_{k\rightarrow + \infty}\frac{d_{k-1}}{b_k a_1 P_1^p} 
&=\lim_{k\rightarrow + \infty} \frac{\frac{d_{k-1}}{ a_1b_k}}{\frac{d_{k-1}}{ a_1b_k}+ \frac{d_k}{ a_2b_k}} \cdot \frac{P_1^\mu}{P_1^p}
=\lim_{k\rightarrow + \infty} \frac{B_1(\mu)}{1+ \frac{a_1 d_k}{a_2 d_{k-1}}}\\
&= \frac{B_1(\mu)}{1+ \frac{a_1 }{a_2} \lambda_1(\rho(\gamma))}.
\end{aligned}
\end{align}
We similarly obtain:
\begin{align*}
\lim_{k\rightarrow -\infty} \frac{e_k}{c_k a_1 P_1^p}
=\frac{B_1(\mu')}{1+ \frac{a_1 }{a_2} \lambda_1(\rho(\gamma))}
\end{align*}
Since $B_1(\mu)+ B_1(\mu')=1$, the two summands add to
\begin{align*}
\frac{1}{1+ \frac{a_1}{a_2} \lambda_1(\rho(\gamma))}.
\end{align*}
Moreover, by Lemma~\ref{proposition:lim}, observe that
\begin{align}
\begin{aligned}
\label{eq:aal}
T(\tilde{p},\gamma \tilde{p}, \gamma^+)
&=\lim_{k \rightarrow + \infty} \frac{a_1 b_{k-1} c_k}{a_2 c_{k-1} b_k} = \frac{a_1}{a_2} \cdot \lim_{k \rightarrow + \infty}\frac{b_{k-1}}{b_k} \cdot \lim_{k \rightarrow + \infty} \frac{c_{k}}{c_{k-1}} \\
&= \frac{a_1}{a_2}\cdot \frac{1}{\lambda_1(\rho(\gamma))}\cdot \lambda_1(\rho(\gamma))\lambda_2(\rho(\gamma))= \frac{a_1 \lambda_2(\rho(\gamma))}{a_2}.
\end{aligned}
\end{align}
Thus the McShane identity summand for $\gamma\in\xvec{\mathcal{C}}_{1,1}$ is
\begin{align*}
\left(1+T(\tilde{p},\gamma \tilde{p}, \gamma^+)\cdot \frac{\lambda_1(\rho(\gamma))}{\lambda_2(\rho(\gamma))}\right)^{-1}
&=\left(1+e^{\ell_1(\gamma)+\tau(\gamma)}\right)^{-1},\\
\text{and hence}
\quad
\sum_{\gamma\in{\subvec{\mathcal{C}}_{1,1}}}
\frac{1}{1+e^{\ell_1(\gamma)+\tau(\gamma)}}
&=1\quad\text{as desired.}
\end{align*}
\end{proof}

\begin{rmk}
\label{remark:c11}
In the $3$-Fuchsian locus, we know that $\tau(\gamma)=0$ and $\ell_1(\gamma)=\ell_1(\gamma^{-1})$, the gap term $\frac{1}{1+e^{\ell_1(\gamma)+\tau(\gamma)}}$ is same as $\frac{1}{1+e^{\ell_1(\gamma^{-1})+\tau(\gamma^{-1})}}$ for any $\gamma \in \xvec{\mathcal{C}}_{1,1}$. After catering for the canonical $2:1$ orientation-forgetting map between $\xvec{\mathcal{C}}_{1,1}$ and $\mathcal{C}_{1,1}$, one immediately obtains the classical McShane identity.
\end{rmk}

\begin{rmk}
\label{rmk:trianglewelldefined}
Recall from the statement of Theorem~\ref{thm:firstn3S11}, that there are two possible candidate ideal triangles spiraling to $\gamma$ which may be used to define $\tau(\gamma)$. Equation~\ref{eq:aal} ensures that their triangle invariants are equal, and hence that $\tau(\gamma)$ is well-defined.
\end{rmk}

\subsubsection{Extra symmetries in the $S_{1,1},\;n=3$ case.}
\label{sec:surprisesymmetry}

\begin{lem}
\label{lem:12termcompare}
Let $\rho:\pi_1(S_{1,1})\to\operatorname{PGL}_3(\mathbb{R})$ be a positive representation with unipotent boundary, then the McShane identity summand indexed by every $\gamma\in\xvec{C}_{1,1}$ satisfies
\begin{align*}
\frac{1}{1+e^{\ell_1(\gamma^{-1})+\tau(\gamma^{-1})}}=
\frac{1}{1+e^{\ell_2(\gamma)-\tau(\gamma)}}.
\end{align*}
\end{lem}

\begin{proof}
First observe that
\[
\ell_2(\gamma)
=\log\left(\tfrac{\lambda_2(\rho(\gamma))}{\lambda_3(\rho(\gamma))}\right)
=\log\left(\tfrac{\lambda_2(\rho(\gamma^{-1}))^{-1}}{\lambda_1(\rho(\gamma^{-1}))^{-1}}\right)
=\ell_1(\gamma^{-1}).
\]
Then, utilizing properties of triple ratios and  Lemma~\ref{proposition:lim}, we obtain that
\begin{align*}
&T(\tilde{p},\gamma^{-1} \tilde{p}, \gamma^-)
=T(\gamma \tilde{p}, \tilde{p},\gamma^-)
=T(\tilde{p},\gamma \tilde{p}, \gamma^-)^{-1} 
=\left(\lim_{k \rightarrow - \infty} \frac{a_1 b_{k-1} c_k}{a_2 c_{k-1} b_k}\right)^{-1} 
\\=& \left(\frac{a_1}{a_2} \cdot \lim_{k \rightarrow - \infty}\frac{b_{k-1}}{b_k} \cdot \lim_{k \rightarrow - \infty} \frac{c_{k}}{c_{k-1}}\right)^{-1} 
= \frac{a_2}{a_1 \lambda_2(\rho(\gamma))}.
\end{align*}
Combined with Equation~\eqref{eq:aal}, we obtain 
\begin{align}
\label{equation:triplesym}
T(\tilde{p},\gamma \tilde{p}, \gamma^+) \cdot T(\tilde{p},\gamma^{-1} \tilde{p}, \gamma^-) =1,\quad
\text{and hence}\quad
\tau(\gamma^{-1})=-\tau(\gamma).
\end{align}
\end{proof}
As an immediate corollary to Lemma~\ref{lem:12termcompare}, we obtain the following (alternative form of the) $S_{1,1}$ McShane identity:

\begin{prop}
\label{prop:dualidentity}
Let $\rho:\pi_1(S_{1,1})\rightarrow\operatorname{PGL}_3(\mathbb{R})$ be a positive representation with unipotent boundary monodromy. Let ${\xvec{\mathcal{C}}_{1,1}}$ denote the collection of oriented simple closed curves up to homotopy on $S_{1,1}$. Then 
\begin{align*}
\sum_{\gamma\in{\subvec{\mathcal{C}}_{1,1}}}
\frac{1}{1+e^{\ell_2(\gamma)-\tau(\gamma)}}
= 1,
\end{align*}
where $\tau(\gamma)=\log T(\tilde{p},\gamma \tilde{p}, \gamma^+)$ is defined as per Figure~\ref{fig:torussummand}.
\end{prop}

In fact, Proposition~\ref{prop:dualidentity} is precisely the identity that one obtains if one uses the $i=2$ Goncharov--Shen potential measure instead of the the $i=1$ measure. The following lemma, specific to $(\bar{\rho},\bar{\xi})\in \mathcal{A}_3(S_{1,1})$, is the root cause of this extra symmetry.

\begin{lem}
\label{lemma:extrasym}
Adopting notation as per Figure~\ref{Figure:cludyn}, then we have:
\begin{align}
\label{equation:extrasym}
\begin{aligned}
\tfrac{P_1(z_k;t,y_k)}{P_1^p}=\tfrac{P_2(x;y_k,t)}{P_2^p},
\tfrac{P_2(z_k;t,y_k)}{P_2^p}=\tfrac{P_1(x;y_k,t)}{P_1^p},\\
\tfrac{P_1(y_k;z_k,x)}{P_1^p}=\tfrac{P_2(t;x,z_k)}{P_2^p},
\tfrac{P_1(y_k;z_k,x)}{P_1^p}=\tfrac{P_2(t;x,z_k)}{P_2^p}.
\end{aligned}
\end{align}
\end{lem}

\begin{rmk}
\label{remark:extrasym}
The equations in \eqref{equation:extrasym} all take the following form: 
\[
B_1(I)=B_2(I'),
\]
where $I$ and $I'$ are horocyclic segments topologically related by the (topological) hyperelliptic involution of $S_{1,1}$. This is highly suggestive that this fact could be true for any two horocyclic segment $I,I'$ related by hyperelliptic involution. In any case, this symmetry is specific to $S_{1,1}$, and one can easily construct counter-examples for general surfaces. 
\end{rmk}

\begin{proof}[Proof of Lemma~\ref{lemma:extrasym}]
We first prove that 
\[
\tfrac{P_2(z_k;t,y_k)}{P_1(x;y_k,t)} = \tfrac{P_2(t;x,z_k)}{P_1(y_k;z_k,x)},
\] 
noting that this is tantamount to showing:
\begin{align}
\label{equation:p1p2s}
\tfrac{a_2 e_k c_{k-1}}{a_1^2 c_k e_{k-1}}
+ \tfrac{e_k b_{k-1}}{a_1 b_k e_{k-1}} 
= \tfrac{a_2 d_{k-1} b_{k+1}}{a_1^2 b_k d_k }
+ \tfrac{d_{k-1} c_{k+1}}{a_1 c_k d_k}.
\end{align}
Using mutation formulae, we make the following substitutions
\[
b_{k+1}=\tfrac{b_k d_k+a_1 e_k}{d_{k-1}},\;
c_{k+1}=\tfrac{e_k c_k+d_k a_2}{e_{k-1}},\;
e_k c_{k-1}=c_k e_{k-1}+ a_1 d_{k-1}
\]
to reduce Equation~\eqref{equation:p1p2s} to the mutation formula at $d_k$:
\begin{align*}
b_{k-1} d_k = a_2 e_{k-1} + b_k d_{k-1},
\end{align*}
which we know to be true. Then, by symmetry, we have:
\[
\tfrac{P_2(x;y_k,t)}{P_1(z_k;t,y_k)}
=\tfrac{P_2(y_k;z_k,x)}{P_1(t;x,z_k)}.
\]
By computing directly
\[
\tfrac{P_2(z_k;t,y_k)}{P_1(x;y_k,t)} 
=\tfrac{P_2(x;y_k,t)}{P_1(z_k;t,y_k)}.
\]
Thus we obtain:
\[
\tfrac{P_2(z_k;t,y_k)}{P_1(x;y_k,t)} 
=\tfrac{P_2(t;x,z_k)}{P_1(y_k;z_k,x)}
=\tfrac{P_2(x;y_k,t)}{P_1(z_k;t,y_k)}
=\tfrac{P_2(y_k;z_k,x)}{P_1(t;x,z_k)}.
\]

We now use the fact that: for $a,b,c,d\in\mathbb{R}_{>0}$, if $\frac{a}{b}=\frac{c}{d}$, then $\frac{a}{b}=\frac{c}{d}=\frac{a+c}{b+d}$, to conclude that 
\begin{align*}
\tfrac{P_2(x;y_k,t)}{P_1(z_k;t,y_k)} 
=\tfrac{P_2(z_k;t,y_k)}{P_1(x;y_k,t)} 
=\tfrac{P_2(t;x,z_k)}{P_1(y_k;z_k,x)} 
=\tfrac{P_2(y_k;z_k,x)}{P_1(t;x,z_k)}
=\tfrac{P_2^p}{P_1^p},
\end{align*}
rearranging yields the desired result.
\end{proof}

\begin{rmk}
We previously asserted that Proposition~\ref{prop:dualidentity} follows from Lemma~\ref{lemma:extrasym}. To see this, recall Remark~\ref{remark:extrasym} and observe that:
\begin{itemize}
\item
the horocyclic intervals $I_1$ and $I_2$ (in Figure~\ref{Figure:gap}, with $\beta$ identified with $\gamma$) are respectively related to the intervals $I_3$ and $I_4$ via hyperelliptic involution;
\item
each of these intervals arises as a limit of the intervals (implicitly) described in Equation~\eqref{equation:extrasym}.
\end{itemize}
\end{rmk}

We also make mention of the following simple corollary of Lemma~\ref{lemma:extrasym}:
\begin{lem}
\label{lem:compareBs}
Consider an oriented simple closed geodesic $\gamma\in\xvec{C}_{1,1}$ on a finite-area $1$-cusped convex real projective torus $\Sigma_{1,1}$, and let $\gamma_p$ be the unique oriented simple ideal geodesic which is boundary-parallel to $\gamma$. The oriented geodesics $\gamma$ and $\gamma_p$ constitute the boundaries of two boundary-parallel pair of half-pants $\mu$ and $\mu'$. Then,
\[
B_1(\mu)=B_2(\mu')
\quad\text{and}\quad
B_1(\mu')=B_2(\mu).
\]
\end{lem}

\subsection{Generalizing step~3: McShane identity for series over pairs of half-pants}
\label{subsection:bphpsummand}
\begin{thm}
\label{theorem:equsl3}
Let $\rho:\pi_1(S_{g,m})\rightarrow\operatorname{PGL}_3(\mathbb{R})$ be a positive representation with unipotent boundary monodromy and let $p$ be a distinguished cusp on $S_{g,m}$. Let $\xvec{\mathcal{H}}_{p}$ be the set of the homotopy classes of boundary-parallel pairs of half-pants containing $p$ (Definition~\ref{defn:hp}). Then,
\begin{align*}
\sum_{(\gamma,\gamma_p)\in\subvec{\mathcal{H}}_{p}}
\frac{B_1(\gamma,\gamma_p)}{1+e^{\ell_1(\gamma)+\tau(\gamma,\gamma_p)}}
=\sum_{(\gamma,\gamma_p)\in\subvec{\mathcal{H}}_{p}}
\frac{B_2(\gamma,\gamma_p)}{1+e^{\ell_2(\gamma)-\tau(\gamma,\gamma_p)}}
= 1,
\end{align*}
where $B_i(\gamma,\gamma_p)$ is the $i$-th half-pants ratio (Definition~\ref{definition:hpr}) and $\tau(\gamma,\gamma_p):=\log T(\tilde{p},\gamma\cdot\tilde{p},\gamma^+)$ as per Figure~\ref{fig:hpidealtriangle}.
\end{thm}

\begin{proof}
We need to compute the McShane identity summand indexed by an arbitrary boundary-parallel pair of half-pants $\mu = (\gamma,\gamma_p)\in\xvec{\mathcal{H}}_p$. The relevant summand is the Goncharov--Shen potential measure of $I_1$, as depicted in Figure~\ref{Figure:gap}, and we use essentially the same derivation as for the $S_{1,1}$ case.\medskip

First observe that the summand is only dependant upon the geometry of $(\gamma,\gamma_p)$ and we may therefore assume without loss of generality that there exists another embedded pair of half-pants $\hat{\mu}$ on the given surface which glues with $\mu$ along $\gamma$ to yield an ``annulus'' as per the right hand side of Figure~\ref{Figure:cpp}. We designate a $k\in\mathbb{Z}$ family of $\mathcal{A}$-coordinates (left hand side of Figure~\ref{Figure:cpp}) related by $k$-th Dehn twist along $\gamma$ in analogy to the coordinates described in Figure~\ref{Figure:cludyn}.

\begin{figure}[h!]
\includegraphics[scale=0.4]{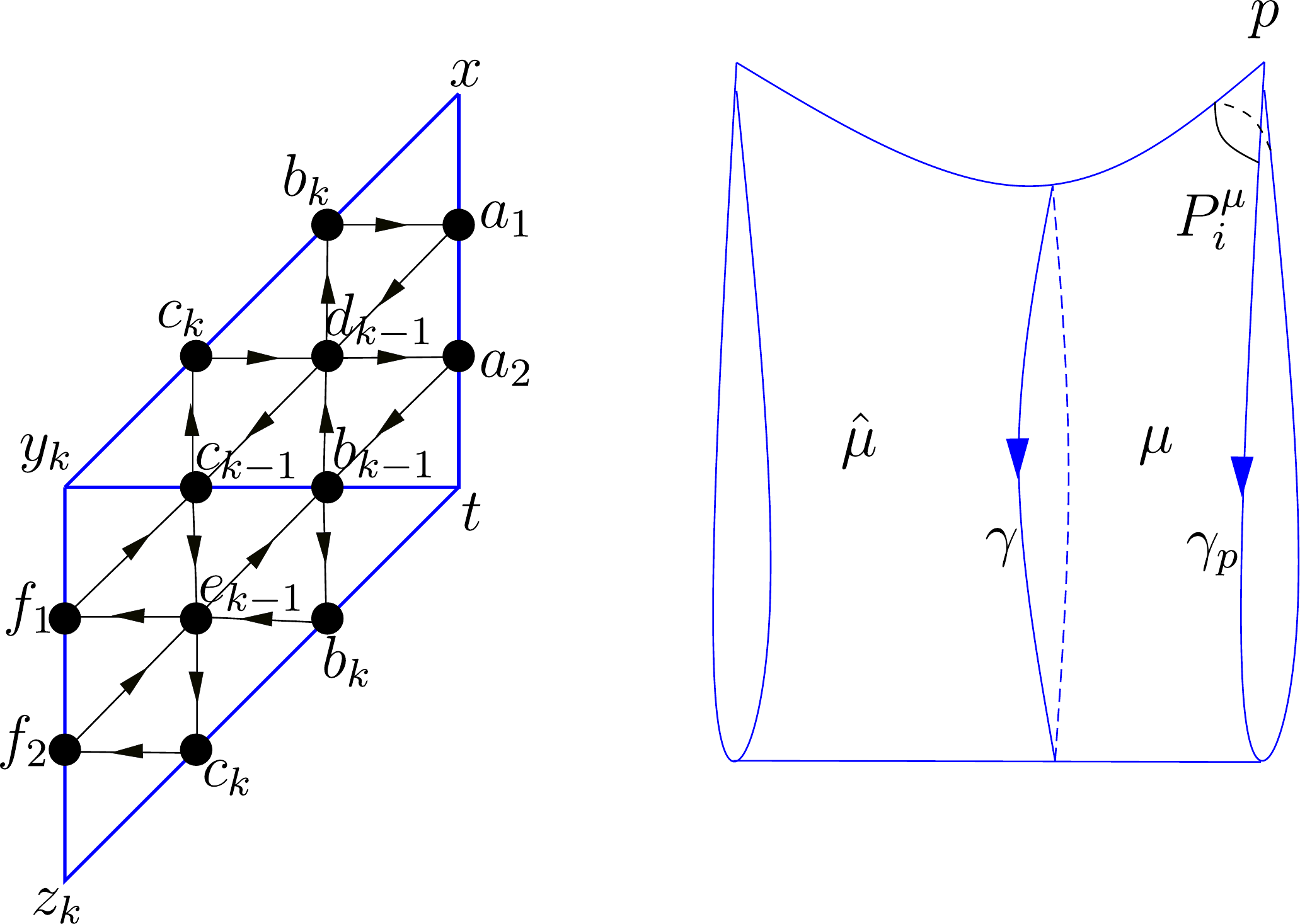}
\caption{The embedded boundary-parallel pair of half-pants $\mu$ is glued to another embedded boundary-parallel pair of half-pants $\hat{\mu}$.}
\label{Figure:cpp}
\end{figure}

We again have $t= \gamma\cdot x$, and $z_k=y_{k+1}=\gamma^{k+1}\cdot y_0$, and the summand that we require is once again
\begin{align}
\label{equation:sgmsl3}
\begin{aligned}
\lim_{k\rightarrow + \infty}\frac{d_{k-1}}{b_k a_1 P_1^p} 
=\lim_{k\rightarrow + \infty} \tfrac{\frac{d_{k-1}}{b_k a_1}}{\frac{d_{k-1}}{b_k a_1}+ \frac{d_k}{b_k a_2}} \cdot \frac{P_1^\mu}{P_1^p}
=\lim_{k\rightarrow + \infty} \tfrac{B_1(\mu)}{1+ \frac{a_1 d_k}{a_2 d_{k-1}}}= \tfrac{B_1(\mu)}{1+ \frac{a_1 }{a_2} \lambda_1(\rho(\gamma))},
\end{aligned}
\end{align}
where Lemma~\ref{proposition:lim} and Remark~\ref{remark:generallim} allow us to take the limit. Again invoking Equation~\eqref{eq:aal}, we obtain the desired summand of
\[
\frac{B_1(\gamma,\gamma_p)}{1+e^{\ell_1(\gamma)+\tau(\gamma,\gamma_p)}}.
\]
The analogous computed $i=2$ Goncharov--Shen potential measure for $I_1$ is
\[
\frac{B_2(\gamma,\gamma_p)}{1+e^{\ell_2(\gamma)-\tau(\gamma,\gamma_p)}}.
\]
\end{proof}

\subsection{Generalizing step~3: McShane identity for series over pairs of pants}
\label{subsection:bpppsummand}
Finally, we prove the pairs of pants summation form of the McShane identity. Recall that the present series is indexed by $\xvec{\mathcal{P}}_p$ --- the set of boundary-parallel pairs of pants (Definition~\ref{definition:Pp}). 

\begin{thm}
\label{thm:equsl3pp}
Let $\rho:\pi_1(S_{g,m})\rightarrow\operatorname{PGL}_3(\mathbb{R})$ be a positive representation with unipotent boundary monodromy and let $p$ be a distinguished cusp on $S_{g,m}$. Then 
\begin{align*}
\sum_{(\beta,\gamma)\in\subvec{\mathcal{P}}_p}
\left(
1+\tfrac{\cosh \frac{e_1(\beta,\gamma)}{2}}
{\cosh \frac{d_1(\beta,\gamma)}{2}}
\cdot 
e^{\frac{1}{2}\left(\ell_1(\beta)+\tau(\beta,\beta_p)+\ell_1(\gamma)+\tau(\gamma,\gamma_p)\right)}
\right)^{-1}
= 1.
\end{align*}
where $d_1(\beta,\gamma)$ and $e_1(\beta,\gamma)$ are edge invariants (Definition~\ref{definition:edgeinvariant}), and $\tau(\gamma,\gamma_p)$ and $\tau(\beta,\beta_p)$ triangle invariants (Definition~\ref{definition:Tgammap}).
\end{thm}

\begin{proof}
The McShane identity summand indexed by the boundary-parallel pairs of pants $(\beta,\gamma)\in\xvec{\mathcal{P}}_p$ is the Goncharov--Shen potential measure of $I_1\cup I_2$ as illustrated in Figure~\ref{Figure:gap}. We adapt our previous strategy for computing half-pants gap terms (i.e.: the half-pants McShane identity summands) by describing a $k\in\mathbb{Z}$ family of ideal triangulations of the surface depicted in the right hand side of Figure~\ref{Figure:ppkflp}, such that this family of ideal triangulations are related by simultaneous Dehn-twists in $\beta^{-1}$ and $\gamma$. \medskip

The intervals $I_1$ and $I_2$ respectively lie on the boundary-parallel pairs of half-pants $\mu=(\gamma,\gamma_p)$ and $\mu'=(\beta,\beta_p)$. As in the proof of Theorem~\ref{theorem:equsl3}, we may respectively attach boundary-parallel pairs of half-pants $\hat{\mu}$ and $\hat{\mu}'$ to $\mu$ and $\mu'$ along $\gamma$ and $\beta$. Having done so, 
\begin{itemize}
\item
fix any ideal triangulation $\mathcal{T}^0$ of $\hat{\mu}\cup \mu \cup \mu'\cup \hat{\mu}'$ which contains the unoriented ideal geodesic underlying $\beta_p=\gamma_p$;
\item
let $\mathcal{T}^k$ denote the ideal triangulation $\mathcal{T}^k$ of $\hat{\mu}\cup \mu \cup \mu'\cup \hat{\mu}'$ obtained from Dehn twisting $\mathcal{T}^0$ by $(\mathrm{tw}_{\beta}^{-1} \cdot \mathrm{tw}_{\gamma})^k$.
\end{itemize}
With a slight abuse of notation, we now regard the oriented simple geodesics $\beta,\gamma$ as homotopy representatives $\beta,\gamma\in\pi_1(S_{g,m})$ such that $\gamma\beta^{-1}$ corresponds to a simple loop going around cusp $p$ once. Then we construct a sequence of lifts of $\{\mathcal{T}^k\}_{k\in\mathbb{Z}}$ (see the left hand side of Figure~\ref{Figure:ppkflp}) as follows: let
\begin{itemize}
\item
$x=\tilde{p}$ denote the unique lift of $p$ which is fixed by $\gamma^{-1}\beta$, then $\gamma\cdot x=\beta\cdot x$;
\item
$y_0$ and $v_0$ be lifts of $p$ such that the ideal triangles $(x,y_0,\gamma x)$ and $(x,\gamma x,v_0)$ are lifts anti-clockwise oriented ideal triangles in $\mathcal{T}^0$;
\item
$\{z_k:=y_{k+1}:=\gamma^{k+1}\cdot y_0\}$ be the orbit of $y_0$ with respect to the subgroup $\langle\gamma\rangle\leq\pi_1(S_{1,1})$ generated by $\gamma$;
\item
$\{w_k:=v_{k+1}:=\beta^{k+1}\cdot v_0\}$ be the orbit of $v_0$ with respect to the subgroup $\langle\beta\rangle\leq\pi_1(S_{1,1})$ generated by $\beta$.
\end{itemize}

\begin{figure}[h!]
\includegraphics[scale=0.35]{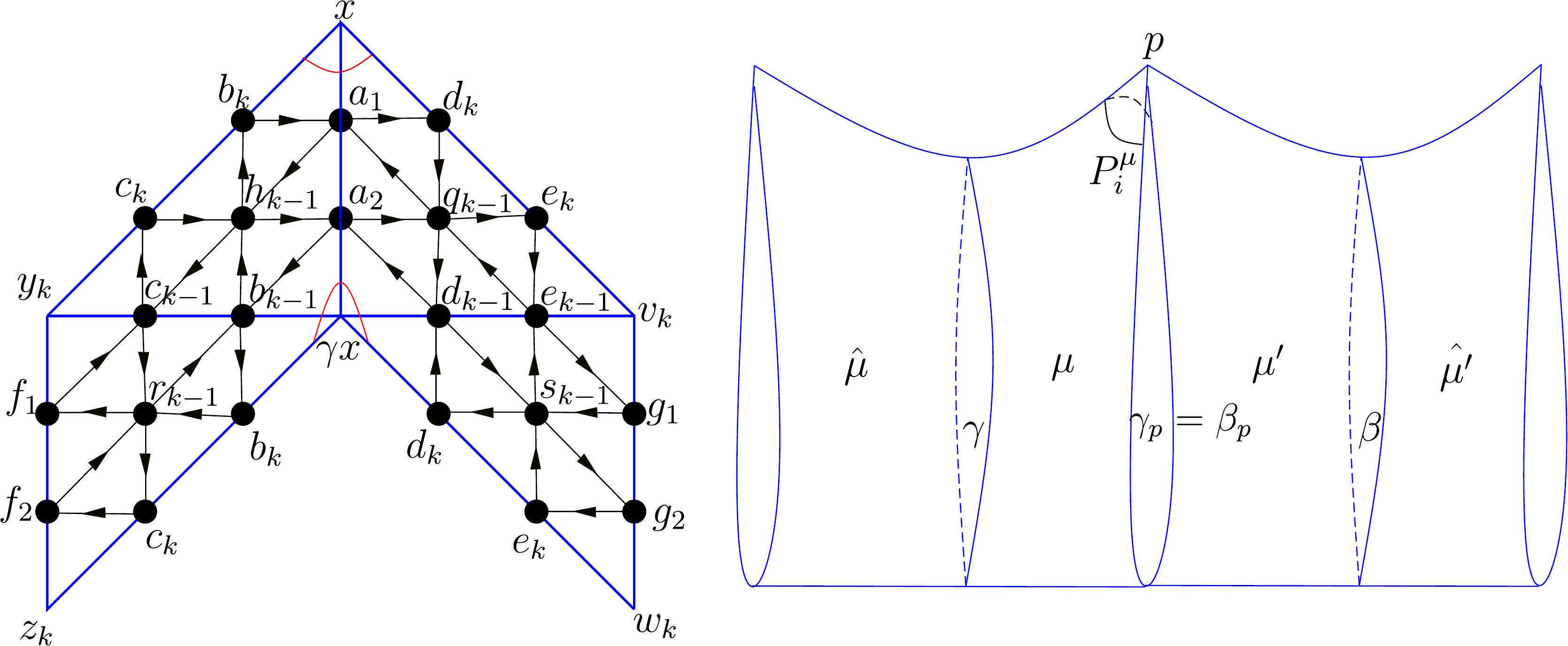}
\caption{The embedded boundary-parallel pair of half-pants $\mu$ is glued to $\hat{\mu}$ and $\mu'$ is glued to $\hat{\mu}'$.}
\label{Figure:ppkflp}
\end{figure}

Going forwards, we label $\mathcal{A}$-coordinates with respect to $\{\mathcal{T}^k\}_{k\in\mathbb{Z}}$ as per Figure~\ref{Figure:ppkflp}. We first express $P_1^p$ in terms of $\mathcal{A}$-coordinates:
\begin{align*}
P_1^p
&={\color{red}P_1(x;y_k,\gamma x)}
+{\color{blue}P_1(\gamma x;x,z_k)}
+{\color{red}P_1(x;\gamma x,v_k)}
+{\color{blue}P_1(\gamma x, w_k,x)}\\
&={\color{red}\frac{h_{k-1}}{a_1 b_k}}
+{\color{blue}\frac{h_k}{a_2b_k}}
+{\color{red}\frac{q_{k-1}}{a_1d_k}}+
{\color{blue}\frac{q_k}{a_2d_k}}
= \frac{b_k q_{k-1}+d_k h_{k-1}}{a_1 b_k d_k}
+ \frac{b_{k} q_{k}+d_{k} h_{k}}{a_2 b_{k} d_{k}}.
\end{align*}
Note that the {\color{red}red} terms above are obtained from $\mathcal{A}$-coordinates with respect to $\mathcal{T}^k$, whereas the {\color{blue}blue} terms are obtained with respect to $\mathcal{T}^{k+1}$. The gap term for $(\beta,\gamma)$ is the Goncharov--Shen potential measure of $I_1\cup I_2$ and is given by the following limit:
\begin{align*}
\begin{aligned}
\lim_{k\rightarrow +\infty}\left( \frac{h_{k-1}}{b_k a_1 P_1^p} +\frac{q_{k-1}}{d_k a_1 P_1^p}\right)
&=\lim_{k\rightarrow +\infty} \frac{\frac{b_k q_{k-1}+d_k h_{k-1}}{a_1 b_k d_k}}{\frac{b_k q_{k-1}+d_k h_{k-1}}{a_1 b_k d_k}+ \frac{b_{k} q_{k}+d_{k} h_{k}}{a_2 b_{k} d_{k}}}\\
&=\lim_{k\rightarrow +\infty} \frac{\frac{b_k q_{k-1}+d_k h_{k-1}}{a_1 b_k d_k}}{\frac{b_k q_{k-1}+d_k h_{k-1}}{a_1 b_k d_k}+ \frac{b_{k-1} q_{k-1}+d_{k-1} h_{k-1}}{a_2 b_{k-1} d_{k-1}}}\\
&=\lim_{k\rightarrow +\infty} \left(1+ \tfrac{a_1 d_k}{a_2 d_{k-1}} \cdot \tfrac{1+\frac{h_{k-1} d_{k-1}}{b_{k-1} q_{k-1}}}{1+\frac{d_k h_{k-1}}{b_k q_{k-1}}}\right)^{-1}.
\end{aligned}
\end{align*}
We invoke Lemma~\ref{proposition:lim} and Remark~\ref{remark:generallim} assert that the following limits exist, and may be (essentially definitionally) expressed as edge functions
\[
\lim_{k\rightarrow +\infty}\tfrac{h_{k-1} d_{k-1}}{b_{k-1} q_{k-1}}
= D_2(x,\gamma x, \beta^+,\gamma^+)^{-1},\quad 
\lim_{k\rightarrow +\infty} \tfrac{d_k h_{k-1}}{b_k q_{k-1}}
=D_1(x,\gamma x, \beta^+,\gamma^+), 
\]
thereby obtaining:
\begin{align}
&\left(
1+ \tfrac{a_1 \lambda_1(\rho(\beta))}{a_2} 
\cdot 
\tfrac{1+\frac{1}{D_2(x,\gamma x,\beta^+,\gamma^+)}}{1+D_1(x,\gamma x,\beta^+,\gamma^+)}
\right)^{-1}\notag\\
=&
\left(
1+\tfrac{a_1 \lambda_1(\rho(\beta))}{a_2} 
\cdot 
e^{-\frac{d_1(\beta,\gamma)}{2}-\frac{e_1(\beta,\gamma)}{2}} 
\cdot 
\tfrac{\cosh \frac{e_1(\beta,\gamma)}{2}}{\cosh \frac{d_1(\beta,\gamma)}{2}}
\right)^{-1}
\label{equation:sl3sgmbr}
\end{align}
Using Equation~\eqref{eq:aal} and 
\[
D_1(x,\gamma x, \beta^+,\gamma^+) 
\cdot D_2(x,\gamma x, \beta^+,\gamma^+) 
= \tfrac{\lambda_1(\rho(\beta))}{\lambda_1(\rho(\gamma))},
\] 
we see that \eqref{equation:sl3sgmbr} equals to 
\begin{align*}
\begin{aligned}
&\left(1+
\sqrt{\tfrac{a_1 \lambda_1(\rho(\beta))}{a_2}} 
\sqrt{\tfrac{a_1 \lambda_1(\rho(\gamma))}{a_2}}  
\cdot 
\tfrac{\cosh \frac{e_1(\beta,\gamma)}{2}}{\cosh \frac{d_1(\beta,\gamma)}{2}}
\right)^{-1}\\
&=\left(1+ 
\tfrac{\cosh \frac{e_1(\beta,\gamma)}{2}}{\cosh \frac{d_1(\beta,\gamma)}{2}}  
\cdot e^{\frac{1}{2}\left(\ell_1(\beta)+\tau(\beta,\beta_p)+\ell_1(\gamma)+\tau(\gamma,\gamma_p)\right)}
\right)^{-1}.
\end{aligned}
\end{align*}
\end{proof}

\begin{rmk}[Recovering the fuchsian identity]
In the $3$-Fuchsian locus, we can show that $\tau(\beta,\beta_p)=\tau(\gamma,\gamma_p)=0$ and $d_1(\beta,\gamma)=e_1(\beta,\gamma)$, and thus the above identity recovers the classical McShane identity for cusped hyperbolic surfaces \cite{mcshane_allcusps}.
\end{rmk}

\begin{rmk}[Recovering the $S_{1,1}$, $n=3$ identity]
\label{remark:sl3s11td}
When $(g,m)=(1,1)$, the homotopy classes $\beta$ and $\gamma$ are different representatives of the same $\pi_1(S_{1,1})$ conjugacy class, i.e.: $\beta= \delta^{-1} \gamma \delta$ for some $\delta\in\pi_1(S_{1,1})$. In this setting, the ideal vertices $(x,\beta x=\gamma x , y_k, z_k, v_k, w_k)$ may be chosen to be
\[
(x,\;
\beta x=\gamma x,\;
y_k=\gamma^{k}\delta x,\; 
z_k=\gamma^{k+1}\delta x,\; 
v_k=\beta^{k}\delta^{-1} x,\;
w_k=\beta^{k+1} \delta^{-1} x ).
\]
In any case, by Lemma~\ref{proposition:lim}, we have
\begin{equation}
\label{equation:tts11}
T(x,\beta x ,\beta^+)= \lim_{k\rightarrow +\infty} \frac{a_1 d_{k-1} e_k}{a_2 d_k e_{k-1}}= \frac{a_1 \lambda_2(\rho(\beta))}{a_2}= \frac{a_1 \lambda_2(\rho(\gamma))}{a_2}=T(x,\gamma x ,\gamma^+).
\end{equation}
Moreover, we see that
\begin{equation}
\label{equation:d121s11}
D_1(x,\gamma x, \beta^+,\gamma^+)\cdot D_2(x,\gamma x, \beta^+,\gamma^+)= \lim_{k\rightarrow +\infty} \frac{b_{k-1} d_k}{b_k d_{k-1}}=1.
\end{equation}
Therefore, $\tau(\gamma,\gamma_p)=\tau(\beta,\beta_p)$ and $\cosh \frac{e_1(\beta,\gamma)}{2}=\cosh \frac{d_1(\beta,\gamma)}{2}$, thereby recovering Theorem~\ref{theorem:inequsl3s11} from Theorem~\ref{thm:equsl3pp}.
\end{rmk}

\clearpage

\section{Simple geodesic sparsity for convex real projective surfaces}
\label{sec:sparsity}

The theory of convex real projective surfaces is a natural geometric avatar of positive $\operatorname{PGL}_3(\mathbb{R})$ representation theory: Goldman and Choi \cite{G1990convex,choi1993convex} established that for holonomy representations for closed convex real projective surfaces $S_g$ correspond to (conjugacy classes of) $n=3$ positive representations of $\pi_1(S_g)$; Marquis \cite{Mar10,Mar12} generalized this picture for surfaces $S_{g,m}$ with cusps (i.e.: $m> 0$), showing that holonomy representations of cusped convex real projective surfaces correspond to $n=3$ positive representations of $\pi_1(S_{g,m})$ with unipotent boundary monodromy. We make use of this dictionary to bring convex real projective geometric techniques to the study of $n=3$ higher Teichm\"uller theory, and vice versa.\medskip

We first give some background for convex real projective surfaces, before moving onto our main goal of this chapter: to generalize the Birman-Series geodesic sparsity theorem to the context of finite-area convex real projective surface context. Our proof is fundamentally geometric topological in nature, and we adjust our language accordingly. This complements the primarily algebraic treatment we give in the previous chapters.

\subsection{Convex real projective surfaces}

\begin{defn}[Convex sets]
A domain $\Omega\subset\mathbb{RP}^2$ contained in an affine patch is called \emph{convex} if the intersection of $\Omega$ with every line in $\mathbb{R}^2$ is connected. Furthermore, a convex domain $\Omega$ is called
\begin{itemize}
\item
\emph{properly convex}, if the closure $\overline{\Omega}$ is convex and contained within the complement $\mathbb{R}^2=\mathbb{RP}^2-\mathbb{RP}^1$ of some $\mathbb{RP}^1$ linearly embedded in $\mathbb{RP}^2$;
\item
\emph{strictly convex}, if the boundary $\partial \Omega$ of the properly convex domain $\Omega$ contains no line segments.
\end{itemize}
\end{defn}

\begin{defn}[Convex real projective surface]
A \emph{real projective surface} $\Sigma$ is a topological surface $S$ equipped with an atlas $\{(U,\varphi: U\rightarrow \mathbb{RP}^2)\}$, with
\begin{itemize}
\item
coordinate patches $U$ embedded as open sets in $\mathbb{RP}^2$ and 
\item
transition maps that are (restrictions of) projective linear transformations $\operatorname{PGL}_3(\mathbb{R})$ acting on $\mathbb{RP}^2$. 
\end{itemize}
Equivalently, \emph{convex real projective surface} $\Sigma=(S,\{(U,\varphi)\})$ is the quotient of a properly convex open domain $\Omega$ by a discrete subgroup of $\operatorname{PGL}_3(\mathbb{R})$ which is isomorphic to $\pi_1(S)$.
\end{defn}

Since convex domains are contractible, every convex real projective surface $\Sigma$ inherits a universal cover $\Omega\subset\mathbb{RP}^2$ from its developing map. Every such $\Omega$ lies within some copy of $\mathbb{R}^2$ linearly embedded in $\mathbb{RP}^2$.\medskip

The fact that $\Sigma$ is equal to the quotient of $\Omega$ by a discrete subgroup $\Gamma$ of $\operatorname{PGL}_3(\mathbb{R})$ means that there is a discrete faithful representation
\[
\rho:\pi_1(S)\rightarrow\operatorname{PGL}_3(\mathbb{R}).
\]
We refer to $\rho$ as the holonomy representation for $\Sigma$.

\begin{defn}[Projective equivalence]
We say that two convex real projective surfaces $\Sigma_1$ and $\Sigma_2$ are \emph{projectively equivalent} if, given their respective associated universal covers $\Omega_1,\Omega_2\subset\mathbb{RP}^2$, let $\Gamma_1,\Gamma_2$ be the images of the corresponding holonomy representations, there is a projective linear transformation $g\in\operatorname{PGL}_3(\mathbb{R})$ such that $(\Omega_2,\Gamma_2)=(g \Omega_1,g \Gamma_1 g^{-1})$. The map $g$ sending $\Sigma_1$ to $\Sigma_2$ is called a \emph{projective equivalence} between $\Sigma_1$ and $\Sigma_2$.
\end{defn}

Goldman \cite{G1990convex} studied the space of marked convex real projective structures on a smooth surface $S$
\[
\mathrm{Conv}(S):=
\left\{ (\Sigma,f) \mid  f:S\rightarrow \Sigma\text{ is a diffeomorphism}\right\}/\sim_\mathrm{conv},
\]
where $(\Sigma_1,f_1)\sim_{\mathrm{conv}}(\Sigma_2,f_2)$ if and only if $f_2\circ f_1^{-1}$ is homotopy equivalent to a projective equivalence between $\Sigma_1$ and $\Sigma_2$.\medskip

\begin{thm}
\label{theorem:cgm}
We have the following correspondences between spaces of marked convex real projective structures on smooth surfaces $S$ and positive/Hitchin representation varieties:
\begin{itemize}
\item
For closed surfaces $S=S_{g,0}$, Choi and Goldman \cite{choi1993convex,G1990convex} showed that the space $\mathrm{Conv}(S)$ of marked convex real projective structures is homeomorphic to the $\operatorname{PGL}_3(\mathbb{R})$-positive/Hitchin representation variety $\mathrm{Pos}_3(S)$.
\item
For the cusped surfaces $S=S_{g,m}$, Marquis \cite{Mar10} showed that the space $\mathrm{Conv}^u(S)$ of marked cusped convex real projective structures is homeomorphic to the unipotent bordered $\operatorname{PGL}_3(\mathbb{R})$-positive representation variety $\mathrm{Pos}^u_3(S)$.
\end{itemize}
\end{thm}

\subsection{The geometry of convex real projective surfaces}

\begin{defn}[Hilbert distance]
Given any two distinct points $x,y$ in a convex domain $\Omega\subset\mathbb{R}^2$, extend the straight line segment running between $x$ and $y$ to a segment running between boundary points $p_x,p_y\in\partial\Omega$, where $p_x$ is closer to $x$ and $p_y$ is closer to $y$. We define the \emph{Hilbert distance} to be
\begin{equation}
\label{eq:hilbertsum}
d(x,y):=\frac{1}{2}\log\frac{|x-p_y|\cdot|y-p_x|}{|y-p_y|\cdot|x-p_x|},
\end{equation}
where $|u-v|$ denotes the Euclidean length of the distance between $u,v\in\Omega\subset\mathbb{R}^2$. The Hilbert distance is invariant under projective linear transformations and hence descends to a distance metric on $\Sigma=\Omega/\Gamma$. We refer to both the metric $d$ on $\Omega$ and the metric $d_\Sigma$ on $\Sigma$ as the \emph{Hilbert metric}.
\end{defn}
In the special case when $\Sigma$ is a hyperbolic surface, its universal cover $\Omega$ is an ellipse, and the Hilbert metric on $\Omega$ is the usual hyperbolic metric on $\Omega$ with respect to the Klein model.

\begin{defn}[Area]
The \emph{area} (also known as the area for Busemann measure) on $(\Omega,d_\Omega)$ is defined as the the measure obtained by weighting the Lebesgue measure on $\Omega$ with density 
\begin{align*}
\frac{\mathrm{Leb}(B_1)}{\mathrm{Leb}(B_{d_\Omega}(x,1))}\text{, where:}
\end{align*}
\begin{itemize}
\item $B_1$ is the Euclidean unit ball in $\mathbb{R}^2$;
\item $B_{d_\Omega}(x,1)$ denotes the Hilbert distance unit ball in the tangent space $T_x \Omega\cong 
 \mathbb{R}^2$,
\item $\mathrm{Leb}(\cdot)$ denotes the canonical Lebesgue measure of $\mathbb{R}^2$ which equals to $1$ on the unit square.
\end{itemize}
The measure thus produced is invariant with respect to the action of the fundamental group because Hilbert distance $d_\Omega$ is invariant under projective transformations. It therefore descends to an area measure on the quotient surface $\Sigma=\Omega/\Gamma$. A \emph{finite area convex real surface} is a convex real projective surface with finite area. 
\end{defn}
By \cite{Mar12}, the two cases in Theorem~\ref{theorem:cgm} are finite area convex real projective surfaces.

\subsection{H\"older regularity and convexity of $\partial\Omega$}
We primarily deal with convex real projective surfaces $\Sigma$ with the two cases in Theorem \ref{theorem:cgm}, where the universal cover $\Omega$ for such a surface $\Sigma$ is necessarily strictly convex with $C^1$ boundary regularity. 

\begin{defn}\cite[Definitions~4.1 and 4.3]{benoist2001convexes}
Let $\Omega\subset\mathbb{R}^2$ be a convex open domain of $\mathbb{R}^2\subset\mathbb{RP}^2$ and fix an arbitrary Euclidean metric $d_E$ on $\mathbb{R}^2$. We say that $\partial\Omega$ is \emph{$\alpha$-H\"older}, for $\alpha\in(1,2]$, if for every compact subset $K\subset\partial\Omega$, there exists a constant $C_K>0$ such that, for all $p,q\in K$, we have:
\begin{align*}
d_E(q, T_p\partial\Omega)\leq C_K\cdot d_E(q,p)^\alpha;
\end{align*}
and we say that $\partial\Omega$ is \emph{$\beta$-convex}, for $\beta\in[2,\infty)$, if there exists a constant $C>0$ such that for all $p,q\in\partial\Omega$, we have:
\begin{align*}
d_E(q, T_p\partial\Omega)\geq C^{-1}\cdot d_E(q,p)^\beta.
\end{align*}
\end{defn}

When $\Omega$ covers a closed convex real projective surface $\Sigma$, the boundary regularity of $\partial\Omega$ may be extended to $\alpha_\Sigma$-H\"older, for some $\alpha_\Sigma\in(1,2]$ \cite[Proposition~4.6]{benoist2001convexes}. Using an argument taught to us by Benoist, we show that this is also true when $\Sigma$ is a finite area cusped convex real projective surface:

\begin{prop}[Benoist-Hulin]\label{thm:convprojregularity}
The boundary $\partial\Omega$ for $\Omega$ universally covering a finite area cusped convex real projective surface $\Sigma$ satisfies:
\begin{itemize}
\item
$\alpha_\Sigma$-H\"older for $\alpha_\Sigma\in(1,2]$,
\item
and $\beta_\Sigma$-convex for $\beta_\Sigma\in[2,\infty)$.
\end{itemize}
\end{prop}

\begin{proof}
The proof of this fact relies on another famous metric for convex real projective sets in $\mathbb{R}^2$: Yau-Cheng's \cite{cheng1977regularity} Blaschke metric (also known as the affine metric) for strictly convex domains. This is a negatively curved Riemannian metric on $\Omega$. Proposition~3.1 of \cite{benoist2013cubic} tells us that the curvature on $\Sigma$ approaches a negative constant as one heads deeper into a cusp, and hence is bounded away from $0$ on the entire surface. Combining this with \cite[Corollary~4.7]{benoist2014cubic} then shows that $\Sigma$ (and hence $\Omega$) is Gromov-hyperbolic with respect to the Hilbert metric. Hence, by \cite[ Corollary~1.5]{benoist2003convexes}, the ideal  boundary $\partial\Omega$ satisfies the desired $\alpha_\Sigma$-H\"older and $\beta_\Sigma$-convex.
\end{proof}

Benoist communicated to us the proof for Lemma~\ref{thm:expshrinkball} below, and it is a key estimate in our proof of the Birman-Series geodesic sparsity theorem for finite area convex real projective surfaces.

\begin{lem}[Exponentially shrinking balls, courtesy of Benoist]
\label{thm:expshrinkball}
Fix a point $O\in\Omega=\tilde{\Sigma}$ and a number $R\in\mathbb{R}_{>0}$. For any $u\in\Omega$, let $B(u,R)\subset\Omega$ denote the ball of (Hilbert) radius $R$ about $u$, and for any bounded set $U\subset\mathbb{R}^2$ let $\mathrm{diam}_E(U)$ denote the Euclidean diameter of $U$. Then there exists a positive constant $c=c_{\Omega,O,R}$ such that
\begin{align*}
\mathrm{diam}_E(B(u,R))<ce^{\frac{-d(u,O)}{c}}.
\end{align*}
\end{lem}

\begin{proof}
We show that for the geodesic ray $\{tO+(1-t)p\mid 0<t\leq1\}$ shooting out from $O$ to an arbitrary boundary point $p\in\partial\Omega$, there exists a constant $c(p)>0$ such that for any point $u$ along the ray,
\begin{align*}
\mathrm{diam}_E(B(u,R))<c(p)e^{-\frac{d(u,O)}{c(p)}}.
\end{align*}
In particular, we shall construct $c(p)$ in such a way that $c(\cdot)$ is a function that continuously varies with respect to $p\in\partial\Omega$. Then, we may use the compactness of $\partial\Omega$ to take
\begin{align*}
c_{\Omega,O,R}:=\max_{p\in\partial\Omega} c(p).
\end{align*}
Let us consider the radius $R$ ball $B(u,R)$ based at $u$, where $u$ is a point along the geodesic ray from $O$ to $p\in\partial\Omega$. By applying an affine (Euclidean) isometry on $\mathbb{R}^2$, we assume without loss of generality that $p$ is placed at the origin in $\mathbb{R}^2$ and that the tangent line $T_p\partial\Omega$ is the $x$-axis in $\mathbb{R}^2$. Let $u=(x_0,y_0)$ with respect to this parametrization, and let $p_1$ and $p_2$ respectively denote the left and right intersection points of the line $y=y_0$ with $\partial\Omega$. Further let $D$ denote the (closed) sector of $\Omega$ below $y=y_0$. (see Figure~\ref{fig:expshrink1}). For $u$ taken sufficiently close to $p$, the region $D$ is contained in the rectangle fenced by the horizontal lines $y=0$, $y=y_0$, and the two vertical lines passing through $p_1$ and $p_2$. Let $u(p)$ denote the $u$ closest to $O$ such that its induced $D$ satisfies the above rectangle-fencing property. Thanks to the $C^1$ smoothness of $\partial\Omega$, $u(p)$ varies continuously with respect to $p$. This partitions $\Omega$ into the union of a compact set 
\[
\Omega_1:=\left\{ u\in\Omega\;\mid\; u\text{ lies on the line segment between }u(p)\text{ and }O
\right\}\] and its open complement $\Omega_2:=\Omega-\Omega_1$.

\medskip

\begin{figure}[h]
\includegraphics[scale=1.25]{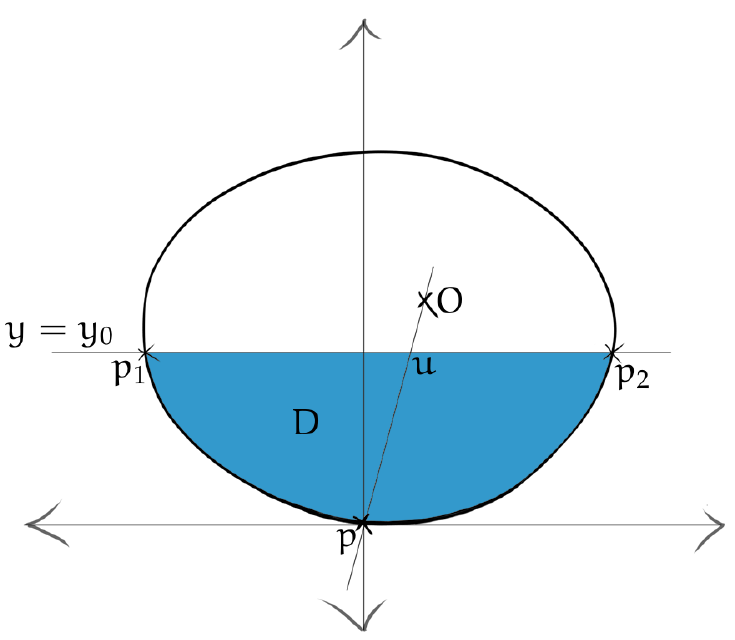}
\caption{$D$ is the shaded region below the $y=y_0$ horizontal line.}
\label{fig:expshrink1}
\end{figure}

For any $u\in\Omega_1$, the compactness of $\Omega_1$ ensures that there is a constant $C$ such that $\mathrm{diam}_E(B(u,R))<ce^{\frac{-d(u,O)}{c}}$. Let us consider the remaining case of $u\in\Omega_2$. Any complete geodesic going through $u$ consists of two geodesic rays, at least one of which lies in $D$. The Euclidean length of any such geodesic ray must then be less than $\mathrm{diam}_E(D)$, which is in turn less than:
\begin{align*}
\mathrm{diam}_E(D\cap\{x\leq 0\})+ \mathrm{diam}_E(D\cap\{x\geq 0\})=
d_E(p,p_1)+d_E(p,p_2). 
\end{align*}
Now invoking the $\beta$-convexity of $\partial\Omega$, we see that:
\begin{align*}
d_E(p,p_1)+d_E(p,p_2)\leq 2(Cy_0)^{\frac{1}{\beta}}\leq 2(C\cdot d_E(u,p))^{\frac{1}{\beta}}.
\end{align*}
We are now equipped to estimate the Euclidean diameter of $B(u,R)$. The triangle inequality tells us that $\mathrm{diam}_E(B(u,R))$ is at most $2$ times the Euclidean length $r$ of the longest geodesic segment $\sigma$ joining $u$ and the boundary of $B(u,R)$. Such a geodesic segment lies on the unique complete geodesic in $\Omega$ joining $u$ and some boundary point $q\in D\cap\partial\Omega$. If $\sigma$ lies on the geodesic ray $\overline{uq}$, then Equation~\eqref{eq:hilbertsum} tells us that
\begin{align*}
R>\frac{1}{2}\log\left(\frac{d_E(u,q)}{d_E(u,q)-r}\right)\text{, and hence }r<(1-e^{-2R})d_E(u,q).
\end{align*}
Similarly, if $\sigma$ lies on the geodesic ray complementary to $\overline{uq}$, then
\begin{align*}
R>\frac{1}{2}\log\left(\frac{d_E(u,q)+r}{d_E(u,q)}\right)\text{, and hence }r<(e^{2R}-1)d_E(u,q).
\end{align*}
Therefore, the diameter of $B(u,R)$ is bounded above by
\begin{align*}
2r<2(e^{2R}-1)d_E(u,q)<2e^{2R}(d_E(p,p_1)+d_E(p,p_2))\leq 4e^{2R}(C\cdot d_E(u,p))^{\frac{1}{\beta}}.
\end{align*}
We substitute in the Hilbert length
\begin{align*}
d(u,O)=\log\left(\frac{d_E(O,p)\cdot d_E(u,\hat{p})}{d_E(u,p)\cdot d_E(O,\hat{p}))}\right),
\end{align*}
where $\hat{p}$ is the ``antipodal" ideal point to $p$ on the opposite side of $O$ (i.e.: $p,\hat{p}$ and $O$ are collinear). This then gives us $\mathrm{diam}_E(B(u,R))<c(p)e^{-\frac{d(u,O)}{c(p)}}$, with
\begin{align*}
c(p):=\max\left\{\beta,4e^{2R}\left(\frac{C\cdot d_E(O,p)\cdot d_E(u,\hat{p})}{d_E(O,\hat{p}))}\right)^{\frac{1}{\beta}}\right\}.
\end{align*}
Since $\hat{p}$ varies continuously with respect to $p$, we conclude that $c(\cdot)$ is a continuous function, as required.
\end{proof}

\subsection{Geodesic Sparsity for finite-area convex real projective surfaces}\label{sec:birmanseries}

Let $\Sigma$ be a finite-area convex real projective surface, and let:
\begin{itemize}
\item
$I_k$ denote the collection of complete geodesics on $\Sigma$ with at most $k$ (geometric) self-intersections (counted with multiplicity);
\item
$|I_k|$ denote the subset of $\Sigma$ consisting of every single point which lies on (at least one) complete geodesic in the collection $I_k$ of geodesics with at most $k$ self-intersections. 
\end{itemize}

The goal of this subsection is to prove the following claim:

\begin{thm}[Geodesic sparsity]
\label{thm:birmanseries}
The area of $|I_k|$ is $0$ and the Hausdorff dimension of $|I_k|$ is $1$.
\end{thm}

When the surface $\Sigma$ is hyperbolic, the above result is referred to as the Birman--Series theorem \cite{birman_series}. They construct a descending filtration of subsets of $\Sigma$ such that:
\begin{itemize}
\item each subset covers $|I_k|$,
\item each subset is a union of finitely many convex geodesic quadrilaterals,
\item the number of convex quadrilaterals at the $k$-th level of the filtration asymptotically grows as a polynomial in $k$,
\item the Euclidean area of the quadrilaterals shrinks exponentially in $k$.
\end{itemize}
The polynomial growth in the number of quadrilaterals versus the exponential shrinkage their area gives us the requisite (Busemann) area $0$ conclusion. The fact that these quadrilaterals become exponentially thin then gives the desired Hausdorff dimension $1$ conclusion.\medskip

Much of the proof is topological, and we use Birman-Series' original arguments. However, we introduce the following tweaks:
\begin{itemize}
\item insteading of encoding geodesics as segments on a single geodesically bordered fundamental domain (such as a Ford domain), we use geodesic triangulations (Lemma~\ref{thm:triangles}). This is to avoid justifying why finitely sided geodesic fundamental domains exist, to highlight the flexibility of the Birman-Series construction and partially to use convexity to replace traditional hyperbolic geometric arguments (such as in Lemma~\ref{thm:linearinN}).
\item we require Lemma~\ref{thm:expshrinkball} to show that Hilbert radius $R$ balls shrink uniformly exponentially as one approaches the boundary.
\end{itemize}

\begin{lem}
\label{thm:triangles}
Any finite-area strictly convex real projective surface $\Sigma$ decomposes into a finite collection of (convex) geodesic triangles $\{\triangle_1,\ldots,\triangle_l\}$ glued along a finite collection of geodesic edges $\Gamma$.
\end{lem}

\begin{proof}
For cusped convex real projective surfaces, we may take an ideal triangulation. For compact $\Sigma$, \cite[Theorem~3.2]{G1990convex} tells us that every essential simple curve is uniquely realizable as a simple geodesic, we may therefore employ standard hyperbolic-surface-case arguments for finding a filling set of simple closed geodesics on $\Sigma$. Namely, if one of the complementary regions of a given collection of geodesics isn't contractible, then it must contain an essential simple closed curve, and one then adds this to the collection of geodesics. In any case, since simple geodesics lift to straight lines in the universal cover $\Omega$, the complementary regions of a filling collection of geodesics is made up of polytopes. These polytopes must be convex because the region is expressible as the intersection of convex regions in $\mathbb{R}^2$. Each polytope then cuts into finitely many triangles, as desired.
\end{proof}

For the remainder of this subsection, we fix one such collection $\{\triangle_1,\ldots,\triangle_l\}$ of geodesic triangles for $\Sigma$ glued along $\Gamma$ as described by Lemma~\ref{thm:triangles}. 

\subsection{Polynomial growth of the number of $k$-diagrams}

\begin{defn}[$k$-diagrams]
\label{defn:kdiag}
Let $J_k$ denote the set of geodesic arcs on $\Sigma$ which:
\begin{itemize}
\item
start and end on $\Gamma$ and/or cusps,
\item
have at most $k$ self-intersections.
\end{itemize}
Further let $J_k(N)$ denote the subset of geodesic arcs in $J_k$ that are cut up into $N$ geodesic segments by $\Gamma$. Also let $[J_k]$ denote the equivalence classes of geodesic arcs in $J_k$ with respect to isotopies of $\Sigma$ which preserve $\Gamma$ as a set. Similarly define $[J_k(N)]$. We refer to the elements of $[J_k]$ as \emph{$k$-diagrams} and the elements of $[J_0]$ as \emph{simple diagrams}.
\end{defn}

\begin{lem}
The cardinality of $[J_k(N)]$ is bounded above by a polynomial $P_k(N)$ in $N$.
\end{lem}

\begin{proof}
Every $k$-diagram $[\gamma]\in [J_k(N)]$ may be encoded as the ordered sequence $\sigma_1,\ldots,\sigma_N$ of elements of $[J_0(1)]$ obtained from cutting $[\gamma]$ along $\Gamma$. The key observation is that we do not need to retain the entire ordering of the sequence to recover a $k$-diagram: any simple diagram $[\gamma]\in [J_0(N)]$ may be completely recovered from the following data:
\begin{itemize}
\item 
the (unordered) multiset of $N$ segments in $[J_0(1)]$ constitute $[\gamma]$;
\item
the starting and ending segments for $[\gamma]$ (including the direction of the starting and ending segment).
\end{itemize}
This efficient encoding is used in the original proof of the Birman--Series theorem (\cite[Lemma~2.1]{birman_series}).\medskip

The consequence of this encoding is that
\begin{align*}
\mathrm{Card}[J_0(N)]\leq
N^2\cdot{{\mathrm{Card}[J_0(1)]+N-1} \choose {N-1}}=: P_0(N).
\end{align*}
For general $k$-diagrams $[\gamma]$, we need to introduce additional data to specify the intersection loci. Since two segments may intersect at most once, the degree of freedom introduced by this intersection data is bounded above by the number of ways of designating at most $k$ unordered pairs of segments to denote the intersections out of all possible unordered pairs of segments. Therefore:
\begin{align*}
\mathrm{Card}[J_k(N)]\leq 
P_0(N)\cdot\left[
{{N \choose 2} \choose {0}}+\ldots+{{N \choose 2} \choose {k}}
\right]=:P_k(N).
\end{align*}
\end{proof}

\subsection{Topological versus geometric length}

We have so far introduced $k$-diagrams, which afford us topological control over geodesics with $k$ self-intersections. We now show that the number of segments constituting a $k$-diagram is proportional to the Hilbert length of the segment it encodes. This promotes our topological control to geometric control.

\begin{lem}\label{thm:linearinN}
For any finite-area convex real projective surface $\Sigma$, there exists a positive constant $\alpha_{\Sigma,\Gamma}>0$ so that for any complete geodesic $\hat{\gamma}$ with at most $k$ self-intersections, the length of any geodesic subarc $\gamma\subset\hat{\gamma}$, such that $\gamma$ is an element of $J_k(N)$, grows at least linearly in $N$ for $N$ large enough. That is: there exists an integer $N_{\Sigma,\Gamma}\geq 0$ such that the Hilbert length
\begin{align*}
\ell_\gamma\geq\alpha_{\Sigma,\Gamma}\cdot N\text{ for all }\gamma\in J_k(N),\quad\text{where } N\geq N_{\Sigma,\Gamma}.
\end{align*}
\end{lem}

\begin{proof}[Proof of Lemma~\ref{thm:linearinN} for closed $\Sigma$]
We first prove this for closed $\Sigma$. Fix a disjoint collection of embedded open balls $B_{r_i}(x_i)$ around every vertex $x_i$ of $\Gamma$. Let $N_{\Sigma,\Gamma}$ be $3l+1$ (recall here that $l$ is the number of geodesic triangles constituting $\Sigma$) and let $\alpha_{\Sigma,\Gamma}>0$ be $\frac{\ell_\mathrm{min}}{2N_{\Sigma,\Gamma}}$, where $\ell_\mathrm{min}$ is the length of the shortest geodesic arc in $J_0(1)$ with end points on $\Gamma\backslash \cup B_{r_i}(x_i)$. The fact that $\ell_\mathrm{min}$ is well-defined is because the subset of $J_0(1)$ with end points on $\Gamma\backslash \cup B_{r_i}(x_i)$ is a compact set. To be precise: this subset of $J_0(1)$ is the disjoint union of $3l$ closed (solid) rectangles formed by taking products of distinct pairs of segments in $\Gamma\backslash \cup B_{r_i}(x_i)$ which lie on the same triangle. Moreover, we know that $\ell_\mathrm{min}>0$ because segments have starting and ending points on distinct edges and hence cannot be of length $0$.\medskip

Next observe that $\Gamma$ cuts each $B_{r_i}(x_i)$ into at most $3l$ convex sectors. Since the intersection of convex sets is convex and hence contractible, the intersection of any contiguous subarc of $\gamma$ with $B_{r_i}(x_i)$ may meet each sector at most once. This means that we may have at most $3l$ consecutive segments of $\gamma$ lying within $B_{r_i}(x_i)$ and hence any $3l+1$ consecutive segments on $\gamma$ must have length strictly greater than $\ell_\mathrm{min}$. This in turn gives us our choice of $N_{\Sigma,\Gamma}$ and $\alpha_{\Sigma,\Gamma}$ when $\Sigma$ is compact.
\end{proof}

We now look to the situation when $\Sigma$ is a (finite-area) cusped strictly convex real projective surface. We show that geodesics with $k$ self-intersections cannot penetrate arbitrarily far into a cusp (unless it goes straight into the cusp), thus effectively reducing the analysis to being on a compact subset of the surface:

\begin{prop}[Cuspidal collar neighborhood]
\label{prop:compact}
Fix a finite-area (cusped) convex real projective surface $\Sigma$ and some integer $k\geq0$. There is a compact subset $K\subset \Sigma$ which contains all (complete) compactly-supported geodesics on $\Sigma$ which self-intersect at most $k$ times when counted with multiplicity.
\end{prop}

\begin{rmk}
The complement of this compact subset $K$ in $\Sigma$ consists of annular neighborhoods around cusps and we refer to them as \emph{cuspidal collar neighborhoods} --- our nomenclature alludes to collar neighborhoods.
\end{rmk}

\begin{proof}
Consider a length $R$ embedded horocycle $\eta_R$ bounding an annular neighborhood $C_R$ of a given cusp. Now choose an even shorter horocycle $\eta_r$ bounding a smaller cuspidal annular neighborhood $C_r\subset C_R$, so that the minimal distance between $\eta_r$ and $\eta_R$ is at least $\frac{R(k+1)}{2}$ (this is always possible since $C_R$ is infinitely long). We claim that no geodesic arc $\gamma\in J_k$ enters and then exits $C_r$, that is: $C_r$ is a cuspidal collar neighborhood.\medskip

Assume otherwise that $\gamma$ enters and exits $C_r$. The complete geodesic extension $\hat{\gamma}$ is the union of two overlapping geodesic rays $\hat{\gamma}^+$ and $\hat{\gamma}^-$ with overlap given by a subarc of $\gamma$ lying within $C_r$ and with end points on $\eta_r$. In order for the ray $\hat{\gamma}^\pm$ to lie completely within $C_R$, the ideal end point of any lift of $\hat{\gamma}^\pm$ would need to be the unique ideal boundary point of the horodisk in the universal cover of $\Sigma$ covering $C_r$. This in turn characterizes $\hat{\gamma}^\pm$ as a geodesic going straight up the cusp, and therefore hitting every horocycle at most once. This is a contradiction as $\hat{\gamma}^\pm$ meets $C_r$ in two places. Therefore, both $\hat{\gamma}^+$ and $\hat{\gamma}^-$ leave $C_r$ at some point and hence there is a geodesic subarc $\bar{\gamma}$ of $\hat{\gamma}$ which:
\begin{itemize}
\item
lies completely within $C_R$;
\item
has both its endpoints on $\eta_R$;
\item
enters and exits $C_r$.
\end{itemize}
Since $\bar{\gamma}$ joins $\eta_r$ and $\eta_R$ along two subarcs, it has length at least $R(k+1)$. On the other hand, the geodesic arc $\bar{\gamma}$ is (endpoint-fixing) homotopy equivalent to a horocyclic path along $\eta_R$ which wraps around $\eta_R$ at most $k$ times (this can be shown by unwrapping $C_R$ to a $k$-fold cover of $C_R$ that undoes the self-intersections of $\bar{\gamma}$). This in turn means that the length of $\bar{\gamma}$ must be strictly less than $R(k+1)$, leading to a contradiction. Therefore, no geodesic arc $\gamma\in J_k$ which extends to a complete geodesic $\hat{\gamma}$ with at most $k$ self-intersections may enter $C_r$.
\end{proof}

We now return to the proof of Lemma~\ref{thm:linearinN}, but addressing the cusped case.

\begin{proof}[Proof of Lemma~\ref{thm:linearinN} for cusped $\Sigma$]
Finally, we complete our proof for the cusped case as follows: fix a horocyclic neighborhood $C_r$ for each cusp on $\Sigma$ and take $N_{\Sigma,\Gamma}=1$ and $\alpha_{\Sigma,\Gamma}>0$ to be the length (in the closed interval $[0,\infty]$ rather than $\mathbb{R}_{>0}$) of the shortest geodesic arc in $J_0(1)$ with endpoints outside of the horocyclic regions. Again, such a length exists due to compactness and is finite. These choices for constants clearly work because every segment on $\gamma$ lies outside of $C_r$ and hence must be at least of length $\alpha_{\Sigma,\Gamma}$.
\end{proof}

\subsection{Geodesic sparsity: area $0$}

We are now prepared to prove the geodesic sparsity theorem for finite-area convex real projective surfaces. Fix a fundamental domain $F\subset \tilde{\Sigma}=:\Omega$ made up of lifts of the triangles $\triangle_1,\ldots,\triangle_{l}$ decomposing $\Sigma$. Represent $\Omega$ as a subset of $\mathbb{R}^2\subset\mathbb{RP}^2$, and let $\overline{F}$ denote the closure of $F$ in $\mathbb{R}^2$. Define the following collection of geodesic arcs
\[
\hat{I}_k:=\left\{\sigma=\tilde{\gamma}\cap\triangle_i \mid \text{for some }i=1,\ldots, l\text{ and where }\tilde{\gamma}\text{ is a lift of }\gamma\in J_k\right\},
\]
and further define $|\hat{I}_k|\subset\overline{F}$ to be the collection of points lying on geodesic arcs $\sigma$ in $\hat{I}_k$. Our goal is to show that $|\hat{I}_k|\cap F$ has zero area. However, since the area on $\Omega$ is definitionally in the same measure class as the Lebesgue measure, we see that we just need to show that $|\hat{I}_k|$ occupies zero Euclidean area.

\begin{proof}[Proof of Theorem~\ref{thm:birmanseries} for compact $\Sigma$ -- area $0$.]
We first consider the case when $\Sigma$ is compact. For each $N\geq N_{\Sigma,\Gamma}$, we partition $\hat{I}_k$ using the fact that any geodesic arc\footnote{We may ignore the case when $\sigma$ is a vertex of $F$ as it does not affect the measure or the Hausdorff dimension of $|\hat{I}_k|$.} $\sigma$ is uniquely expressible as the middle (i.e.: $(N+1)^{\text{st}}$) segment of a lift of some representative of $[\gamma]\in J_k(2N+1)$. This gives us a partition of $\hat{I}_k$ into at most $P_k(2N+1)$ sets.\medskip

We next show that the Euclidean area occupied by all the lifts of representatives of $[\gamma]\in[J_k(2N+1)]$ with the middle segment in $\overline{F}$ is exponentially decreasing in $N$. Consider an arbitrary lift $\gamma$ of a representative of $[\gamma]$ positioned so that its middle segment is in $\overline{F}$ and let the endspoints of $\gamma$ lie on $\overline{F'}$ and $\overline{F''}$ --- two deck transformation translates of the fundamental domain $\overline{F}$. By the unique path lifting property, every such $\gamma$ necessarily ends on the same $\overline{F'}$ and $\overline{F''}$ pair. In particular, this means that the union of every such representative of $[\gamma]$ is contained within the convex hull of $F'\cup F''$. We know from Lemma~\ref{thm:linearinN} that both $\overline{F'}$ and $\overline{F''}$ are at least distance $\alpha_{\Sigma,\Gamma}N$ away from $F$. We now use this fact to control the Euclidean area for the convex hull of $F'\cup F''$.\medskip

Let $O$ be an arbitrary point on the interior of $F$. Since $\bar{F}$ is compact, for some $R>0$ the domain $\overline{F}\subset B(0,R)$. The domains $\overline{F'}$ and $\overline{F''}$ are deck transform translates of $\overline{F}$ and the corresponding translated points $x'\in F'$ and $x''\in F''$ of $O\in F$ satisfy that $d(x',O),d(x'',O)>\alpha_{\Sigma,\Gamma} N$. Therefore, the Euclidean diameters of $F'$ and $F''$ must both be less than $ce^{\frac{-\alpha_{\Sigma,\Gamma} N}{c}}$. This in turn means that the convex hull of $F'\cup F''$ may be covered by an Euclidean rectangle of width $ce^{\frac{-\alpha_{\Sigma,\Gamma} N}{c}}$ and length $\mathrm{diam}_E(\Omega)$. We absorb $\mathrm{diam}_E(\Omega)$ into $c$ and ignore it henceforth.\medskip

We next note that the convex hull of $F'\cup F''$ necessarily covers every representative geodesic segment in $[\gamma]\in[J_k(2N+1)]$. Since there are fewer than $P_k(2N+1)$ homotopy classes $[\gamma]$ constituting $[J_k(2N+1)]$ and each class is covered by a rectangle of area $ce^{\frac{-\alpha_{\Sigma,\Gamma} N}{c}}$, this means that the set $|\hat{I}_k|$ has Euclidean area less than $P_k(2N+1)\cdot ce^{\frac{-\alpha_{\Sigma,\Gamma} N}{c}}$. Since $N$ may be set to be arbitrarily large, this means that $|\hat{I}_k|$ has zero Euclidean area and hence zero area. Finally observe that $|\hat{I}_k|$ is the lift of $|I_k|$ to $\overline{F}$ (except for perhaps finitely many closed geodesics lying completely on $\Gamma$) and hence $|I_k|\cap F$ has zero area.
\end{proof}

\begin{proof}[Proof of Theorem~\ref{thm:birmanseries} for cusped $\Sigma$ -- area $0$.]
We now turn to the case when $\Sigma$ is non-compact, that is: we are dealing with a cusped convex real projective surface. Given a geodesic segment $\sigma\in\hat{I}_k$, when we try to geodesically extend $\sigma$ using deck transform translates of segments in $\hat{I}_k$, one of the following three things occurs:
\begin{enumerate}
\item 
$\sigma$ can be extended by $N$ segments in both directions, this produces a geodesic arc in $J_k(2N+1)$;
\item
$\sigma$ can be extended by $N$ segments in one direction and hits a cusp in the other direction, this produces an arc in $J_k(M)$, for $M\leq 2N$;
\item
$\sigma$ cannot be extended by $N$ segments in either direction and hits a cusp in the both directions, this produces an arc in $J_k(M)$, for $M\leq 2N-1$;
\end{enumerate}
This behavioral classification allows us to partition $\hat{I}_k$ into the following three classes of objects:
\begin{enumerate}
\item
$\sigma$ is the middle (i.e.: $(N+1)^{\mathrm{st}}$) segment of a lift of some representative of $[\gamma]\in J_k(2N+1)$;
\item
$\sigma$ is a segment of a lift of some representative $\gamma$ of $[\gamma]\in J_k(M)$, for $M\leq 2N$, where $\gamma$ is a geodesic ray (i.e.: one of the ends of $\gamma$ is a cuspidal ideal point) and $\sigma$ is the $i^\text{th}$ segment, for $1\leq i\leq M-(N+1)$, indexed from the cuspidal end;
\item
$\sigma$ is a segment of a lift of the unique representative $\gamma$ of $[\gamma]\in J_k(M)$, for $M\leq 2N-1$, where $\gamma$ is a bi-infinite geodesic (i.e.: both end points of $\gamma$ are cuspidal ideal points) and $\sigma$ has index (strictly) less than $N+1$ indexed from both ends of $\gamma$.
\end{enumerate}
Case~1 is identical to the previous compact closed $\Sigma$ analysis, and each homotopy class $[\gamma]$ may be covered by a Euclidean rectangle of Euclidean area $ce^{\frac{-\alpha_{\Sigma,\Gamma} N}{c}}$. For Case~2, note that one end of $\gamma$ is a single cuspidal ideal point on $\partial\Omega$, and therefore $[\gamma]$ may be covered by a Euclidean trapezium with Euclidean area less than $ce^{\frac{-\alpha_{\Sigma,\Gamma} N}{c}}$. Case~3 concerns bi-infinite geodesics joining two cuspidal ideal points and may be covered by a single line. This means that $|\hat{I}_k|$ may be covered by a finite collection of quadrilaterals (and lines) of total Euclidean area less than
\begin{align*}
ce^{\frac{-\alpha_{\Sigma,\Gamma} N}{c}}(P_k(2N+1)+P_k(2N)+\ldots+P_k(1)) < ce^{\frac{-\alpha_{\Sigma,\Gamma} N}{c}}\cdot(2N+1)\cdot P_k(2N+1).
\end{align*}
Once again, by taking $N$ to be arbitrarily large, we see that the Euclidean area of $|\hat{I}_k|$ is zero and hence area of $|I_k|$ is zero.
\end{proof}

\subsection{Geodesic sparsity: Hausdorff dimension $1$}

Finally, we show that $|I_k|$, or equivalently $\hat{I}_k$, has Hausdorff dimension $1$. 

\begin{proof}[Proof of Theorem~\ref{thm:birmanseries} -- Hausdorff dimension $1$.]
Consider $\Omega$ equipped with the Hilbert metric $d$ (which is Finsler) in comparison with $\Omega$ endowed with the Euclidean metric $d_E$ (but regarded as a Finsler manifold). The Finsler metric for $(\Omega,d)$ is a $C^1$ rescaling of $(\Omega,d_E)$ due to the dependence on the boundary smoothness (which is at least $C^1$). This means that, for any (possibly non-compact) subset $K$ of a compact subset of $\Omega$, the identity map between $(\Omega,d)$ and $(\Omega,d_E)$ restricts to a bi-Lipschitz map between $(K,d)$ and $(K,d_E)$. Combined with the fact that Hausdorff dimension is preserved under bi-Lipschitz maps, this means that when $\Sigma$ (and hence $\overline{F}$) is compact the Hausdorff dimension of $(|I_k|,d)$ and $(|\hat{I}_k|,d_E)$ are the same. Combined with the further fact that the Hausdorff dimension is preserved with respect to taking countable unions of sets with the same Hausdorff dimension, the equivalence in Hausdorff dimension between $(|I_k|,d)$ and $(|\hat{I}_k|,d_E)$ is true when $\Sigma$ is cusped.\medskip

We have reduced our Hausdorff dimension derivation problem to that of $(|I_k|,d_E)$. We first show that the $(1+\epsilon)$-dimensional Hausdorff content of $(|I_k|,d_E)$ is $0$ for every $\epsilon>0$. Recall from earlier in this proof that for every $N\geq N_{\Sigma,\Gamma}$, there we may cover $|I_k|$ with fewer than $N\cdot P_k(2N+1)$ Euclidean rectangles of length $\mathrm{diam}_E(\Omega)$ and width $ce^{\frac{-\alpha_{\Sigma,\Gamma} N}{c}}$. Each such rectangle may be covered by $\left\lceil \frac{\mathrm{diam}_E(\Omega)}{ce^{\frac{-\alpha_{\Sigma,\Gamma} N}{c}}}\right\rceil$ Euclidean balls of radius $\frac{3}{2}ce^{\frac{-\alpha_{\Sigma,\Gamma} N}{c}}$. The $(1+\epsilon)$-dimensional Hausdorff content of $(|\hat{I}_k|,d_E)$ is $0$, because:
\begin{align*}
\lim_{N\to\infty}N\cdot P_k(2N+1)
\cdot
\left\lceil \frac{\mathrm{diam}_E(\Omega)}{ce^{\frac{-\alpha_{\Sigma,\Gamma} N}{c}}}\right\rceil
\cdot
\left(
 \frac{3}{2}ce^{\frac{-\alpha_{\Sigma,\Gamma} N}{c}}
\right)^{1+\epsilon}=0.
\end{align*}
This means that the Hausdorff dimension of $(|I_k|,d)$ is at most $1$. On the other hand, since $|I_k|$ contains geodesic arcs, it necessarily has Hausdorff dimension at least $1$.
\end{proof}

\clearpage

\section{McShane identities for higher Teichm\"uller space}
\label{sec:highermcshane}
In this section, we derive McShane identities (Theorem \ref{theorem:boundaryith}) for $\operatorname{PGL}_n(\mathbb{R})$-positive representations with loxodromic boundary monodromy from $i$-th (Goncharov--Shen) potential ratio (Definition \ref{defn:ratio}). Such ratio is a symmetry breaking analogue of the general cross ratio \cite{Lab07}, which is called the ratio (Definition \ref{definition:4ratio}).

\subsection{Labourie--McShane identities}
Let us firstly recall the general cross ratio introduced by Labourie.
\begin{defn}[\cite{Lab07} Cross ratio]
\label{definition:crossratio}
Let 
\[\partial_{\infty}\pi_1(S_{g,m})^{4*}=\{(x,y,z,t)\in \partial_{\infty}\pi_1(S_{g,m})^4 \;|\; x \neq t, y \neq z\}.\]
A {\em cross ratio} on $S_{g,m}$ is a $\pi_1(S_{g,m})$-invariant H\"older function $\mathbb{B}$ from $\partial_{\infty}\pi_1(S_{g,m})^{4*}$ to $\mathbb{R}$ which satisfies the following rules:
\begin{enumerate}
\item (Normalization): $\mathbb{B}(x,y,z,t)=0$ if and only if $x=z$ or $y=t$,
\item (Normalization): $\mathbb{B}(x,y,z,t)=1$ if and only if $x=y$ or $z=t$, 
\item (cocycle): $\mathbb{B}(x,y,z,t)=\mathbb{B}(x,y,z,w) \cdot \mathbb{B}(x,y,w,t)$,
\item (cocycle): $\mathbb{B}(x,y,z,t)=\mathbb{B}(x,w,z,t) \cdot \mathbb{B}(w,y,z,t)$. 
\end{enumerate}
An {\em ordered cross ratio} is a cross ratio $\mathbb{B}$ on $S_{g,m}$ which satisfies, for four different points $x,y,z,t \in \partial_{\infty}\pi_1(S_{g,m})$:
\begin{enumerate}
\item $\mathbb{B}(x,y,z,t)>0$ if $z,t$ are on the same side of $(x,y)$,
\item $\mathbb{B}(x,y,z,t)>1$ if $x,y,z,t$ are cyclically ordered.
\end{enumerate}
\end{defn}

\begin{defn}
For any non-trivial $\alpha\in \pi_1(S_{g,m})$, let $y \neq \alpha^+, \alpha^-$, the {\em period} of $\alpha$ with respect to the ordered cross ratio $\mathbb{B}$, which does not depend on $y$, is
\[\ell^{\mathbb{B}}(\alpha):=
\log \left(\mathbb{B}(\alpha^+,\alpha^-,y,\alpha(y))\right).
\]
\end{defn}
By Definition \ref{definition:crossratio}, we have $\ell^{\mathbb{B}}(\alpha)=\ell^{\mathbb{B}}(\alpha^{-1})$.

Given $\rho \in \mathrm{Pos}_n(S_{g,m})$ (Definition \ref{definition:posrep}) with loxodromic boundary mondromy which is also a Hitchin representation \cite[Section 9]{LM09} by Remark \ref{remark:hitpos}, let $\xi_\rho:\partial_\infty \pi_1(S_{g,m})\rightarrow \mathcal{B}$ be the $\rho$-equivariant positive map with respect to the canonical lift (Definition \ref{definition:ratioperiod}). Then we define the associated ordered cross ratio $\mathbb{B}^{\rho}$ using $\xi_\rho$. For example, we define $\mathbb{B}^{\rho}_i$ for $i=1,\cdots, n-1$ using $\xi_\rho$ as follows. Let $\tilde{\xi}_\rho^{i}$ be the lift of $\xi_\rho^{i}$ with values in $\wedge^i(\mathbb{R}^n)$ for $i=1,\cdots, n-1$. Recall $\Delta$ is the volume form of $\mathbb{R}^n$. For any four different points $x,y,z,t\in\partial_{\infty}(\pi_1(S_{g,m}))$:
\begin{equation}
\label{equation:labcr}
\mathbb{B}^{\rho}_i(x,y,z,t) = \frac{\Delta\left(\tilde{\xi}_\rho^{n-i}(x) \wedge \tilde{\xi}_\rho^{i}(z)\right) }{\Delta\left(\tilde{\xi}_\rho^{n-i}(x) \wedge \tilde{\xi}_\rho^{i}(t)\right)} \cdot \frac{\Delta\left(\tilde{\xi}_\rho^{n-i}(y) \wedge \tilde{\xi}_\rho^{i}(t)\right)}{\Delta\left(\tilde{\xi}_\rho^{n-i}(y) \wedge \tilde{\xi}_\rho^{i}(z)\right)}
\end{equation}
which does not depend the lift of $\xi_\rho$. The cross ratio $\mathbb{B}^{\rho}_i$ is indeed the ordered cross ratio by positivity \cite{FG06}. For $i=1$ or $n-1$, it is called \emph{rank $n$ weak cross ratio}. The {\em period} of $\alpha$ with respect to $\mathbb{B}^{\rho}_i$ is
$\log \frac{\lambda_1\left(\rho(\alpha)\right)\cdots \lambda_i\left(\rho(\alpha)\right)}{\lambda_{n-i+1}\left(\rho(\alpha)\right)\cdots \lambda_n\left(\rho(\alpha)\right)}.$

By splitting the ordered cross ratio the same way as the classical McShane identity for the horocycle, Labourie and McShane \cite{LM09} obtained the following identity. Recall the loxodromic-bordered $\operatorname{PGL}_n(\mathbb{R})$-representation variety $\mathrm{Pos}_n^h(S_{g,m})$ in Definition \ref{definition:posrep}.
\begin{thm}\cite[Theorem 4.1.2.1]{LM09}
\label{theorem:labmc}
Given $\rho \in \mathrm{Pos}_n^h(S_{g,m})$ with loxodromic boundary monodromy, let $\alpha$ be an oriented boundary component of $S_{g,m}$ such that $S_{g,m}$ is on the left side of $\alpha$ and $\mathbb{B}^\rho$ is an ordered cross ratio for $\rho$. Then
\begin{equation*}
\begin{aligned}
\sum_{(\beta,\gamma)\in \subvec{\mathcal{P}}_\alpha}
 \log \mathbb{B}^{\rho}(\alpha^-,\alpha^+,\gamma^+,\beta^+)  + \sum_{(\beta,\gamma)\in \subvec{\mathcal{P}}^{\partial}_\alpha}
\log \mathbb{B}^{\rho}(\alpha^-,\alpha^+,\gamma^-,\gamma^+)=\ell^{\mathbb{B}^\rho}(\alpha),
\end{aligned}
\end{equation*}
where $\xvec{\mathcal{P}}_\alpha$ is the set of the homotopy classes of boundary-parallel pairs of pants in Definition \ref{definition:Pp}, and $\xvec{\mathcal{P}}^{\partial}_\alpha$ is a subset of $\vec{\mathcal{P}}_\alpha$ with $\alpha$ and another boundary component of the pair of pants as the boundary components of $S_{g,m}$. For each boundary-parallel pair of pants, we fix a marking on the boundary components $\alpha,\beta,\gamma$ such that $\alpha \beta^{-1}\gamma =1$ as in Figure \ref{Figure:pti}.
\end{thm}

\begin{rmk}
\label{remark:holderanosov}
The proof of \cite[Theorem 4.1.2.1]{LM09} is different from the step 2 in the proof of Theorem \ref{theorem:equsl3} where we compare the horocyclic Hilbert length measure and the Goncharov--Shen potential measure through $\partial_\infty \pi_1(S_{g,m})\backslash \tilde{p}$. Labourie and McShane compare directly the horocyclic measure with respect to the auxiliary hyperbolic metric for $\partial_\infty \pi_1(S_{g,m})$ and the ordered cross ratio measure induced from the $\rho$-equivariant positive H\"older map $\xi_\rho:\partial_\infty \pi_1(S_{g,m})\rightarrow \mathcal{B}$. The Anosov property of the representation $\rho$ induces the H\"older property of $\xi_\rho$ which ensures that these two measures are comparable. But the Anosov property fails for the positive representation with unipotent boundary monodromy (cusps).
\end{rmk}
\subsection{Ordered ratios and identities}

We introduce a generalization of ordered cross ratios, which called \emph{ordered ratios}. 
\begin{defn}[Ratio]
\label{definition:4ratio}
Consider the following collection of $4$-tuples
\[
\partial_{\infty}\pi_1(S_{g,m})^{4**}
=\left\{(x,y,z,t)\in \partial_{\infty}\pi_1(S_{g,m})^4 \mid x \neq y, x \neq z, x \neq t, y \neq z\right\}.
\]
A \emph{ratio} $B:\partial_{\infty}\pi_1(S_{g,m})^{4**}\to\mathbb{R}$ is a $\pi_1(S_{g,m})$-invariant continuous function which satisfies the following three \emph{ratio conditions}:
\begin{enumerate}
\item (normalization): $B(x,y,z,t)=0$ if and only if $y=t$,
\item (normalization): $B(x,y,z,t)=1$ if and only if $z=t$, 
\item (cocycle): $B(x,y,z,t)=B(x,y,z,w) \cdot B(x,y,w,t)$,
\end{enumerate}
An \emph{ordered ratio} is a ratio $B$ which satisfies two \emph{order conditions}: for four different points $x,y,z,t \in \partial_{\infty}\pi_1(S_{g,m})$:
\begin{enumerate}
\item $B(x,y,z,t)>0$ if $z,t$ are on the same side of $(x,y)$,
\item $B(x,y,z,t)>1$ if $x,y,z,t$ are cyclically ordered.
\end{enumerate}
\end{defn}

\begin{rmk}
Let us start with a few examples for examining the conditions to be a ratio.
\begin{enumerate}
\item For $(\bar{\rho},\bar{\xi}) \in \mathcal{A}_{\operatorname{SL}_n,S_{g,m}}$, we define 
\[B(x,y,z,t):=\frac{P_i(x;y,t)}{P_i(x;y,z)}.\] 
It is a ratio but not an ordered ratio in general. It is an ordered ratio when $(\bar{\rho},\bar{\xi})\in \mathcal{A}_n(S_{g,m})$. This is the main ordered ratio that we study in this section.
\item For $\rho\in \mathcal{X}_n(S_{g,m})$, we define the ordered ratio $B(x,y,z,t)$ to be the edge function $-D_i(x,y,z,t)$ in Definition \ref{definition:edgefunction}.
\end{enumerate}
\end{rmk}

\begin{defn}[Periods for ratios]
\label{defn:period}
For non-trivial $\alpha \in \pi_1(S_{g,m})$ and $y \neq\alpha^-, \alpha^+$, the \emph{period} of $\alpha$ for the ordered ratio $B$ is 
\[
\ell^B(\alpha):=\log B\left(\alpha^-,\alpha^+,\alpha(y),y\right),
\]
\end{defn}
As with periods for cross ratios, periods for ratios are also independent of the choice of $y$. For any $z \in \partial_\infty \pi_1(S_{g,m}) \backslash \{\alpha^-, \alpha^+\}$, by
\begin{itemize}
\item
 $\pi_1(S_{g,m})$-invariance:
  $B\left(\alpha^-,\alpha^+,\alpha(y),\alpha(z)\right)= B\left(\alpha^-,\alpha^+,y,z\right)$,
  \item
  and the cocycle identity for the ordered ratios,
  \end{itemize}
  we obtain that:
\begin{align*}
&B\left(\alpha^-,\alpha^+,\alpha(y),y\right)
\\=&B\left(\alpha^-,\alpha^+,\alpha(y),\alpha(z)\right)\cdot B\left(\alpha^-,\alpha^+,\alpha(z),z\right)\cdot B\left(\alpha^-,\alpha^+,z,y\right)
\\=&B\left(\alpha^-,\alpha^+,y,z\right)\cdot B\left(\alpha^-,\alpha^+,\alpha(z),z\right)\cdot B\left(\alpha^-,\alpha^+,z,y\right)
\\=&B\left(\alpha^-,\alpha^+,\alpha(z),z\right).
\end{align*}

Ordered ratios satisfy one fewer cocycle axiom than ordered cross ratios. As a consequence periods $\ell^B$ of an ordered ratio $B$ do not necessarily satisfy $\ell^B(\alpha)= \ell^B(\alpha^{-1})$. One immediate advantage of ordered ratios is that $i$-th lengths (Definition \ref{definition:ilength}) can now be periods.

Before we state the generalized McShane identity for ordered ratios, we define the boundary-parallel pairs of half-pants when an oriented boundary component $\alpha$ of $S_{g,m}$ is playing the role of the cusp $p$ in Definition \ref{defn:hp}. 
\begin{defn}
\label{definition:hparital}
Let $\alpha$ be an oriented boundary component of $S_{g,m}$ such that $S_{g,m}$ is on the left side of $\alpha$ as in Figure~\ref{fig:hpspiral}. An (embedded) {\em boundary-parallel pairs of half-pants} on $S_{g,m}$ is one of the two pieces obtained by cutting along the unique simple bi-infinite geodesic of an embedded pair of pants with both ends emanating from $\alpha^-$. Usually, denoted by $(\gamma,\gamma_{\alpha^-})$ or $\mu$.

Let $\xvec{\mathcal{H}}_\alpha$ denotes the collection of (embedded) boundary-parallel pairs of half-pants on $S_{g,m}$ with both ends of its seam emanating from $\alpha^-$. Moreover, let $\xvec{\mathcal{H}}^\partial_\alpha\subset\xvec{\mathcal{H}}_\alpha$ denote the subset of half-pants with a peripheral cuff.
\end{defn}

\begin{figure}[h!]
\includegraphics[scale=1.5]{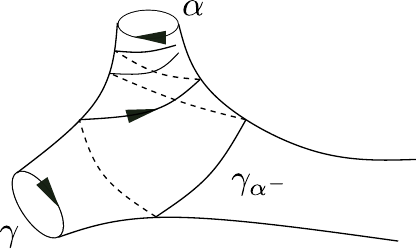}
\caption{An example of a boundary-parallel pair of pants $(\gamma,\gamma_{\alpha^-})$ with both ends of the seam $\bar{\gamma}_{\alpha^-}$ emanating from ${\alpha^-}$.}
\label{fig:hpspiral}
\end{figure}

\begin{rmk}
Given a boundary-parallel pair of half-pants $(\gamma,\gamma_{\alpha_-})\in\xvec{\mathcal{H}}^{\partial}_\alpha$, there is a unique boundary-parallel pair of pants in $\xvec{\mathcal{P}}_\alpha$ that contains $(\gamma,\gamma_{\alpha_-})$ and agrees with its boundary orientations. We thereby identify $\xvec{\mathcal{H}}^{\partial}_\alpha$ with the subset $\xvec{\mathcal{P}}^{\partial}_\alpha\subset\xvec{\mathcal{P}}_\alpha$ of boundary-parallel pairs of pants with a peripheral cuff.
\end{rmk}

\begin{thm}[McShane identity for loxodromic-bordered positive representations]
\label{thm:boundary}
For a $\operatorname{PGL}_n(\mathbb{R})$-positive representation $\rho \in \mathrm{Pos}_n^h(S_{g,m})$ with loxodromic boundary monodromy, let $\alpha$ be a distinguished oriented boundary component $S_{g,m}$ such that $S_{g,m}$ is on the left side of $\alpha$ and $B^\rho$ is an ordered ratio for $\rho$, we have the equality:
\begin{align*}
\begin{aligned}
&\ell^{B^\rho}(\alpha)
=\sum_{(\beta,\gamma)\in \subvec{\mathcal{P}}_\alpha}
 \log B^\rho(\alpha^-,\alpha^+,\gamma^+,\beta^+)  + \sum_{(\beta,\gamma)\in \subvec{\mathcal{P}}^{\partial}_\alpha}
\log B^\rho(\alpha^-,\alpha^+,\gamma^-,\gamma^+)
\\&=\sum_{(\delta,\delta_{\alpha^-})\in \subvec{\mathcal{H}}_\alpha}
\left| \log B^\rho\left(\alpha^-;\alpha^+,\delta(\alpha^-),\delta^+\right) \right| 
+ \sum_{(\gamma,\gamma_{\alpha^-})\in \subvec{\mathcal{H}}^\partial_\alpha}
\log B^\rho\left(\alpha^-;\alpha^+,\gamma^-,\gamma^+\right).
\end{aligned}
\end{align*}
where $\xvec{\mathcal{P}}_\alpha$ is the set of the homotopy classes of boundary-parallel pairs of pants, and $\xvec{\mathcal{P}}^{\partial}_\alpha$ is a subset of $\vec{\mathcal{P}}_\alpha$ containing another boundary component of $S_{g,m}$ as in Definition \ref{definition:Pp}. And $\xvec{\mathcal{H}}_\alpha$ and $\xvec{\mathcal{H}}^{\partial}_\alpha$ are the sets of boundary-parallel pairs of half-pants in Definition \ref{definition:hparital}. For each boundary-parallel pair of pants or pair of half-pants, we fix a marking on the boundary components $\alpha,\beta,\gamma$ such that $\alpha \beta^{-1}\gamma =1$ as in Figure \ref{Figure:pti}.
\end{thm}

\begin{figure}[h!]
\includegraphics[scale=0.28]{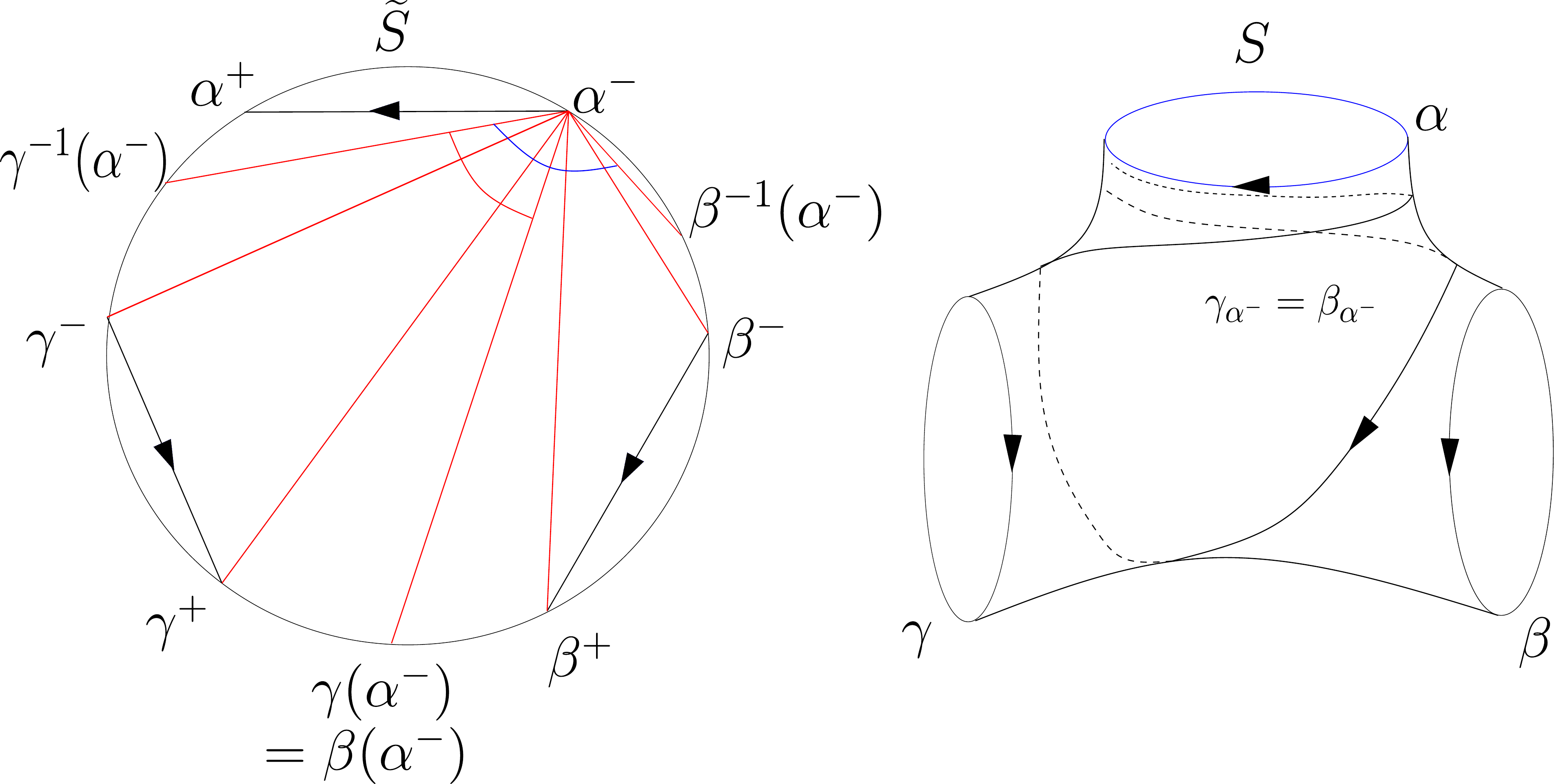}
\caption{The boundary-parallel pair of pants $(\beta,\gamma)$ has the boundary components $\alpha$, $\beta$, $\gamma$ with $\alpha \beta^{-1} \gamma= 1$ and $(\beta,\gamma)$ is cut into $(\beta,\beta_{\alpha^-}),(\gamma,\gamma_{\alpha^-})$ along the simple bi-infinite geodesic $\gamma_{\alpha^-}=\beta_{\alpha^-}$. }
\label{Figure:pti}
\end{figure}

\begin{proof}
Notice that only one cocycle property is used in the proof of \cite[Theorem 4.1.2.1]{LM09}.
Using the cocycle property, positivity and H\"older property of $B^\rho$ (due to H\"older property of $\xi_\rho$), we follow the proof presented in \cite[Theorem 4.1.2.1]{LM09} almost line by line and replace the ordered cross ratio $\mathbb{B}^\rho$ by the ordered ratio $B^\rho$. We obtain
\begin{align*}
\ell_i(\alpha)=\sum_{(\beta,\gamma)\in \subvec{\mathcal{P}}_\alpha}
 \log B^\rho\left(\alpha^-;\alpha^+,\gamma^+,\beta^+\right)
+ \sum_{(\beta,\gamma)\in \subvec{\mathcal{P}}^{\partial}_\alpha}
\log B^\rho\left(\alpha^-;\alpha^+,\gamma^-,\gamma^+\right).
\end{align*}

Recall the refinement in the proof of Theorem \ref{theorem:equsl3} step (4). The right hand side of the above equation equals
\begin{align*}
\begin{aligned}
&=\sum_{(\beta,\gamma)\in \subvec{\mathcal{P}}_\alpha}
\left( \log B^\rho\left(\alpha^-;\alpha^+,\gamma^+,\gamma(\alpha^-)\right)+  \log B^\rho\left(\alpha^-;\alpha^+, \beta(\alpha^-),\beta^+\right) \right) 
\\&+ \sum_{(\beta,\gamma)\in \subvec{\mathcal{P}}^{\partial}_\alpha}
\log B^\rho\left(\alpha^-;\alpha^+,\gamma^-,\gamma^+\right)
\\&=\sum_{(\delta,\delta_{\alpha^-})\in \subvec{\mathcal{H}}_\alpha}
\left| \log B^\rho\left(\alpha^-;\alpha^+,\delta(\alpha^-),\delta^+\right) \right| + \sum_{(\gamma,\gamma_{\alpha^-})\in \subvec{\mathcal{H}}^\partial_\alpha}
\log B^\rho\left(\alpha^-;\alpha^+,\gamma^-,\gamma^+\right).
\end{aligned}
\end{align*}
\end{proof}

As is, the identity is not expressed in terms of explicit geometric/projective invariants attached to the representation $\rho$. We do this crucial step later in this section.

\subsection{$i$-th potential ratio}
\label{section:potentialratio}

The $i$-th character $P_i(f;g,h)$ (Definition \ref{definition:ichar}) depends on the choice of a lift of the flag $\xi_\rho(f)$ in $\mathcal{A}$. For elements of $\mathcal{A}_n(S_{g,m})$, this is canonically assigned, but not so for $\mathcal{X}_n(S_{g,m})$. To resolve this issue, we consider taking ratios of two $i$-th characters, thereby providing a well-defined regular function on $\mathcal{X}_n(S_{g,m})$. 

\begin{defn}[$i$-th potential ratio]
\label{defn:ratio}
For a $\operatorname{PGL}_n(\mathbb{R})$-positive representation $\rho \in \mathrm{Pos}_n^h(S_{g,m})$ with loxodromic boundary monodromy, let $(\rho,\xi) \in \mathcal{X}_n(S_{g,m})$ be the canonical lift (Definition \ref{definition:ratioperiod}) of $\rho$ (or $\rho \in \mathrm{Pos}_n^u(S_{g,m})$ with unipotent boundary monodromy with a unique lift $(\rho,\xi) \in \mathcal{X}_n(S_{g,m})$). Recall the unique $\rho$-equivariant map $\xi_\rho:\partial_\infty \pi_1(S)\rightarrow \mathcal{B}$ for $(\rho,\xi)$ in Definition~\ref{definition:rhoequivariant}. For any $(x,y,z,t)\in \partial_{\infty}\pi_1(S_{g,m})^{4**}$ (Definition \ref{definition:4ratio}), choose a lift $X$ of $\xi_\rho(x)$ in $\mathcal{A}$. We define the \emph{$i$-th potential ratio} for $\rho$ as:
\begin{align*}
B_i(x;y,z,t)
:=\frac{P_i(x;y,t)}{P_i(x;y,z)}
:=\frac{P_i(X;\xi_\rho(y),\xi_\rho(t))}{P_i(X;\xi_\rho(y),\xi_\rho(z))}.
\end{align*}
\end{defn}
\begin{rmk}
The $i$-th potential ratio can be defined directly for $(\rho,\xi) \in \mathcal{X}_n(S_{g,m})$ using $\xi_\rho$. Thus a different choice of the lift for any $\rho \in \mathrm{Pos}_n^h(S_{g,m})$ induces a different gap function expression similar as in \cite[Appendix A]{LM09}.
\end{rmk}

\begin{figure}
\includegraphics[scale=0.7]{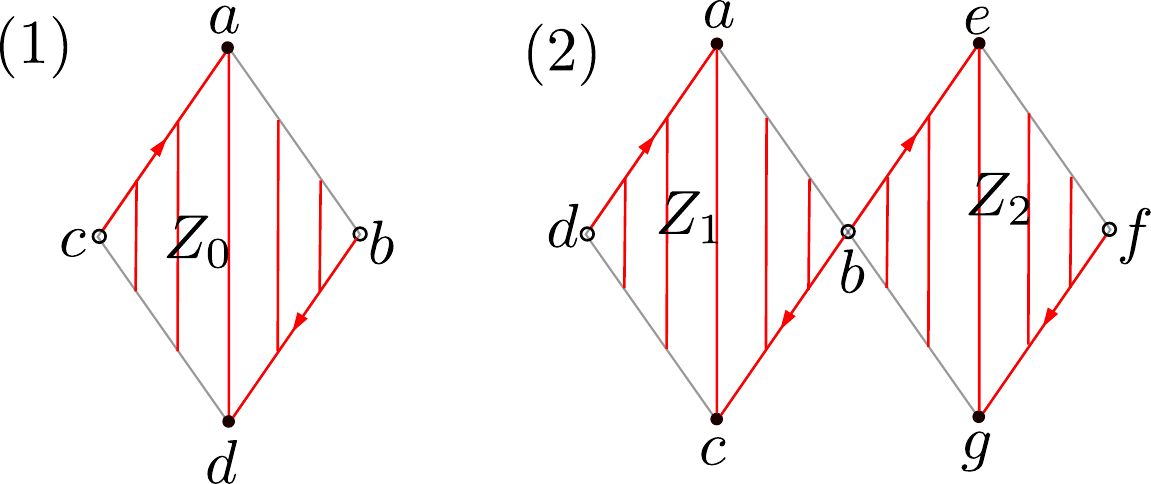}
\caption{The $\mathcal{A}$ coordinates labelled on the vertices correspondingly, $Z_0=\frac{ad}{bc}$ is some $\alpha$,\; $Z_1/Z_2=\frac{acf}{dge}$ is the triple ratio.}
\label{Figure:lozengeratio}
\end{figure}
We ``visualize'' our computations in the following way. Given a $n$-triangulation (Definition \ref{definition:ntri}), the {\em oriented edge ratio} (\cite[Definition 5.9]{Sun20a}) for an oriented edge of the $n$-triangulation is the $\mathcal{A}$ coordinate (Definition \ref{defn:FGAcoordinate}) at the head divided by the $\mathcal{A}$ coordinate at the tail of the oriented edge. By Lemma \ref{lemma:lozenge}, as shown in Figure \ref{Figure:lozengeratio} (1), each $\alpha$, represented by a lozenge, is the ratio of two oriented edge ratio. By Lemma \ref{lem:triplea}, as shown in Figure \ref{Figure:lozengeratio} (2), each triple ratio is the ratio of two lozenges.
The $i$-th level of $x$ (\cite[Definition 5.11]{Sun20a}) is union of the edges of the $n$-triangulation where for any vertex of any edge, the least number of edges towards $x$ is $i$.
Equation \eqref{equ:potential} shows us that $P_i(x;y,z)$ is the summation of lozenges crossed by the $i$-th level of $x$ in the triangle $(x,y,z)$ (Figure \ref{Figure:edgeratio}).

\begin{figure}
\includegraphics[scale=0.7]{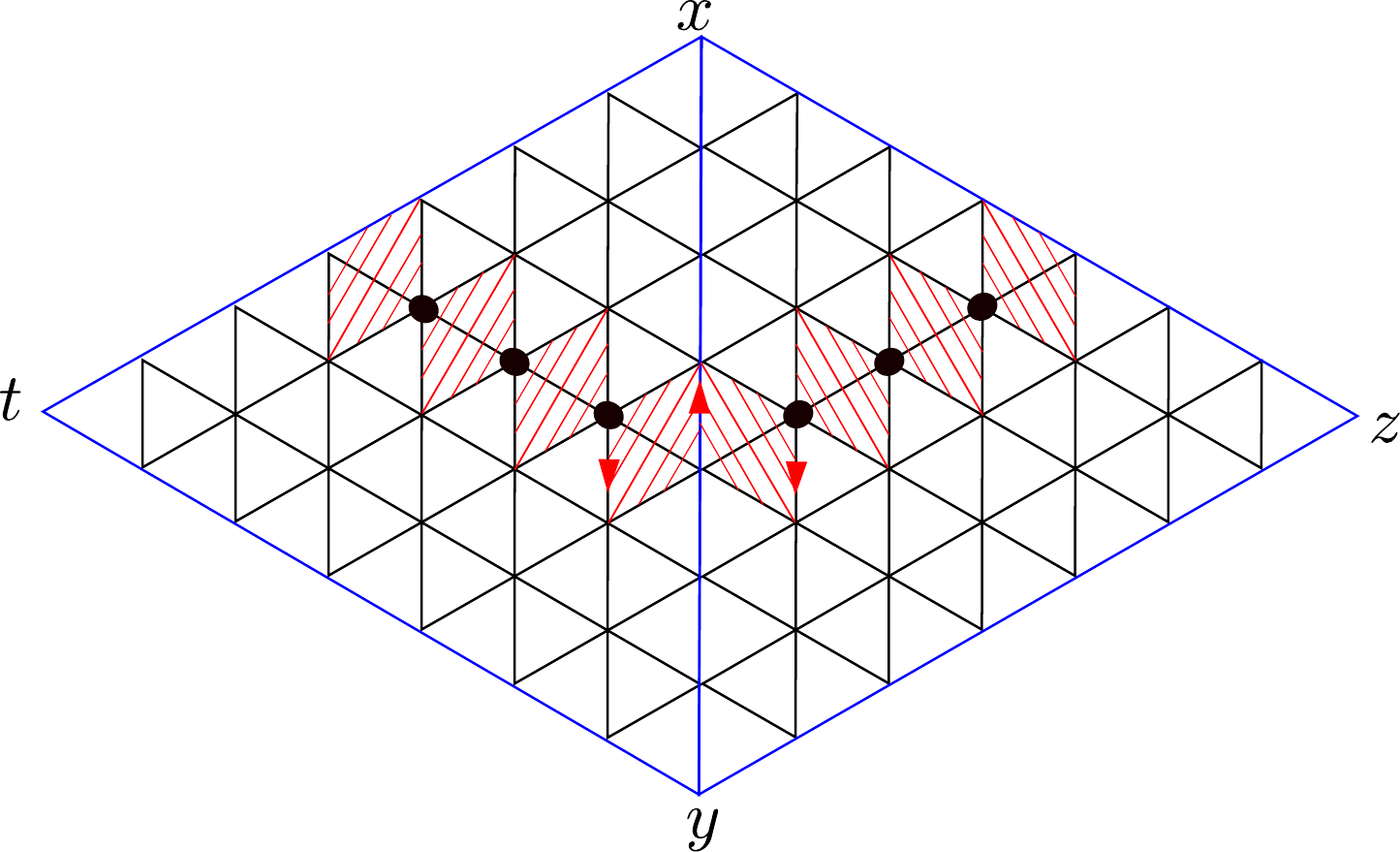}
\caption{Figure for $B_i(x;y,z,t)$.}
\label{Figure:edgeratio}
\end{figure}

\begin{prop}
\label{prop:PT}
For any positive triple $(F,G,H)\in \mathcal{A}^3$, we have
\begin{align*}
P_i(F;G,H) = \alpha^{F;G,H}_{n-i,i,0} \left(1+\sum_{c=1}^{i-1} \prod_{j=1}^c \frac{1}{ T_{n-i,i-j,j}(F,G,H)}\right).
\end{align*}
\end{prop}
\begin{proof}
Iteratively applying Lemma~\ref{lem:triplea} $c$ times, we obtain
\begin{align*}
\alpha^{F;G,H}_{n-i,i-c,c} = \alpha^{F;G,H}_{n-i,i,0} \cdot \prod_{j=1}^c \frac{1}{T_{n-i,i-j,j}(F,G,H)}.
 \end{align*}

Re-expressing Equation~\eqref{equ:potential}, we get:
\begin{align*}
P_i(F;G,H) =\sum_{c = 0}^{i-1} \alpha^{F;G,H}_{n-i,i-c,c}= \alpha^{F;G,H}_{n-i,i,0} \left(1+\sum_{c=1}^{i-1} \prod_{j=1}^c \frac{1}{ T_{n-i,i-j,j}(F,G,H)}\right).
\end{align*}
\end{proof}

\begin{prop}
\label{prop:BC}
The $i$-th potential ratio $B_i(x;y,z,t)$ is expressed as follows:
\begin{align*}
B_i\left(x;y,z,t\right)= \frac{1+\sum_{c=1}^{i-1} \prod_{j=1}^c \frac{1}{ T_{n-i,i-j,j}(x,y,t)}}{1+\sum_{c=1}^{i-1} \prod_{j=1}^c \frac{1}{ T_{n-i,i-j,j}(x,y,z)}} \cdot (-D_i(x,y,z,t)).
\end{align*}
Thus the $i$-th potential ratio does not depend on the lift $X$ of $\xi_\rho(x)$ and is an ordered ratio. 
\end{prop}
\begin{proof}
By Proposition~\ref{prop:PT}, we have
\begin{align*}
\begin{aligned}
&P_i(x;y,t)
=  \alpha^{x;y,t}_{n-i,i,0} \left(1+\sum_{c=1}^{i-1} \prod_{j=1}^c \frac{1}{ T_{n-i,i-j,j}(x,y,t)}\right),
\end{aligned}
\end{align*}
and
\begin{align*}
\begin{aligned}
&P_i(x;y,z)
=  \alpha^{x;y,z}_{n-i,i,0} \left(1+\sum_{c=1}^{i-1} \prod_{j=1}^c \frac{1}{ T_{n-i,i-j,j}(x,y,z)}\right).
\end{aligned}
\end{align*}
Moreover, by Lemma~\ref{lemma:lozenge} (or by observing the red arrows in Figure \ref{Figure:edgeratio}), we get
\begin{align*}
\begin{aligned}
&\frac{\alpha^{x;y,t}_{n-i,i,0}}{\alpha^{x;y,z}_{n-i,i,0}}
= \frac{\frac{\Delta\left(x^{n-i-1}\wedge t^1  \wedge y^{i}\right)\cdot \Delta\left(x^{n-i+1} \wedge y^{i-1}\right)}{\Delta\left(x^{n-i} \wedge  y^{i} \right) \cdot \Delta\left(x^{n-i}  \wedge t^{1} \wedge y^{i-1}\right)}}{\frac{\Delta\left(x^{n-i-1}\wedge z^1  \wedge y^{i}\right)\cdot \Delta\left(x^{n-i+1} \wedge y^{i-1}\right)}{\Delta\left(x^{n-i} \wedge  y^{i} \right) \cdot \Delta\left(x^{n-i}  \wedge z^{1} \wedge y^{i-1}\right)}}
\\&=\frac{\Delta\left(x^{n-i-1}\wedge t^1  \wedge y^{i}\right)}{\Delta\left(x^{n-i}  \wedge t^{1} \wedge y^{i-1}\right)}\cdot \frac{\Delta\left(x^{n-i}  \wedge z^{1} \wedge y^{i-1}\right)}{ \Delta\left(x^{n-i-1}\wedge z^1  \wedge y^{i}\right)}
\\&=-D_i(x,y,z,t).
\end{aligned}
\end{align*}

Thus we obtain
\begin{align*}
\begin{aligned}
B_i\left(x;y,z,t\right)=\frac{P_i(x;y,t)}{P_i(x;y,z)}
&=\frac{1+\sum_{c=1}^{i-1} \prod_{j=1}^c \frac{1}{ T_{n-i,i-j,j}(x,y,t)}}{1+\sum_{c=1}^{i-1} \prod_{j=1}^c \frac{1}{ T_{n-i,i-j,j}(x,y,z)}} \cdot \frac{\alpha^{x;y,t}_{n-i,i,0}}{\alpha^{x;y,z}_{n-i,i,0}}
\\
& = \frac{1+\sum_{c=1}^{i-1} \prod_{j=1}^c \frac{1}{ T_{n-i,i-j,j}(x,y,t)}}{1+\sum_{c=1}^{i-1} \prod_{j=1}^c \frac{1}{ T_{n-i,i-j,j}(x,y,z)}} \cdot (-D_i(x,y,z,t)).
\end{aligned}
\end{align*}
Thus $i$-th potential ratio does not depend on the lift $X$ of $\xi_\rho(x)$. And it is an ordered ratio.
\end{proof}
Thus we have the well-definedness of the $i$-th potential ratio.

Recall Equation \eqref{equation:labcr}, we have 
\begin{equation*}
\begin{aligned}
&\mathbb{B}_1^\rho(x,y,z,t)=\prod_{i=1}^{n-1} \frac{\Delta\left(x^{n-i-1}\wedge t^1  \wedge y^{i}\right)}{\Delta\left(x^{n-i}  \wedge t^{1} \wedge y^{i-1}\right)}\cdot \frac{\Delta\left(x^{n-i}  \wedge z^{1} \wedge y^{i-1}\right)}{ \Delta\left(x^{n-i-1}\wedge z^1  \wedge y^{i}\right)}
\\&=\prod_{i=1}^{n-1} (-D_i(x,y,z,t))
\\&= \prod_{i=1}^{n-1} \left(B_i(x,y,z,t)\cdot \frac{1+\sum_{c=1}^{i-1} \prod_{j=1}^c \frac{1}{ T_{n-i,i-j,j}(x,y,z)}}{1+\sum_{c=1}^{i-1} \prod_{j=1}^c \frac{1}{ T_{n-i,i-j,j}(x,y,t)}} \right).
\end{aligned}
\end{equation*}

We therefore obtain:
\begin{cor}
\label{corollary:LM-HS}
The rank $n$ weak cross ratio and the $i$-th potential ratio is related by:
\begin{equation*}
\mathbb{B}_1^\rho(x,y,z,t)= \prod_{i=1}^{n-1} B_i(x,y,z,t)\cdot \prod_{i=1}^{n-1} \left( \frac{1+\sum_{c=1}^{i-1} \prod_{j=1}^c \frac{1}{ T_{n-i,i-j,j}(x,y,z)}}{1+\sum_{c=1}^{i-1} \prod_{j=1}^c \frac{1}{ T_{n-i,i-j,j}(x,y,t)}} \right),
\end{equation*}
which relate Labourie--McShane's rank $n$ weak cross ratio used in their explicit identities \cite[Section 10]{LM09} to the $i$-th potential ratio.
\end{cor}

\begin{prop}
\label{prop:potentialspectral}
Given $\rho \in \mathrm{Pos}_n^h(S_{g,m})$ and its canonical lift $(\rho,\xi) \in \mathcal{X}_n(S_{g,m})$, we define $P_i$ as in Definition \ref{defn:ratio}. Let $\alpha$ be an oriented boundary component of $S_{g,m}$. For any distinct $x,y,z \in \partial_\infty \pi_1(S_{g,m})$ where $x=\alpha^-$, for any integers $a,b,c$ such that $a,b\geq1$ and $c=n-a-b\geq0$, we have
\begin{align*}
\frac{\alpha^{x;y,z}_{a,b,c}}{\alpha^{x;\alpha y,\alpha z}_{a,b,c}}= e^{\ell_{n-a}(\alpha)}.
\end{align*}
Thus, for $i=1,\cdots,n-1$, the following $i$-th potential ratio
\begin{align*}
\frac{P_{i}(\alpha^-;y,z)}{P_{i}(\alpha^-;\alpha y,\alpha z)}
=e^{\ell_i(\alpha)}.
\end{align*}

\end{prop}
\begin{proof}
Recall the notation \ref{notation:flag}. By Definition \ref{definition:ratioperiod}, $\rho(\alpha)$ has a lift into $\operatorname{SL}_n(\mathbb{R})$ where its eigenvalues are $\lambda_1>\cdots>\lambda_n>0$. For any non-negative integer $u, v$ with $a+u+v=n$, we obtain
\begin{align*}
\Delta(x^a \wedge z^u \wedge y^v)
&= \Delta(\rho(\alpha)x^a \wedge \rho(\alpha) z^u \wedge \rho(\alpha) y^v)\\
&=\frac{1}{\lambda_n \cdots \lambda_{n-a+1}} \cdot \Delta(x^a \wedge (\alpha z)^u \wedge (\alpha y)^v).
\end{align*}
Thus
\begin{align*}
\alpha^{x;y,z}_{a,b,c}
&=\frac
{\Delta\left(x^{a-1} \wedge z^{c+1} \wedge y^{b}\right)
\cdot \Delta\left(x^{a+1} \wedge z^{c} \wedge y^{b-1}\right)}
{\Delta\left(x^{a} \wedge z^{c}\wedge y^{b} \right) 
\cdot \Delta\left(x^{a}  \wedge z^{c+1} \wedge y^{b-1}\right)}\\
&=\frac{\lambda_{n-a} \cdot \Delta\left(x^{a-1} \wedge (\alpha z)^{c+1} \wedge (\alpha y)^{b}\right) \cdot \Delta\left(x^{a+1} \wedge (\alpha z)^{c} \wedge (\alpha y)^{b-1} \right)}{\lambda_{n-a+1} \cdot \Delta\left(x^{a} \wedge (\alpha z)^{c} \wedge (\alpha y)^{b}\right) \cdot \Delta\left(x^{a} \wedge (\alpha z)^{c+1}\wedge (\alpha y)^{b-1} \right)}
\\&=\frac{\lambda_{n-a}}{\lambda_{n-a+1}} \cdot \alpha^{x;\alpha y,\alpha z}_{a,b,c}=e^{\ell_{n-a}(\alpha)} \cdot \alpha^{x;\alpha y,\alpha z}_{a,b,c}.
\end{align*}
Then by Equation \eqref{equ:potential}, for $i=1,\cdots, n-1$, we get
\[
\frac{P_{i}(x;y,z)}{P_i(x;\alpha y,\alpha z)}=\frac{\lambda_{i}}{\lambda_{i+1}}=e^{\ell_i(\alpha)}.
\] 
\end{proof}

As a consequence of Proposition~\ref{prop:potentialspectral}, we obtain
\begin{cor}
\label{cor:iratiocycle}
For $\rho \in \mathrm{Pos}_n^h(S_{g,m})$ and its canonical lift $(\rho,\xi) \in \mathcal{X}_n(S_{g,m})$, we define the $i$-th potential ratio $B_i$ for $\rho$. Let $\alpha$ be an oriented boundary component of $S_{g,m}$. The \emph{$i$-th period} of $\alpha$ for $B_i$:
\begin{align*}
\log B_i\left(\alpha^-;\alpha^+,\alpha(y),y\right)=\ell_i(\alpha),\quad \text{where $y\neq \alpha^\pm$.}
\end{align*}
\end{cor}

The following lemma is crucial to obtain a closed-form formula for the gap functions with respect to $i$-th potential ratios in Theorem \ref{theorem:boundaryith}.
\begin{prop}
\label{prop:ratioP}
For $\rho\in \mathrm{Pos}_n^h(S_{g,m})\cup \mathrm{Pos}_n^u(S_{g,m})$ and its canonical lift $(\rho,\xi)\in \mathcal{X}_n(S_{g,m})$, we define $P_i$ as in Definition \ref{defn:ratio}. Let $x\neq \delta^-, \delta^+$ (Usually $x=\alpha^-,p$ and $\delta\in\{\beta,\beta^{-1},\gamma,\gamma^{-1}\}$ for an embedded pair of pants containing $\alpha$ or $p$). We have
\begin{align*}
\frac{P_i(x;\delta^+,\delta^{-1}x)}{P_i(x;\delta x,\delta^+)}=K_i(\delta,\delta_x)\cdot  \frac{\lambda_i(\rho(\delta))}{\lambda_{i+1}(\rho(\delta))},
\end{align*}
where
\begin{align}
\label{equation:kappa}
K_i(\delta,\delta_x) = \frac{1+\sum_{c=1}^{i-1} \prod_{j=1}^c T_{n-i,j,i-j}(\delta x,\delta^+,x)}{1+\sum_{c=1}^{i-1} \prod_{j=1}^c T_{n-i,j,i-j}(x,\delta x,\delta^+)} \cdot \frac{\prod_{j=1}^{n-i-1} T_{n-i-j,j,i}(x,\delta x,\delta^+)}{\prod_{j=1}^{i-1} T_{j,n-i,i-j}(x,\delta x,\delta^+)}.
\end{align}
Let us define
\begin{align*}
\kappa_i(\delta,\delta_x):=\log K_i(\delta,\delta_x).
\end{align*}

\end{prop}
\begin{proof}
Firstly, we have 
\begin{align*}
\frac{P_i(x;\delta^+,\delta^{-1}(x))}{P_i(x;\delta x,\delta^+)}=\frac{P_i(\delta x;\delta^+,x)}{P_i(x;\delta x,\delta^+)}.
\end{align*}
We compute the right hand side of the above equation. By Proposition \ref{prop:PT}, we have
\begin{align*}
P_i(x;\delta^+,\delta x) =\alpha^{x;\delta^+,\delta x}_{n-i,i,0} \left(1+\sum_{c=1}^{i-1} \prod_{j=1}^c T_{n-i,j,i-j}(x,\delta x,\delta^+)\right),
\end{align*}
\begin{align*}
P_i(\delta x;x,\delta^+) =\alpha^{\delta x;x,\delta^+}_{n-i,i,0} \left(1+\sum_{c=1}^{i-1} \prod_{j=1}^c T_{n-i,j,i-j}(\delta x,\delta^+,x)\right).
\end{align*}

\begin{figure}
\includegraphics[scale=0.7]{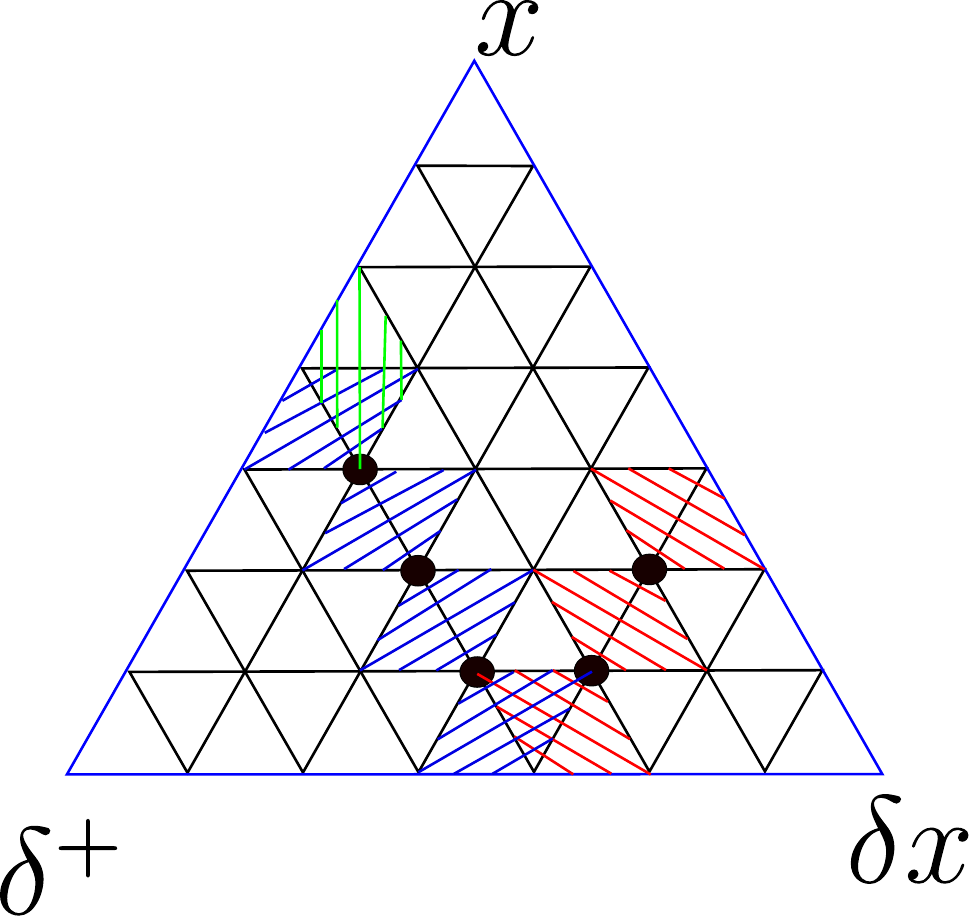}
\caption{Lozenges for Proposition \ref{prop:ratioP}.}
\label{Figure:Pi}
\end{figure}

Then
\begin{align*}
\begin{aligned}
&\frac{P_i(\delta x;\delta^+,x)}{P_i(x;\delta x,\delta^+)}=\frac{P_i(\delta x;x,\delta^+)}{P_i(x;\delta^+,\delta x)}
=\frac{1+\sum_{c=1}^{i-1} \prod_{j=1}^c T_{n-i,j,i-j}(\delta x,\delta^+,x)}{1+\sum_{c=1}^{i-1} \prod_{j=1}^c T_{n-i,j,i-j}(x,\delta x,\delta^+)} \cdot \frac{\alpha^{\delta x;x,\delta^+}_{n-i,i,0}}{\alpha^{x;\delta^+,\delta x}_{n-i,i,0}}.
\end{aligned}
\end{align*}

Observing Figure \ref{Figure:Pi}, to compute $\alpha^{\delta x;x,\delta^+}_{n-i,i,0}/\alpha^{x;\delta^+,\delta x}_{n-i,i,0}$, we need to divide the the red lozenge with one edge on $(x,\delta x)$ by the green lozenge with one edge on $(x,\delta^+)$. Then we decompose this red/green as:
\begin{align}
\label{equation:lem7171}
\frac{\alpha^{\delta x;x,\delta^+}_{n-i,i,0}}{\alpha^{x;\delta^+,\delta x}_{n-i,i,0}}
= \frac{\alpha^{\delta x;x,\delta^+}_{n-i,i,0}}{\alpha^{\delta x;x,\delta^+}_{n-i,1,i-1}}  \cdot \frac{\alpha^{\delta^+;\delta x, x}_{i,1,n-i-1}}{ \alpha^{\delta^+;\delta x,x}_{i,n-i,0}} \cdot \frac{\alpha^{\delta x;x,\delta^+}_{n-i,1,i-1} \alpha^{\delta^+;\delta x,x}_{i,n-i,0}}{\alpha^{\delta^+;\delta x,x}_{i,1,n-i-1} \alpha^{x;\delta^+,\delta x}_{n-i,i,0}},
\end{align}
where the first (second resp.) term of the product on the right hand side corresponds to the two corner red (blue resp.) lozenges crossed by the $i$-th level of $\delta x$ ($(n-i)$-th level of $\delta^+$ resp.).   Using all the intermediate lozenges crossed by the $i$-th level of $\delta x$ ($(n-i)$-th level of $\delta^+$ resp.), the right hand side of Equation \eqref{equation:lem7171} equals 
\begin{align*}
\begin{aligned}
&\prod_{j=1}^{i-1} T_{n-i,j,i-j}(\delta x, x,\delta^+) \cdot \frac{1}{\prod_{j=1}^{n-i-1} T_{i,j,n-i-j}(\delta^+,\delta x,x)} \cdot  \frac{\alpha^{\delta x;x,\delta^+}_{n-i,1,i-1} \alpha^{\delta^+;\delta x,x}_{i,n-i,0}}{\alpha^{\delta^+;\delta x,x}_{i,1,n-i-1} \alpha^{x;\delta^+,\delta x}_{n-i,i,0}}
\\& =\frac{\prod_{j=1}^{n-i-1} T_{n-i-j,j,i}(x,\delta x,\delta^+)}{\prod_{j=1}^{i-1} T_{j,n-i,i-j}(x,\delta x,\delta^+)} \cdot \frac{\alpha^{\delta x;x,\delta^+}_{n-i,1,i-1} \alpha^{\delta^+;\delta x,x}_{i,n-i,0}}{\alpha^{\delta^+;\delta x,x}_{i,1,n-i-1} \alpha^{x;\delta^+,\delta x}_{n-i,i,0}}.
\end{aligned}
\end{align*}

\begin{figure}
\includegraphics[scale=0.7]{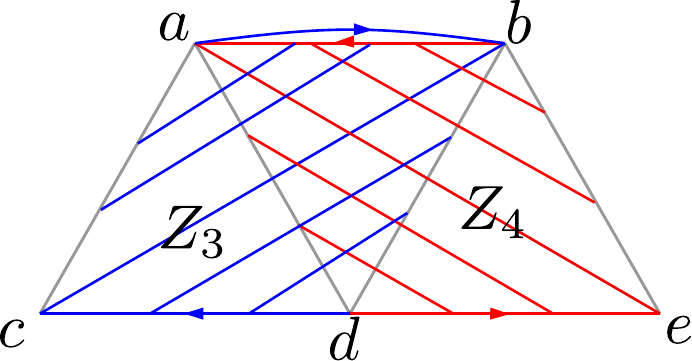}
\caption{$\mathcal{A}$ coordinates are labelled on the vertices.}
\label{Figure:combdimer}
\end{figure}
Since the product of the two lozenges $Z_3=\frac{bc}{ad}$ and $Z_4=\frac{ae}{bd}$ in Figure \ref{Figure:combdimer} is $\frac{ce}{d^2}$, two similar products
$\alpha^{\delta x;x,\delta^+}_{n-i,1,i-1} \alpha^{\delta^+;\delta x,x}_{i,n-i,0}$ and $\alpha^{\delta^+;\delta x,x}_{i,1,n-i-1} \alpha^{x;\delta^+,\delta x}_{n-i,i,0}$ along the edge $(\delta x,\delta^+)$ and $(x,\delta^+)$ respectively can be explicitly expressed. We obtain

\begin{align*}
\begin{aligned}
&\frac{\alpha^{\delta x;x,\delta^+}_{n-i,1,i-1} \alpha^{\delta^+;\delta x,x}_{i,n-i,0}}{\alpha^{\delta^+;\delta x,x}_{i,1,n-i-1} \alpha^{x;\delta^+,\delta x}_{n-i,i,0}}
\\&= \frac{\Delta \left((
\delta x)^{n-i+1}  \wedge \delta^{i-1}  \right)}{\Delta \left(x^{n-i+1}  \wedge \delta^{i-1}  \right)}\cdot \frac{\Delta \left((
\delta x)^{n-i-1}  \wedge \delta^{i+1}  \right)}{\Delta \left(x^{n-i-1}  \wedge \delta^{i+1}  \right)} \cdot \left(\frac{\Delta \left((\delta x)^{n-i}  \wedge \delta^{i}  \right)}{\Delta \left((x^{n-i}  \wedge \delta^{i}  \right)}\right)^{-2}
\\&=\frac{1}{\lambda_1\cdots \lambda_{i-1}(\rho(\delta)) } \cdot \frac{1}{\lambda_1\cdots \lambda_{i+1}(\rho(\delta)) } \cdot \left(\lambda_1\cdots \lambda_{i}(\rho(\delta))\right)^2
= \frac{\lambda_i(\rho(\delta))}{\lambda_{i+1}(\rho(\delta))}.
\end{aligned}
\end{align*}
We conclude that 
\begin{align*}
\frac{P_i(x;\delta^+,\delta^{-1}(x))}{P_i(x;\delta x,\delta^+)}=K_i(\delta,\delta_x)\cdot  \frac{\lambda_i(\rho(\delta))}{\lambda_{i+1}(\rho(\delta))}.
\end{align*}
\end{proof}

\subsection{McShane identities for $i$-th potential ratio}

\begin{thm}[McShane identity for $i$-th potential ratio]
\label{theorem:boundaryith}
For a $\operatorname{PGL}_n(\mathbb{R})$-positive representation $\rho \in \mathrm{Pos}_n^h(S_{g,m})$ with loxodromic boundary monodromy, let $(\rho,\xi) \in \mathcal{X}_n(S_{g,m})$ be the canonical lift (Definition \ref{definition:ratioperiod}) of $\rho$ which induces the $i$-th potential ratio $B_i$. Let $\alpha$ be a distinguished oriented boundary component of $S_{g,m}$ such that $S_{g,m}$ is on the left side of $\alpha$. Recall Equation~\eqref{equation:mirid} 
\begin{equation*}
\mathcal{D}(x, y, z) = \, \log \left( \frac{e^{\frac{x}{2}} + e^{\frac{y+z}{2}}}{e^{\frac{-x}{2}} + e^{\frac{y+z}{2}}} \right).
\end{equation*}
For $i=1,\cdots,n-1$, we have the equality:
\begin{align}
\label{eq:altsummand}
\begin{aligned}
&\sum_{(\beta,\gamma)\in \subvec{\mathcal{P}}_\alpha}
 \mathcal{D}(\ell_i(\alpha), \phi_i(\beta,\gamma)+\kappa_i(\beta,\beta_{\alpha^-})+\ell_i(\beta), \phi_i(\beta,\gamma)+\kappa_i(\gamma,\gamma_{\alpha^-})+\ell_i(\gamma))
\\&  + \sum_{(\beta,\gamma)\in \subvec{\mathcal{P}}^{\partial}_\alpha}
(\mathcal{D}(\ell_i(\alpha), \phi_i'(\beta,\gamma)+\kappa_i(\beta,\beta_{\alpha^-})+\ell_i(\beta), \phi_i'(\beta,\gamma)-\kappa_i(\gamma^{-1},\gamma^{-1}_{\alpha^-})-\ell_i(\gamma^{-1}))
\\&-\mathcal{D}(\ell_i(\alpha), \phi_i(\beta,\gamma)+\kappa_i(\beta,\beta_{\alpha^-})+\ell_i(\beta), \phi_i(\beta,\gamma)+\kappa_i(\gamma,\gamma_{\alpha^-})+\ell_i(\gamma)))=\ell_i(\alpha),
\end{aligned}
\end{align}
where $\xvec{\mathcal{P}}_\alpha$ is the set of the homotopy classes of boundary-parallel pairs of pants, and $\xvec{\mathcal{P}}^{\partial}_\alpha$ is a subset of $\vec{\mathcal{P}}_\alpha$ containing another boundary component of $S_{g,m}$ as in Definition \ref{definition:Pp}. For each boundary-parallel pair of pants, we fix a marking on the boundary components $\alpha,\beta,\gamma$ such that $\alpha \beta^{-1}\gamma =1$ as in Figure \ref{Figure:pti}.
Recall Proposition \ref{prop:ratioP} for the definition of $\kappa_i(\delta,\delta_x)$.
Here
\begin{align*}
d_i(\beta,\gamma):=\log (-B_i(\alpha^-;\gamma(\alpha^-),\beta^+,\gamma^+)),
\end{align*}
\begin{align*}
 e_i(\beta,\gamma):= \log(-B_i(\alpha^-;\gamma^{-1}(\alpha^-),\gamma^+, \gamma^{-1}(\beta^+))),
\end{align*}
\begin{align*}
d_i'(\beta,\gamma):=\log (-B_i(\alpha^-;\gamma(\alpha^-),\beta^+,\gamma^-)),
\end{align*}
\begin{align*}
e_i'(\beta,\gamma):= \log(-B_i(\alpha^-;\gamma^{-1}(\alpha^-),\gamma^-, \gamma^{-1}(\beta^+))),
\end{align*}
\begin{align*}
\phi_i(\beta,\gamma):=\log \frac{\cosh \frac{e_i(\beta,\gamma)}{2}}{\cosh \frac{d_i(\beta,\gamma)}{2}},\;\;
\phi_i'(\beta,\gamma):=\log \frac{\cosh \frac{e_i'(\beta,\gamma)}{2}}{\cosh \frac{d_i'(\beta,\gamma)}{2}}.
\end{align*}
\end{thm}

Let us start with the following lemma before we prove the above theorem.

\begin{lem}
\label{lemma:cosh12}
Set up the notations as in Theorem \ref{theorem:boundaryith} where $B_i$ is defined with respect to $\rho \in \mathrm{Pos}_n^h(S_{g,m})$ and $(\rho,\xi)\in \mathcal{X}_n(S_{g,m})$, we have
\begin{align*}
\begin{aligned}
&e^{\phi_i(\beta,\gamma)} \cdot e^{\frac{1}{2}\left(\kappa_i(\gamma,\gamma_{\alpha^-})+\ell_i(\gamma)+\kappa_i(\beta,\beta_{\alpha^-})+\ell_i(\beta)+\ell_i(\alpha)\right)}
\\&=-e^{\ell_i(\alpha)} \cdot B_i(\alpha^-;\gamma^+,\beta^+,\gamma^{-1}(\beta^+)).
\end{aligned}
\end{align*}
\end{lem}

\begin{proof}
Since $\gamma(\alpha^-)=\beta(\alpha^-)$, by Proposition~\ref{prop:ratioP} we have
\begin{align}
\label{equation:lem719comp1}
\begin{aligned}
& e^{d_i(\beta,\gamma)+e_i(\beta,\gamma)+\kappa_i(\gamma,\gamma_{\alpha^-})+\ell_i(\gamma)} 
\\= & \frac{P_i(\alpha^-;\gamma^+,\gamma(\alpha^-))}{P_i(\alpha^-;\beta(\alpha^-),\beta^+)} \cdot 
\frac{P_i(\alpha^-;\gamma^{-1}(\beta^+),\gamma^{-1}(\alpha^-))}{P_i(\alpha^-;\gamma^{-1}(\alpha^-),\gamma^+)}
\cdot \frac{P_i(\alpha^-;\gamma^+,\gamma^{-1}(\alpha^-))}{P_i(\alpha^-;\gamma(\alpha^-),\gamma^+)}
\\&=\frac{P_i(\alpha^-;\gamma^{-1}(\beta^+),\gamma^{-1}(\alpha^-))}{P_i(\alpha^-;\beta(\alpha^-),\beta^+)}.
\end{aligned}
\end{align}

By $\alpha^{-1} \gamma^{-1}=\beta^{-1}$ and Proposition \ref{prop:potentialspectral}, we get
\begin{equation}
\label{equation:lem719comp2}
P_i(\alpha^-;\gamma^{-1}(\beta^+),\gamma^{-1}(\alpha^-))=e^{-\ell_i(\alpha)}\cdot P_i(\alpha^-;\beta^+,\beta^{-1}(\alpha^-)).
\end{equation}
Thus the right hand side of Equation \eqref{equation:lem719comp1} is equal to
\begin{align*}
\frac{e^{-\ell_i(\alpha)}\cdot P_i(\alpha^-;\beta^+,\beta^{-1}(\alpha^-))}{P_i(\alpha^-;\beta(\alpha^-),\beta^+)}=e^{\kappa_i(\beta,\beta_{\alpha^-})+\ell_i(\beta)-\ell_i(\alpha)} .
\end{align*}
Thus 
\begin{align}
\label{equation:lem719comp3}
\begin{aligned}
& e^{\phi_i(\beta,\gamma)} \cdot e^{\frac{1}{2}\left(\kappa_i(\gamma,\gamma_{\alpha^-})+\ell_i(\gamma)+\kappa_i(\beta,\beta_{\alpha^-})+\ell_i(\beta)+\ell_i(\alpha)\right)} 
=  \frac{1+ e^{-e_i(\beta,\gamma)}}{1+e^{d_i(\beta,\gamma)}}\cdot e^{\kappa_i(\beta,\beta_{\alpha^-})+\ell_i(\beta)}.
\end{aligned}
\end{align}
By Equation \eqref{equation:lem719comp2} and additivity of $i$-th characters, the right hand side of Equation \eqref{equation:lem719comp3} is equal to

\begin{align*}
\begin{aligned}
& \frac{\frac{P_i(\alpha^-;\gamma^{-1}(\beta^+),\gamma^+)}{P_i(\alpha^-;\gamma^{-1}(\beta^+),\gamma^{-1}(\alpha^-))}}{\frac{P_i(\alpha^-;\gamma^+,\beta^+)}{P_i(\alpha^-;\beta(\alpha^-),\beta^+)}} \cdot \frac{P_i(\alpha^-;\beta^+,\beta^{-1}(\alpha^-))}{P_i(\alpha^-;\beta(\alpha^-),\beta^+)}
= \frac{e^{\ell_i(\alpha)}\cdot P_i(\alpha^-;\gamma^{-1}(\beta^+),\gamma^+)}{P_i(\alpha^-;\gamma^+,\beta^+)}
\\&=-e^{\ell_i(\alpha)} \cdot B_i(\alpha^-;\gamma^+,\beta^+,\gamma^{-1}(\beta^+)).
\end{aligned}
\end{align*}
\end{proof}

\begin{proof}[Proof of Theorem \ref{theorem:boundaryith}]
Firstly, let us show that
\begin{equation}
\begin{aligned}
\label{equation:mcloxo1}
&\log B_1(\alpha^-;\alpha^+,\gamma^+,\beta^+)
\\&= \log \left(\frac{e^{\ell_i(\alpha)}+e^{\phi_i(\beta,\gamma)}  \cdot e^{\frac{1}{2}\left(\kappa_i(\gamma,\gamma_{\alpha^-})+\ell_i(\gamma)+\kappa_i(\beta,\beta_{\alpha^-})+\ell_i(\beta)+\ell_i(\alpha)\right)}}{1 +e^{\phi_i(\beta,\gamma)}\cdot e^{\frac{1}{2}\left(\kappa_i(\gamma,\gamma_{\alpha^-})+\ell_i(\gamma)+\kappa_i(\beta,\beta_{\alpha^-})+\ell_i(\beta)+\ell_i(\alpha)\right)}}\right).
\end{aligned}
\end{equation}

We show two sides of the above equation are equal by evaluating two sides at the strictly increasing function $f(A)=\frac{e^A-1}{e^{\ell_i(\alpha)}-1}$. Then the left hand side of Equation \eqref{equation:mcloxo1} becomes $\frac{P_i(\alpha^-; \gamma^+,\beta^+)}{P_i(\alpha^-; \alpha^+,\gamma^+)\cdot (e^{\ell_i(\alpha)}-1)}$, and by Lemma \ref{lemma:cosh12} the right hand side of Equation \eqref{equation:mcloxo1} becomes $\frac{1}{1+\frac{e^{\ell_i(\alpha)}\cdot P_i(\alpha^-;\gamma^{-1}(\beta^+),\gamma^+)}{P_i(\alpha^-;\gamma^+,\beta^+)}}$. 
Equation \eqref{equation:mcloxo1} is equivalent to 
\begin{align}
\label{equation:mcloxo2}
P_i(\alpha^-; \alpha^+,\gamma^+)\cdot (e^{\ell_i(\alpha)}-1)=e^{\ell_i(\alpha)}\cdot P_i(\alpha^-;\gamma^{-1}(\beta^+),\gamma^+)+P_i(\alpha^-;\gamma^+,\beta^+).
\end{align}

By Proposition \ref{prop:potentialspectral}, we have
\begin{align*}
\begin{aligned}
&P_i(\alpha^-; \alpha^+,\gamma^+)\cdot (e^{\ell_i(\alpha)}-1)
\\=& P_i(\alpha^-; \alpha^+,\alpha^{-1}(\gamma^+)) - P_i(\alpha^-; \alpha^+,\gamma^+)
\\=& P_i(\alpha^-; \gamma^+,\alpha^{-1}(\gamma^+)).
\end{aligned}
\end{align*}
Thus Equation \eqref{equation:mcloxo2} is equivalent to
\begin{align*}
P_i(\alpha^-; \beta^+,\alpha^{-1}(\gamma^+))
=e^{\ell_i(\alpha)}\cdot P_i(\alpha^-;\gamma^{-1}(\beta^+),\gamma^+),
\end{align*}
which is a consequence of Proposition \ref{prop:potentialspectral}. Hence we obtain Equation \eqref{equation:mcloxo1}.

For each gap function for $\xvec{\mathcal{P}}^{\partial}_\alpha$, similarly, we have
\begin{equation*}
\begin{aligned}
&\log B_i\left(\alpha^-;\alpha^+,\gamma^-,\beta^+\right) 
\\&=\log 
\left(\frac{e^{\ell_i(\alpha)}+e^{\phi_i'(\beta,\gamma)}  \cdot e^{\frac{1}{2}\left(-\kappa_i(\gamma^{-1},\gamma^{-1}_{\alpha^-})-\ell_i(\gamma^{-1})+\kappa_i(\beta,\beta_{\alpha^-})+\ell_i(\beta)+\ell_i(\alpha)\right)}}{1 +e^{\phi_i'(\beta,\gamma)}\cdot e^{\frac{1}{2}\left(-\kappa_i(\gamma^{-1},\gamma^{-1}_{\alpha^-})-\ell_i(\gamma^{-1})+\kappa_i(\beta,\beta_{\alpha^-})+\ell_i(\beta)+\ell_i(\alpha)\right)}}\right).
\end{aligned}
\end{equation*}
Then, we use
\begin{align*}
\begin{aligned}
&\log B_i\left(\alpha^-;\alpha^+,\gamma^-,\gamma^+\right) = \log B_i\left(\alpha^-;\alpha^+,\gamma^-,\beta^+\right)-\log B_i\left(\alpha^-;\alpha^+,\gamma^+,\beta^+\right).
\end{aligned}
\end{align*}
Finally, we conclude the theorem by Theorem \ref{thm:boundary} with $B^\rho=B_i$.
\end{proof}

The similar formulas as Equations \eqref{equation:tts11}\eqref{equation:d121s11} hold after replacing $x$ by $\alpha^-$ for $\rho \in \mathrm{Pos}_3^h(S_{1,1})$, thus we can simplify Equation \eqref{eq:altsummand} as follows. 
\begin{cor}
\label{corollary:S11sl3h}
For a $\operatorname{PGL}_3(\mathbb{R})$-positive representation $\rho \in \mathrm{Pos}_3^h(S_{1,1})$ with loxodromic boundary monodromy, let $(\rho,\xi) \in \mathcal{X}_3(S_{1,1})$ be the canonical lift (Definition \ref{definition:ratioperiod}) of $\rho$ which induces the $1$st potential ratio $B_1$. Let $\alpha$ be the oriented boundary component of $S_{1,1}$ such that $S_{1,1}$ is on the left side of $\alpha$. Let $\xvec{\mathcal{C}}_{1,1}$ be the collection of oriented simple closed curves up to homotopy on $S_{1,1}$. We have
\begin{align*}
&\sum_{\gamma \in \subvec{\mathcal{C}}_{1,1}}
 \log \left(\frac{e^{\frac{\ell_1(\alpha)}{2}}+ e^{\tau(\gamma)+\ell_1(\gamma)}}{e^{-\frac{\ell_1(\alpha)}{2}} +e^{\tau(\gamma)+\ell_1(\gamma)}}\right)=\ell_1(\alpha),
\end{align*}
where $\tau(\gamma):=\log T(\alpha^-,\gamma(\alpha^-),\gamma^+)$.
\end{cor}
For $\rho \in \mathrm{Pos}_n^h(S_{1,1})$ with $n\geq 4$, we do not have similar formula as that in Corollary \ref{corollary:S11sl3h}, since we do not have similar formulas as Equations \eqref{equation:tts11}\eqref{equation:d121s11}.

\subsubsection{Newly inserted parameters} 
we briefly study the parameters arising in Theorem \ref{theorem:boundaryith}.
Since $\kappa_i(\beta,\beta_{\alpha^-})$ is the log of a positive rational function of triple ratios (Equation \eqref{equation:kappa}), by Theorem \ref{thm:bounded} we have
\begin{cor}
For $\rho \in \mathrm{Pos}_n(S_{g,m})$, the collection $\{\kappa_i(\beta,\beta_{\alpha^-}),\kappa_i(\gamma,\gamma_{\alpha^-})\}_{(\beta,\gamma)\in \subvec{\mathcal{P}}_\alpha}$ is bounded within some compact interval.
\end{cor}

The function $\phi_i'(\beta,\gamma)$ is similar to $\phi_i(\beta,\gamma)$. Let us consider $\phi_i(\beta,\gamma)=\log \frac{\cosh \frac{e_i(\beta,\gamma)}{2}}{\cosh \frac{d_i(\beta,\gamma)}{2}}$. By Proposition \ref{prop:BC}, we have
\[e_i(\beta,\gamma)=\log\left(\frac{1+\sum_{c=1}^{i-1} \prod_{j=1}^c \frac{1}{ T_{n-i,i-j,j}(\gamma(\alpha^-),\alpha^-,\beta^+)}}{1+\sum_{c=1}^{i-1} \prod_{j=1}^c \frac{1}{ T_{n-i,i-j,j}(\gamma(\alpha^-),\alpha^-,\gamma^+)}} \cdot D_{n-i}(\alpha^-;\gamma(\alpha^-),\beta^+,\gamma^+)\right),\]
and
\[d_i(\beta,\gamma)=\log\left(\frac{1+\sum_{c=1}^{i-1} \prod_{j=1}^c \frac{1}{ T_{n-i,i-j,j}(\alpha^-;\gamma(\alpha^-),\gamma^+)}}{1+\sum_{c=1}^{i-1} \prod_{j=1}^c \frac{1}{ T_{n-i,i-j,j}(\alpha^-,\gamma(\alpha^-),\beta^+)}} \cdot D_i(\alpha^-;\gamma(\alpha^-),\beta^+,\gamma^+)\right).\]
By Theorem \ref{thm:bounded}, the collection 
\begin{small}
\[\left\{\frac{1+\sum_{c=1}^{i-1} \prod_{j=1}^c \frac{1}{ T_{n-i,i-j,j}(\gamma(\alpha^-),\alpha^-,\beta^+)}}{1+\sum_{c=1}^{i-1} \prod_{j=1}^c \frac{1}{ T_{n-i,i-j,j}(\gamma(\alpha^-),\alpha^-,\gamma^+)}},\; \frac{1+\sum_{c=1}^{i-1} \prod_{j=1}^c \frac{1}{ T_{n-i,i-j,j}(\alpha^-;\gamma(\alpha^-),\gamma^+)}}{1+\sum_{c=1}^{i-1} \prod_{j=1}^c \frac{1}{ T_{n-i,i-j,j}(\alpha^-,\gamma(\alpha^-),\beta^+)}}\right\}_{(\beta,\gamma)\in \subvec{\mathcal{P}}_\alpha}\]
\end{small}
is bounded. 
\begin{conj}
\label{conjecture:phii}
For any $\operatorname{PGL}_n(\mathbb{R})$-positive representation $\rho$, the collection $\{\phi_i(\beta,\gamma)\}_{(\beta,\gamma)\in \subvec{\mathcal{P}}_\alpha}$ is bounded within some compact interval.
\end{conj}


\begin{rmk}
For each boundary-parallel pair of pants $(\beta,\gamma)\in \xvec{\mathcal{P}}_\alpha$, comparing the corresponding gap function 
\[\mathcal{D}(\ell_i(\alpha), \phi_i(\beta,\gamma)+\kappa_i(\beta,\beta_{\alpha^-})+\ell_i(\beta), \phi_i(\beta,\gamma)+\kappa_i(\gamma,\gamma_{\alpha^-})+\ell_i(\gamma))\]
in Theorem \ref{theorem:boundaryith} with the gap function $\mathcal{D}(\ell(\alpha),\ell(\beta),\ell(\gamma))$ in hyperbolic case, 
\[\kappa_i(\beta,\beta_{\alpha^-}),\;\;\;\kappa_i(\gamma,\gamma_{\alpha^-}),\;\;\;\phi_i(\beta,\gamma)\]
are the newly inserted parameters. In $n$-Fuchsian case,
\begin{enumerate}
\item all the triple ratios are $1$, by Equation \eqref{equation:kappa}, for $i=1,\cdots,n-1$,
\[\kappa_i(\beta,\beta_{\alpha^-})=\kappa_i(\gamma,\gamma_{\alpha^-})=0;\] 
\item and for any $i,j=1,\cdots,n-1$,
\[D_i(\alpha^-;\gamma(\alpha^-),\beta^+,\gamma^+)=D_j(\alpha^-;\gamma(\alpha^-),\beta^+,\gamma^+),\]
thus 
\[\phi_i(\beta,\gamma)=\phi_i'(\beta,\gamma)=0;\]
\item and for any $\delta\in\pi_1(S)$ and any $i,j=1,\cdots,n-1$,
\[\ell_i(\delta)=\ell_j(\delta).\]
\end{enumerate}
 Hence the $(n-1)$ McShane identities in that case are all the same as the Mizakhani's generalized McShane identity \cite[Theorem 4.2]{mirz_simp} after rearrangement as remarks in \cite[Theorem 4.1.2.1]{LM09}. 
\end{rmk}

Using Corollary \ref{corollary:LM-HS} and Theorem \ref{theorem:boundaryith}, we have the following corollary.
\begin{cor}
The Labourie--McShane's identities in \cite[Section 10]{LM09} can be written as summation of regular expressions of the Fock-Goncharov coordinates, where the regular expressions are the gap functions in Theorem \ref{theorem:boundaryith} and the logs of certain rational functions of triple ratios. 
\end{cor}

\textbf{Boundary-parallel pairs of half-pants summation}

\begin{defn}
\label{definition:Ri}
Set up as in Theorem \ref{theorem:boundaryith}.
For $\delta \in \{\beta,\beta^{-1}, \gamma, \gamma^{-1}\}$, we define 
\begin{align*}
r_i(\delta,\delta_{\alpha^-})  :=\frac{\log B_i(\alpha^-;\alpha^+,\delta(\alpha^-), \delta^{-1}(\alpha^-)) }{\ell_i(\alpha)}.
\end{align*}
Then for $\delta\in \{\beta,\gamma^{-1}\}$, we have 
\begin{align*}
r_i(\delta,\delta_{\alpha^-})=-r_i(\delta^{-1}, \delta^{-1}_{\alpha^-})>0,
\end{align*} 
and 
\begin{align*}
r_i(\beta,\beta_{\alpha^-})+r_i(\gamma^{-1},\gamma^{-1}_{\alpha^-})=1.
\end{align*} 
In the case $S_{g,m}=S_{1,1}$, for $\mu \in \xvec{\mathcal{H}}_\alpha$, we denote $r_i(\mu)$ instead.
\end{defn}
Later, we will relate $r_i(\delta,\delta_{\alpha^-})$ to the $i$-th half-pants ratio in Definition \ref{definition:hpr}.

\begin{lem}
\label{lem:equb}
Let $B_i$ be the $i$-th potential ratio for $\rho \in \mathrm{Pos}_n^h(S_{g,m})$. Let $a,b,c,d,e$ be five distinct points in $\partial_\infty \pi_1(S_{g,m})$. Then
\begin{align*}
B_i\left(a;b,c,d\right) = \frac{B_i(a;b,c,e)- B_i(a;d,c,e)}{1 - B_i(a;d,c,e)}.
\end{align*}
\end{lem}
\begin{proof}
Using additivity of the $i$-th character, we have
\begin{align*}
\begin{aligned}
&B_i(a;b,c,e)- B_i(a;d,c,e)=\frac{P_i(a;b,e)}{P_i(a;b,c)}- \frac{P_i(a;d,e)}{P_i(a;d,c)}
\\&=\frac{P(a;b,e)\cdot P(a;d,c)-P(a;d,e)\cdot P(a;b,c)}{P(a;b,c) \cdot P(a;d,c)}
\\&=\frac{(P(a;b,c)+P(a;c,e))\cdot (P(a;d,e)+P(a;e,c))-P(a;d,e)\cdot P(a;b,c)}{P(a;b,c)\cdot P(a;d,c)}
\\&=\frac{P(a;b,e)\cdot P(a;e,c)+ P(a;c,e)\cdot P(a;d,e)}{P(a;b,c)\cdot P(a;d,c)}
\\&=\frac{P(a;b,d)\cdot P(a;e,c)}{P(a;b,c)\cdot P(a;d,c)}.
\end{aligned}
\end{align*}
Thus 
\begin{equation*}
\begin{aligned}
&\frac{B_i(a;b,c,e)- B_i(a;d,c,e)}{1 - B_i(a;d,c,e)}
= \frac{P(a;b,d)\cdot P(a;e,c)}{P(a;b,c)\cdot P(a;d,c)} \big/\frac{P(a;e,c)}{P(a;d,c)}
=\frac{P(a;b,d)}{P(a;b,c)}
\\&=B_i(a;,b,c,d).
\end{aligned}
\end{equation*}

\end{proof}

\begin{thm}[Boundary-parallel pairs of Half-pants summation]
\label{thm:equgeo}
For a $\operatorname{PGL}_n(\mathbb{R})$-positive representation $\rho \in \mathrm{Pos}_n^h(S_{g,m})$ with loxodromic boundary monodromy, let $(\rho,\xi) \in \mathcal{X}_n(S_{g,m})$ be the canonical lift (Definition \ref{definition:ratioperiod}) of $\rho$ which induces the $i$-th potential ratio $B_i$. Let $\alpha$ be a distinguished oriented boundary component of $S_{g,m}$ such that $S_{g,m}$ is on the left side of $\alpha$. For $i=1,\cdots,n-1$, we have the equality:
\begin{align}
\begin{aligned}
&\ell_i(\alpha)
=\sum_{(\delta,\delta_{\alpha^-})\in \subvec{\mathcal{H}}_\alpha}
 \left|\log \frac{e^{r_i(\delta,\delta_{\alpha^-}) \cdot \ell_i(\alpha)}+e^{\kappa_i(\delta,\delta_{\alpha^-})+  \ell_i(\delta)}}{1 +e^{\kappa_i(\delta,\delta_{\alpha^-})+  \ell_i(\delta)}}\right|
\\&  + \sum_{(\beta,\gamma)\in \subvec{\mathcal{H}}^{\partial}_\alpha}
\log 
\left(B_i(\alpha^-;\alpha^+,\gamma^-,\gamma^+)\right),
\end{aligned}
\end{align}
where $\xvec{\mathcal{H}}_\alpha$ is the set of the homotopy classes of boundary-parallel pairs of half-pants, and $\xvec{\mathcal{H}}^{\partial}_\alpha$ is a subset of $\xvec{\mathcal{H}}_\alpha$ containing another boundary component of $S_{g,m}$ as in Definition \ref{definition:hparital}. Recall $r_i(\delta,\delta_{\alpha^-})$ in Definition \ref{definition:Ri}. For each pair of half-pants, we fix a marking as in Figure \ref{Figure:pti}.
\end{thm}

\begin{proof}
For any $(\delta,\delta_{\alpha^-})\in \xvec{\mathcal{H}}_\alpha$, by Lemma \ref{lem:equb}, we have
\begin{align*}
 B_i\left(\alpha^-;\alpha^+,\delta(\alpha^-),\delta^+\right)=\frac{B_i(\alpha^-;\alpha^+,\delta(\alpha^-), \delta^{-1}(\alpha^-))-B_i\left(\alpha^-;\delta^+,\delta(\alpha^-),\delta^{-1}(\alpha^-)\right)}{1-B_i\left(\alpha^-;\delta^+,\delta(\alpha^-),\delta^{-1}(\alpha^-)\right)}.
\end{align*}

Then by definition of $r_i(\delta,\delta_{\alpha^-})$ and Proposition \ref{prop:ratioP}, we obtain
\begin{align*}
\left|\log B_i\left(\alpha^-;\alpha^+,\delta(\alpha^-),\delta^+\right)\right| =\left|\log \frac{e^{r_i(\delta,\delta_{\alpha^-}) \cdot \ell_i(\alpha)}+e^{\kappa_i(\delta,\delta_{\alpha^-})+  \ell_i(\delta)}}{1 +e^{\kappa_i(\delta,\delta_{\alpha^-})+  \ell_i(\delta)}}\right|.
\end{align*}

We conclude our theorem after Theorem~\ref{thm:boundary}.
\end{proof}

\subsection{McShane-type inequalities for unipotent-bordered positive representations}

We in fact have two strategies for deriving McShane-type inequalities for unipotent-bordered positive representations. The first is to follow the Goncharov--Shen potential splitting idea we employed in Section~\ref{sec:splitting}. The second is to take the loxodromic-bordered identities we just obtained and to consider them under deformation to the unipotent-bordered locus in the representation variety. We choose to illustrate the second strategy which may work for a general positive representation with certain boundary simple root length zero; the necessary ingredients for computing via the first strategy is nevertheless contained in what follows.

\begin{thm}[McShane-type inequality for unipotent-bordered positive representations]
\label{theorem:puncturecase}
Consider a $\operatorname{PGL}_n(\mathbb{R})$-positive representation $\rho$ with unipotent boundary monodromy and let $p\in m_p$ be a distinguished puncture/cusp on $S_{g,m}$. Then, for $i=1,\cdots, n-1$, we have 
\begin{align}
\begin{aligned}
\label{equation:unipps}
&\sum_{(\beta,\gamma)\in \subvec{\mathcal{P}}_p}
 \frac{1}{1 +e^{\phi_i(\beta,\gamma)}  \cdot e^{\frac{1}{2}\left(\kappa_i(\gamma,\gamma_{p})+\ell_i(\gamma)+\kappa_i(\beta,\beta_{p})+\ell_i(\beta)\right)}}
\leq 1,
\end{aligned}
\end{align}
\begin{align}
\begin{aligned}
\label{equation:unihps}
&\sum_{(\delta,\delta_p)\in \subvec{\mathcal{H}}_p}
\frac{B_i(\delta,\delta_p)}{1+e^{\kappa_i(\delta,\delta_p)+\ell_i(\delta)}}
\leq 1,
\end{aligned}
\end{align}
where $\xvec{\mathcal{P}}_p$ is the set of boundary-parallel pairs of pants containing $p$ in Definition~\ref{definition:Pp} and $\xvec{\mathcal{H}}_p$ is the set of the homotopy classes of boundary-parallel pairs of half-pants containing $p$ in Definition \ref{defn:hp}. 
For any pair of pants or pair of half-pants, we fix a marking as in Figure \ref{Figure:pti} taking $\alpha^-$ as a lift of the cusp $p$.
Recall $\phi_i(\beta,\gamma)$ defined in Theorem \ref{theorem:boundaryith}, $\kappa_i(\delta,\delta_p)$ defined in Proposition \ref{prop:ratioP} and $B_i(\delta,\delta_p)$ is the $i$-th half-pants ratio in Definition \ref{definition:hpr}.
\end{thm}
\begin{rmk}
As in the proof of Theorem \ref{thm:boundary}, the reason that we can establish inequality is that we can split using positivity, but this does not ensure the Cantor set introduced in Section \ref{sec:proofideasummary} with respect to the ($i$-th) Goncharov--Shen measure has measure zero.
\end{rmk}

\begin{defn}[Path $l$]
\label{definition:hypara}
For $(\rho,\xi) \in \mathcal{X}_n(S_{g,m})$ with (purely) loxodromic-bordered monodromy representation $\rho$, we choose an analytic path $l$ in $\mathcal{X}_n(S_{g,m})$ satisfying the following conditions:
\begin{enumerate}
\item $l(0) =(\rho_0,\xi_0)$;
\item every element of $l([0,1)) \subset \mathcal{X}_n(S_{g,m})$ has loxodromic monodromy around all of its boundary components;
\item $l(1)=(\rho,\xi) \in \mathcal{X}_n(S_{g,m})$ where $\rho\in \mathrm{Pos}_n^u(S_{g,m})$ is the positive representation with unipotent boundary monodromy as in Theorem \ref{theorem:puncturecase}.
\end{enumerate}
We denote the limit of a function $f$ on $\mathcal{X}_n(S_{g,m})$ under a sequence of elements in $l([0,1))$ that converges to $l(1)$ by {\em $\lim_{\text{hyp$\to$para}} f$}. Under the sequence, the $i$-th length $\ell_i$ of each oriented boundary component converges to $0$ for $i=1,\cdots,n-1$. Geometrically speaking, this is tantamount to the boundary $\alpha$ of $S_{g,m}$ deforms to a cusp $p$.
\end{defn}

\begin{lem}
\label{lem:Sa}
Consider a path $l$ as in Definition \ref{definition:hypara} and the second summation in Theorem \ref{thm:boundary} (or Theorem \ref{theorem:boundaryith}). For any pair of half-pants $\mu \in \xvec{\mathcal{H}}^{\partial}_{\alpha}$ with its cuff a boundary component $\gamma$, as $\gamma$ deforms to a unipotent boundary, we have:
\begin{align*}
\lim_{\text{hyp$\to$para}} \frac{ \log B_i\left(\alpha^-;\alpha^+,\gamma^-,\gamma^+\right)
}{ \ell_i(\alpha)}=0.
\end{align*}
\end{lem}

\begin{proof}
We have
\begin{align*}
&\lim_{\text{hyp$\to$para}} 
\frac{ \log B_i\left(\alpha^-;\alpha^+,\gamma^-,\gamma^+\right)}
{\ell_i(\alpha)}
=\lim_{\text{hyp$\to$para}} 
\frac{\frac{P_i(\alpha^-;\alpha^+,\gamma^+)}
{P_i(\alpha^-;\alpha^+,\gamma^-)}-1}
{\frac{P_i(\alpha^-;\alpha^+,\alpha^{-1}(\gamma^-))}
{P_i(\alpha^-;\alpha^+,\gamma^-)}-1}
\\
&=\lim_{\text{hyp$\to$para}} 
\frac{P_i(\alpha^-;\gamma^-,\gamma^+)}
{P_i(\alpha^-;\gamma^-,\alpha^{-1}(\gamma^-))}.
\end{align*}
Since $\gamma$ is another boundary component which converges to a cusp under $\lim_{\text{hyp$\to$para}}$, $\gamma^+$ converges to $\gamma^-$. Hence we obtain
\begin{align*}
\begin{aligned}
& \lim_{\text{hyp$\to$para}} \frac{P_i(\alpha^-;\gamma^-,\gamma^+)}{P_i(\alpha^-;\gamma^-,\alpha^{-1}(\gamma^-))}
=0.
\end{aligned}
\end{align*}
\end{proof}

\begin{rmk}
The previous result explains why there are no $\xvec{\mathcal{P}}^{\partial}_{\alpha}\cong \xvec{\mathcal{H}}^{\partial}_{\alpha}$ summands in the unipotent-bordered McShane identity.
\end{rmk}

\begin{lem}
\label{lem:lim1}
Consider a path $l$ as in Definition \ref{definition:hypara} and the first summation in Theorem \ref{theorem:boundaryith}. For any pair of pants $(\beta,\gamma)\in \xvec{\mathcal{P}}_{\alpha}$, we have
\begin{align*}
\lim_{\text{hyp$\to$para}}
\frac{\log B_i\left(\alpha^-;\alpha^+,\gamma^+,\beta^+\right) 
}{\ell_i(\alpha)}
= \frac{1}{1 +e^{\phi_i(\beta,\gamma)}  \cdot e^{\frac{1}{2}\left(\kappa_i(\gamma,\gamma_{p})+\ell_i(\gamma)+\kappa_i(\beta,\beta_{p})+\ell_i(\beta)\right)}}.
\end{align*}
\end{lem}
\begin{proof}
Take the formula in Theorem \ref{theorem:boundaryith}, let 
\[Z(x):=e^{\phi_i(\beta,\gamma)}  \cdot e^{\frac{1}{2}\left(\kappa_i(\gamma,\gamma_{x})+\ell_i(\gamma)+\kappa_i(\beta,\beta_{x})+\ell_i(\beta)+\ell_i(\alpha)\right)}.\]
Then 
\begin{equation*}
\begin{aligned}
&\lim_{\text{hyp$\to$para}}
\frac{\log B_i\left(\alpha^-;\alpha^+,\gamma^+,\beta^+\right) } {\ell_i(\alpha)}  =\lim_{\text{hyp$\to$para}} \frac{\log \frac{e^{\ell_i(\alpha)+Z(\alpha^-)}}{1+Z(\alpha^-)} } {\ell_i(\alpha)} 
\\&=\lim_{\text{hyp$\to$para}} \frac{ \frac{e^{\ell_i(\alpha)+Z(\alpha^-)}}{1+Z(\alpha^-)} -1} {\ell_i(\alpha)} =\frac{1}{1+Z(p)}.
\end{aligned}
\end{equation*}
\end{proof}

By a choice of fundamental domain, we define the normalized $(\mu,i)$-Goncharov--Shen potential for the boundary case.
\begin{defn}
For a $\operatorname{PGL}_n(\mathbb{R})$-positive representation $\rho \in \mathrm{Pos}_n^h(S_{g,m})$ with loxodromic boundary monodromy, let $(\rho,\xi) \in \mathcal{X}_n(S_{g,m})$ be the canonical lift (Definition \ref{definition:ratioperiod}) of $\rho$ which induces the $i$-th potential ratio $B_i$. For any $(\beta,\gamma)\in \xvec{\mathcal{P}}_\alpha$ and a choice of its fundamental domain as in Figure \ref{Figure:pti}. Then we define 
\begin{align*}
B_i(\gamma,\gamma_{\alpha^-})
:=\frac{P_i(\alpha^-;\gamma^{-1}(\alpha^-),\gamma(\alpha^-))}{P_i(\alpha^-;\gamma^{-1}(\alpha^-),\beta^{-1}(\alpha^-))},
\end{align*}
\begin{align*}
B_i(\beta,\beta_{\alpha^-})
:=\frac{P_i(\alpha^-;\beta(\alpha^-),\beta^{-1}(\alpha^-))}{P_i(\alpha^-;\gamma^{-1}(\alpha^-),\beta^{-1}(\alpha^-))}.
\end{align*}
\end{defn}
Thus we have
\begin{align*}
B_i(\gamma,\gamma_{\alpha^-})+B_i(\beta,\beta_{\alpha^-})=1
\end{align*}
When we take the limit $\lim_{\text{hyp$\to$para}}$ along the path $l$ in Definition \ref{definition:hypara}, $\alpha^-$ converges to $p$, the ratio $B_i(\delta,\delta_{\alpha^-})$ converges to the $i$-th half-pants ratio $B_i(\delta,\delta_p)$ for $\delta\in (\beta,\gamma)$. Actually $B_i(\delta,\delta_p)$ does not depend on the fundamental domain that we choose.

The following lemma provides the relation between $r_i(\delta,\delta_{\alpha^-})$ and $B_i(\delta,\delta_{\alpha^-})$.
\begin{lem}
\label{lem:equl}
Set up as in Theorem \ref{thm:equgeo}. We have 
\begin{align*}
 \frac{ e^{-r_i(\gamma,\gamma_{\alpha^-})\cdot \ell_i(\alpha)}  
-1}{  e^{\ell_i(\alpha)}-1}=B_i(\gamma,\gamma_{\alpha^-}).
\end{align*}

\begin{align*}
 \frac{ e^{r_i(\beta,\beta_{\alpha^-})\cdot \ell_i(\alpha)} 
-1}{  e^{\ell_i(\alpha)}-1}=\frac{B_i(\beta,\beta_{\alpha^-})}{B_i(\beta,\beta_{\alpha^-})+B_i(\gamma,\gamma_{\alpha^-})\cdot e^{\ell_i(\alpha)}}.
\end{align*}

\end{lem}
\begin{proof}
By direct computation
\begin{align*}
\begin{aligned}
&\frac{ B_i\left(\alpha^-;\alpha^+,\beta(\alpha^-),\beta^{-1}(\alpha^-)\right) 
-1}{ e^{\ell_i(\alpha)}-1}
=\frac{\frac{P_i(\alpha^-;\alpha^+,\beta^{-1}(\alpha^-)) - P_i(\alpha^-;\alpha^+,\beta(\alpha^-)) }{P_i(\alpha^-;\alpha^+,\beta(\alpha^-))}}{ e^{\ell_i(\alpha)}-1}
\\&=  \frac{P_i(\alpha^-;\beta(\alpha^-),\beta^{-1}(\alpha^-))}{(e^{\ell_i(\alpha)}-1)\cdot P_i(\alpha^-;\alpha^+,\beta(\alpha^-)) }
\\&=  \frac{P_i(\alpha^-;\beta(\alpha^-),\beta^{-1}(\alpha^-))}{P_i(\alpha^-;\alpha^+,\alpha^{-1}\beta(\alpha^-))- P_i(\alpha^-;\alpha^+,\beta(\alpha^-)) }
\\&=  \frac{P_i(\alpha^-;\beta(\alpha^-),\beta^{-1}(\alpha^-))}{P_i(\alpha^-;\beta(\alpha^-),\beta^{-1}(\alpha^-))+P_i(\alpha^-;\beta^{-1}(\alpha^-),\alpha^{-1}\beta(\alpha^-)) }
\\&=  \frac{P_i(\alpha^-;\beta(\alpha^-),\beta^{-1}(\alpha^-))}{P_i(\alpha^-;\beta(\alpha^-),\beta^{-1}(\alpha^-))+e^{\ell_i(\alpha)}\cdot P_i(\alpha^-;\gamma^{-1}(\alpha^-),\gamma(\alpha^-)) }
\\&=\frac{B_i(\beta,\beta_{\alpha^-})}{B_i(\beta,\beta_{\alpha^-})+B_i(\gamma,\gamma_{\alpha^-})\cdot e^{\ell_i(\alpha)}}.
\end{aligned}
\end{align*}
Similarly for the other formula.
\end{proof}

Then Lemma \ref{lem:equb} and Lemma \ref{lem:equl} allow us to compute the following.
\begin{lem}
\label{lemma:B11X}
Set up as in Theorem \ref{thm:equgeo}.
Evaluating the function $f(A)=\frac{e^A-1}{e^{\ell_i(\alpha)}-1}$ at the four gap functions for $\xvec{\mathcal{H}}_\alpha$ in Figure~\ref{Figure:pti}, we get:
\begin{enumerate}
\item 
\begin{align*}
\frac{ B_i\left(\alpha^-;\alpha^+,\gamma^{-1}(\alpha^-),\gamma^-\right)-1 
}{ e^{\ell_i(\alpha)}-1}
= B_i(\gamma,\gamma_{\alpha^-}) \cdot \frac{1}{1 +\frac{P_i(\alpha^-;\gamma^-,\gamma(\alpha^-))}{P_i(\alpha^-;\gamma^{-1}(\alpha^-),\gamma^-)}};
\end{align*}
\item 
\begin{align*}
\frac{ B_i\left(\alpha^-;\alpha^+,\gamma^+,\gamma(\alpha^-)\right)-1 
}{ e^{\ell_i(\alpha)}-1}
= B_i(\gamma,\gamma_{\alpha^-}) \cdot \frac{1}{1 + e^{-r_i(\gamma,\gamma_{\alpha^-})\cdot \ell_i(\alpha)} \cdot \frac{P_i(\alpha^-;\gamma^+,\gamma^{-1}(\alpha^-))}{P_i(\alpha^-;\gamma(\alpha^-),\gamma^+)}};
\end{align*}
\item 
\begin{align*}
\begin{aligned}
&\frac{ B_i\left(\alpha^-;\alpha^+,\beta(\alpha^-),\beta^+\right)-1 
}{ e^{\ell_i(\alpha)}-1}
\\=& \frac{B_i(\beta,\beta_{\alpha^-})}{B_i(\beta,\beta_{\alpha^-})+B_i(\gamma,\gamma_{\alpha^-})\cdot e^{\ell_i(\alpha)}} \cdot \frac{1}{1 + \frac{P_i(\alpha^-;\beta^+,\beta^{-1}(\alpha^-))}{P_i(\alpha^-;\beta(\alpha^-),\gamma^-)}};
\end{aligned}
\end{align*}
\item 
\begin{align*}
\begin{aligned}
&\frac{ B_i\left(\alpha^-;\alpha^+,\beta^-,\beta^{-1}(\alpha^-)\right)-1 
}{ e^{\ell_i(\alpha)}-1}
\\=& \frac{B_i(\beta,\beta_{\alpha^-})}{B_i(\beta,\beta_{\alpha^-})+B_i(\gamma,\gamma_{\alpha^-})\cdot e^{\ell_i(\alpha)}} \cdot \frac{1}{1 + e^{r_i(\beta,\beta_{\alpha^-})\cdot \ell_i(\alpha)}\cdot \frac{P_i(\alpha^-;\beta^-,\beta(\alpha^-))}{P_i(\alpha^-;\beta^{-1}(\alpha^-),\beta^-)}}.
\end{aligned}
\end{align*}
\end{enumerate}

\end{lem}
\begin{proof}
By direct compuation, we have
\begin{align}
\begin{aligned}
\\&\frac{ B_i\left(\alpha^-;\alpha^+,\gamma^{-1}(\alpha^-),\gamma^-\right)-1 
}{ e^{\ell_i(\alpha)}-1}
\\&=\frac{ \frac{B_i\left(\alpha^-;\alpha^+,\gamma^{-1}(\alpha^-),\gamma(\alpha^-)\right)- B_i\left(\alpha^-;\gamma^-,\gamma^{-1}(\alpha^-),\gamma(\alpha^-)\right)}{1 - B_i\left(\alpha^-;\gamma^-,\gamma^{-1}(\alpha^-),\gamma(\alpha^-)\right)}-1}{e^{\ell_i(\alpha)}-1}      
\quad\text{ by Lemma~\ref{lem:equb}}
\\&= \frac{ B_i\left(\alpha^-;\alpha^+,\gamma^{-1}(\alpha^-),\gamma(\alpha^-)\right)    
-1}{ e^{\ell_i(\alpha)}-1} \cdot \frac{1}{1 - B_i\left(\alpha^-;\gamma^-,\gamma^{-1}(\alpha^-),\gamma(\alpha^-)\right)} 
\\&=  
\frac{B_i(\gamma,\gamma_{\alpha^-})}{1 +\frac{P_i(\alpha^-;\gamma^-,\gamma(\alpha^-))}{P_i(\alpha^-; \gamma^{-1}(\alpha^-),\gamma^-)}}     
\quad\text{by Lemma~\ref{lem:equl} }.
\end{aligned}
\end{align}
Similarly for the other cases.
\end{proof}

A direct consequence of Lemma \ref{lemma:B11X} is the following.
\begin{cor}
\label{cor:lim2}
Consider a path $l$ as in Definition \ref{definition:hypara}. Suppose $\delta \in \{\beta,\beta^{-1},\gamma,\gamma^{-1}\}$. Let $x$ be a lift of the cusp $p$ such that $(x, \delta x,\delta^+)$ is a lift of the ideal triangle. For the four cases in Lemma~\ref{lemma:B11X}, we have
\begin{align*}
\lim_{\text{hyp$\to$para}}
\frac{\left|\log B_i\left(\alpha^-;\alpha^+,\delta(\alpha^-),\delta^+\right) \right|
}{\ell_i(\alpha)}
= \frac{B_i(\delta,\delta_p) }{1 +\frac{P_i(x;\delta^+,\delta^{-1}x)}{P_i(x;\delta x,\delta^+)}}.
\end{align*}
\end{cor}

\begin{proof}[Proof of Theorem \ref{theorem:puncturecase}]
Let us study the formulas in Theorem \ref{theorem:boundaryith} and Theorem \ref{thm:equgeo} under the sequence of elements in $l([0,1))$ that converges to $l(1)$ as in Definition \ref{definition:hypara}. Lemma \ref{lem:Sa} shows us that there are no $\xvec{\mathcal{P}}^{\partial}_{\alpha}\cong \xvec{\mathcal{H}}^{\partial}_{\alpha}$ summands. Formula \eqref{equation:unipps} is a consequence of Lemma \ref{lem:lim1} and Formula \eqref{equation:unihps} is deduced by Corollary \ref{cor:lim2}. 
\end{proof}

\begin{rmk}
\label{remark:xninq}
Theorem \ref{theorem:puncturecase} can be extended to the general $(\rho,\xi)\in \mathcal{X}_n(S_{g,m})$ where the distinguished oriented boundary component $\alpha$ has $i$-th length zero. Following the above proof, the only thing that we need to modify is Lemma \ref{lem:Sa}. For any $(\gamma,\gamma_p)\in \xvec{\mathcal{H}}_p^\partial\cong \xvec{\mathcal{P}}_p^\partial$, using
\[B_i\left(\alpha^-;\alpha^+,\gamma^-,\gamma^+\right)=B_i\left(\alpha^-;\alpha^+,\gamma^-,\gamma(\alpha^-)\right)-B_i\left(\alpha^-;\alpha^+,\gamma^+,\gamma(\alpha^-)\right),\]
we have the corresponding gap function:
\begin{align*}
\lim_{\text{hyp$\to$para}} \frac{ \log B_i\left(\alpha^-;\alpha^+,\gamma^-,\gamma^+\right)
}{ \ell_i(\alpha)}=B_p(\gamma,\gamma_p)\cdot \frac{e^{\kappa_i(\gamma,\gamma_p)+\ell_i(\gamma)}-1}{e^{\kappa_i(\gamma,\gamma_p)+\ell_i(\gamma)}+1}.
\end{align*}

Similarly, Theorems \ref{theorem:boundaryith} and \ref{thm:equgeo} can be extended to the general $(\rho,\xi)\in \mathcal{X}_n(S_{g,m})$, but the equalities will be changed into inequalities. 

\end{rmk}

\subsection{A strategy for establishing the unipotent-bordered McShane identity}

\begin{lem}
\label{lemma:B1}
Let $\rho$ be a $\operatorname{PGL}_n(\mathbb{R})$-positive representation with only unipotent boundary monodromy (only loxodromic boundary monodromy resp.). Let $x$ be a cusp $p$ (a boundary component $\alpha$ resp.). Let $\xvec{\mathcal{H}}_x(\gamma)$ be the subset of $\xvec{\mathcal{H}}_x$ containing $\gamma$ as the cuff. Then  
\begin{align}
\label{eq:prob}
\sum_{\mu\in\subvec{\mathcal{H}}_x(\gamma)}
B_i(\mu)\leq 1.
\end{align} 
\end{lem}
\begin{proof}
Let us prove for the case $\rho \in \mathrm{Pos}_n^u(S_{g,m})$ and $x$ is a cusp $p$. The argument for $\rho \in \mathrm{Pos}_n^h(S_{g,m})$ is the same since the following is a topological argument.
For each $\mu\in\xvec{\mathcal{H}}_p(\gamma)$, $B_i(\mu)$ is the probability (with respect to the Goncharov--Shen potential measure) that the portion of a geodesic launched from cusp $p$ up to its first point of self-intersection will either:
\begin{itemize}
\item intersect $\gamma$, or
\item be completely contained on a pair of half-pants with $\gamma$ as its cuff.
\end{itemize}
Suppose two $\mu,\mu'\in\xvec{\mathcal{H}}_p(\gamma)$ has non-empty interior intersection. We argue the same way as \cite[Proposition 1]{mcshane_allcusps}. If $\mu$ and $\mu'$ contain a common geodesic launched from cusp $p$ up to $\gamma$, there is a unique pair of half-pants in $\xvec{\mathcal{H}}_p(\gamma)$ containing $\gamma$ and that geodesic. If $\mu$ and $\mu'$ contain a common geodesic launched from cusp $p$ up to its first point of self-intersection completely contained on a pair of half-pants with $\gamma$ as its cuff, there is a unique pair of half-pants in $\xvec{\mathcal{H}}_p(\gamma)$ containing $\gamma$ and that geodesic. We have $\mu=\mu'$ for both of the above two cases. We conclude Formula \eqref{eq:prob}.

\end{proof}

\begin{thm}[Equality under assumption]
\label{theorem:equuni}
Let $\xvec{\mathcal{C}}_{g,m}$ is the set of free homotopy classes of oriented simple closed curves on $S_{g,m}$. For any $\operatorname{PGL}_n(\mathbb{R})$-positive representation $\rho$, let 
\[
D_i(N,\rho) := \#\left\{[\delta] \in \xvec{\mathcal{C}}_{g,m}\;
\begin{array}{|l}
\; \log \frac{\lambda_i(\rho(\delta))}{\lambda_{i+1}(\rho(\delta))} \leq N
\end{array}
\right\}.
\]

Assume that we have the following property:

\emph{Consider a path $l$ in Definition \ref{definition:hypara} where $l(1)=\rho\in \mathrm{Pos}_n^u(S_{g,m})$ has only unipotent boundary monodromy, there exists a positive continuous function for $s\in [0,1]$ such that
\begin{align}
\label{equation:assumption}
D_i(N,l(s)) \leq  c(l(s))\cdot N^{6g-6+2m}.
\end{align}  }

Then the inequalities in theorem \ref{theorem:puncturecase} are equalities.
\end{thm}

\begin{proof}
It is enough to prove only Formula \eqref{equation:unihps} is an equality. 

Take a path $l$ in Definition \ref{definition:hypara}.
For $s\in [0,1]$, let 
\begin{equation}
\begin{aligned}
\label{equation:fs}
&F(s):=\sum_{\delta\in \subvec{\mathcal{C}}_{g,m}}\sum_{(\delta,\delta_{\alpha^-})\in \subvec{\mathcal{H}}_\alpha(\delta)} \frac{\left|\log B_i^{l(s)}\left(\alpha^-;\alpha^+,\delta(\alpha^-),\delta^+\right) \right|
}{\ell_i^{l(s)}(\alpha)} +
\\&\sum_{\gamma}\sum_{(\gamma,\gamma_{\alpha^-})\in \subvec{\mathcal{H}}_\alpha(\gamma)}\frac{\log B_i^{l(s)}\left(\alpha^-;\alpha^+,\gamma^-,\gamma^+\right)
}{\ell_i^{l(s)}(\alpha)},
\end{aligned}
\end{equation} 
where in the first summation $\delta$ is enumerated with respect to $\ell_i(\delta)$, and in the second summation $\gamma$ goes over the finitely many boundary components of $S_{g,m}$ (oriented such that $S_{g,m}$ is on the left side of $\gamma$) except $\alpha$.

By Corollary \ref{cor:lim2} and Lemma \ref{lem:Sa}, each summand of $F(s)$ is a continuous function on $s\in [0,1]$. By Theorem \ref{thm:boundary}, $F(s)=1$ for $s\in [0,1)$ and $F(1)\leq 1$ by Theorem \ref{theorem:puncturecase}. To prove $F(1)= 1$, we will show that $F(s)$ is uniformly convergent on $s\in[0,1]$.

Since the path $l$ is compact, we have the following bounds in the path $l$:
\begin{enumerate}
\item the limit of $\frac{e^{\ell_i(\alpha)}-1}{\ell_i(\alpha)}$ under $\lim_{\text{hyp$\to$para}}$ is $1$, so $\frac{e^{\ell_i(\alpha)}-1}{\ell_i(\alpha)}$ is bounded above by a constant $C_0>0$;
\item following the compactness argument in Theorem \ref{thm:bounded}, the collection of all the triple ratios for all the ideal triangles is uniformly bounded within a compact positive interval, thus the collection $\{K(\delta,\delta_{\alpha^-})\}_{(\delta,\delta_{\alpha^-})\in \subvec{\mathcal{H}}_\alpha}$ (Equation \eqref{equation:kappa}) is bounded below by a constant $K>0$.
\end{enumerate}
For the first summation of Equation \eqref{equation:fs}:
\begin{align*}
\begin{aligned}
&\sum_{(\delta,\delta_{\alpha^-})\in \subvec{\mathcal{H}}_\alpha}\frac{\left|\log B_i^{l(s)}\left(\alpha^-;\alpha^+,\delta(\alpha^-),\delta^+\right) \right|
}{\ell_i^{l(s)}(\alpha)} 
\\ 
\leq&\frac{e^{\ell_i(\alpha)}-1}{\ell_i(\alpha)}\cdot \sum_{(\delta,\delta_{\alpha^-})\in \subvec{\mathcal{H}}_\alpha}
\left(\frac{B_i(\delta,\delta_{\alpha^-})}{1+ K_i(\delta,\delta_{\alpha^-})\cdot  e^{\ell_i(\delta)}}
 \right)
 \quad\text{Lemma~\ref{lemma:B11X}}
\\ \leq &C_0\cdot \sum_{\delta\in \subvec{\mathcal{C}}_{g,m}}\sum_{(\delta,\delta_{\alpha^-})\in \subvec{\mathcal{H}}_\alpha(\delta)}
\left(\frac{B_i(\delta,\delta_{\alpha^-})}{1+ K_i(\delta,\delta_{\alpha^-})\cdot   e^{\ell_i(\delta)}}
 \right)
 \\ \leq &C_0\cdot \sum_{\delta\in \subvec{\mathcal{C}}_{g,m}}\sum_{(\delta,\delta_{\alpha^-})\in \subvec{\mathcal{H}}_\alpha(\delta)}
\left(\frac{B_i(\delta,\delta_{\alpha^-})}{1+ K\cdot  e^{\ell_i(\delta)}}
 \right)        
  \\ \leq &C_0\cdot \sum_{\delta\in \subvec{\mathcal{C}}_{g,m}} \frac{1}{ K\cdot  e^{\ell_i(\delta)}}       \quad\text{Lemma~\ref{lemma:B1}}
    \\ \leq &C_0\cdot \sum_{t=1}^{+\infty} \frac{ D_i(t,l(s))}{ K\cdot  e^t}.
\end{aligned}
\end{align*} 
By our assumption Equation \eqref{equation:assumption}, we have $D_i(t,l(s)) \leq  c(l(s))\cdot t^{6g-6+2m}$ for $s\in [0,1]$. Since $[0,1]$ is compact, there is $Q>0$ such that $c(l(s))\leq Q$. Thus the rest sum of the above summation is bounded above by 
\[\frac{C_0 Q}{K}\cdot \sum_{t=r+1}^{+\infty} \frac{ t^{6g-6+2m}}{  e^t}\]
which converges to zero when $r$ goes to infinity.

The second summation of Equation \eqref{equation:fs} is finite over $\gamma$ and
\[\sum_{(\gamma,\gamma_{\alpha^-})\in \subvec{\mathcal{H}}_\alpha(\gamma)}\frac{\log B_i^{l(s)}\left(\alpha^-;\alpha^+,\gamma^-,\gamma^+\right)
}{\ell_i^{l(s)}(\alpha)}\leq 1.\]
Hence we have $F(s)$ uniformly convergent on $s\in [0,1]$. Thus $F(1)=1$.

\end{proof}

\newpage

\section{Applications}
\label{section:app}
\subsection{Simple spectral discreteness}\label{sec:spectraldiscreteness}

\begin{defn}[Simple spectra]
Recall $S=S_{g,m}$ where $2g-2+m>0$.
Let $\rho\in \mathrm{Pos}_n^u(S)$ be a $\operatorname{PGL}_n(\mathbb{R})$-positive representation with unipotent boundary monodromy, and let $\xvec{\mathcal{C}}_{g,m}$ denote the collection of oriented simple closed geodesics on $S$ up to homotopy. We define the following spectra:
\begin{enumerate}
\item
the \emph{simple $\ell_i$-spectrum}:
\[
\left\{\ell_i^\rho(\gamma)\mid \gamma\in\xvec{\mathcal{C}}_{g,m}\right\},
\]
\item
the \emph{simple largest-eigenvalue spectrum}:
\[
\left\{\lambda_1(\rho(\gamma))\mid \gamma\in\xvec{\mathcal{C}}_{g,m}\right\},
\]
\item
and the \emph{simple $\ell:=\sum_{i=1}^{n-1} \ell_i$-spectrum}:
\[
\left\{\ell^\rho(\gamma)\mid \gamma\in\xvec{\mathcal{C}}_{g,m}\right\}.
\]
\end{enumerate}
\end{defn}

Our goal in this subsection is to prove that the above simple spectra are discrete for any positive representation $\rho \in \mathrm{Pos}_n^u(S_{g,m})$. Our proof relies on the Theorem~\ref{theorem:puncturecase}. Let $\xvec{\mathcal{H}}_p(\gamma)$ denote the subset of $\xvec{\mathcal{H}}_p$ consisting of all boundary-parallel half-pants with $\gamma$ as its oriented cuff.

\begin{lem}\label{thm:halfpantslbound}
Given a $\operatorname{PGL}_n(\mathbb{R})$-positive representation $\rho\in \mathrm{Pos}_n^u(S)$ with unipotent boundary monodromy equipped with an auxiliary cusped complete hyperbolic surface $\Sigma=(S,h_\rho)$ as in Definition \ref{definition:bin}, let $p$ be a distinguished puncture of $S$. There is a universal constant $b^\rho>0$ such that for every oriented simple closed curve $\gamma\in\xvec{\mathcal{C}}_{g,m}$, there exists an embedded pair of half-pants $(\gamma,\gamma_p)\in\xvec{\mathcal{H}}_\alpha(\gamma)$ such that:
\begin{align*}
B_i(\gamma,\gamma_p) \geq b^\rho.
\end{align*}
\end{lem}

\begin{proof}
Much like the proof of the boundedness of triple ratios (Theorem~\ref{thm:triplebounded}), we rely on a compactness argument. With respect to the hyperbolic metric $\Sigma$, the length $1$ horocycle $\eta_p$ around cusp $p$ separates $\Sigma$ into two connected components: an (open) annular cuspidal neighborhood $C_p\subset \Sigma$ as well as a (closed) homotopy retract $\Sigma^{(p)}:=\Sigma-C_p$. Also let $\Sigma^{\geq1}\subset \Sigma^{(p)}\subset\Sigma$ denote the compact subsurface of $\Sigma$ obtained from truncating every cusp of $\Sigma$ at its length $1$ horocycle.\medskip

Consider the following subset of the unit tangent bundle $T^1\Sigma$:
\begin{align*}
\Xi:=
\left\{
(x,v)\in T^1\Sigma
\;
\begin{array}{|l}
x\text{ is a point lying in }\Sigma^{\geq1}\text{ and the geodesic ray $\sigma_{(x,v)}$}\\
\text{shooting out from $x$ with initial vector $v$ is simple,}\\
\text{approaches the cusp $p$, and the arc $\sigma_{(x,v)}\cap\Sigma^{\geq1}$}\\
\text{realizes the distance between $x$ and $\eta_p$} 
\end{array}
\right\}.
\end{align*}
Let $d\Sigma^{\geq1}$ be the closed surface obtained by doubling $\Sigma^{\geq1}$. Then $T^1{\Sigma^{\geq1}}$ is a closed subset of the compact space $T^1 d\Sigma^{\geq1}$.
We now show that $\Xi$ is a closed subset of the restricted unit tangent bundle $T^1{\Sigma^{\geq1}}$ of $T^1\Sigma$ to $\Sigma^{\geq1}$, and is hence compact. Consider a sequence $\{(x_n,v_n)\in\Xi\}$ which converges to a point $(x_\infty,v_\infty)$. Since $\Sigma^{\geq1}$ is a closed subset of $\Sigma$, the limiting base point $x_\infty$ must lie on $\Sigma^{\geq1}$. Next, to show that the geodesic ray $\sigma_{(x_\infty,v_\infty)}$ approaches the cusp $p$, choose a fundamental domain $F$ for $\Sigma$ containing a lift $\tilde{x}_\infty$ of $x_\infty$ in the interior. The lifts to $F$ of the sequence $\{(x_k,v_k)\}$ for $k$ large enough necessarily all induce rays which shoot into the same lift $\tilde{p}$ of the cusp $p$, and hence $\sigma_{(x_\infty,v_\infty)}$ also shoots into $p$. Hence $\Xi$ is a compact set.\medskip

Since $\sigma_{(x,v)}$ shoots into cusp $p$, the corresponding subset to $\Xi$ in $\mathit{Tri}(\Sigma)$ is a compact subset with every point of the form $[\tilde{p},b,c]_\rho$, where $\tilde{p}$ is a lift of $p$. In particular, this means that the (strictly) positive function $P_i(\tilde{p};b,c)/P_i^p$ in Definition~\ref{defn:ratio} is well-defined and continuous on a compact set and achieves its minimum. We denote this minimum by $b^\rho>0$.\medskip

Given an arbitrary oriented (essential) simple closed curve $\gamma\in\xvec{\mathcal{C}}_{g,m}$, let $\gamma$ denote its geodesic realization on $\Sigma$. Further let $x_0\in\gamma$ be the point on $\gamma$ closest to the horocycle $\eta_p$, let $\sigma$ be one of the geodesic arcs realizing the distance between $x_0$ and $\eta_p$, and let $v_0$ denote the initial vector of $\sigma$. By construction, the geodesic ray $\sigma_{(x_0,v_0)}$ contains $\sigma$. Since $\sigma$ is a distance minimizing arc, it must meet $\eta_p$ perpendicularly and hence $\sigma_{(x_0.v_0)}$ shoot up straight into cusp $p$ after passing $\eta_p$. Moreover, the arc $\sigma$ must also be simple (so as to be distance minimizing), and hence $\sigma_{(x_0,v_0)}$ is the concatenation of $\sigma$ and a simple geodesic ray which lies in $C_p$ (and hence cannot intersect $\sigma$) and is thus simple. Therefore, we see that $(x_0,v_0)\in\Xi$. We denote its corresponding point in $\mathit{Tri}(\Sigma)$ by $[\tilde{p},b_0,c_0]_\rho$. Let $(\bar{\gamma},\bar{\gamma}_p)\in\mathcal{H}_p$ denote the unique embedded pair of half-pants on $S$ containing $\bar{\gamma}\cup\sigma_{(x_0,v_0)}$ (Figure~\ref{fig:ratiobound}).\medskip

\begin{figure}[h]
\includegraphics[scale=1]{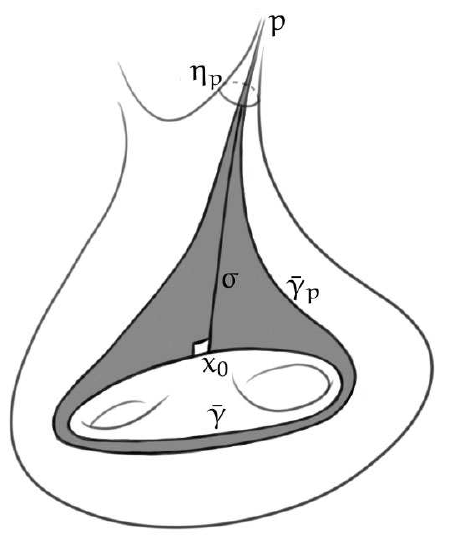}
\caption{The pair of half-pants $(\bar{\gamma},\bar{\gamma}_p)$ is the unique embedded pair of half-pants that contains $\bar{\gamma}$ and $\sigma(x_0,v_0)\supset\sigma$.}
\label{fig:ratiobound}
\end{figure}

We also know that $\sigma$ is perpendicular to $\gamma$, and by possibly replacing $\tilde{p}$ with a different lift of $p$, the point $b_0$ must be of one of the two fixed points of $\rho(\gamma)$. This in turn means that $B_i(\gamma,\gamma_p)\geq P_i(\tilde{p};b,c)/P_i^p\geq b^\rho>0$, thereby demonstrating the desired lower bound.
\end{proof}

\begin{thm}
\label{thm:discretespectra}
Let $\rho\in \mathrm{Pos}_n^u(S)$ be a $\operatorname{PGL}_n(\mathbb{R})$-positive representation with unipotent boundary monodromy. Then the simple $\ell_i$-spectrum for $1\leq i\leq n-1$, the simple largest eigenvalue spectrum and the simple $\ell$-spectrum for $\rho$ are all discrete.
\end{thm}

\begin{proof}
We begin by rearranging the inequality Theorem \ref{theorem:puncturecase} to obtain the following expression for $i=1,\cdots,n-1$:

\begin{align*}
\sum_{\gamma\in\subvec{\mathcal{C}}_{g,m}}
\sum_{(\gamma,\gamma_p)\in\subvec{\mathcal{H}}_p(\gamma)}
\frac{B_i(\gamma,\gamma_p)}
{1+e^{\ell_i(\gamma)+\tau(\gamma,\gamma_p)}}
\leq 1.
\end{align*}
Invoking Theorem~\ref{thm:bounded} to assert that there exists some $\tau_{\min}$ such that $\tau(\gamma,\gamma_p)\leq \tau_{\max}$, we obtain:
\begin{align*}
\sum_{\gamma\in\subvec{\mathcal{C}}_{g,m}}
\left(
\frac{1}{1+e^{\ell_i(\gamma)+\tau_{\max}}}
\sum_{(\gamma,\gamma_p)\in\subvec{\mathcal{H}}_p(\gamma)}
B_i(\gamma,\gamma_p)
\right)
\leq 1.
\end{align*}
Further invoking Lemma~\ref{thm:halfpantslbound} to uniformly bound 
\begin{align*}
\sum_{(\gamma,\gamma_p)\in\subvec{\mathcal{H}}_p(\gamma)}
B_i(\gamma,\gamma_p)
\geq \sup_{(\gamma,\gamma_p)\in\subvec{\mathcal{H}}_p(\gamma)}
B_i(\gamma,\gamma_p)
\geq b^\rho.
\end{align*}
Hence:
\begin{align*}
\sum_{\gamma\in\subvec{\mathcal{C}}_{g,m}}
\frac{b^\rho}{1+e^{\ell_i(\gamma)+\tau_{\max}}}
\leq & 1.
\end{align*}
This suffices to ensure the discreteness of the simple $\ell_i$-spectrum.

Then
\[
\ell=\sum_{i=1}^{n-1} \ell_i
\] 
ensures that the $\ell$-spectrum is also discrete.
Furthermore, the fact that
\[
e^{\ell(\gamma)}>\lambda_1(\rho(\gamma))
\]
then ensures that the simple largest-eigenvalue spectrum is also discrete.
\end{proof}

\subsection{The collar lemma}
\label{sec:collarlemma}

As a second application of our McShane identity, we establish the collar lemma for $\rho \in \mathrm{Pos}_3(S_{g,m})$ which corresponds to a certain convex $\mathbb{RP}^2$ structure. We require our McShane identities for $\mathrm{Pos}_3(S_{1,1})$. Note also that we do not need the full force of the McShane identity, and only require the inequality.

\begin{lem}
\label{lemma;collarsimple}
Let $\rho\in \mathrm{Pos}_3^u(S_{1,1})$ be a $\operatorname{PGL}_3(\mathbb{R})$-positive representation with unipotent boundary monodromy (or equivalently cusped strictly convex $\mathbb{RP}^2$ structure on $S_{1,1}$).
We define 
\[T(\beta):=T(\tilde{p},\beta \tilde{p},\beta^+).\]
For distinct (oriented) simple closed geodesics  $\beta,\gamma\in\xvec{\mathcal{C}}_{1,1}$, let
\begin{align*}
u_1 &= T(\beta) \tfrac{\lambda_1(\rho(\beta))}{\lambda_2(\rho(\beta))},\;
u_2 = T(\beta^{-1}) \tfrac{\lambda_1(\rho(\beta^{-1}))}{\lambda_2(\rho(\beta^{-1}))},\\
u_3 &= T(\gamma) \tfrac{\lambda_1(\rho(\gamma))}{\lambda_2(\rho(\gamma))},\;
u_4 = T(\gamma^{-1}) \tfrac{\lambda_1(\rho(\gamma^{-1}))}{\lambda_2(\rho(\gamma^{-1}))}.
\end{align*}
Then, for any configuration of $\{i,j,k,l\}= \{1,2,3,4\}$, we have:
\begin{align*}
\left(\left(u_i u_j\right)^{\frac{1}{2}}-1 \right) \cdot  \left(\left(u_k u_l\right)^{\frac{1}{2}}-1 \right)>4.
\end{align*}
\end{lem}

\begin{proof}
By Theorem~\ref{theorem:inequsl3s11}, we have:
\begin{align*}
\sum_{s=1}^4\frac{1}{1+u_s}< 
\sum_{\delta\in{\subvec{\mathcal{C}}_{1,1}}}
\frac{1}{1+T(\delta)\frac{\lambda_1(\rho(\delta))}{\lambda_2(\rho(\delta))}}
\leq 1.
\end{align*}
Multiplying both sides by $\prod_{s=1}^4 (1+u_s)$ and rearranging the resulting terms, we obtain:
\begin{align*}
3+ 2 \sum_{i=s}^4 u_s + \sum_{s<t} u_s u_t< \prod_{s=1}^4 u_s.
\end{align*}
Further adding $(1-u_i u_j- u_k u_l)$ to both sides, we get:
\begin{align*}
(2+u_i+u_j)(2+u_k+u_l)  < (1-u_i u_j)(1- u_k u_l).
\end{align*}
By the algebraic mean-geometric mean inequality, we obtain:
\begin{align*}
(2+2 (u_i u_j)^{\frac{1}{2}})(2+2(u_k u_l)^{\frac{1}{2}})  < (1-u_i u_j)(1- u_k u_l),
\end{align*}
and hence:
\begin{align*}
\left(\left(u_i u_j\right)^{\frac{1}{2}}-1 \right) \cdot  \left(\left(u_k u_l\right)^{\frac{1}{2}}-1 \right)>4.
\end{align*}
\end{proof}

\begin{prop}\label{thm:collarminint}
Let $\rho\in \mathrm{Pos}_3(S_{1,1})$ be a $\operatorname{PGL}_3(\mathbb{R})$-positive representation. The twice of Hilbert lengths $\ell$ of any two distinct simple closed geodesics $\beta$ and $\gamma$ satisfy the following inequality:
\begin{align}
\label{eq:collars10}
(e^{\frac{1}{2}\ell(\beta)}-1)(e^{\frac{1}{2}\ell(\gamma)}-1)
>4.
\end{align}
\end{prop}

\begin{proof}
We first consider the unipotent case. Recall from Equation~\eqref{equation:triplesym} that
\begin{align*}
T(\tilde{p},\delta \tilde{p}, \delta^+) \cdot T(\tilde{p},\delta^{-1} \tilde{p}, \delta^-) =1,
\end{align*}
where $\tilde{p}$ is a lift of the puncture $p$ such that $(\tilde{p},\delta \tilde{p}, \delta^+)$ is a lift of an ideal triangle. This means that the product terms $u_1 u_2$ and $u_3u_4$ satisfy
\[
u_1 u_2=\tfrac{\lambda_1(\rho(\beta))}{\lambda_3(\rho(\beta))}=e^{\ell(\beta)}
\quad\text{ and }\quad
u_3 u_4=\tfrac{\lambda_1(\rho(\gamma))}{\lambda_3(\rho(\gamma))}=e^{\ell(\gamma)},
\] 
and hence we obtain Equation~\eqref{eq:collars10} as desired.\medskip

Similarly for the distinguished oriented boundary component $\alpha$ (such that $S_{1,1}$ is on the left) has $1$st length zero.

We now turn to the case where the distinguished oriented boundary component $\alpha$ has $1$st length non-zero. For any simple closed geodesic $\delta$ on $\Sigma$, let $\mu_1^\delta,\mu_2^\delta\in\xvec{\mathcal{H}}_{\alpha}$ denote two boundary-parallel pairs of half-pants which have $\delta$ as its oriented cuff such that their underlying half-pants are distinct. Recall Definition~\ref{definition:Ri}, we have 
\begin{equation}
\label{equ:R121}
r_1(\mu_1^\delta)+ r_1(\mu_2^{\delta^{-1}})=1,\;\; r_1(\mu_1^\delta)=-r_1(\mu_1^{\delta^{-1}}).
\end{equation}
We consider two gap functions in Theorem~\ref{thm:equgeo} associated to one pair of half-pants:
$\log\left(
\frac{e^{r_1(\mu_i^\delta) \ell_1(\alpha)}+e^{\ell_1(\delta)+\tau(\delta)}}
{1+e^{\ell_1(\delta)+\tau(\delta)}}
\right)$ for $r_1(\mu_i^\delta)$
and 
$\log\left(
\frac{1+e^{\ell_1(\delta^{-1})+\tau(\delta^{-1})}}
{e^{-r_1(\mu_i^\delta) \ell_1(\alpha)}+e^{\ell_1(\delta^{-1})+\tau(\delta^{-1})}}\right)$ for $r_1(\mu_i^{\delta^{-1}})$.

We require the following fact: 
\begin{align*}
XY>1\Rightarrow (1+X^2)^{-1}+(1+Y^2)^{-1}\geq 2(1+XY)^{-1}.
\end{align*}
By taking $X=e^{\frac{-r_1(\mu_i^\delta) \ell_1(\alpha)+\ell_1(\delta)+\tau(\delta)}{2}}$ and $Y=e^{\frac{r_1(\mu_i^\delta) \ell_1(\alpha)+\ell_1(\delta^{-1})+\tau(\delta^{-1})}{2}}$, we obtain:
\begin{align}
\frac{2 r_1(\mu_i^\delta)}
{1+e^{\frac{1}{2}\ell(\delta)}}
\leq \frac{r_1(\mu_i^\delta) e^{r_1(\mu_i^\delta) \ell_1(\alpha)}}
{e^{r_1(\mu_i^\delta) \ell_1(\alpha)}+e^{\ell_1(\delta)+\tau(\delta)}}
+\frac{r_1(\mu_i^\delta) e^{-r_1(\mu_i^\delta) \ell_1(\alpha)}}
{e^{-r_1(\mu_i^\delta) \ell_1(\alpha)}+e^{\ell_1(\delta^{-1})+\tau(\delta^{-1})}}.\label{eq:s11holeprelim}
\end{align}
The above inequality in turn leads to the following comparison: for $\ell_1(\alpha)>0$,
\begin{align}
\frac{2r_1(\mu_i^\delta) \ell_1(\alpha)}
{1+e^{\frac{1}{2}\ell(\delta)}}
\leq \log\left(
\frac{e^{r_1(\mu_i^\delta) \ell_1(\alpha)}+e^{\ell_1(\delta)+\tau(\delta)}}
{1+e^{\ell_1(\delta)+\tau(\delta)}}\right)
+\log\left(
\frac{1+e^{\ell_1(\delta^{-1})+\tau(\delta^{-1})}}
{e^{-r_1(\mu_i^\delta) \ell_1(\alpha)}+e^{\ell_1(\delta^{-1})+\tau(\delta^{-1})}}\right).
\label{eq:s11holeineq}
\end{align}
To see this, note that Equation~\eqref{eq:s11holeineq} is an obvious equality when $\ell_1(\alpha)=0$ and its derivative with respect to $\ell_1(\alpha)$ satisfies Equation~\eqref{eq:s11holeprelim} for $\ell_1(\alpha)>0$. We see that:
\begin{equation}
\label{eq:hyptermcompare}
\begin{aligned}
\frac{2 r_1(\mu_i^\delta)}
{1+e^{\frac{1}{2}\ell(\delta)}}
\leq&
\frac{1}{\ell_1(\alpha)}
\log\left(
\frac{e^{r_1(\mu_i^\delta) \ell_1(\alpha)}+e^{\ell_1(\delta)+\tau(\delta)}}
{1+e^{\ell_1(\delta)+\tau(\delta)}}
\right)\\
&+
\frac{1}{\ell_1(\alpha)}
\log\left(
\frac{1+e^{\ell_1(\delta^{-1})+\tau(\delta^{-1})}}
{e^{-r_1(\mu_i^\delta) \ell_1(\alpha)}+e^{\ell_1(\delta^{-1})+\tau(\delta^{-1})}}\right).
\end{aligned}
\end{equation}
There is one inequality of the same form as Equation~\eqref{eq:hyptermcompare} for each choice of $\delta=\beta,\gamma$ and $i=1$ or $2$. This makes a total of four such inequalities, and hence eight right-hand side terms. Crucially, these eight terms are distinct summands of the McShane identity for $\rho \in \mathrm{Pos}_3(S_{1,1})$ by Theorem~\ref{thm:equgeo}, and hence by Remark \ref{remark:xninq}:
\begin{align*}
\frac{2r_1(\mu_1^\beta)}{1+e^{\frac{1}{2}\ell(\beta)}}
+\frac{2r_1(\mu_2^\beta)}{1+e^{\frac{1}{2}\ell(\beta)}}
+\frac{2r_1(\mu_1^\gamma)}{1+e^{\frac{1}{2}\ell(\gamma)}}
+\frac{2 r_1(\mu_2^\gamma)}{1+e^{\frac{1}{2}\ell(\gamma)}}
< 1.
\end{align*}
By Equation~\eqref{equ:R121}, we then obtain
\begin{align*}
\frac{2}{1+e^{\frac{1}{2}\ell(\beta)}}
+\frac{2}{1+e^{\frac{1}{2}\ell(\gamma)}}
< 1,
\end{align*}
which rearranges to give Equation~\eqref{eq:collars10} as desired.
\end{proof}

\begin{thm}[Collar lemma]
\label{thm:collar}
Let $\rho\in \mathrm{Pos}_3(S_{g,m})$ be a $\operatorname{PGL}_3(\mathbb{R})$-positive representation. Any two intersecting simple closed geodesics $\beta,\gamma$ satisfy the following inequality:
\begin{align}
(e^{\frac{1}{2}\ell(\beta)}-1)(e^{\frac{1}{2}\ell(\gamma)}-1)>4.\label{eq:collars11}
\end{align}
\end{thm}

\begin{proof}
We first note that Proposition~\ref{thm:collarminint}, coupled with the fact that the twice of Hilbert length $\ell(\delta)$ (which we call {\em length} for short in this proof) of a curve $\delta$ is equal to 
\[
\ell(\delta)=\ell_1(\delta)+\ell_2(\delta), 
\]
tells us that Equation~\eqref{eq:collars11} is true if the convex hull of $\beta\cup\gamma$ is a $1$-holed torus. Furthermore, whenever the convex hull of $\beta\cup\gamma$ is a $4$-holed sphere $\Sigma_{0,4}$, then $\Sigma_{0,4}$ is the quotient of a $4$-holed torus $\Sigma_{1,4}$ with respect to the action of an isometric involution (see Figure~\ref{fig:doublecover}):

\begin{figure}[h!]
\includegraphics[scale=1]{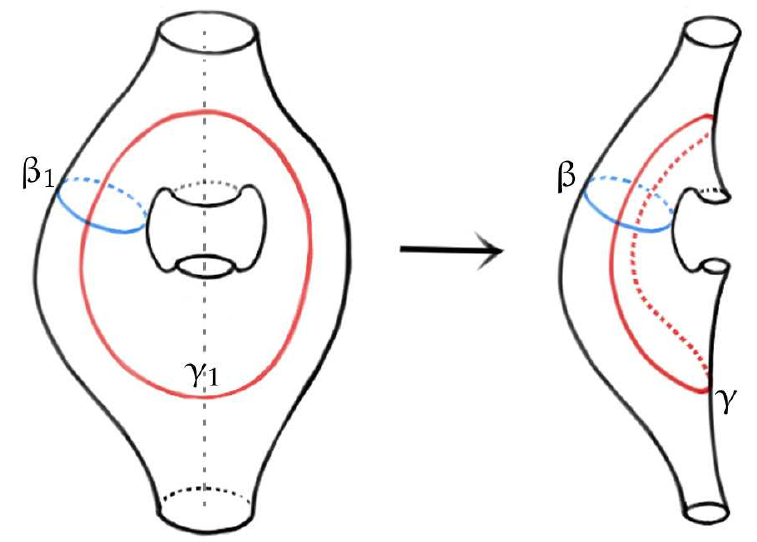}
\caption{The left $4$-holed torus double covers the right $4$-holed sphere, with identification given by $\pi$-rotation about the central vertical axis. The curves $\beta_1,\gamma_1$ respectively cover $\beta,\gamma$ precisely once and the convex hull of $\beta_1\cup\gamma_1$ is a $1$-holed torus.}
\label{fig:doublecover}
\end{figure}

The curve $\beta$ lifts to two simple connected geodesics $\beta_1,\beta_2$ in $\Sigma_{1,4}$, each of length equal to length of $\beta$. Likewise, the curve $\gamma$ also lifts to $\gamma_1$ and $\gamma_2$. The convex hull of $\beta_1\cup\gamma_1$ is a $1$-holed torus, and hence we once again obtain Equation~\eqref{eq:collars11}.\medskip

The above cases cover all possibilities where there are two or fewer (geometric) intersection points between $\beta$ and $\gamma$. We now turn to the case when there are at least three intersections. Let us assume without loss of generality that $\beta$ is shorter than or equal to $\gamma$. We also assume that the intersection points $\beta\cap\gamma$ are generic, our arguments still apply when there are non-generic intersection points with the small caveat that some of the geodesic segments we concatenate may be of length zero.\medskip

Consider the geodesic subarcs $\{\sigma\}$ on $\gamma$ with ends in $\beta\cap\gamma$, but not interior points. Note that this collection of subarcs may be bipartitioned into those whose endpoint tangent directions point to the same side of $\beta$ (left hand side of Figure~\ref{fig:abarcs}) and those whose endpoint directions point to opposite sides (right hand side of Figure~\ref{fig:abarcs}). We refer to the former as a \emph{type-A arc} and the latter as a \emph{type-B arc}.

\begin{figure}[h]
\includegraphics[scale=1]{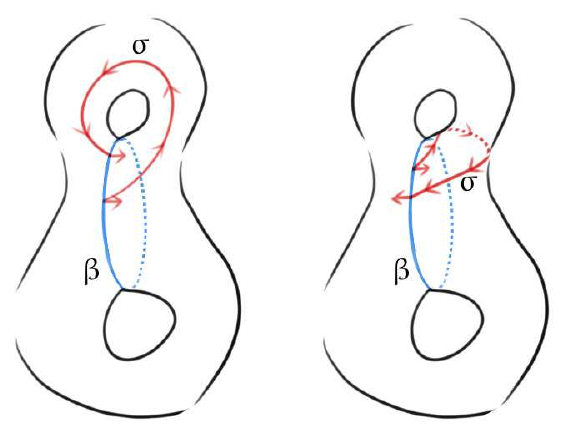}
\caption{A type-A arc (left) versus a type-B arc (right).}
\label{fig:abarcs}
\end{figure}

\textbf{Case 1: $\exists$ type-A arc ${\sigma}$ on ${\gamma}$ of length $\ell(\sigma)\leq\frac{1}{2}\ell(\gamma)$.} Join the two ends of ${\sigma}$ with the shorter of the two subarcs of ${\beta}$ traversing between the endpoints of ${\sigma}$. The resulting concatenated broken geodesic shortens to a unique simple closed geodesic ${\gamma}'$ which intersects ${\beta}$ precisely once. The length of ${\gamma}'$ satisfies:
\[
\ell(\gamma')\leq\tfrac{1}{2}(\ell(\beta)+\ell(\gamma))\leq\ell(\gamma),
\]
and the convex hull of ${\beta}\cup{\gamma}'$ is a $1$-holed torus. Therefore:
\begin{align}
(e^{\frac{1}{2}\ell(\beta)}-1)(e^{\frac{1}{2}\ell(\gamma)}-1)\geq
(e^{\frac{1}{2}\ell(\beta)}-1)(e^{\frac{1}{2}\ell(\gamma')}-1)>4,\label{eq:comparison}
\end{align}
as desired.\medskip

\textbf{Case 2: no type-A arcs on ${\gamma}$.} Let $N$ denote the number of intersection points in ${\beta}\cap{\gamma}$ (non-generic intersection points are counted with multiplicity). The no type-A arcs condition forces $N$ to be even. Hence, there are $N\geq4$ type-B arcs ${\sigma}_1,\ldots,{\sigma}_N$ which concatenate to form ${\gamma}$. Consider the $N$ geodesic arcs of the form ${\sigma}_i*{\sigma}_{i+1}$ (and ${\sigma}_N*{\sigma}_1$) obtained from concatenating consecutive type-B arcs. The total sum of the lengths of these concatenated arcs is $2\ell(\gamma)$, and the pigeonhole principle tells us that at least one has length shorter than $\frac{2\ell(\gamma)}{N}\leq\frac{\ell(\gamma)}{2}$.\medskip

Let ${\sigma}$ denote one such $\frac{\ell(\gamma)}{2}$-short concatenated arc and consider the closed broken geodesic formed by joining the endpoints of ${\sigma}$ with the shorter of the two arcs on ${\beta}$ adjoining the endpoints of ${\sigma}$, and denote its geodesic representative by ${\gamma}'$. The curve ${\gamma}'$ is either simple or may have one self-intersection. In the former case, we have two simple closed geodesics $\beta$ and $\gamma'$ with geometric intersection number equal to $2$ but algebraic intersection number equal to $0$. Hence $\beta\cup\gamma'$ lies on a $4$-holed sphere, and we once again obtain Equation~\eqref{eq:comparison}. In the latter case, the convex hull of ${\gamma}'$ is a pair of pants. and precisely one of the two ways of resolving the intersection point on ${\gamma}'$ produces an essential simple closed geodesic ${\gamma}''$ (see Figure~\ref{fig:collarcase2}). In particular, since Hilbert metric is a distance metric, the triangle inequality ensures that resolving crossings results in shorter rectifiable curves with even shorter geodesic representatives. Thus, we replace ${\gamma}'$ with ${\gamma}''$, and wind up with the former case.\medskip

In either of the two cases as in Figure~\ref{fig:collarcase2}, 
\begin{figure}[h]
\includegraphics[scale=1]{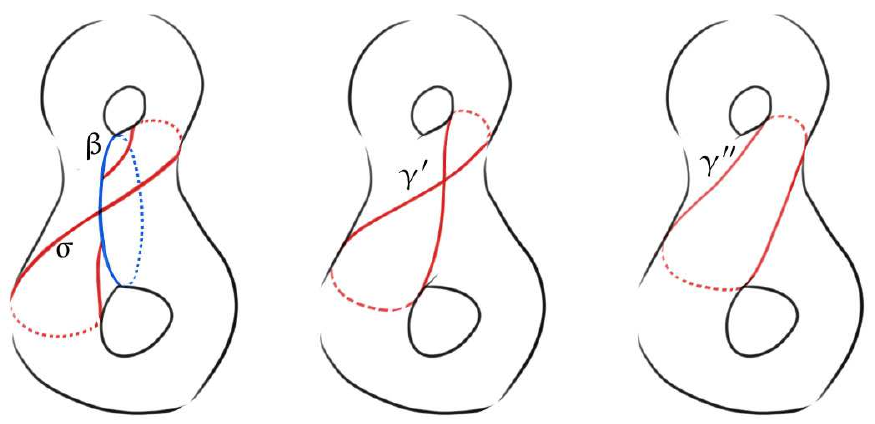}
\caption{An example of the how the arc $\sigma$ (left) is used to produce curves $\gamma'$ (center) and $\gamma''$ (right).}
\label{fig:collarcase2}
\end{figure}

\textbf{Case 3: $\exists$ type-A arc $b$ on $\gamma$ of length $\ell(b)>\frac{1}{2}\ell(\gamma)$.} Our argument here is similar to Case~2. Let $N$ again denote the number of intersection points $\beta\cap\gamma$. By assuming disjointness from Case~1, we may assume without loss of generality that there are $N-1$ consecutive type-B arcs $\sigma_1,\ldots,\sigma_{N-1}$ which, along with $b$, concatenate to form $\gamma$. The sum of the length of the following list of $N$ concatenated arcs
\[
\sigma_1*\sigma_2,\ldots,\sigma_i*\sigma_{i+1},\ldots,\sigma_{N-2}*\sigma_{N-1},\sigma_{N-1}*b,b*\sigma_1
\]
is equal to $2\ell(\gamma)$. By the pigeonhole principle, there must be at least one concatenated arc of the form $\sigma=\sigma_k*\sigma_{k+1}$ of length shorter than 
\[
\frac{2\ell(\gamma)-\ell(\sigma_{N-1}*b)-\ell(b)*\sigma_1)}{N-2}
<\frac{2\ell(\gamma)-2\ell(b)}{N-2}
<\frac{\ell(\gamma)}{N-2}.
\]
If $N>3$, the above inequality ensures that $\ell(\sigma)<\frac{\ell(\gamma)}{2}$. If $N=3$, then $\sigma$  must be $\sigma_1*\sigma_2$, and is the complementary arc to $b$. Hence $\sigma$ is again of length less than $\frac{\ell(\gamma)}{2}$. We may now run the latter half of the argument for Case 3 to obtain equation~\eqref{eq:collars11}.
\end{proof}

\begin{rmk}
Multiply both sides of Equation~\eqref{eq:collars11} by $(4 e^{\frac{\ell(\beta)}{4}} e^{\frac{\ell(\gamma)}{4}})^{-1}$ and we obtain
\begin{align*}
\sinh\left(\tfrac{1}{4}\ell(\beta)\right)
\cdot
\sinh\left(\tfrac{1}{4}\ell(\gamma)\right)> e^{-\frac{\ell(\beta)+\ell(\gamma)}{4}}.
\end{align*}
Our inequality is weaker than the ``sharp" inequality described in \cite[Conjecture 3.8]{lee2017collar}.
\end{rmk}

\subsection{Thurston-type ratio metrics}\label{sec:thurston}

Thurston showed in \cite[Theorem~3.1]{Thu98} that it is impossible for the simple marked length spectrum of one hyperbolic structure on a closed surface $S$ to dominate that of another. This non-domination ensures that Thurston's simple length ratio metric on $\mathit{Teich}(S)$ is positive.\medskip

Non-domination breaks down for bordered hyperbolic surfaces, and it is possible to map from a bordered surface to one where every geodesic is shorter \cite{papadopoulos2010shortening}. The way that Papadopoulous and Th{\'e}ret resolve this issue is to introduce orthogeodesic arcs into the collection of objects that one takes length ratios over. We show using McShane identities that the na\"{i}ve length ratio metric suffices provided that one fixes all boundary lengths.

\begin{thm}\label{thm:nondomborder}
Given marked hyperbolic surfaces $\Sigma_1,\Sigma_2\in\mathit{Teich}_{g,m}(L_1,\ldots,L_m)$ with fixed boundary lengths $L_1,\ldots,L_m\geq 0$ for $\alpha_1,\cdots,\alpha_m$. Then the marked simple geodesic spectrum for $\Sigma_1$ dominates the marked simple geodesic spectrum $\Sigma_2$ if and only if $\Sigma_1=\Sigma_2$.
\end{thm}

\begin{proof}
Assume without loss of generality that the simple length spectrum of $\Sigma_1$ dominates that of $\Sigma_2$. We first consider the case where at least one of the boundaries $\alpha=\alpha_i$ is strictly greater than $0$. Let $\mathcal{P}_{\alpha}$ be the set of the homotopy classes of pairs of pants and $\mathcal{P}_{\alpha}^\partial$ is a subset of $\mathcal{P}_{\alpha}$ which have two borders, say $\alpha$ and $\gamma$, as boundary components of $S_{g,m}$. The summands in the McShane identities for bordered surface \cite{mirz_simp,tan_zhang_cone}:
\begin{equation}
\label{equation:mirmc}
\begin{aligned}
&L_i=\sum_{(\beta,\gamma)\in \mathcal{P}_{\alpha}}   2\log  \frac{e^{\frac{\ell(\alpha)}{2}} + e^{\frac{\ell(\beta)+\ell(\gamma)}{2}}}{e^{\frac{-\ell(\alpha)}{2}} + e^{\frac{\ell(\beta)+\ell(\gamma)}{2}}}  + \sum_{(\beta,\gamma)\in \mathcal{P}_{\alpha}^{\partial}} \log  \frac{\cosh(\frac{\ell(\beta)}{2}) + \cosh(\frac{\ell(\alpha)+\ell(\gamma)}{2})}{\cosh(\frac{\ell(\beta)}{2}) + \cosh(\frac{\ell(\alpha)-\ell(\gamma)}{2})} 
\end{aligned}
\end{equation}
 have summands which are strictly decreasing with respect to increasing the lengths of (interior) simple closed geodesics. Since the simple length spectrum of $\Sigma_1$ dominates that of $\Sigma_2$, this forces each pair of corresponding summands in the McShane identities for $\Sigma_1$ and $\Sigma_2$ to be equal. This forces the length of multicurves $\ell^{\Sigma_1}(\beta)+\ell^{\Sigma_1}(\gamma)$ to be equal to $\ell^{\Sigma_2}(\beta)+\ell^{\Sigma_2}(\gamma)$ for each $(\beta, \gamma)\in \mathcal{P}_\alpha$, and $\ell^{\Sigma_1}(\beta)=\ell^{\Sigma_2}(\beta)$ for each $(\beta, \gamma)\in \mathcal{P}_\alpha^\partial$. For each $(\beta, \gamma)\in \mathcal{P}_\alpha$, domination then tells us that
\[
\ell^{\Sigma_1}(\beta)
=\ell^{\Sigma_2}(\beta)
\text{ and }
\ell^{\Sigma_1}(\gamma)
=\ell^{\Sigma_2}(\gamma).
\]
Therefore, the marked simple length spectra for $\Sigma_1$ and $\Sigma_2$ are equal and $\Sigma_1=\Sigma_2$.\medskip

The remaining case is where every boundary is length $0$ is classically due to Thurston \cite{Thu98}, but can also be demonstrated by applying the same arguments to McShane's identities for cusped surfaces \cite{mcshane_allcusps}.
\end{proof}

The above non-domination result immediately implies the following:

\begin{defn}[Thurston metric for bordered surfaces]
Let $\mathcal{C}_{g,m}$ be the set of simple closed curves up to homotopy on $S_{g,m}$. The non-negative real function $d_{\mathit{Th}}: \mathit{Teich}_{g,m}(L_1,\ldots,L_m)\times \mathit{Teich}_{g,m}(L_1,\ldots,L_m)\rightarrow \mathbb{R}_{\geq0}$ defined by 
\begin{align*}
d_{\mathit{Th}}(\Sigma_1,\Sigma_2)
:=\log \sup_{\bar{\gamma}\in\mathcal{C}_{g,m}}
\frac{\ell^{\Sigma_2}(\bar{\gamma})}{\ell^{\Sigma_1}(\bar{\gamma})},
\end{align*}
is a mapping class group invariant asymmetric Thurston-type length ratio metric on the Teichm\"uller space $\mathit{Teich}_{g,m}(L_1,\ldots,L_m)$ of surfaces with fixed boundary lengths $L_1,\ldots, L_m$.
\end{defn}

We are now interested in the spaces of geometric structures underlying the $\operatorname{PGL}_3(\mathbb{R})$-positive representations. By \cite{Mar10}, the space $\mathrm{Conv}^u(S_{g,m})$ of cusped strictly convex $\mathbb{RP}^2$ structures on $S_{g,m\geq 1}$ is homeomorphic to $\mathrm{Pos}_3^u(S_{g,m})$. 

Tholozan \cite{tholozan2017volume} showed that, for any strictly convex $\mathbb{RP}^2$ structure $\rho$ on a closed surface, it is always possible to find a hyperbolic structure $j$ such that the length spectrum of $j$ is uniformly smaller than that of $\rho$ (which should also works for the cusped strictly convex $\mathbb{RP}^2$ structure on $S_{g,m}$). Thus, the na\"{i}ve length ratio expression for the Thurston metric, when extended to the space $Conv_3^u(S_{g,m})$, results in a function which may be negative. To deal with this, we reverse engineer our McShane identities-based proof for the non-negativity of the length ratio metric (Theorem~\ref{thm:nondomborder}) and propose the following candidate for a metric on $Conv_3^u(S_{1,1})$:

\begin{align}\label{eq:gapmetric}
d_{\mathit{Gap}}(\Sigma_1,\Sigma_2)
:=
\log
\sup_{\gamma\in\subvec{\mathcal{C}}_{1,1}}
\left(
\frac{\log(1+e^{\ell_1^{\Sigma_2}(\gamma)+\tau^{\Sigma_2}(\gamma)})}
{\log(1+e^{\ell_1^{\Sigma_1}(\gamma)+\tau^{\Sigma_1}(\gamma)})}
\right)
\end{align}

To show that this is a well-defined function, we use the following comparison:

\begin{thm}[{\cite[Corollary~5.3]{benoist2001convexes}}, $\ell$ vs. $\ell_i$-length comparison]
\label{thm:simpcomp}
For any cusped strictly convex $\mathbb{RP}^2$ structure $\rho$, there exists $K_{\rho}>1$ such that for every non-trivial non-peripheral closed curve $\gamma$ on $S$, we have:
\begin{align*}
\ell_1(\gamma)\leq\ell(\gamma)\leq K_{\rho}\cdot\ell_1(\gamma).
\end{align*}
\end{thm}

\begin{rmk}
Although \cite[Corollary~5.3]{benoist2001convexes} is stated for closed surfaces, Proposition~\ref{thm:convprojregularity} allows us to extend this result to the cusped strictly convex $\mathbb{RP}^2$ surfaces. 
\end{rmk}

\begin{prop}
The function $d_{\mathit{Gap}}$ is well-defined.
\end{prop}

\begin{proof}
We need to show that the supremum in \eqref{eq:gapmetric} is bounded. If the supremum is realized by some simple geodesic $\gamma$, then obviously the gap metric is well-defined. If not, then there is a sequence of distinct geodesics $\{\gamma_k\}$ for which the expression in \eqref{eq:gapmetric} tends to the supremum. Then, by the discreteness of the simple length spectrum (Theorem~\ref{thm:discretespectra}) and the uniform boundedness of triple ratios (Theorem~\ref{thm:bounded}), showing that the supremum exists is equivalent the existence of the following supremum:
\begin{align}\label{eq:supexists}
\sup_{\gamma\in\subvec{\mathcal{C}}_{1,1}}
\frac{\ell_1^{\Sigma_2}(\gamma)}
{\ell_1^{\Sigma_1}(\gamma)}
\leq
K_{\Sigma_1}\cdot
\sup_{\gamma\in\subvec{\mathcal{C}}_{1,1}}
\frac{\ell^{\Sigma_2}(\gamma)}
{\ell^{\Sigma_1}(\gamma)},
\end{align}
where the $K_{\Sigma_1}$ in the right hand side is the coefficient in Theorem~\ref{thm:simpcomp}. However, we know from \cite[Theorem~2]{thurston2016rubber} that the twice of Hilbert lengths $\ell^{\Sigma_1}(\cdot)$ and $\ell^{\Sigma_2}(\cdot)$ extend continuously to the space of (compactly supported) measured laminations on $S_{1,1}$. In particular, the homogeneity of these length functions on multicurves means that they must be homogeneous over all of measured lamination space, and hence $\ell^{\Sigma_2}/\ell^{\Sigma_1}$ defines a continuous function on the space of (compactly supported) projective measured laminations. This is a compact codomain, and hence must be bounded above. Therefore, the left-hand side supremum in \eqref{eq:supexists} exists and $d_{\mathit{Gap}}$ is well-defined.
\end{proof}

\begin{thm}[Gap metric for $\mathrm{Conv}^u(S_{1,1})$]\label{thm:genthurstonmetric}
The non-negative function $d_{\mathit{Gap}}$ defines a mapping class group invariant aymmetric metric on $\mathrm{Conv}^u(S_{1,1})$. 
\end{thm}

\begin{proof}
It is clear that $d_{\mathit{Gap}}$ is mapping class group invariant and satisfies the triangle inequality. The McShane-type identity (Theorem~\ref{theorem:inequsl3s11}) tells us that the gap summands of for $\Sigma_1$ cannot dominate those for $\Sigma_2$, and this gives us the requisite non-negativity.\medskip

All that remains is to show that $d_{\mathit{Gap}}(\Sigma_1,\Sigma_2)=0$ if and only if $\Sigma_1=\Sigma_2$. One way is obvious. For the converse, assume that $d_{\mathit{Gap}}(\Sigma_1,\Sigma_2)=0$, then the McShane identity tells us that the corresponding gap summands must each be equal, and hence
\begin{align*}
\forall\gamma\in\xvec{\mathcal{C}}_{1,1},\quad
\ell_1^{\Sigma_1}(\gamma)+\tau^{\Sigma_1}(\gamma)
=\ell_1^{\Sigma_2}(\gamma)+\tau^{\Sigma_2}(\gamma).
\end{align*}
Consider the sequence of curves $\{\beta\gamma^k\}_{k\in\mathbb{Z}}$ obtained from applying Dehn-twists along $\gamma$ to a $\beta$ which once-intersects $\gamma$. The eigenvalues for the monodromy for two matrices are minimal/maximal when they are simultaneously diagonalizable, and hence we obtain the bounds:
\begin{align*}
k\ell_1(\gamma)+\log\lambda_3(\beta)-\log\lambda_1(\beta)
=k\ell_{1}(\gamma)-\ell(\beta)
&<\ell_{1}(\beta\gamma^k)\text{ and}\\
k\ell_{1}(\gamma)+\log\lambda_1(\beta)-\log\lambda_3(\beta)
=k\ell_{1}(\gamma)+\ell(\beta)
&>\ell_{1}(\beta\gamma^k).
\end{align*}
Hence we see that
\begin{align}\label{eq:genthurstonmetric}
\ell_1(\gamma)=\lim_{k\to\infty}\tfrac{1}{k}\ell_1(\beta\gamma^k),
\end{align}
which in turn implies that:
\begin{align*}
\frac{\ell_1^{\Sigma_2}(\gamma)}{\ell_1^{\Sigma_1}(\gamma)}
=\lim_{k\to\infty}
\frac{\frac{1}{k}\ell_1^{\Sigma_2}(\beta\gamma^k)}{\frac{1}{k}\ell_1^{\Sigma_1}(\beta\gamma^k)}
=\lim_{k\to\infty}
\frac{\frac{1}{k}
(\ell_1^{\Sigma_2}(\beta\gamma^k)+\tau^{\Sigma_2}(\beta\gamma^k))}
{\frac{1}{k}
(\ell_1^{\Sigma_1}(\beta\gamma^k)+\tau^{\Sigma_1}(\beta\gamma^k))}
=1.
\end{align*}
Therefore, the marked simple $\ell_1$ (and $\ell_2$) spectra for $\Sigma_i\in \mathrm{Conv}^u(S_{1,1})$ must be congruent. Which means that the simple marked $\lambda_1$ spectra for $\Sigma_1$ and $\Sigma_2$ must be equal. By \cite{bridgeman2017simple}, this means that $\Sigma_1=\Sigma_2$.
\end{proof}

\begin{prop}\label{thm:thurston11}
The restriction of the metric $d_{\mathit{Gap}}$ to the Fuchsian locus of $\mathrm{Conv}^u(S_{1,1})$ is precisely the Thurston metric $d_{\mathit{Th}}$.
\end{prop}

\begin{proof}
We first note that on the Fuchsian locus, triple ratios are all equal to $1$, and the simple root length $\ell_1(\gamma)$ of every geodesic $\gamma$ is equal to $\frac{1}{2}\ell(\gamma)$. Since $f(x)=\log(1+x)/\log(x)$ is a monotonically decreasing for $x>0$, whenever $\ell_1^{\Sigma_2}(\gamma)>\ell_1^{\Sigma_1}(\gamma)$, we have
\begin{align*}
\frac{\log(1+e^{\ell_1^{\Sigma_2}(\gamma)})}
{\log(1+e^{\ell_1^{\Sigma_1}(\gamma)})}
<
\frac{\log(e^{\ell_1^{\Sigma_2}(\gamma)})}
{\log(e^{\ell_1^{\Sigma_1}(\gamma)})}
=
\frac{\ell_1^{\Sigma_2}(\gamma)}{\ell_1^{\Sigma_1}(\gamma)}
=
\frac{\ell^{\Sigma_2}(\gamma)}{\ell^{\Sigma_1}(\gamma)}.
\end{align*}
Therefore $d_{\mathit{Gap}}\leq d_{\mathit{Th}}$. On the other hand, by equation~\eqref{eq:genthurstonmetric},
we have 
\begin{align*}
\lim_{k\to+\infty}
\frac{\log (1+e^{l_1^{\Sigma_2}(\beta\gamma^k)})}{\log (1+e^{l_1^{\Sigma_1}(\beta\gamma^k)})}
=\lim_{k\to+\infty}
\frac{\frac{1}{k}\ell_1^{\Sigma_2}(\beta\gamma^k)}{\frac{1}{k}\ell_1^{\Sigma_1}(\beta\gamma^k)}
=\frac{\ell_1^{\Sigma_2}(\gamma)}{\ell_1^{\Sigma_1}(\gamma)}
\end{align*}
 gives us the converse comparison $d_{\mathit{Gap}}\geq d_{\mathit{Th}}$, hence allowing us to conclude that the two metrics are equal on the Fuchsian locus.
\end{proof}

\subsubsection{Two generalizations to $S_{g,m}$}

We now turn to the space $\mathrm{Conv}^u(S_{g,m})$. We consider two possible generalizations. The first is equal to the Thurston metric on the Fuchsian slice and is conjecturally generalizable for $\mathrm{Pos}_n^u(S_{g,m})$ with $n\geq4$.

\begin{defn}[Pants-gap metric for $\mathrm{Conv}^u(S_{g,m})$]
\label{defn:pantsgapmetric}
For a cusped convex real projective surface $\Sigma\in \mathrm{Conv}^u(S_{g,m})$ with cusps $p_1,\ldots,p_m$, we define the \emph{pants gap function} $\mathit{PGap}^{\Sigma}(\beta,\gamma)$ for a boundary-parallel pair of pants $(\beta,\gamma)$ as the McShane identity summand corresponding to $(\beta,\gamma)$: 
\begin{align*}
PGap^{\Sigma}(\beta,\gamma)
:=
\left(1+e^{\phi_1(\beta,\gamma)}
\cdot e^{\frac{1}{2}\left(\tau(\gamma,\gamma_p)+\ell_1(\gamma)+\tau(\beta,\beta_p)+\ell_1(\beta)\right)}\right)^{-1}.
\end{align*}
We define the \emph{pants gap metric} as:
\begin{align*}
d_{\mathit{PGap}}(\Sigma_1,\Sigma_2)
:=
\log
\sup_{(\beta,\gamma)\in\subvec{\mathcal{P}}}
\frac{\log(\mathit{PGap}^{\Sigma_1}(\beta,\gamma))}
{\log(\mathit{PGap}^{\Sigma_2}(\beta,\gamma))},
\end{align*}
where $(\beta,\gamma)$ varies over the set $\xvec{\mathcal{P}}=\xvec{\mathcal{P}}_{p_1}\cup\ldots\cup\xvec{\mathcal{P}}_{p_m}$ of all boundary-parallel pairs of pants on $S_{g,m}$.
\end{defn}

\begin{defn}[Total gap metric for $\mathrm{Conv}^u(S_{g,m})$]
\label{defn:totalgapmetric}
For a cusped convex $\mathbb{RP}^2$ surface $\Sigma\in \mathrm{Conv}^u(S_{g,m})$ with cusps $p_1,\ldots,p_m$, we define the \emph{total gap function}  as:
\begin{align*}
\mathit{TGap}^{\Sigma}(\gamma):=
\frac{1}{m}\sum_{j=1}^m
\left(
\sum_{(\gamma,\gamma_{p})\in\subvec{\mathcal{H}}_{p_j}(\gamma)}
\frac{B_1(\gamma,\gamma_{p})}{1+e^{\ell_1(\gamma)+\tau(\gamma,\gamma_{p})}}
\right).
\end{align*}
We define the \emph{total gap metric} as:
\begin{align*}
d_{\mathit{TGap}}(\Sigma_1,\Sigma_2)
:=
\log
\sup_{\gamma\in\subvec{\mathcal{C}}_{g,m}}
\frac{\log(\mathit{TGap}^{\Sigma_2}(\gamma))}
{\log(\mathit{TGap}^{\Sigma_1}(\gamma))}.
\end{align*}
\end{defn}

\begin{rmk}
When $(g,m)=(1,1)$, both of these two metrics agree with the gap metric we defined for $1$-cusped convex real projective tori.
\end{rmk}

\begin{rmk}
The proof that the total gap metric is a mapping class group invariant asymmetric metrics on $\mathrm{Conv}^u(S_{g,m})$ is essentially the same as for the $\mathrm{Conv}^u(S_{1,1})$ case with the help of Lemma \ref{lemma:B1} and Lemma~\ref{thm:halfpantslbound}. The well-definedness of Definition \ref{defn:pantsgapmetric} requires Conjecture \ref{conjecture:phii}.
\end{rmk}

\begin{prop}
The restriction of the pants gap metric $d_{\mathit{PGap}}$ to the Fuchsian locus is equal to the classical Thurston metric.
\end{prop}

\begin{proof}
The proof is essentially identical to the proof of Proposition~\ref{thm:thurston11}, provided that one uses the following fact: 
\begin{align*}
d_{\mathit{Th}}(\Sigma_1,\Sigma_2)=
\log
\sup_{(\bar{\beta},\bar{\gamma})\in\mathcal{P}_p}
\frac{\ell^{\Sigma_2}(\bar{\beta},\bar{\gamma})}{\ell^{\Sigma_1}(\bar{\beta},\bar{\gamma})},
\end{align*}
where $\ell^{\Sigma_1}(\bar{\beta},\bar{\gamma}):=\ell^{\Sigma_1}(\bar{\beta})+ \ell^{\Sigma_1}(\bar{\gamma})$,
which comes from the fact that the projection of $\mathcal{P}_p$, regarded as a set of multicurves, in projective measured lamination space is dense.
\end{proof}

\begin{rmk}
It is unclear whether the restriction of the total gap metric $d_{\mathit{TGap}}$ to the Fuchsian locus is the Thurston metric, although it is fairly straight-forward to show that $d_{\mathit{TGap}}\geq d_{\mathit{Th}}$.
\end{rmk}

It is also possible to extend the pants gap metric over $\mathrm{Pos}_3^h(S_{g,m})$. \medskip

\begin{defn}[Pants gap metric for $\mathrm{Pos}_3^h(S_{g,m})$]
Let $\alpha_1,\ldots,\alpha_m$ be the boundary components of $S_{g,m}$. Let $\mathrm{Pos}_3^h(S_{g,m})(\mathbf{L})$ be the space of $\operatorname{PGL}_3(\mathbb{R})$-positive representations with fixed loxodromic boundary monodromy $\mathbf{L}$. Recall the notations: $\xvec{\mathcal{P}}_{\alpha}$ denotes the set of the homotopy classes of boundary-parallel pairs of pants containing $\alpha$, and $\xvec{\mathcal{P}}^{\partial}_{\alpha}$ denotes the set of the homotopy classes of boundary-parallel pairs of pants in $\xvec{\mathcal{P}}_{\alpha}$ which have two borders being boundary components of $S_{g,m}$.\medskip
\begin{itemize}
\item
For any $(\beta,\gamma)\in\xvec{\mathcal{P}}_{\alpha}\setminus\xvec{\mathcal{P}}^{\partial}_{\alpha}$ we set $\mathit{PGap}^{\Sigma}(\beta,\gamma)$ to be $\frac{1}{\ell_1(\alpha)}$ times the $i=1$ McShane identity summand in Theorem \ref{theorem:boundaryith}; 
\item
for any $(\beta,\gamma)\in\mathcal{P}^{\partial}_\alpha$, we set $\mathit{PGap}^{\Sigma}(\beta,\gamma)$ to be $\frac{1}{\ell_1(\alpha)}$ times the $i=1$ summand in Theorem \ref{theorem:boundaryith}.
\end{itemize}
The pants gap metric $d_{\mathit{PGap}}(\Sigma_1,\Sigma_2)$ on $\mathrm{Pos}_3^h(S_{g,m})(\mathbf{L})$ is defined as:
\begin{align*}
\log
\max_{j=1,\ldots,m}\left\{
\sup_{[\beta,\gamma]\in\subvec{\mathcal{P}}_{\alpha_j}\setminus\subvec{\mathcal{P}}_{\alpha_j}}
\frac{\log(\mathit{PGap}^{\Sigma_1}(\beta,\gamma))}
{\log(\mathit{PGap}^{\Sigma_2}(\beta,\gamma))},
\sup_{(\beta,\gamma)\in{\mathcal{P}}^{\partial}_{\alpha_j}}
\frac{\log(\mathit{PGap}^{\Sigma_1}(\beta,\gamma))}
{\log(\mathit{PGap}^{\Sigma_2}(\beta,\gamma))}
\right\}.
\end{align*}
\end{defn}

The proof that this is a well-defined metric is essentially the same as for the cusped case and we again require Conjecture \ref{conjecture:phii}.

\begin{rmk}
We expect Conjecture \ref{conjecture:phii} to be true. Provided that this can be demonstrated, it is possible to generalize the pants gap metric to define asymmetric metrics on the loxodromic-bordered positive representation variety of arbitrary rank. Moreover, the $(n-1)$ different McShane identities we obtain induce a $(n-1)$-dimensional positive ``quadrant'' of such metrics.
\end{rmk}

\clearpage

\section*{acknowledgements}
The authors wish to thank Yves Benoist, Vladimir Fock, Maxim Kontsevich, Fran\c{c}ois Labourie, Bob Penner, Andres Sambarino, Linhui Shen, Nicolas Tholozan, Yunhui Wu, Binbin Xu, Wenyuan Yang and Tengren Zhang for your active interest in our project and for helpful and stimulating discussions. We also wish to thank many, \emph{many} others in the community whose support has sustained us throughout this long project: your words have been truly encouraging. The bulk of this work was conducted at the Yau Mathematical Sciences Center (Tsinghua University), but key insights were also developed thanks to the generous hospitality of the Institut des Hautes \'{E}tudes Scientifique, the Institute for Mathematical Sciences (National University of Singapore) and the University of Melbourne.

\section*{Funding}
The first author wishes to state that this work was supported by the China Postdoctoral Science Foundation [grant numbers 2016M591154, 2017T100058]. The second author wishes to state that this work was supported by the China Postdoctoral Science Foundation 2018T110084 and FNR AFR bilateral grant COALAS 11802479-2. The second author also acknowledges support from U.S. National Science Foundation grants DMS-1107452, 1107263, 1107367 ``RNMS:GEometric structures And Representation varieties'' (the GEAR Network).

\appendix

\section{Fuchsian rigidity}
\label{sec:appendixrigidity}

\begin{prop}[Triple ratio rigidity for $n=3,4$]
For $n=3,4$, a positive representation $\rho\in \mathrm{Pos}_n(S_{g,m})$ is $n$-Fuchsian if and only if $\rho$ satisfies the triple ratio rigidity condition.
\end{prop}

\begin{proof}
We invoke Remark~\ref{rmk:rigidityproof}, and also lift our discussion to the universal cover to avoid dealing with different cases involving topologically distinct triangulations of the surface. Given any ideal triangulation $\mathcal{T}$, consider an ideal edge $\overline{xz}$ common to two ideal triangles $(x,y,z)$ and $(x,z,t)$ in $\widetilde{\mathcal{T}}$ as depicted in Figure~\ref{Figure:triple1}.

\begin{figure}[h!]
\includegraphics[scale=0.5]{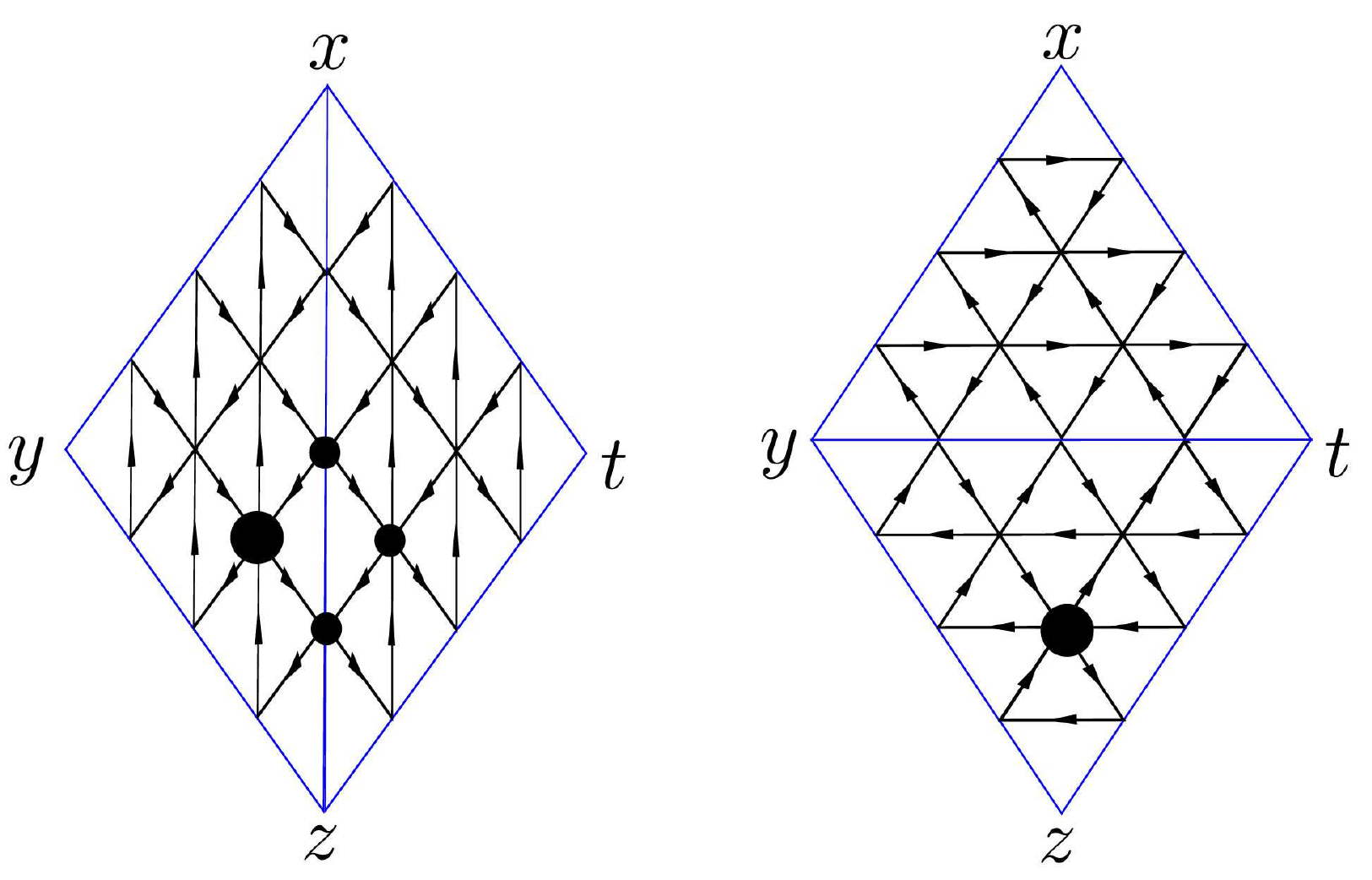}
\caption{flip at $\overline{xz}$}
\label{Figure:triple1}
\end{figure}
We compute $X'_{v_{t,y,z}^{1,1,n-2}}$ after flipping at edge $\overline{xz}$ via the cluster transformation: 
\begin{align*}
\frac{X'_{v^{t,y,z}_{1,1,n-2}}}{X_{v^{x,y,z}_{1,1,n-2}}}= \frac{1+ X_{v^{x,z}_{2,n-2}}+ X_{v^{x,z,t}_{1,n-2,1}} X_{v^{x,z}_{2,n-2}} + X_{v^{x,z}_{1,n-1}} X_{v^{x,z,t}_{1,n-2,1}} X_{v^{x,z}_{2,n-2}}}{1+ X_{v^{x,z}_{1,n-1}}+ X_{v^{x,y,z}_{1,1,n-2}} X_{v^{x,z}_{1,n-1}} + X_{v^{x,z}_{2,n-2}} X_{v^{x,y,z}_{1,1,n-2}} X_{v^{x,z}_{1,n-1}}}.
\end{align*}
By assumption, triple ratios are all equal to $1$, and the equation above tells us that
\begin{align*}
X_{v^{x,z}_{2,n-2}}=X_{v^{x,z}_{1,n-1}}.
\end{align*}
By symmetry, we also have
\begin{align*}
X_{v^{x,z}_{n-2,2}}=X_{v^{x,z}_{n-1,1}}.
\end{align*}
For $n=3,4$ there are at most $3$ coordinates along $\overline{xz}$, and hence must all be equal. Since this applies to any arbitrary edge, we see that $\rho$ is $n$-Fuchsian.
\end{proof}

\begin{prop}[Edge function rigidity for $n=3$]
\label{proposition:bul0}
For $n=3$, a positive representation $\rho\in \mathrm{Pos}_n(S_{g,m})$ is $n$-Fuchsian if and only if $\rho$ satisfies the edge function rigidity condition.
\end{prop}

\begin{proof}
We again invoke Remark~\ref{rmk:rigidityproof}, and we again work in the universal cover (see Figure~\ref{Figure:triple1}). By assumption, we have $X_{v_{1,2}^{x,z}}=X_{v_{2,1}^{x,z}}$. After flipping the edge $\overline{xz}$, we obtain
\begin{align*}
X'_{v_{1,2}^{x,y}}
=\frac{X_{v_{1,2}^{x,y}} X_{v_{1,1,1}^{x,y,z}} X_{v_{1,2}^{x,z}} (1+ X_{v_{1,2}^{x,z}})}
{1+ X_{v_{1,2}^{x,z}} + X_{v_{1,2}^{x,z}} X_{v_{1,1,1}^{x,y,z}}
+ X_{v_{1,2}^{x,z}} X_{v_{1,1,1}^{x,y,z}} X_{v_{1,2}^{x,z}}}=
\frac{X_{v_{1,2}^{x,y}} X_{v_{1,1,1}^{x,y,z}} X_{v_{1,2}^{x,z}}}{1+ X_{v_{1,1,1}^{x,y,z}} X_{v_{1,2}^{x,z}}},
\end{align*}
and 
\begin{align*}
X'_{v_{2,1}^{x,y}} 
= \frac{X_{v_{2,1}^{x,y}} X_{v_{2,1}^{x,z}}}{1+X_{v_{2,1}^{x,z}}}
= \frac{X_{v_{1,2}^{x,y}} X_{v_{1,2}^{x,z}}}{1+X_{v_{1,2}^{x,z}}},
\end{align*}
which satisfies $X'_{v_{1,2}^{x,y}}=X'_{v_{2,1}^{x,y}}$ by assumption. Solving for $X_{v_{1,1,1}^{x,y,z}}$ yields $X_{v_{1,1,1}^{x,y,z}}=1$ as desired.
\end{proof}

\section{More on the length spectrum of cusped strictly convex real projective surfaces}
\label{sec:lengthspecproj}

The proof of the Birman-Series theorem implies that the simple length spectrum of a cusped (or a closed) convex real projective surface necessarily has at least polynomial asymptotic growth rate. We now show something stronger, that it is asymptotically of order $N^{6g-6+2m}$.

\begin{prop}
\label{prop:polygrowth}
Given a cusped strictly convex real projective surface $\Sigma\in\mathrm{Conv}^u_3(S_{g,m})$, let $D(N,\Sigma)$ denote the number of (non-peripheral) simple closed geodesics on $\Sigma$ of length less than $N$. Then there exist constants $k_{\Sigma},K_{\Sigma}>0$ such that
\[
k_{\Sigma}N^{6g-6+2m}
\leq 
D(N,\Sigma)
\leq
K_{\Sigma}N^{6g-6+2m},\text{ for all sufficiently large }N.
\]
\end{prop}

\begin{proof}
We first endow $S_{g,m}$ with an auxiliary cusped hyperbolic metric $h_0$ which is smoothly compatible with $\Sigma$, and note that by McShane--Rivin \cite{MR95} and Rivin \cite{Riv01}, there exist constants $k_0, K_0>0$ such that \[
k_0N^{6g-6+2m}\leq D(N,(S_{g,m},h_0))\leq K_0N^{6g-6+2m}.
\]
Now, consider the Birman--Series set $\mathcal{BS}(\Sigma)$ of $\Sigma$ and note that its closure $\overline{\mathcal{BS}(\Sigma)}$ is compact thanks to Proposition~\ref{prop:compact}. \footnote{In fact, the Birman-Series set itself is compact because it is the intersection of a filtration of closed sets, however, that is not essential to our current proof.} The restriction of the unit tangent bundle of $\Sigma$ over the $\overline{\mathcal{BS}(\Sigma)}$ is therefore compact, and since this restricted unit tangent bundle is also a subset of the (whole) tangent bundle of $(S_{g,m},h_0)$, there is a uniform constant $c>0$ such that the Riemannian norm of every tangent vector in this restricted bundle is at least $c$.\medskip

Consider an arbitrary simple closed geodesic $\gamma$ on $\Sigma$. Note that $\gamma$ defines a smooth curve on $(S_{g,m},h_0)$, which then shortens to a unique simple closed geodesic $\gamma_0$ on $(S_{g,m},h_0)$. We have the following inequality:
\[
\ell^{\Sigma}(\gamma)\geq c\ell^{h_0}(\gamma)\geq c\ell^{h_0}(\gamma_0),
\]
where $\ell^{h_0}$ denotes the hyperbolic length of a curve on $(S_{g,m},h_0)$. Switching the roles of $\Sigma$ and $(S_{g,m},h_0)$ and using the Finsler norm instead of the Riemannian, we see that there exists a positive constant $C^{-1}>0$ such that
\[
\ell^{h_0}(\gamma_0)\geq C^{-1}\ell^{\Sigma}(\gamma_0)\geq C^{-1}\ell^{\Sigma}(\gamma).
\]
Therefore, we see that for every simple closed geodesic $\gamma$ on $\Sigma$ and $\gamma_0$ on $(S_{g,m},h_0)$ which are homotopic as topological curves on $S_{g,m}$, their lengths satisfy the following comparison:
\[
C\ell^{h_0}(\gamma_0)\geq \ell^{\Sigma}(\gamma)\geq c\ell^{h_0}(\gamma_0).
\]
This suffices to give the desired asymptotic growth rate of the simple length spectrum for $\Sigma$.
\end{proof}

\begin{rmk}
An immediate corollary of Proposition~\ref{prop:polygrowth} and Theorem~\ref{thm:simpcomp} is that the simple $i$-length spectra for cusped convex real projective surfaces also satisfy the same polynomial growth rate, albeit with possibly different constant coefficients.
\end{rmk}

\begin{prop}
All cusped convex real projective surfaces have discrete length spectra.
\end{prop}

\begin{proof}
Consider a cusped convex real projective surface $\Sigma$, and assume that there exists a sequence $\{\gamma_i\}_{i\in\mathbb{N}}$ of distinct closed geodesics on $\Sigma$ with bounded length. Fix an ideal triangulation $\mathcal{T}$ of $\Sigma$, we first observe that for $k\geq {N\choose 2}$, the number of homotopy classes of $k$-diagrams in $[J_k(N)]$ (Definition~\ref{defn:kdiag}) no longer increases, because $N$ geodesic segments can only has at most ${N\choose 2}$ intersections. Therefore, the number of geodesic segments that $\mathcal{T}$ cuts $\gamma_i$ into cannot remain bounded as $i\to\infty$. This in turn means that some of these segments must eventually penetrate arbitrarily far into the cusps of the ideal triangles in $\mathcal{T}$, and hence arbitrarily far into the cusps of $\Sigma$. Since cuspidal annular neighborhoods of $\Sigma$ cannot support non-peripheral geodesics, for $i$ sufficiently large, $\gamma_i$ must traverse from the $\epsilon$-thick part of $\Sigma$ (for some fixed $\epsilon$) arbitrarily far into the cusps. This contradicts the boundedness of the lengths of the $\gamma_i$.
\end{proof}

\begin{rmk}
Again, applying Theorem~\ref{thm:simpcomp} immediately tells us that cusped convex real projective surfaces have discrete $i$-length spectra.
\end{rmk}

\bibliography{Bibliography}   

\begin{thebibliography}{}
%
%

\bibitem[AC15]{adeboye2015area}
Ilesanmi Adeboye and Daryl Cooper, \emph{The area of convex projective surfaces and Fock--Goncharov coordinates}, Journal of Topology and Analysis (2018).

\bibitem[AMS04]{akiyoshi2004refinement}
Hirotaka Akiyoshi, Hideki Miyachi, and Makoto Sakuma, \emph{A refinement of McShane's identity for quasifuchsian punctured torus groups}, Contemporary
  Mathematics \textbf{355} (2004), 21--40.

\bibitem[AMS06]{MR2258748}
---, \emph{Variations of {M}c{S}hane's identity for punctured surface
  groups}, Spaces of {K}leinian groups, London Math. Soc. Lecture Note Ser.,
  vol. 329, Cambridge Univ. Press, Cambridge, 2006, pp.~151--185.

\bibitem[B60]{B60}
Jean--Paul Benz\'ecri, \emph{Sur les vari\'et\'es localement affines et localement projectives}, Bull. Soc. Math. France, \textbf{88} (1960), 229--332. 


\bibitem[Bas93]{basmajian1993orthogonal}
Ara Basmajian, \emph{The orthogonal spectrum of a hyperbolic manifold},
  American Journal of Mathematics \textbf{115} (1993), no.~5, 1139--1159.

\bibitem[BCL20]{bridgeman2017simple}
Martin Bridgeman, Richard Canary, and Fran{\c{c}}ois Labourie, \emph{Simple
  length rigidity for Hitchin representations}, Advances in Mathematics \textbf{360} (2020), 106901.


\bibitem[BCS18]{pressure}
Martin Bridgeman, Richard Canary, and Andr\'{e}s Sambarino, \emph{An introduction to pressure metrics for higher Teichm\"uller spaces}, Ergodic Theory Dynam. Systems \textbf{38} (2018), no.~6,
2001--2035.


\bibitem[BD14]{bonahon2014parameterizing}
Francis Bonahon and Guillaume Dreyer, \emph{Parameterizing Hitchin components},
  Duke Mathematical Journal \textbf{163} (2014), no.~15, 2935--2975.

\bibitem[BD17]{bonahon2017hitchin}
---, \emph{Hitchin characters and geodesic laminations}, Acta Mathematica
  \textbf{218} (2017), no.~2, 201--295.

\bibitem[Ben01]{benoist2001convexes}
Yves Benoist, \emph{Convexes divisibles I}, Algebraic groups and arithmetic. Tata Inst. Fund. Res. Stud. Math. \textbf{17} (2004), 339-374.


\bibitem[Ben03]{benoist2003convexes}
---, \emph{Convexes hyperboliques et fonctions quasisym{\'e}triques},
  Publications Math{\'e}matiques de l'IH{\'E}S \textbf{97} (2003), no.~1,
  181--237.

\bibitem[BH13]{benoist2013cubic}
Yves Benoist and Dominique Hulin, \emph{Cubic differentials and finite volume
  convex projective surfaces}, Geometry \& Topology \textbf{17} (2013), no.~1,
  595--620.

\bibitem[BH14]{benoist2014cubic}
---, \emph{Cubic differentials and hyperbolic convex sets}, Journal of
  Differential Geometry \textbf{98} (2014), no.~1, 1--19.

\bibitem[BS85]{birman_series}
Joan~S. Birman and Caroline Series, \emph{Geodesics with bounded intersection
  number on surfaces are sparsely distributed}, Topology \textbf{24} (1985),
  no.~2, 217--225.


\bibitem[Bon96]{MR1413855}
Francis Bonahon, \emph{Shearing hyperbolic surfaces, bending pleated surfaces
  and {T}hurston's symplectic form}, Ann. Fac. Sci. Toulouse Math. (6)
  \textbf{5} (1996), no.~2, 233--297.

\bibitem[Bow96]{Bow96}
Brian~H Bowditch, \emph{A proof of McShane's identity via Markoff triples}, B. Lond. Math. Soc., \textbf{28} (1996), no. 1, 73--78.

\bibitem[Bow97]{bowditch1997variation}
Brian~H Bowditch, \emph{A variation of McShane's identity for once-punctured
  torus bundles}, Topology \textbf{36} (1997), no.~2, 325--334.

\bibitem[Bow98]{markofftriples}
B.~H. Bowditch, \emph{Markoff triples and quasi-{F}uchsian groups}, Proc.
  London Math. Soc. (3) \textbf{77} (1998), no.~3, 697--736.

\bibitem[Bri11]{orthospectra}
Martin Bridgeman, \emph{Orthospectra of geodesic laminations and dilogarithm
  identities on moduli space}, Geom. Topol. \textbf{15} (2011), no.~2,
  707--733.
  

\bibitem[BK10]{bridgeman2010hyperbolic}
Martin Bridgeman and Jeremy Kahn, \emph{Hyperbolic volume of manifolds with
  geodesic boundary and orthospectra}, Geometric and Functional Analysis
  \textbf{20} (2010), no.~5, 1210--1230.



\bibitem[BK09]{BK09}
Stephen Buckley, and Simon Kokkendorff. \textit{Comparing the Floyd and ideal boundaries of a metric space}, Transactions of the American Mathematical Society \textbf{361} (2009), vol.~2, 715--734.



\bibitem[CG93]{choi1993convex}
Suhyoung Choi and William~M Goldman, \emph{Convex real projective structures on
  closed surfaces are closed}, Proceedings of the American Mathematical Society
  \textbf{118} (1993), no.~2, 657--661.

\bibitem[CY77]{cheng1977regularity}
Shiu-Yuen Cheng and Shing-Tung Yau, \emph{On the regularity of the
  monge-amp{\`e}re equation det $(\partial^2 u/\partial x_i\partial x_j)= f (x,
  u)$}, Communications on Pure and Applied Mathematics \textbf{30} (1977),
  no.~1, 41--68.

\bibitem[FG06]{FG06}
Vladimir Fock and Alexander Goncharov, \emph{Moduli spaces of local systems and
  higher Teichm{\"u}ller theory}, Publications Math{\'e}matiques de l'Institut
  des Hautes {\'E}tudes Scientifiques \textbf{103} (2006), no.~1, 1--211.

\bibitem[FG07]{FG07}
Vladimir~V Fock and Alexander~B Goncharov, \emph{Moduli spaces of convex
  projective structures on surfaces}, Advances in Mathematics \textbf{208}
  (2007), no.~1, 249--273.

\bibitem[Flo80]{Flo80}
William J. Floyd, \textit{Group completions and limit sets of Kleinian groups}. Inventiones mathematicae \textbf{57} (1980), no.~3, 205--218.


\bibitem[FP16]{FP16}
Federica Fanoni and Maria Beatrice Pozzetti, \emph{Basmajian-type inequalities for maximal representations}, to appear in Journal of Differential Geometry, arXiv:1611.00286 (2016).

\bibitem[GK41]{GK41}
Feliks R. Gantmacher and Mark G. Krein, \textit{Oscillation matrices and kernels and small vibrations of mechanical systems}, Revised edition of the 1941 Russian original. 

\bibitem[G90]{G1990convex}
William~M Goldman, \emph{Convex real projective structures on compact surfaces}, Journal of Differential Geometry
  \textbf{31} (1990), no.~3, 791--845.

\bibitem[GH78]{GH78}
Phillip Griffiths and Joseph Harris, \emph{Principles of algebraic geometry}, John Wiley \& Sons, 1978.

\bibitem[GS15]{GS15}
Alexander Goncharov and Linhui Shen, \emph{Geometry of canonical bases and
  mirror symmetry}, Inventiones mathematicae \textbf{202} (2015), no.~2,
  487--633.

\bibitem[Guo13]{guo2013characterizations}
Ren Guo, \emph{Characterizations of hyperbolic geometry among Hilbert
  geometries: a survey}, Handbook of Hilbert Geometry. IRMA Lectures in
  Mathematics and Theoretical Physics \textbf{22} (2013), 147--158.

\bibitem[He19]{He19}
Yan Mary He, \emph{Identities for hyperconvex Anosov representations}, arxiv preprint, arXiv:1909.11595.

\bibitem[Hit92]{Hit92}
Nigel J. Hitchin, \emph{Lie groups and Teichm\"ller space}, Topology \textbf{31}(1992), no.~3, 449-473.

\bibitem[HN17]{HN17}
Yi Huang and Paul Norbury, \textit{Simple geodesics and Markoff quads}, Geometriae Dedicata \textbf{186} (2017), no.~1, 113-148.

\bibitem[HPZ19]{HPZ19}
Yi Huang, Robert C. Penner, and Anton M. Zeitlin, \textit{Super McShane identity}, arXiv preprint, arXiv:1907.09978 (2019).


\bibitem[HSY18]{HSY18}
Hengnan Hu, Ser Peow Tan, and Ying Zhang, \textit{Polynomial automorphisms of $\mathbb {C}^n$ preserving the Markoff--Hurwitz polynomial}. Geometriae Dedicata \textbf{192}(2018), no.~1, 207-243.


\bibitem[Hua14]{huang_thesis}
Yi~Huang, \emph{Moduli spaces of surfaces}, Ph.D. thesis, The University of Melbourne, June 2014.

\bibitem[Hua15]{huangclosed}
---, \emph{A McShane-type identity for closed surfaces}, Nagoya Mathematical Journal \textbf{219} (2015), 65--86.

\bibitem[Hua18]{huang2018mcshane}
---, \emph{McShane-type identities for quasifuchsian representations of nonorientable surfaces}, arXiv preprint, arXiv:1802.03069 (2018).


\bibitem[Kim19]{kim2019}
Inkang Kim, \emph{Primitive stable representations in higher rank semisimple Lie groups}, arXiv preprint arXiv:1504.08056v4 (2019).

\bibitem[KT98]{KT98}
A. Knuston and T. Tao, \emph{The honeycomb model of $\operatorname{GL}_n(\mathbb{C})$ tensor products I: Proof of the saturation conjecture}, J. Amer. Math. Soc. \textbf{12} (1999), 1055--1090.


\bibitem[Lab06]{Lab06}
Fran{\c{c}}ois Labourie, \emph{Anosov flows, surface groups and curves in projective space}, Inventiones mathematicae, \textbf{165} (2006), no.~1, 51--114.


\bibitem[Lab07]{Lab07}
Fran{\c{c}}ois Labourie, \emph{Cross ratios, surface groups, $\operatorname{PSL}(n,\mathbb{R})$ and
  diffeomorphisms of the circle}, Publications math{\'e}matiques \textbf{106}
  (2007), no.~1, 139--213.

\bibitem[LM09]{LM09}
Fran{\c{c}}ois Labourie and Gregory McShane, \emph{Cross ratios and identities
  for higher {T}eichm\"uller-{T}hurston theory}, Duke Math. J. \textbf{149}
  (2009), no.~2, 279--345.  

\bibitem[Lu94]{Lu94}
G. Lusztig, \emph{Total positivity in reductive groups. InLie theory and geometry}, \textbf{123} (1994) Progr. Math., pages 531--568. 


\bibitem[LS13]{lee2013variation}
Donghi Lee and Makoto Sakuma, \emph{A variation of McShane's identity for
  2--bridge links}, Geometry \& Topology \textbf{17} (2013), no.~4, 2061--2101.

\bibitem[LT11]{luo_tan}
Feng Luo and Ser~Peow Tan, \emph{A dilogarithm identity on moduli spaces of curves}, J. Differential. Geom. \textbf{97} (2014), no. 2, 255--274.

\bibitem[LZ17]{lee2017collar}
Gye-Seon Lee and Tengren Zhang, \emph{Collar lemma for Hitchin representations}, Geometry \& Topology \textbf{21} (2017), no. 4, 2243--2280.


\bibitem[Mar10]{Mar10}
Ludovic Marquis, \emph{Espaces des modules des surfaces projectives proprement convexes de volume fini}, Geometry and Topology \textbf{14} (2010), no.~4, 2103--2149.


\bibitem[Mar12]{Mar12}
Ludovic Marquis, \emph{Surface projective convexe de volume fini}, Ann. Inst.
  Fourier (Grenoble) \textbf{62} (2012), no.~1, 325--392.

\bibitem[McS91]{mcshane_thesis}
Greg McShane, \emph{A remarkable identity for lengths of curves}, Ph.D. thesis,
  may 1991.

\bibitem[McS98]{mcshane_allcusps}
---, \emph{Simple geodesics and a series constant over {T}eichm\"uller
  space}, Invent. Math. \textbf{132} (1998), no.~3, 607--632.

\bibitem[MR95]{MR95}
Greg McShane and Igor Rivin, \textit{Geometry of geodesics and a norm on homology}, International Mathematics Research Notices \textbf{2} (1995), 61-69.

\bibitem[Mir07a]{mirz_simp}
Maryam Mirzakhani, \emph{Simple geodesics and {W}eil--{P}etersson volumes of
  moduli spaces of bordered {R}iemann surfaces}, Invent. Math. \textbf{167}
  (2007), no.~1, 179--222.
  
\bibitem[Mir07b]{mirz_witten}
Maryam Mirzakhani, \emph{Weil-Petersson volumes and intersection theory on the moduli space of curves}, 
J. Amer. Math. Soc. \textbf{20}
  (2007), no.~1, 1-23.

\bibitem[Miy05]{miyachi2005limit}
Hideki Miyachi, \emph{The limit sets of quasifuchsian punctured surface groups
  and the Teichm{\"u}ller distances}, Kodai Mathematical Journal \textbf{28}
  (2005), no.~2, 301--309.

\bibitem[Nor08]{MR2399656}
Paul Norbury, \emph{Lengths of geodesics on non-orientable hyperbolic
  surfaces}, Geom. Dedicata \textbf{134} (2008), 153--176.

\bibitem[Pen87]{pennercoords}
R.~C. Penner, \emph{The decorated {T}eichm\"uller space of punctured surfaces},
  Comm. Math. Phys. \textbf{113} (1987), no.~2, 299--339.

\bibitem[PS17]{potrie2017eigenvalues}
Rafael Potrie and Andr{\'e}s Sambarino, \emph{Eigenvalues and entropy of a
  Hitchin representation}, Inventiones mathematicae \textbf{209} (2017), no.~3,
  885--925.

\bibitem[PT10]{papadopoulos2010shortening}
Athanase Papadopoulos and Guillaume Th{\'e}ret, \emph{Shortening all the simple
  closed geodesics on surfaces with boundary}, Proceedings of the American
  Mathematical Society \textbf{138} (2010), no.~5, 1775--1784.

\bibitem[Riv01]{Riv01}
Igor Rivin, \textit{Simple curves on surfaces}, Geometriae Dedicata, \textbf{87} (2001): no.~1-3, 345-360.

\bibitem[Sch33]{Sch33}
Isaac J. Schoenberg, \textit{Convex domains and linear combinations of continuous functions}, Bulletin of the American Mathematical Society, \textbf{39} (1933), no.~4, 273--280.

\bibitem[Sun20a]{Sun20a}
Zhe Sun, \emph{Rank $n$ swapping algebra for $\operatorname{PGL}_n$ Fock--Goncharov $\mathcal{X}$ moduli space}, Math. Ann. (2020). https://doi.org/10.1007/s00208-020-02025-1

\bibitem[Sun20b]{Sun20b}
Zhe Sun, \emph{Volume of the moduli space of unmarked bounded positive convex $\mathbb{RP}^2$ structures}, arxiv preprint, arXiv:2001.01295 (2020).

\bibitem[SZ17]{SZ17}
Zhe Sun and Tengren Zhang, \emph{The Goldman symplectic form on the $\operatorname{PGL}(V)$-Hitchin component}, arxiv preprint, arXiv:1709.03589 (2017).


\bibitem[SWZ20]{sun2017flows}
Zhe Sun, Anna Wienhard, and Tengren Zhang, \emph{Flows on the $\operatorname{PGL}(V)$-Hitchin
  component}, Geom. Funct. Anal. \textbf{30} (2020), 588-692.


\bibitem[Tho17]{tholozan2017volume}
Nicolas Tholozan, \emph{Volume entropy of Hilbert metrics and length spectrum
  of Hitchin representations into $\operatorname{PSL}(3,\mathbb{R})$}, Duke
  Mathematical Journal \textbf{166} (2017), no.~7, 1377--1403.

\bibitem[Thu98]{Thu98}
William~P Thurston, \emph{Minimal stretch maps between hyperbolic surfaces}, arXiv preprint, math/9801039 (1998).

\bibitem[Thu16]{thurston2016rubber}
Dylan~P Thurston, \emph{From rubber bands to rational maps: a research report},
  Research in the Mathematical Sciences \textbf{3} (2016), no.~1, 15.

\bibitem[TWZ06]{tan_zhang_cone}
Ser~Peow Tan, Yan~Loi Wong, and Ying Zhang, \emph{Generalizations of
  {M}c{S}hane's identity to hyperbolic cone-surfaces}, J. Differential Geom.
  \textbf{72} (2006), no.~1, 73--112.

\bibitem[TWZ08]{tan2008mcshane}
---, \emph{McShane's identity for classical Schottky groups}, Pacific
  Journal of Mathematics \textbf{237} (2008), no.~1, 183--200.

\bibitem[VY17]{VY17}
Nicholas Vlamis and Andrew Yarmola, \emph{Basmajian's identity in higher Teichm\"uller--Thurston theory}, Journal of Topology \textbf{10} (2017), no. 3, 744--764.

\bibitem[W18]{W18}
Anna Wienhard, \emph{An invitation to higher Teichm\"uller theory}, Proceedings of the ICM 2018, 2018.






\end{thebibliography}


\end{document}